\documentclass{amsart}
\usepackage[a4paper, margin=2.3cm]{geometry}

\usepackage[utf8]{inputenc}
\usepackage{amsmath}
\usepackage{amssymb}
\usepackage{amsthm}
\usepackage{graphicx}
\usepackage[dvipsnames]{xcolor}
\usepackage{enumitem}
\usepackage{tikz}
\usetikzlibrary{arrows.meta, positioning, calc}
\usepackage{tikz-cd}
\usepackage{subcaption}
\usepackage{csquotes}
\usepackage[T1]{fontenc}
\usepackage[hidelinks]{hyperref}
\hypersetup{
    pdftitle={From Whole-Space Theory to Boundary Value Problems for Mixed-Order PDEs},
    pdfauthor={Guillaume Neuttiens and Jonas Sauer},
    pdfsubject={Mixed-order partial differential equations under general boundary conditions},
    pdfkeywords={time-periodic partial differential equations, mixed-order systems, boundary value problems, Newton polygons, dynamic boundary conditions}
}
\usepackage{geometry}
\usepackage{comment}
\usepackage[colorinlistoftodos,prependcaption,color=yellow,textsize=tiny,textwidth=2cm]{todonotes}
\usepackage{epigraph}
\usepackage{bbm}

\makeatletter

\def\l@section{\@tocline{1}{2pt}{1pc}{}{\bfseries}}

\makeatother
\newcommand{\eps}{\varepsilon}

\newcommand{\dd}{\,\mathrm{d}}

\renewcommand{\Re}{\mathop{\mathrm{Re}}}
\renewcommand{\Im}{\mathop{\mathrm{Im}}}

\newcommand{\calb}{{\mathcal B}}

\newcommand{\cald}{{\mathcal D}}
\newcommand{\cale}{{\mathcal E}}
\newcommand{\calf}{{\mathcal F}}
\newcommand{\calg}{{\mathcal G}}
\newcommand{\calh}{{\mathcal H}}

\newcommand{\calk}{{\mathcal K}}
\newcommand{\call}{{\mathcal L}}

\newcommand{\caln}{{\mathcal N}}
\newcommand{\calo}{{\mathcal O}}

\newcommand{\cals}{{\mathcal S}}

\newcommand{\calu}{{\mathcal U}}
\newcommand{\calv}{{\mathcal V}}

\newcommand{\R}{\mathbb{R}}

\newcommand{\Z}{\mathbb{Z}}
 \newcommand{\C}{\mathbb{C}}
\newcommand{\N}{\mathbb{N}}

\newcommand{\E}{\mathbb{E}}

\DeclareMathOperator{\Ext}{Ext}

\DeclareMathOperator{\T}{T}

\DeclareMathOperator{\id}{\mathsf{id}}

\DeclareMathOperator{\supp}{supp}

\DeclareMathOperator{\diag}{diag}

\DeclareMathOperator{\ran}{ran}

\DeclareMathOperator{\sgn}{sgn}
\DeclareMathOperator{\ord}{ord}
\DeclareMathOperator{\conv}{conv}
\DeclareMathOperator{\ad}{ad}

\newcommand{\opA}{\mathsf{A}}
\newcommand{\opB}{\mathsf{B}}

\newcommand{\opL}{\mathsf{L}}

\newcommand{\set}[1]{\ensuremath{\{#1\}}}

\newcommand{\seqK}[1]{\ensuremath{(#1)_{k\in\N}}}

\newcommand{\op}{\mathsf{op}}
\newcommand{\Tr}{\mathsf{Tr}}
\newcommand{\ic}{\mathsf{i}}

\newcommand{\opnorm}[1]{{\lvert\kern-0.25ex\lvert\kern-0.25ex\lvert #1 \rvert\kern-0.25ex\rvert\kern-0.25ex\rvert}}

\newcommand{\CR}[1]{C^{#1}}

\newcommand{\CRi}{\CR \infty}
\newcommand{\CRci}{\CR \infty_c}

\newcommand{\newCCtr}[2][d]{
\newcounter{#2}\setcounter{#2}{0}
\expandafter\xdef\csname kyedtheconst#2\endcsname{#1}
}
\newcommand{\Cc}[2][nolabel]{
\stepcounter{#2}
\expandafter\ensuremath{\csname kyedtheconst#2\endcsname_{\arabic{#2}}}
\ifthenelse{\equal{#1}{nolabel}}
{}
{\expandafter\xdef\csname kyedconst#1\endcsname
{\expandafter\ensuremath{\csname kyedtheconst#2\endcsname_{\arabic{#2}}}}}
}
\newcommand{\Ccn}[2][nolabel]{
\expandafter\ensuremath{\csname kyedtheconst#2\endcsname}
\ifthenelse{\equal{#1}{nolabel}}
{}
{\expandafter\xdef\csname kyedconst#1\endcsname
{\expandafter\ensuremath{\csname kyedtheconst#2\endcsname}}}
}

\newcommand{\Cclast}[1]{
\expandafter\ensuremath{\csname kyedtheconst#1\endcsname_{\arabic{#1}}}
}

\newcommand{\Ccllast}[1]{
\addtocounter{#1}{-1}
\expandafter\ensuremath{\csname kyedtheconst#1\endcsname_{\arabic{#1}}}
\addtocounter{#1}{1}
}
\newcommand{\const}[1]{
\expandafter{\ifcsname kyedconst#1\endcsname
  \csname kyedconst#1\endcsname
\else
  \errmessage{Undefined Kyedconstant #1.}%
\fi}
}

\newcommand{\Comp}{\mathsf{Comp}}
\newcommand{\opC}{\mathsf{C}}
\newcommand{\opE}{\mathsf{E}}
\newcommand{\opG}{\mathsf{G}}
\newcommand{\opT}{\mathsf{T}}
\newcommand{\opV}{\mathsf{V}}
\theoremstyle{plain}
\newtheorem{theorem}{Theorem}[section]
\newtheorem{proposition}[theorem]{Proposition}
\newtheorem{lemma}[theorem]{Lemma}
\newtheorem{corollary}[theorem]{Corollary}
\theoremstyle{definition}
\newtheorem{definition}[theorem]{Definition}
\newtheorem{example}[theorem]{Example}

\newtheorem{remark}[theorem]{Remark}

\title[Boundary Value Problems for Mixed-Order PDEs]{From Whole-Space Theory to Boundary Value Problems for Mixed-Order PDEs}
\author{Guillaume~Neuttiens}
\address{Guillaume~Neuttiens\newline Friedrich-Schiller-Universit\"at Jena \newline Inselplatz 5, 07743 Jena, Germany}
\email{guillaume.neuttiens@uni-jena.de}
\author{Jonas~Sauer}
\address{Jonas~Sauer\newline Friedrich-Schiller-Universit\"at Jena \newline Inselplatz 5, 07743 Jena, Germany}
\email{jonas.sauer@uni-jena.de}

\begin{document}


\begin{abstract}
We develop a general principle for passing from the analysis of partial differential operators on the whole space to boundary value problems on the half-space.
The reduction separates the boundary problem into restricted invertibility of the interior operator and a complementing problem formulated entirely on the trace space.
For a broad class of operators admitting suitable factorizations, the restricted invertibility is governed by the whole-space theory of the factors, while the complementing problem determines both well-posedness and the admissible combinations of interior and boundary data.
In this way, boundary value theory is reduced to whole-space operator theory together with an explicit analysis on the trace space.

We realize this principle for time-periodic partial differential equations and systems of mixed order, developing the required Newton-polygon function spaces, interpolation theory, and $L^p$-trace theory.
For mixed-order differential operators under general boundary conditions, well-posedness and the admissible data space are determined by an explicit complemented boundary matrix.
This yields new $L^p$-well-posedness results for Cahn--Hilliard equations with dynamic boundary conditions and for Cahn--Hilliard--Gurtin systems.
Most notably, we obtain the first maximal $L^p$-regularity result for a genuinely mixed-order interior system coupled to dynamic boundary conditions carrying a distinct mixed-order structure.
\end{abstract}

\maketitle

\tableofcontents

\section{Introduction}
The systematic analysis of boundary value problems of the form
\begin{align*}
    \left\{
    \begin{array}{rcl}
        \opA u &=& f,\\
        B_i u &=& g_i,\qquad i\in\{1,\ldots,m\},
    \end{array}
    \right.
\end{align*}
goes back to the seminal work of Agmon, Douglis, and Nirenberg \cite{agmon1959estimates,agmon1964estimates}.
Their theory established that broad classes of elliptic boundary value problems can be treated through structural conditions on the interior and boundary operators, rather than by separate arguments for individual equations.
Parameter-elliptic and parabolic theories developed this principle further and provide powerful $L^p$-results whenever the underlying problem is governed by a single (quasi)-homogeneous scaling \cite{AgV63,DHP07,PrS16}.

In the present work, we pursue this principle from a different direction.
We ask when a well-posedness theory for an operator on the whole space can be converted into a boundary value theory on the half-space.
Our starting point is an abstract reduction which separates two aspects of the boundary problem: restricted invertibility of the interior operator and compatibility of the trace data.
After identifying the appropriate trace structure, the former amounts to invertibility of the interior operator on the subspace of vanishing traces, while the latter is expressed by a complementing condition on the trace space.
The decisive point is that, for large classes of differential operators, a suitable factorization allows the required restricted invertibility to be verified by whole-space methods, making the reduction robust across rather different boundary value problems.

This viewpoint is particularly effective for mixed-order equations.
Such operators exhibit several competing homogeneities, naturally encoded by Newton polygons, and admit a well-developed whole-space theory \cite{denk2013general,denk1998newton,gindikin2012method}.
By contrast, boundary value problems with mixed-order structure have so far been treated only in particular settings \cite{DeV02,denk2008parabolic,NeS25}, and a comparable general $L^p$-boundary theory has largely been unavailable.
In particular, in the mixed-order setting the pointwise Lopatinski\u{\i}--Shapiro condition, although still necessary at the symbolic level, no longer provides sufficient quantitative control for $L^p$-well-posedness on the natural product of data spaces.

We show that the missing quantitative information can be captured by an appropriate Newton-polygon ellipticity condition on the boundary operator.
Combined with the abstract reduction and a factorization separating the boundary-sensitive from the boundary-insensitive part of the interior operator, this yields a general $L^p$-theory for mixed-order equations under general boundary conditions.

The framework has three consequences which are difficult to obtain uniformly by classical methods.
First, it identifies the admissible data space, including nontrivial compatibility conditions, intrinsically.
Second, it treats boundary operators containing temporal derivatives and therefore covers dynamic boundary conditions, including problems in which a genuinely mixed-order interior system is coupled to an independent mixed-order structure on the boundary.
Third, the abstract complementing condition becomes an explicit condition on a complemented boundary matrix.
The theory rests on three main ingredients.
\begin{enumerate}
\item \emph{An abstract reduction of boundary value problems.}
We develop an operator-theoretic framework which separates the whole-space analysis of the interior operator from the additional structure created by the boundary.
Its central result is Theorem \ref{js007}.
Once an abstract trace operator is available and the interior operator is invertible on the kernel of this trace operator, well-posedness of the boundary value problem is characterized by an abstract complementing boundary condition.
In this way, a whole-space isomorphism result is converted into a half-space boundary value theory in which the admissible data are themselves part of the conclusion.
Importantly, the compatibility between interior and boundary data is not prescribed in advance.
It arises as part of the reduction itself: the abstract theory identifies the admissible data space from the trace structure and restricted invertibility of the interior operator.
In particular, this space need not be the full product of the natural interior and boundary data spaces.

More precisely, given an interior operator and a boundary operator $\opB$, we construct a complementary operator $\opC$ such that well-posedness reduces to the injectivity of the combined operator $(\opB,\opC)$.
Its range identifies exactly which combinations of interior and boundary data are compatible.
In particular, all data in the natural product space are admissible precisely when $(\opB,\opC)$ is an isomorphism, see Lemma \ref{js012}.
Thus, the abstract theory reduces the boundary problem to restricted invertibility of the interior operator together with a problem entirely formulated on the trace space.

    \vspace{.2cm}
\item \emph{Admissible factorizations.}
For the differential operators under consideration, the restricted invertibility required by the abstract theorem is generally not available for the full operator $\opA$ itself.
We overcome this by introducing the class of admissible operators, which admit a suitable factorization $\opA=\opA^-\opA^+$.
Here, the factor $\opA^-$ may be viewed as the boundary-insensitive part of the operator: its inverse preserves the relevant one-sided supports, so that one may extend a function to the whole space, apply the whole-space inverse, and restrict back to the half-space independently of the chosen extension.
The complementary factor $\opA^+$ contains precisely the part of the problem which sees the boundary.
It is invertible only on the subspace of vanishing traces and is coupled with the trace and boundary operators through the abstract theory.

Combining this notion of admissibility with Theorem \ref{js007} yields Theorem \ref{thm: abstract_complementing_condition}.
It shows that, for an admissible operator, well-posedness of the boundary value problem is characterized by an abstract complementing condition on the trace space in which the interior operator enters only through $\opA^+$.

It remains to turn admissibility into verifiable conditions for concrete differential operators.
Following an idea of Arkeryd \cite{Ark67} and extending the approach of \cite{kyed2019time,NeS25}, the required one-sided support properties can be obtained by means of a Paley--Wiener type theorem from suitable estimates for the factors.
A substantial part of the analysis is therefore devoted to constructing the appropriate factorizations and establishing the estimates needed to apply this argument.
For mixed-order operators, these estimates are naturally formulated in terms of $\mathcal N$-ellipticity in the sense of Definition \ref{def: Newton_ellipticity}.
This leads from the abstract notion of admissibility to concrete criteria that can be checked directly for the differential operators and boundary value problems considered below.

    \vspace{.2cm}
    \item \emph{Concrete verification for mixed-order problems.}
    To turn the preceding abstract theory into a workable criterion for mixed-order differential operators, we develop the function-space, trace, and ellipticity machinery needed to verify its hypotheses.
    We begin by extending the classical theory of Newton polygon spaces through a unified scale of potential spaces generated by order functions.
    The geometric idea is that the regularity of a function can be encoded by a Newton polygon in the $(|\xi|,\tau)$-plane.\footnote{\label{footnote001}Historically, Newton polygons are drawn in the $(|\xi|,\tau)$-plane rather than the $(\tau,|\xi|)$-plane.
    Consequently, a point $(r,s)$ of a Newton polygon corresponds naturally to the mixed Sobolev space $W^{s,p}(\mathbb T;W^{r,p}(\mathbb R^n))$, where the temporal regularity $s$ is conventionally written before the spatial regularity $r$.
    Accordingly, throughout this article spaces of the form $H^{(s,r),p}$ or $E(s,r)$ are represented by the point $(r,s)$ in the Newton polygon.}
    Equivalently, one may replace the polygon by its support function.
    Interpreting the order functions of \cite{denk2013general} in this way makes the algebraic structure of Newton polygons transparent: Minkowski addition becomes ordinary addition of order functions, and differences of positive order functions lead naturally to potential spaces which need no longer correspond to Newton polygons.

    A further complication arises from interpolation and traces.
    Even when the interior spaces are Newton polygon spaces based on Bessel-potential spaces, their traces generally do not share this property for $p\neq 2$.
    Instead, real interpolation naturally produces hybrid spaces involving both Besov and Bessel-potential regularity.
    We therefore introduce the notion of a Newton polygon scale, which contains the classical Newton polygon spaces together with these interpolation spaces and is stable under the operations needed in the boundary theory.
    Within this framework, we identify the trace spaces explicitly in Theorem \ref{thm: trace_space} and show that they admit a natural description as interpolation spaces between closely related Newton polygon spaces based on Bessel-potential spaces.
    This provides the final ingredient required to apply the abstract reduction to mixed-order boundary value problems.
\end{enumerate}

Together, these three ingredients yield the $L^p$-well-posedness theory for mixed-order boundary value problems associated with $\mathcal N$-elliptic operators established in Theorem \ref{thm: main_result}.
This theorem also clarifies the precise role of the Lopatinski\u{\i}--Shapiro condition.
In our formalism, this condition corresponds to pointwise invertibility of the complemented boundary matrix, see Proposition 5.1 in \cite{NeS25}.
We show that this condition already yields well-posedness, but in general only on a subspace of the natural data space determined by compatibility conditions between the interior and boundary data.
Our abstract framework identifies this compatibility space intrinsically and proves well-posedness on this restricted space.
If, in addition, the inverse of the complemented boundary matrix satisfies the quantitative estimates encoded by $\mathcal N$-ellipticity, then the compatibility space is the full product of the natural data spaces, so that the interior and boundary data can be prescribed independently.

To illustrate the scope of the resulting theory and to facilitate direct comparison with existing results, we state the applications considered in this work explicitly in terms of classical Sobolev and Besov spaces, rather than in the order-function notation used in the general theory.
They are posed on the time-periodic half-space $\mathbb T\times\R^n_+$ with $\mathbb T:=\R/T\Z$ for some fixed period $T>0$.
\begin{enumerate}
    \item \emph{Higher-Order Regularity for the Heat Equation:}
    Since Dirichlet boundary conditions satisfy the Lopatinski\u{\i}-Shapiro conditions, it follows from the general theory for parabolic boundary problems that the problem
        \begin{align}\label{eqn: heat_equation_intro}
        \left\{\begin{array}{rlll}
        \partial_t u - \Delta u&=&f     & \text{on } \mathbb{T}\times \R^n_+, \\
         u&=&g    & \text{on } \mathbb{T}\times \R^{n-1},
        \end{array}\right.
    \end{align}
    admits for each $f\in \mathbb{F}$ and $g\in \mathbb{G}$ a unique solution $u\in \mathbb{E}$, where
    \begin{align*}
        \mathbb{E}&:=W^{1,p}_\perp(\mathbb{T};L^p(\R^n_+))\cap L^p_\perp(\mathbb{T};W^{2,p}(\R^n_+)),\\
        \mathbb{F}&:=L^p_\perp(\mathbb{T}\times \R^n_+),\\
        \mathbb{G}&:=W^{1-\frac{1}{2p},p}_\perp(\mathbb{T};L^p(\R^{n-1}))\cap L^p_\perp(\mathbb{T};W^{2-\frac1p,p}(\R^{n-1})),
    \end{align*}
    see \cite{kyed2019time} for this result in the time-periodic setting.
    By standard operator-theoretic arguments, and for $g=0$, this can be extended to higher regularity in the sense that one may replace $\mathbb{E}$ by
    \begin{align*}
        \mathbb{E}'&:= W^{1,p}_\perp(\mathbb{T};W^{1,p}(\R^n_+))\cap L^p_\perp(\mathbb{T};W^{3,p}(\R^n_+)),
    \end{align*}
    if one also replaces $\mathbb{F}$ by $\mathbb{F}':= L^p_\perp(\mathbb{T};W^{1,p}_0(\R^n_+))$, where $W^{1,p}_0(\R^n_+)$ refers to the Sobolev space with vanishing trace.
    This higher-regularity result is, however, not optimal.
    Namely, the condition $f\in \mathbb{F}'$ is not necessary: there exist solutions $u\in \mathbb{E}'$ with $g=0$ for which $f\notin \mathbb{F}'$.
    Indeed, formally taking the trace of $\partial_t u-\Delta u=f$ and using the boundary condition $u=0$ yields $-\Delta u=f$ on the boundary.
    Consequently, besides the interior regularity $f\in L^p(\mathbb{T};W^{1,p}(\R^n_+))$ one is naturally led to the boundary compatibility condition $\Tr_0 f=-\Tr_0\Delta u\in W^{\frac12-\frac1{2p},p}_\perp(\mathbb{T};L^p(\R^{n-1}))$, rather than the stronger condition $\Tr_0 f=0$. 
    Our framework turns this into a purely mechanical computation and shows automatically that this compatibility condition is also sufficient, see Theorem \ref{js011}.
    In fact, it immediately identifies a linear data space $\mathbb{D}$ for the pair $(f,g)$ such that \eqref{eqn: heat_equation_intro} becomes an isomorphism between the solution space and the data space.
    In particular, the inhomogeneous boundary case $g\ne 0$ is treated simultaneously.
    \vspace{.2cm}
    \item \emph{Parabolic Problems with Dynamic Boundary Conditions:}
    Dynamic boundary conditions arise naturally in many models from physics and chemistry.
    For example, the linearized Cahn-Hilliard equation with dynamic boundary conditions and surface diffusion
     \begin{align}\label{eqn: ch_dynamic_diff_intro}
            \left\{\begin{array}{rcll}
            \partial_tu+\Delta^2u&=& f     & \text{ on }\mathbb{T}\times \mathbb{R}^n_+ \\
             \partial_n \Delta u&=&g_1    & \text{ on }\mathbb{T}\times \mathbb{R}^{n-1}\\
             \partial_t u +  \partial_n u-\Delta_{\mathbb{R}^{n-1}}u &=& g_2& \text{ on }\mathbb{T}\times \mathbb{R}^{n-1},
            \end{array}\right.
     \end{align}
    models phase separation in binary mixtures when the interaction between the container wall and the mixture is short-ranged, leading to an evolution equation on the boundary, see \cite{denk2008parabolic,PRZ06}.
    Here, our framework yields that for any $f\in \mathbb{F}$ and $g=(g_1,g_2)\in \mathbb{G}_1\times \mathbb{G}_2$ there is a unique solution $u\in \mathbb{E}$, where
    \begin{align*}
        \mathbb{E}&:= W^{1,p}_\perp(\mathbb{T};L^p(\R^n_+))\cap L^p_\perp(\mathbb{T};W^{4,p}(\R^n_+)), \\
        \mathbb{F}&:= L^p_\perp(\mathbb{T}\times\R^n_+),\\
        \mathbb{G}_1&:=W^{\frac14-\frac1{4p},p}_\perp(\mathbb{T};L^p(\R^{n-1}))\cap L^p_\perp(\mathbb{T};W^{1-\frac1p,p}(\R^{n-1})),\\
        \mathbb{G}_2&:=(\partial_t-\Delta_{\R^{n-1}})\mathbb{G}', \quad \mathbb{G}':=W^{1-\frac{1}{4p},p}_\perp(\mathbb{T};L^p(\R^{n-1}))\cap L^p_\perp(\mathbb{T};W^{4-\frac1p,p}(\R^{n-1})).
    \end{align*}
    The norm on $\mathbb{G}_2$ is initially defined as the quotient norm $\|g_2\|_{\mathbb{G}_2}:=\inf\{\|h\|_{\mathbb{G}'}\mid h\in \mathbb{G'}, \partial_t h-\Delta_{\R^{n-1}}h=g_2\}$.
    Corollary \ref{cor: ellipticity_HB_BH} below shows that $\partial_t-\Delta_{\R^{n-1}}$ is in fact an isomorphism between $\mathbb{G}'$ and $\mathbb{G}_2$.
    Consequently, the infimum is attained at the unique $h\in \mathbb{G}'$ satisfying $\partial_t h-\Delta_{\R^{n-1}}h=g_2$.
    Our framework furthermore admits a complementary formulation: we may replace the solution space $\mathbb{E}$ by the stronger space
    \begin{align*}
        \mathbb{E}'&:=\{u\in \mathbb{E}\mid \Tr_0 u \in \mathcal E_0 \},
        \qquad
        \|u\|_{\mathbb{E}'}:= \|u\|_{\mathbb{E}}+\|\Tr_0u\|_{\mathcal E_0},
    \end{align*}
    where 
    \begin{align*}
        \mathcal E_0:= W^{\frac{7}{4}-\frac{1}{4p},p}_\perp(\mathbb{T};L^p(\R^{n-1}))\cap W^{1,p}_\perp(\mathbb{T};W^{3-\frac1p,p}(\R^{n-1}))\cap L^p_\perp(\mathbb{T};W^{5-\frac1p,p}(\R^{n-1})),
    \end{align*}
    if we also replace the space $\mathbb{G}_2$ for the dynamic boundary condition by
    \begin{align*}
        \mathbb{G}_2'&:= W^{\frac{3}{4}-\frac{1}{4p},p}_\perp(\mathbb{T};L^p(\R^{n-1}))\cap L^p_\perp(\mathbb{T};W^{3-\frac1p,p}(\R^{n-1})),
    \end{align*}
    see Theorem \ref{thm:CH}.
    The well-posedness result based on the solution space $\mathbb{E}'$ is the time-periodic analogue of Theorem 2.1 in \cite{PRZ06} for the half-space (see also Theorem 6.2 in \cite{denk2008parabolic} for the case $p=2$), whereas the well-posedness result for the larger solution space $\mathbb{E}$ appears to be new.
    This example also provides a point of comparison with the Newton-polygon approach of \cite{denk2008parabolic}, which allows nonstandard scaling in the boundary operators while retaining a parabolic interior operator.
    
    \medskip
    
    A considerably more challenging example is provided by the dynamic boundary condition introduced by Liu and Wu \cite{LiW19}.
In the absence of surface diffusion, the boundary trace is no longer governed by a parabolic surface operator of positive spatial order, as it is in \eqref{eqn: ch_dynamic_diff_intro}.
Instead, the additional spatial scale is generated entirely through the coupling of the boundary dynamics with the fourth-order bulk equation, resulting in a boundary evolution whose highest spatial order is three.
Liu and Wu observe that the corresponding linearized problem falls outside the maximal $L^p$-regularity framework of \cite{DPZ08} and state that it is not clear whether the available maximal regularity results can be extended to this case.
The Newton-polygon structure of the boundary operators suggests that in $L^2$-based spaces the problem may also be accessible by the approach of \cite{denk2008parabolic}.
Here, we show directly that it fits into the present framework and obtain a new $L^p$-well-posedness result in the natural mixed-order spaces.
More precisely, we consider
     \begin{align}\label{eqn: ch_lw_intro}
            \left\{\begin{array}{rcll}
            \partial_tu+\Delta^2u&=& f     & \text{ on }\mathbb{T}\times \mathbb{R}^n_+ \\
             \partial_n \Delta u&=&g_1    & \text{ on }\mathbb{T}\times \mathbb{R}^{n-1}\\
             \partial_t u -  \partial_n \Delta_{\mathbb{R}^{n-1}}u &=& g_2& \text{ on }\mathbb{T}\times \mathbb{R}^{n-1},
            \end{array}\right.
     \end{align}
     and show that for any $f\in \mathbb{F}$ and $g=(g_1,g_2)\in \mathbb{G}_1\times \mathbb{G}_2'$ there is a unique solution $u\in \mathbb{E}'$, where compared to the earlier result we choose $\mathbb{G}_2':=\mathbb{G}_1$, and the space $\mathcal E_0$ determining the regularity of the trace $\Tr_0 u$ in the solution space $\mathbb{E}'$ is replaced by
     \begin{align*}
     	\mathcal E_0:= W^{\frac54-\frac1{4p},p}_\perp(\mathbb{T};L^p(\R^{n-1}))\cap W^{1,p}_\perp(\mathbb{T};W^{1-\frac1p,p}(\R^{n-1}))\cap L^p_\perp(\mathbb{T};W^{4-\frac1p,p}(\R^{n-1})).
     \end{align*}
     Thus, the same framework that recovers and extends the established theory for \eqref{eqn: ch_dynamic_diff_intro} also applies to the substantially different boundary dynamics in \eqref{eqn: ch_lw_intro}.
     In particular, no problem-specific modification of the general theory is required: after establishing admissibility of the interior operator, this nonstandard coupling is captured by the $\mathcal N$-ellipticity of the corresponding complemented boundary matrix.
     \item \emph{The Cahn--Hilliard--Gurtin System with Static and Dynamic Boundary Conditions:}
     As a final application, we consider the Cahn--Hilliard--Gurtin system, which combines the two phenomena illustrated separately above.
     Its interior operator is a genuinely mixed-order system, and in the dynamic case it is coupled to a distinct mixed-order structure on the boundary.
     The system was introduced by Gurtin in \cite{Gur96} as an extension of the classical Cahn--Hilliard equation incorporating microscopic relaxation effects.
     We first consider the static boundary value problem
     \begin{align}\label{eqn: CHG_system_1_intro}
     \left\{
    \begin{array}{rcll}
        \partial_tu_1-\Delta u_2&=&f_1
        &\text{in }\mathbb T\times\R^n_+,\\
        \Delta u_1-\partial_tu_1+u_2-c\cdot\nabla u_2&=&f_2
        &\text{in }\mathbb T\times\R^n_+,\\
        \partial_nu_1&=&g_1
        &\text{on }\mathbb T\times\R^{n-1},\\
        \partial_nu_2&=&g_2
        &\text{on }\mathbb T\times\R^{n-1}.
    \end{array}
    \right.
\end{align}
In this model, $u_1$ represents the concentration, $u_2$ the chemical potential and $c\in \R^n$ accounts for anisotropy of the material.
For every $p\in(1,\infty)$ and $c\in\R^n$, Theorem~\ref{thm: CHG_c_stat} yields an isomorphism between the solution space $\mathbb E:=\mathbb E_1\times\mathbb E_2$ and the product of the interior space $\mathbb F:=\mathbb F_1\times\mathbb F_2$ and the boundary data space $\mathbb G:=\mathbb G_1\times\mathbb G_2$, where
\begin{align*}
    \mathbb E_1&:=W^{1,p}_\perp(\mathbb T;W^{1,p}(\R^n_+))
    \cap L^p_\perp(\mathbb T;W^{3,p}(\R^n_+)),
    &
    \mathbb E_2&:=L^p_\perp(\mathbb T;W^{2,p}(\R^n_+)),\\
    \mathbb F_1&:=L^{p}_\perp(\mathbb T;L^{p}(\R^n_+)),
    &
    \mathbb F_2&:=L^p_\perp(\mathbb T;W^{1,p}(\R^n_+)),\\
    \mathbb G_1&:=W^{1-\frac1{2p},p}_\perp(\mathbb T;L^p(\R^{n-1}))
    \cap L^p_\perp(\mathbb T;W^{2-\frac1p,p}(\R^{n-1})),
    &
    \mathbb G_2&:=L^p_\perp(\mathbb T;W^{1-\frac1p,p}(\R^{n-1})).
\end{align*}
The case $c=0$ and $p=2$ was established in our earlier work \cite{NeS25}.
Wilke \cite{Wil12} treated the corresponding initial-value problem by means of a reflection argument, which requires the normal component $c_n$ to vanish.
The present approach imposes no such restriction in the static problem.

The proof also illustrates a feature specific to systems.
The determinant of the interior symbol has a non-triangular Newton polygon, but its support-preserving factorization induces an admissible factorization $\opL:=\opA^-\opA^+$ of the full $2\times2$ system.
The resulting intermediate space $\mathbb H:=\opA^+_+\mathbb E$ is no longer a product of Newton-polygon spaces, so that its trace structure cannot be read off componentwise from the standard trace theory.
We instead construct it abstractly from the kernel of the boundary-sensitive factor $\opA_+^+$ and show that the static boundary operator corresponding to \eqref{eqn: CHG_system_1_intro} satisfies the resulting abstract complementing condition.

We then consider the dynamic boundary conditions introduced in the model of Goldstein and Miranville \cite{GoM13},
\begin{align}\label{eqn: CHG_dynamic_intro}
    \left\{
    \begin{array}{rcll}
        \partial_tu_1-\Delta u_2&=&f_1
        &\text{in }\mathbb T\times\R^n_+,\\
        \Delta u_1-\partial_tu_1+u_2-c\cdot\nabla u_2&=&f_2
        &\text{in }\mathbb T\times\R^n_+,\\
        \partial_tu_1-\Delta_{\R^{n-1}}u_2-\partial_nu_2&=&g_1
        &\text{on }\mathbb T\times\R^{n-1},\\
        (1-c_n)u_2-\partial_tu_1+\Delta_{\R^{n-1}}u_1+\partial_nu_1-u_1&=&g_2
        &\text{on }\mathbb T\times\R^{n-1}.
    \end{array}
    \right.
\end{align}
Here the interior Newton polygon and the mixed-order structure generated by the dynamic boundary operators interact through the traces of both unknowns.
In particular, the problem lies outside the parabolic interior setting underlying the Newton-polygon approach to dynamic boundary conditions in \cite{denk2008parabolic}.

The constitutive coercivity condition imposed in \cite{GoM13} reduces, under the normalization in \eqref{eqn: CHG_dynamic_intro}, to $|c|<2$, as we verify in Section~\ref{sec: CHG}.
Under this condition, Theorem~\ref{thm: CHG_c_dyn} shows that \eqref{eqn: CHG_dynamic_intro} is well posed in the same solution space $\mathbb E$ as the static problem, but the admissible data no longer form the full product of $\mathbb F$ and the natural boundary space $\mathbb G_{\mathrm{dyn}}=\mathbb G_{1,\mathrm{dyn}}\times \mathbb{G}_{2,\mathrm{dyn}}$
 with
\begin{align*}
G_{1,\mathrm{dyn}}:=L^p_\perp(\mathbb T;W^{-\frac1p,p}(\R^{n-1})), \qquad G_{2,\mathrm{dyn}}:=L^p_\perp(\mathbb T;W^{1-\frac1p,p}(\R^{n-1})).
\end{align*}
Instead, the additional temporal regularity $g_2 - \op[\ell]g_1-\Gamma f\in \mathbb G':=W^{\frac12-\frac1{2p},p}_\perp(\mathbb T;L^p(\R^{n-1}))$ acts as a compatibility condition, where $\op[\ell]$ and $\Gamma$ are given explicitly in \eqref{eq: dynamic_CHG_range_c_ell} and \eqref{eq: CHG_Gamma}, respectively.
For $c=0$, the compatibility condition reduces to $g_2-\Tr_0f_2\in \mathbb G'$.
Equipped with the natural norm $\|f\|_{\mathbb F} +\|g\|_{\mathbb G_{\mathrm{dyn}}} +\|g_2-\op[\ell]g_1-\Gamma f\|_{\mathbb G'}$, the resulting compatibility space is precisely the data space for which the dynamic problem defines an isomorphism.

This example therefore combines the main features of the general theory in a single problem.
The determinant factorization yields admissibility of a genuinely mixed-order interior system, the trace structure of the coupled intermediate space is constructed intrinsically, and the dynamic boundary conditions reduce to a mixed-order multiplier problem on this trace space.
The abstract theory then determines both well-posedness and the non-product compatibility space of the interior and boundary data.
To the best of our knowledge, Theorem~\ref{thm: CHG_c_dyn} is the first maximal $L^p$-regularity result for a genuinely mixed-order interior system coupled to dynamic boundary conditions carrying a distinct mixed-order structure.
\end{enumerate}
These examples illustrate different aspects of the general reduction principle, ranging from the intrinsic identification of compatibility conditions to genuinely mixed-order systems with dynamic boundary conditions.
Despite their rather different analytic structures, their treatment follows the same scheme: restricted invertibility is obtained through suitable factorizations of the interior operators, while the remaining boundary analysis is reduced to explicit conditions on the corresponding operators on the trace space.

\medskip

The structure of this work is as follows.
Section \ref{sec:red} develops the abstract reduction from whole-space to half-space problems and introduces admissible factorizations.
In Section \ref{sec:newt}, we develop the Newton-polygon and order-function machinery required for mixed-order operators, while Section \ref{sec:whole_space} establishes the corresponding whole-space isomorphism and interpolation results.
Section \ref{sec:half_space} combines these ingredients with the $L^p$-trace theory to obtain the general boundary value theorem.
Finally, Section \ref{sec:app} applies the theory to the examples described above.

\subsection{Preliminaries and Notation}\label{sec:pre}
The natural numbers will be denoted by $\N:=\set{1,2,3,\ldots}$, and we set $\N_0:=\N\cup\set{0}$. $\Z$ denotes the integers, $\R$ the real numbers, and $\C$ the complex numbers.
The imaginary unit is denoted by $\ic$.
Throughout this work the time period $T>0$, the integrability exponent $p\in (1,\infty)$, and the dimension $n\in\N$ with $n\ge 2$ are fixed.
The notion $A\lesssim B$ denotes an estimate of the form $A\le CB$, where the constant $C$ does not depend on $A$ and $B$.
We will write $A\simeq B$ if both $A\lesssim B$ and $B\lesssim A$.
All topological vector spaces in this work are understood to be Hausdorff.
Given two topological vector spaces $E, F$ we write $\mathcal{L}(E, F)$ for the space of all continuous linear operators $E\to F$ and $\call_{\mathrm{iso}}(E,F)$ for the space of continuous, invertible linear operators $E\to F$ with continuous inverse.
Both spaces are endowed with the usual topology of uniform convergence on bounded subsets.
In particular, if  $E$ and $F$ are normed spaces, the topology on $\call(E,F)$ is given by the operator norm. If $E=F$, we write $\call(E)$ instead of $\call(E,E)$.

\medskip

We use the convention that for a set $Y$, a topological space $X$ and a map $\varphi:X\to Y$, we denote by $\varphi X$ the image space $\varphi(X)$ equipped with the final topology.
If additionally $Y$ is a vector space, $X$ a (locally convex) topological vector space, and $\varphi$ linear with a closed kernel, then $\varphi X$ is a (locally convex) topological vector space as well, see e.g.\@ \cite[Theorem 2.4.3]{Eng89} and \cite[Chapter 4 and Proposition 7.9]{Tre67}.
Similarly, if $X$ is a Fr\'echet or normed space, then so is $\varphi X$, see \cite[Chapters 10 and 11]{Tre67}.
In particular, if $X$ is a Banach space, then so is $\varphi X$, and its norm is given by the quotient norm
\begin{align*}
\|f\|_{\varphi X}=\inf\{\|F\|_{X}\mid F\in X, \ \varphi(F)=f\}.
\end{align*}
For a continuous map $\varphi:X\to Y$ between two topological vector spaces, we define its transpose ${}^t \varphi:Y'\to X'$ via $\langle {}^t \varphi(y'),x\rangle:=\langle y',\varphi(x)\rangle$ for all $y'\in Y'$ and $x\in X$.
Then ${}^t \varphi$ is continuous as well if $X'$ and $Y'$ carry both the weak or both the strong dual topology \cite[Proposition 19.5]{Tre67}.

\medskip

For $\xi\in\R^n$, we write $\langle \xi\rangle:=(1+|\xi|^2)^{\frac12}$.
Two Banach spaces $A_0$ and $A_1$ are called compatible if they embed into a common Hausdorff space.
For two such compatible couples $\set{A_0,A_1}$ and $\set{B_0,B_1}$ we write $\call(\set{A_0,A_1},\set{B_0,B_1})$ for the set of linear operators $T:A_0+A_1\to B_0+B_1$ such that $T|_{A_0}\in \call(A_0,B_0)$ and $T|_{A_1}\in \call(A_1,B_1)$.
For $\theta\in (0,1)$ and $q\in [1,\infty]$, the complex interpolation is denoted by $[A_0,A_1]_\theta$ and the real interpolation with fine index $q$ by $(A_0,A_1)_{\theta,q}$.

\section{Reduction From Whole Space to Half Space}\label{sec:red}
In this section, we develop a theory for boundary value problems on a half-space.
We first establish an abstract framework which reduces well-posedness to invertibility on the kernel of a trace operator and an appropriate complementing boundary condition.
We then show how support-preserving operators allow this framework to be applied to purely oscillatory Newton polygon spaces and derive the corresponding $L^p$-theory.

\subsection{An Abstract Operator Framework}\label{sec:abstr_frame}
As we shall see later, many boundary value problems share the following abstract structure.
After a harmless transformation, the underlying differential operator can be realized as an operator between topological vector spaces that becomes invertible only after restriction to distinguished subspaces.
The objective of this section is to identify the additional linear constraints that restore invertibility.
Throughout, we adopt a notation that is tailored to the half-space problems considered in the subsequent sections.
In this subsection, $U$, $V$, $\calu$, $\calv$, and $\calg$ denote topological vector spaces, and $\dot U\subseteq U$ and $\dot V\subseteq V$ are closed subspaces endowed with the respective subspace topology.

\medskip

Consider $\opG\in\call(\calu,V)$.
Then $\dot{V}\times\calu$ is naturally isomorphic to a compatibility space
\begin{align}\label{js005_e1}
    \Comp_\opG&:=\{(f,b)\in V\times \calu\mid f-\opG b \in \dot{V}\}
\end{align}
endowed with the product topology of $V\times \calu$.
Indeed, it suffices to observe that the operator $\Psi_\opG\in \call(\Comp_\opG, \dot{V}\times \calu)$ given by $\Psi_\opG(f,b):=(f-\opG b,b)$ is an isomorphism with inverse $\Psi_\opG^{-1}(h,b):=(h+\opG b,b)$.
\begin{lemma}\label{js006}
    Let $\opE\in \call(\calu,U)$, and suppose that $\opT\in \call(U,V)$ is such that
    \begin{align*}
        \text{$\dot{\opT}:=\opT|_{\dot{U}}\in \call_{\mathrm{iso}}(\dot{U},\dot{V})$.}
    \end{align*}
    Then the operator $S_{\opT,\opE}:=(\opT,\id_\calu)$ belongs to  $\mathcal{L}_{\mathrm{iso}}(\Comp_\opE,\Comp_{\opT\opE})$, and its inverse is given by $S_{\opT,\opE}^{-1}(f,b)=(v+w,b)$, where $v:=\opE b$ and $w:=\dot{\opT}^{-1}(f-\opT\opE b)$.
    In particular, for any $(f,b)\in \Comp_{\opT\opE}$, there is a unique $u\in U$ such that $\opT u=f$ and $(u,b)\in \Comp_{\opE}$.
\end{lemma}
\begin{proof}
    By the definition of the compatibility space $\eqref{js005_e1}$, we learn that the operator
    \begin{align*}
        \Psi_{\opT\opE}^{-1}\circ
        \begin{pmatrix}
            \dot{\opT} & 0 \\
            0 & \id_\calu
        \end{pmatrix}
        \circ \Psi_{\opE}
    \end{align*}
belongs to $\call_{\mathrm{iso}}(\Comp_{\opE},\Comp_{\opT\opE})$.
Therefore, the result follows from
\begin{align*}
        S_{\opT,\opE}(u,b)=(\opT u,b)=\Psi_{\opT\opE}^{-1}(\opT(u-\opE b),b)=
        \Psi_{\opT\opE}^{-1}
        \begin{pmatrix}
            \dot{\opT} & 0 \\
            0 & \id_\calu
        \end{pmatrix}
        \Psi_{\opE}(u,b)
\end{align*}
and hence
\begin{align*}
        S_{\opT,\opE}^{-1}(f,b) &=\Psi_{\opE}^{-1} 
        \begin{pmatrix}
            \dot{\opT}^{-1} & 0 \\
            0 & \id_\calu
        \end{pmatrix}
        \Psi_{\opT\opE}(f,b)=\Psi_{\opE}^{-1}(\dot{\opT}^{-1}(f-\opT\opE b),b)=(v+w,b).
        \qedhere
\end{align*}
\end{proof}
The previous lemma shows that invertibility of $\dot{\opT}$ induces invertibility between corresponding compatibility spaces once a space $\calu$ and a map $\opE\in\call(\calu,U)$ are chosen.
To recover an invertibility statement for the original space $U$, we need a natural identification of $U$ with the compatibility space $\Comp_{\opE}$, i.e., an isomorphism $\iota\in \call_{\mathrm{iso}}(U,\Comp_{\opE})$.
Such an isomorphism should preserve the first component in the sense $\iota(u)=(u,b)$ for some $b\in \calu$, so that the equation $\opT u=f$ is unaffected.
In other words, the compatibility space $\Comp_{\opE}$ has to be realized as the graph of some operator $\Tr\in\call(U,\mathcal U)$, i.e., $\iota(u)=(u,\Tr\,u)$. In that case, we have
\begin{align*}
    \Tr\, u =0 \iff (u,0)\in\Comp_\opE \iff u\in\dot{U}.
\end{align*}
Moreover, for every $b\in\calu$, the point $(\opE b,b)$ belongs to $\Comp_{\opE}$.
Hence $(\opE b,b)=\iota(\iota^{-1}(\opE b,b))=(u,\Tr\, u)$, where $u:=\iota^{-1}(\opE b,b)$.
Comparing first coordinates gives $u=\opE b$, while comparing second coordinates yields $\Tr\,\opE b=\Tr\, u=b$.
Summarizing, we want $\Tr\circ \opE=\id_{\calu}$ and $\ker\Tr=\dot{U}$.
This is exactly the content of the following definition.
\begin{definition}\label{def: abstract_trace}
    The space $\calu$ is called an \emph{abstract trace space} for $(U,\dot{U})$ if there exist $\Tr\in \mathcal{L}(U,\calu)$ and $\opE \in \mathcal{L}(\calu,U)$ such that
    \begin{align*}
        \Tr\circ \opE = \id_{\calu},
        \quad \text{and} \quad
        \ker \Tr= \dot{U}.
    \end{align*}
    In that case, we call $(U,\dot{U},\calu,\Tr,\opE)$ an \emph{abstract trace tuple}.
\end{definition}
\begin{example}\label{js009}
    \begin{enumerate}
        \item\label{js009i} Let $p\in (1,\infty)$ and $s>\frac1p$ be such that $s-\frac1p$ is not an integer.
        Consider the Sobolev-Slobodecki\u{\i} space $U:=W^{s,p}(\R^d_+)$ on the half-space.
        Denote by $\dot{U}:=W^{s,p}_0(\R^d_+)$ the closure of $C_c^\infty(\R^d_+)$.
        Then it is a classical result, see e.g.\ \cite[Chapter 3.6.2]{triebel1995interpolation}, that for $\calu:=\prod_{j=0}^{k} W^{s-j-\frac1p,p}(\R^{d-1})$, $k:=\lfloor s-\frac1p\rfloor$, there is a trace operator $\Tr=(\Tr_0,\ldots, \Tr_{k})$ in $\call(U,\calu)$ and a corresponding extension operator $\opE\in\call(\calu,U)$ such that $(U,\dot{U},\calu,\Tr,\opE)$ is an abstract trace tuple.
        \item\label{js009ii} If $0\le s<\frac1p$, which contains the important case of $U=L^p(\R^d_+)$ for the choice $s=0$, there does not exist a reasonable Sobolev trace for elements in $W^{s,p}(\R^d_+)$ anymore.
        However, the distinguished subspace $\dot{U}$ coincides with $U=\dot{U}$ in that range, see e.g.\ \cite[Chapter 3.6.2]{triebel1995interpolation}.
        Consequently, $(U,\dot{U},\calu,\Tr,\opE)$ forms an abstract trace tuple for the trivial trace space $\calu=\{0\}$ together with the zero operators $\Tr=0$ and $\opE=0$.
    \end{enumerate}
\end{example}
One may wonder for which pairs $(U,\dot{U})$ such an abstract trace space exists.
Observe that if $\dot{U}$ is a complemented closed subspace of $U$ with projection $P$, then for $Q:=\id_U-P$ the choice $\calu:=QU$ is an abstract trace space with $\Tr:=Q$ and $\opE$ being the natural injection of $\calu$ in $U$.
On the other hand, if $(U,\dot{U},\calu,\Tr,\opE)$ is an abstract trace tuple, then $\dot{U}$ is a complemented closed subspace of $U$, with the projection onto $\dot{U}$ given by $\id_{U}-\opE\circ\Tr$.
Therefore, an abstract trace space exists if and only if $\dot{U}$ is a complemented closed subspace of $U$.
In particular, for $(u,b)\in U\times \calu$ it holds
\begin{align*}
\Tr \, u=b
\iff
u-\opE b\in \dot{U}
=
\ker\Tr
\iff
(u,b)\in\Comp_{\opE}.
\end{align*}
Thus, $\Comp_{\opE}$ is precisely the graph of $\Tr$, i.e., $\iota\in \call(U,\Comp_{\opE})$ given by $\iota(u):=(u,\Tr \, u)$ is an isomorphism with inverse $\iota^{-1}(v,b):=v$.

\begin{proposition}\label{prop: op_preserving_upper_support_abstr}
    Let $(U,\dot{U},\calu,\Tr,\opE)$ be an abstract trace tuple.
    Suppose that $\opT\in \call(U,V)$ is such that
    \begin{align*}
        \dot{\opT}:=\opT|_{\dot{U}}\in \call_{\mathrm{iso}}(\dot{U},\dot{V}).
    \end{align*}
    Then $S:=(\opT,\Tr)$ belongs to $\call_{\mathrm{iso}}(U,\Comp_{\opT\opE})$ with inverse $S^{-1}(f,b)=v+w$, where $v:=\opE b$ and $w:=\dot{\opT}^{-1}(f-\opT\opE b)$.
    In particular, for any $(f,b)\in \Comp_{\opT\opE}$, there exists a unique $u\in U$ satisfying
    \begin{align*}
        \left\{\begin{array}{rcl}
         \opT u    &=&f  \\
          \Tr \, u   &=&b. 
        \end{array}\right.
    \end{align*}
\end{proposition}
\begin{proof}
Since $S=S_{\opT,\opE}\circ \iota$ and $\iota\in \call(U,\Comp_\opE)$, this follows immediately from Lemma \ref{js006}. 
\end{proof}
    As seen above, the constraint encoded by the space $\Comp_{\opE}$ determines the admissible boundary data.
    We now turn to the compatibility condition defining $\Comp_{\opT\opE}$, namely
    \begin{align*}
     f-\opT\opE b\in\dot V.      
    \end{align*}
    Suppose there exists $\opV\in\call(V,\calv)$ such that $\ker\opV=\dot{V}$.\footnote{
    In many applications below, the map $\opV$ will itself arise as the trace operator $\Tr_V$ of an abstract trace tuple $(V,\dot V,\calv,\Tr_V,\Ext_V)$.
    However, the abstract theory developed in this section only requires the weaker assumption $\ker\opV=\dot V$, so we formulate the results at this level of generality.}
    Then a necessary condition for $\opT u=f$ to hold in $V$ is $\opV\opT u=\opV f$.
    If $\Tr \, u=b$, then $u-\opE b\in \dot{U}$, so that $\opT(u-\opE b)\in \dot{V} = \ker \opV$, i.e. $\opV\opT u=\opV\opT\opE b$.
    Thus, the necessary condition presents itself as the compatibility condition
    \begin{align}\label{js005_e2}
        \opV \opT\opE b=\opV f
        \quad \Leftrightarrow \quad
        f-\opT\opE b \in \ker\opV=\dot{V}
        \quad \Leftrightarrow \quad
        (f,b)\in \Comp_{\opT\opE}.
    \end{align}
    In many applications, it actually holds $V=\dot V$, so that there is no nontrivial constraint on the right-hand side $f$.
    As seen in Example \ref{js009}, this is for instance the case for $V=L^p(\R^d_+)$.
    In this case, \eqref{js005_e2} is always satisfied, that is, every pair $(f,b)\in V\times\calu$ is admissible in Proposition \ref{prop: op_preserving_upper_support_abstr}.
    At first sight, this may suggest that keeping track of the compatibility condition serves little practical purpose.
    However, the compatibility condition becomes indispensable when treating general boundary value problems, as we discuss next.

    \medskip
    
    Applying \eqref{js005_e2} to $(f,b)=(\opT u,\Tr\, u)$, which belongs to $\Comp_{\opT\opE}$ by Proposition \ref{prop: op_preserving_upper_support_abstr}, and writing $\opC:=\opV\opT\opE$, we obtain the fundamental intertwining relation
    \begin{equation}\label{js005_e3}
        \opC\Tr=\opV\opT,
        \quad \text{i.e., the diagram} \quad
        \hbox{
        \begin{tikzcd}
        U \arrow[r, "\opT"] \arrow[d, "\Tr"] & V\arrow[d, "\opV"] \\
        \calu \arrow[r, "\opC"] & \calv
        \end{tikzcd}
        }
        \quad \text{commutes.}
    \end{equation}
    We now consider the topological vector space $\calg$, viewed as a boundary space, together with a boundary operator $\opB\in \call(\calu,\calg)$, and study the solvability of
    \begin{align}\label{js007_e1}
        \left\{\begin{array}{rcl}
         \opT u    &=&f  \\
          \opB\Tr \,u   &=&g 
        \end{array}\right.
    \end{align}
    for $f\in V$, $g\in \calg$.

\medskip

Since the operator $\opB$ is generally not invertible, the condition $\opB \Tr\, u=g$ does not determine the trace $\Tr\, u$, which is precisely the information needed to invert $\opT$ in light of Proposition \ref{prop: op_preserving_upper_support_abstr}.
The missing boundary information must therefore be recovered from the equation $\opT u=f$.
Since $\opC\Tr\,u=\opV\opT u =\opV f$ by \eqref{js005_e3}, this additional information is encoded by the operator $\opC$. 
This motivates the following abstract complementing boundary condition. 
\begin{definition}\label{def: abstract_complementing_condition_upper}
    Let $(U,\dot{U},\calu,\Tr,\opE)$ be an abstract trace tuple.
    Suppose that $\opV\in \call(V,\calv)$ is such that $\ker\opV=\dot{V}$.
    We say that the boundary operator $\opB\in \mathcal{L}(\calu, \calg)$ satisfies the \emph{abstract complementing boundary condition} with respect to an operator $\opT\in \call(U,V)$ if for the operator $\opC \in \mathcal{L}( \calu, \calv)$ defined by $\opC:=\opV\circ \opT \circ \opE$, the complemented operator $(\opB, \opC)\in \mathcal{L}(\calu, \calg\times \calv)$ is injective.
\end{definition}
With this definition in place, the general boundary value problem \eqref{js007_e1} can now be solved abstractly.
\begin{theorem}\label{js007}
    Let $(U,\dot{U},\calu,\Tr,\opE)$ be an abstract trace tuple.
    Suppose that $\opV\in \call(V,\calv)$ is such that $\ker\opV=\dot{V}$.
    Assume that $\opT\in \call(U,V)$ is such that $\dot{\opT}:=\opT|_{\dot{U}}\in \call_{\mathrm{iso}}(\dot{U},\dot{V})$, and let $\opB\in\call(\calu,\calg)$ satisfy the abstract complementing boundary condition with respect to $\opT$.
    Endow
    \begin{align*}
        \Comp_{\opB,\opC}:=\{(f,g)\in V\times \calg\mid (g,\opV f)\in \ran(\opB,\opC)\}
    \end{align*}
    with the initial topology of the map
    \begin{align*}
        \Comp_{\opB,\opC}\ni(f,g)\mapsto (f,b(f,g))\in V\times\calu,
    \end{align*}
    where $b(f,g):=(\opB,\opC)^{-1}(g,\opV f)$.
    Then $S_\opB:=(\opT,\opB\Tr)$ belongs to $\call_{\mathrm{iso}}(U,\Comp_{\opB,\opC})$ with inverse
    \begin{align*}
        S_\opB^{-1}(f,g)=\opE b(f,g) + \dot{\opT}^{-1}(f-\opT\opE b(f,g)).
    \end{align*}
    In particular, for all $(f,g)\in \Comp_{\opB,\opC}$ there is a unique solution $u\in U$ to \eqref{js007_e1}.
\end{theorem}
\begin{proof}
Define $J:\Comp_{\opT\opE}\to\Comp_{\opB,\opC}$ and $K:\Comp_{\opB,\opC}\to\Comp_{\opT\opE}$ via
\begin{align*}
    J(f,b):=(f,\opB b), \qquad K(f,g):=(f,b(f,g)),
\end{align*}
where we recall $b(f,g)=(\opB,\opC)^{-1}(g,\opV f)$, i.e.,
\begin{align}\label{js007_e3}
    \opB b(f,g)=g,
    \qquad
    \opC b(f,g)=\opV f.
\end{align}
We first show that $J$ and $K$ are well-defined, continuous, and mutually inverse.
Indeed, if \((f,b)\in\Comp_{\opT\opE}\), then $\opV f=\opV\opT\opE b=\opC b$ by \eqref{js005_e2}, and thus
\begin{align}\label{js007e_2}
    (\opB b,\opV f)=(\opB,\opC)b\in\ran(\opB,\opC),
\end{align}
so \(J(f,b)\in\Comp_{\opB,\opC}\).
On the other hand, by \eqref{js007_e3} we have $\opV(f-\opT\opE b(f,g))=\opV f-\opC b(f,g)=0$ for $(f,g)\in \Comp_{\opB,\opC}$.
Thus $f-\opT\opE b(f,g)\in \ker\opV=\dot{V}$, so that $K(f,g)\in \Comp_{\opT\opE}$.
Moreover, for $(f,b)\in \Comp_{\opT\opE}$, the relation \eqref{js007e_2} shows
\begin{align*}
    KJ(f,b)
    = K(f,\opB b)
    =(f,(\opB,\opC)^{-1}(\opB b,\opV f))
    =(f,b).
\end{align*}
Conversely, for \((f,g)\in\Comp_{\opB,\opC}\), the relation \eqref{js007_e3} shows
\begin{align*}
    JK(f,g)=J(f,b(f,g))=(f,\opB b(f,g))=(f,g).
\end{align*}
By the definition of the initial topology on $\Comp_{\opB,\opC}$, the map $K$ is continuous.
To prove continuity of $J$, it suffices by the same definition to verify continuity of $(f,b)\mapsto (f,b(J(f,b)))=(f,b)$ as a map from $\Comp_{\opT\opE}\to V\times\calu$, which is immediate because the topologies on $\Comp_{\opT\opE}$ and $V\times\calu$ are identical.
Therefore, $J=K^{-1}\in \call_{\mathrm{iso}}(\Comp_{\opT\opE},\Comp_{\opB,\opC})$.

Since $S:=(\opT,\Tr)\in \call_{\mathrm{iso}}(U,\Comp_{\opT\opE})$ by Proposition \ref{prop: op_preserving_upper_support_abstr}, we obtain $S_\opB=J\circ S\in \call_{\mathrm{iso}}(U,\Comp_{\opB,\opC})$.
By the inverse formula from Proposition \ref{prop: op_preserving_upper_support_abstr}, this gives
\begin{align*}
    S_\opB^{-1}(f,g) = S^{-1}K(f,g) = \opE b(f,g) + \dot{\opT}^{-1}\bigl(f-\opT\opE b(f,g)\bigr),
\end{align*}
as claimed.
\end{proof}
We stress that, in applications of Theorem \ref{js007}, the compatibility condition typically requires that a certain difference belongs to a space of higher regularity than either of the individual terms. Thus, rather than prescribing equality of two quantities, one only requires that their discrepancy is sufficiently regular.
The corresponding regularity space is encoded in the topology of $\Comp_{\opB,\opC}$.
In Section \ref{sec:app} we compute this topology explicitly for several concrete examples.
A particularly simple situation occurs when $(\opB,\opC)$ is an isomorphism onto $\calg\times\calv$.
In this case, there are no compatibility conditions relating the interior and boundary data, and the compatibility space is simply the product $V\times\calg$, both algebraically and topologically.
The following lemma makes this precise.
\begin{lemma}\label{js012}
If in the situation of Theorem \ref{js007} it holds $(\opB,\opC)\in \call_{\mathrm{iso}}(\calu,\calg\times\calv)$, then $\Comp_{\opB,\opC}=V\times\calg$, and the initial topology on $\Comp_{\opB,\opC}$ is the product topology of $V\times\calg$.
\end{lemma}
\begin{proof}
Since $\ran(\opB,\opC)=\calg\times\calv$, and since $(g,\opV f)\in \calg\times\calv$ for all $(f,g)\in V\times\calg$ by $\opV\in \call(V,\calv)$, the identity $\Comp_{\opB,\opC}=V\times\calg$ as sets is immediate.

\medskip

It remains to identify the topology.
To prove continuity of $\id:V\times\calg \to \Comp_{\opB,\opC}$, it is enough by the defining property of the initial topology to verify continuity of
\begin{align*}
    i: V\times\calg \longrightarrow V\times\calu,
    \qquad i(f,g):=(f,b(f,g)),
\end{align*}
where we recall $b(f,g)=(\opB,\opC)^{-1}(g,\opV f)$.
Since $\opV$, $(\opB,\opC)^{-1}$ and the projections $\pi_1:V\times\calg\to V$, $\pi_2:V\times\calg\to \calg$ are continuous, so is $b = (\opB,\opC)^{-1} \circ (\id_{\calb}\pi_2,\opV\pi_1)$, and therefore $i$ is continuous.
Conversely, define
\begin{align*}
    i\in \call(\Comp_{\opB,\opC},& V\times\calu),
    \qquad
    i(f,g):=(f,b(f,g)), \\
    j\in \call(V\times\calu,& V\times\calg),
    \qquad
    j(f,b):=(f,\opB b).
\end{align*}
Then $j\circ i=\id$  because $\opB b(f,g)=g$.
In particular $\id:\Comp_{\opB,\opC}\to V\times \calg$ is continuous as well.
Thus the initial topology on $\Comp_{\opB,\opC}$ coincides
with the product topology on $V\times\calg$.
\end{proof}
\subsection{Restricted and Supported Spaces}\label{sec:restr_spc}
We want to apply the abstract results of the previous section to boundary problems on domains of half-space type.
To this end, we consider a smooth product manifold of the form $M=M'\times \mathbb{R}$, where $M'$ is a smooth (second-countable Hausdorff) manifold. We denote by $\mathcal{D}(M):=C^\infty_0 (M)$ the space of smooth, compactly supported functions equipped with the canonical LF-topology, and by $\mathcal{D}'(M)$ its strong topological dual, the space of distributions on $M$. 

We then define the half-manifolds $M_\pm := M'\times \mathbb{R}_\pm$, and again the space of smooth, compactly supported functions $\mathcal{D}(M_\pm)$ with the canonical LF-topology and its strong topological dual $\mathcal{D}'(M_\pm)$.

We also define the natural zero-extensions $e_0^\pm: \mathcal{D}(M_\pm)\to \mathcal{D}(M)$ via
\begin{align}\label{def:ext_op}
e_0^+\varphi(x',x_n):=
\begin{cases}
\varphi(x',x_n) & \text{if } x_n>0,\\
0 & \text{else,}
\end{cases}
\quad \text{and} \quad
e_0^-\varphi(x',x_n):=
\begin{cases}
\varphi(x',x_n) & \text{if } x_n<0,\\
0 & \text{else,}
\end{cases}
\end{align}
for all $x'\in M'$. 
Then $e_0^\pm$ is continuous, and thus we may introduce the continuous restrictions $r^\pm:={}^t e_0^\pm:\mathcal{D}'(M)\to \mathcal{D}'(M_\pm)$. In particular, $\ker r^\pm$ is closed and so is its restriction $\ker r^\pm_{\vert E(M)}$ to any topological vector space $E(M)\hookrightarrow \mathcal{D}'(M)$ continuously embedded in $\mathcal{D}'(M)$.
Function spaces on $M_\pm$ will be defined as follows, where by abuse of language we write $\overline{M_\pm}$ instead of $M'\times \overline{\R_{\pm}}$.
\begin{definition}\label{def:restr_and_supp_space}
Let $E(M)\hookrightarrow\mathcal{D}'(M)$ be a topological vector space.
Then we introduce the \emph{restricted spaces} $\overline{E}(M_\pm)$ and the \emph{supported spaces} $\dot E(\overline{M_\pm})$ via
\begin{align*}
\overline{E}(M_\pm):=r^\pm E(M) \quad \text{and} \quad \dot E(\overline{M_\pm}):=\{f\in E(M)\mid \supp f\subseteq \overline{M_\pm}\}.
\end{align*}
If there is no danger of confusion, we simply write $E:=E(M)$, $\overline{E}_\pm:=\overline{E}(M_\pm)$, and $\dot{E}_\pm:=\dot{E}(\overline{M_\pm})$.
\end{definition}
We recall that by our convention, the topology on $\overline{E}_\pm$ is the final topology of $r^\pm$, while the supported space $\dot E_\pm$ is endowed with the subspace topology inherited from $E$.
If $E$ is normed, we choose the quotient norm $\|u\|_{\overline{E}_\pm}:=\inf\{\|U\|_E:U\in E, r^\pm U=u\}$ to realize this topology.
The following lemma records the expected behavior for the restriction operator on the supported spaces.
\begin{lemma}\label{lem:restr_supp_space}
Let $E\hookrightarrow\mathcal{D}'(M)$ be a topological vector space.
Then $\ker(r^{\pm}|_{E})=\dot E_\mp$.
In particular, $\dot E_\mp$ is a closed subspace of $E$, and it holds $r^\pm\dot{E}_\mp=\{0\}$ and $\overline{E}_\pm\cong E/\dot{E}_\mp$.
\end{lemma}
\begin{proof}
    Let $f\in \dot{E}_\mp$, i.e. $\supp f\subseteq \overline{M_\mp}$.
    Then for $\varphi\in \cald(M_\pm)$ we have $\supp e_0^\pm\varphi\subseteq M_\pm$ and thus
    \begin{align*}
        \langle r^\pm f,\varphi\rangle&=\langle f,e_0^\pm\varphi\rangle=0.
    \end{align*}
    Therefore $r^\pm f=0$, and thus $\dot{E}_\mp\subseteq \ker(r^\pm|_{E})$.
    Let conversely $f\in \ker(r^\pm|_{E})$.
    Observe that for $\varphi\in \mathcal D(M)$ with $\supp\varphi\subset M_\pm$ it holds $\varphi=e_0^\pm(\varphi|_{M_\pm})$, and thus
    \begin{align*}
        \langle f,\varphi\rangle = \langle r^\pm f,\varphi|_{M_\pm}\rangle = 0.    
    \end{align*}
    In other words, $\supp f\subseteq M\setminus M_\pm=\overline{M_\mp}$, i.e., $f\in \dot E_\mp$.
\end{proof}

The interplay between embeddings and the restricted and supported spaces is also straightforward from their definition.
\begin{lemma}\label{lemma:embeddings_restricted_spaces}
    Let $E, F \hookrightarrow \cald'(M)$ be topological vector spaces.
    If $E\hookrightarrow F$, then also $\overline{E}_\pm\hookrightarrow \overline{F}_\pm$ and $\dot E_\pm\hookrightarrow\dot F_\pm$.
\end{lemma}

\subsection{Operators Preserving Lower Support}\label{sec:half_supp}
Given a topological vector space $E\hookrightarrow\mathcal{D}'(M)$ and a (possibly non-local) linear operator $\opA: E\to \mathcal{D}'(M)$, it is not clear whether $\opA$ induces a well defined operator on the restricted spaces. To address this question, we introduce the notion of operators preserving half-support. 
\begin{definition}[half-support preserving operators]\label{def: supp_pres_op}
    Let $E,F\hookrightarrow \cald'(M)$ be topological vector spaces and $\opA\in \call(E,F)$.
    \begin{enumerate}
    \item $\opA$ \emph{preserves support in $\overline{M_\pm}$} (or \emph{preserves upper/lower support}) if for every $f\in \dot E_\pm$, one has $\opA f\in \dot F_\pm$.
    \item If $\opA$ preserves support in $\overline{M_\pm}$, the induced operator $\opA_\mp : \overline{E}_\mp\to \overline{F}_\mp$ is defined via
    \begin{align*}
     \opA_\mp u:=r^\mp \opA U,
    \end{align*}
    where $U\in E$ is any extension of $u$, that is, $r^\mp U=u$.
    \item If in addition $r^\pm$ is injective on $\dot{E}_\pm$ and hence $r^\pm\in \call_{\mathrm{iso}}(\dot{E}_\pm,r^\pm\dot{E}_\pm)$ by definition of the image topology, we call its inverse $e_0^\pm$ and define the operator $\dot{\opA}_{\pm}:r^\pm\dot E_\pm\to r^\pm\dot F_\pm$ via
    \begin{align*}
     \dot{\opA}_\pm u:=r^\pm \opA e_0^\pm u.
    \end{align*}
    \end{enumerate}
\end{definition}
The operator $\opA_\mp$ is well-defined.
Indeed, if $U_1, U_2\in E$ are two such extensions of $u$, then $\supp(U_1-U_2)\subseteq \overline{M_\pm}$. Since $\opA$ preserves support in $\overline{M_\pm}$, we have $\supp(\opA(U_1-U_2))\subseteq \overline{M_\pm}$, which in turn implies $r^\mp\opA U_1=r^\mp  \opA U_2$.
Indeed, by Lemma \ref{lem:restr_supp_space} and the universal property of the kernel, we may characterize $\opA_\mp$ as the unique operator $\opB\in \call(\overline{E}_\mp,\overline{F}_\mp)$ such that $r^\mp\opA=\opB r^\mp$.
Hence, if in addition $r^\mp$ is injective on $\dot{E}_\mp$ and $\opA$ preserves both upper and lower support, then $\dot{\opA}_\mp u = \opA_\mp u$ for all $u\in r^\mp\dot{E}_\mp$.
The same argument shows that if $\opA$ is defined on a larger space $H$, i.e., $E\hookrightarrow H\hookrightarrow \cald'(M)$, then for any $U\in H$ with $r^\mp U=u\in \overline{E}_\mp$ one has $\opA_{\mp}u=r^\mp\opA U$.

\medskip

The following proposition collects the basic properties of operators preserving lower support.
Analogous statements for operators preserving upper support can be obtained by symmetry.
\begin{proposition}\label{prop: boundedness_lower_support}
    Let $E, F,H\hookrightarrow \cald'(M)$ be topological vector spaces, and let $\opA\in \call(E,F)$ and $\opB\in \call(F,H)$ preserve lower support.
    \begin{enumerate}[leftmargin=*]
        \item\label{prop: boundedness_lower_supporti} It holds $\opA_+\in \call(\overline{E}_+,\overline{F}_+)$.
        If $E$, $F$ are normed, the operator norm of $\opA_+$ does not exceed that of $\opA$.
        \item\label{prop: boundedness_lower_supportii} $\opB\opA\in \call(E,H)$ preserves lower support, and $(\opB\opA)_+=\opB_+\opA_+$.
        \item\label{prop: boundedness_lower_supportiii} If $\opA\in \call_{\mathrm{iso}}(E,F)$ and $\opA^{-1}$ preserves lower support, then $\opA_+\in \call_{\mathrm{iso}}(\overline{E}_+,\overline{F}_+)$ with $(\opA_+)^{-1}=(\opA^{-1})_+=:\opA^{-1}_+$.
        \item\label{prop: boundedness_lower_supportiv} If $r^-$ is injective on $\dot{E}_-$, then $\dot\opA_-\in \call(r^-\dot{E}_-,r^-\dot{F}_-)$.
        If $r^-$ is also injective on $\dot{F}_-$, then $(\opB\opA)_-^{\boldsymbol{\cdot}}=\dot\opB_-\dot\opA_-$.
        If $\opA\in \call_{\mathrm{iso}}(E,F)$ and $\opA^{-1}$ preserves lower support as well, then $\dot\opA_-\in \call_{\mathrm{iso}}(r^-\dot{E}_-,r^-\dot{F}_-)$ with $(\dot\opA_-)^{-1}=(\opA^{-1})_-^{\boldsymbol{\cdot}}=:\dot{\opA}_-^{-1}$.
    \end{enumerate}
\end{proposition}
\begin{proof}
 \begin{enumerate}
    \item Since $\opA_+r^+=r^+\opA\in \call(E,\overline{F}_+)$, the continuity of $\opA_+$ follows from the definition of the final topology on $\overline{E}_+$.
    Let now $E$ and $F$ be normed and $u\in \overline{E}_+$.
    For any $U\in E$ such that $r^+U=u$, it holds $\opA_+u=r^+ \opA U$, and thus
    \begin{align*}
        \|\opA_+ u\|_{\overline{F}_+}&=\|r^+\opA U\|_{\overline{F}_+}
        \leq \|\opA U\|_{F}\leq \|\opA\| _{\mathcal{L}(E, F)} \|U\|_{E}.
    \end{align*}
    Since $\|u\|_{\overline{E}_+}=\inf\{ \|U\|_{E} \mid U\in E, r^+U=u\}$, this yields the claimed estimate.
    \item By definition, it is clear that $\opB\opA\in \call(E,H)$ preserves lower support.
    Thus, the previous point yields $(\opB\opA)_+\in \call(\overline{E}_+,\overline{H}_+)$, $\opA_+\in \call(\overline{E}_+,\overline{F}_+)$, and $\opB_+\in \call(\overline{F}_+,\overline{H}_+)$.
    Moreover, for $u\in \overline{E}_+$ and $U\in E$ with $r^+U=u$ we have
    \begin{align*}
        \opB_+\opA_+u=\opB_+r^+\opA U=r^+\opB\opA U=(\opB\opA)_+u,
    \end{align*}
    where $V:=\opA U\in F$ is an extension of $v:=r^+\opA U\in \overline{F}_+$.
    \item Follows directly from the previous point with $\opB:=\opA^{-1}$.
    \item Observe that by assumption, we have $e_0^-\in \call(r^-\dot{E}_-,\dot{E}_-)$, $\opA\in \call(\dot{E}_-,\dot{F}_-)$, and $r^-\in \call(\dot{F}_-,r^-\dot{F}_-)$.
    Therefore, $\dot{\opA}_-\in \call(r^-\dot{E}_-, r^-\dot{F}_-)$ is immediate.
    Moreover
    \begin{align*}
        (\opB\opA)_-^{\boldsymbol{\cdot}}=r^-\opB\opA e_0^-=r^-\opB e_0^-r^-\opA e_0^-=\dot\opB_-\dot\opA_-.
    \end{align*}
    The final assertion follows as in point \ref{prop: boundedness_lower_supportiii} with $\opB:=\opA^{-1}$.
    \qedhere
\end{enumerate}
\end{proof}
\begin{remark}
    A trivial but important consequence of Proposition \ref{prop: boundedness_lower_support} is that the unique solvability of the equation $\opA u=f$ on $M$ with $u\in E$ and $f\in F$ implies that of $\opA_+ u=f$ on $M_+$ with $u\in \overline{E}_+$ and $f\in \overline{F}_+$.
    More precisely, if $\opA\in \call(E,F)$ preserves lower support, then $\opA_+\in \mathcal{L}(\overline{E}_+, \overline{F}_+)$ is well-defined and one can consider the equation $\opA_+u=f$. If furthermore $\opA$ is an isomorphism and its inverse $\opA^{-1}$ preserves lower support as well, then the unique solution is given by $u=\opA^{-1}_+f$.
\end{remark}
\subsection{Admissible Operators}
Most differential operators of interest do not satisfy the hypothesis of Proposition \ref{prop: boundedness_lower_support}.
However, they often admit a factorization into two operators, one verifying the hypothesis of Proposition \ref{prop: boundedness_lower_support} and the other amenable to the abstract framework developed in Section \ref{sec:abstr_frame}, in particular Proposition \ref{prop: op_preserving_upper_support_abstr}.
This motivates the following notion of admissible operators.
\begin{definition}\label{def: admissible}
    Let $E, F \hookrightarrow \mathcal{D}'(M)$ be topological vector spaces such that $r^+$ is injective on $\dot{E}_+$.
    An operator $\opA\in\call_{\textrm{iso}}(E, F)$ is admissible if there exists a topological vector space $H \hookrightarrow\mathcal{D}'(M)$ such that $r^+$ is injective on $\dot{H}_+$, and  operators $\opA^-\in\call_{\textrm{iso}}(H,F)$, $\opA^+\in\call_{\textrm{iso}}(E,H)$ such that the following conditions are verified. 
    \begin{enumerate}
        \item $\opA=\opA^-\opA^+$.
        \item Both $\opA^+$ and $\opA^-$ preserve lower support.
        \item Both $\opA^+$ and $(\opA^+)^{-1}$ preserve upper support. 
        \item $(\opA^-)^{-1}$ preserves lower support.
    \end{enumerate}
\end{definition}

By Proposition \ref{prop: boundedness_lower_support}, $\opA$ preserves lower support.
Moreover, $\opA_+=\opA^-_+\opA^+_+\in \call(\overline{E}_+,\overline{F}_+)$, while $\opA^-_+\in\call_{\mathrm{iso}}(\overline{H}_+, \overline{F}_+)$ with the inverse given by $(\opA^-)^{-1}_+$.
Observe that $\opA^+_+$ need not belong to $\call_{\mathrm{iso}}(\overline{E}_+, \overline{H}_+)$, since it generally fails to be injective.
The crucial point is that $\opA^+$ preserves \emph{both upper and lower} support, so that the restriction of $\opA^+_+$ to $r^+\dot{E}_+$ is given by the operator $\dot{\opA}_+^+$, which is an isomorphism by Proposition \ref{prop: boundedness_lower_support}, namely $\dot{\opA}_+^+\in \call_{\mathrm{iso}}(r^+\dot{E}_+,r^+\dot{H}_+)$.
As discussed in Section \ref{sec:abstr_frame}, invertibility on the whole of $\overline{E}_+$ is thus restored after suitable boundary data are prescribed.

To apply this section, we thus assume that the image topologies on $r^+\dot E_+$ and $r^+\dot H_+$ induced by $r^+$ agree with the corresponding subspace topologies, and that we are given spaces $\cale$ and $\calh$ and operators $\Tr_E\in \call(\overline{E}_+,\cale)$, $\Ext_E\in \call(\cale,\overline{E}_+)$, and $\mathsf V\in \call(\overline{H}_+,\calh)$ such that $(\overline{E}_+,r^+\dot{E}_+,\cale,\Tr_E,\Ext_E)$ is a trace tuple and $\ker \mathsf V=r^+\dot{H}_+$.
Given, in addition, a topological vector space $\mathcal{G}$ and a ``boundary'' operator $\opB\in\call(\mathcal{E},\mathcal{G})$, our strategy for solving the boundary value problem 
\begin{equation*}
    \left\{\begin{array}{l}
     \opA_+ u = \opA^-_+\opA^+_+u = f    \\
     \opB \Tr_E u=g,     
    \end{array}\right.
\end{equation*}
for $f\in \overline{F}_+$ and $g\in \mathcal{G}$ is as follows:
first invert the factor $\opA_+^-$, which is already an isomorphism on the restricted spaces and therefore requires no additional boundary analysis. The remaining problem is then to invert $\opA_+^+$, for which Proposition \ref{prop: op_preserving_upper_support_abstr} provides the appropriate framework.
Therefore, it is natural to require $\opB$ to satisfy the abstract complementing boundary condition with respect to $\opA_+^+$ as defined in Definition \ref{def: abstract_complementing_condition_upper}, that is, that $\opC=\mathsf V\circ \opA_+^+\circ \Ext_E$ is such that $(\opB,\opC)\in \call(\cale,\calg\times\calh)$ is injective.
We now present the main result of this section. 
\begin{theorem}\label{thm: abstract_complementing_condition}
 Let $\opA= \opA^-\opA^+$ be admissible with notations as in Definition \ref{def: admissible}.
 Let $\opB$ satisfy the abstract complementing boundary condition in Definition \ref{def: abstract_complementing_condition_upper} for $\opA_+^+$.
  Endow
  \begin{align*}
      \Comp:=\{(f,g)&\in \overline{F}_+\times \mathcal{G}\mid (g,\mathsf V(\opA_+^-)^{-1}f)\in \mathsf{ran}(\opB,\opC)\}
      \end{align*}
with the initial topology of the map
  \begin{align*}
      \Comp\ni (f,g)\mapsto (f,b(f,g))\in \overline{F}_+\times \mathcal{E},
  \end{align*}
  where $b(f,g):=(\opB,\opC)^{-1}(g,\mathsf V(\opA_+^-)^{-1}f)$.
  Then $S:=(\opA_+,\opB\Tr_E)$ belongs to $\call_{\mathrm{iso}}(\overline{E}_+,\Comp)$ with inverse
  \begin{align*}
      S^{-1}(f,g)=\Ext_E b(f,g) + (\dot{\opA}_+^+)^{-1}((\opA_+^-)^{-1}f - \opA_+^+\Ext_E b(f,g)).
  \end{align*}
  In particular, for all $(f,g)\in \Comp$ there is a unique $u\in \overline{E}_+$ such that
    \begin{equation}\label{eqn: general_problem}
    \left\{\begin{array}{rcl}
     \opA_+ u &=& f    \\
     \opB \Tr_E u &=&g.     
    \end{array}\right.
    \end{equation}
In case the considered spaces are normed, it holds 
\begin{equation*}
    \|u\|_{\overline{E}_+} \lesssim \|f\|_{\overline{F}_+}+ \|b(f,g)\|_{\mathcal{E}}.
\end{equation*}
\end{theorem}
\begin{proof}
    Choose $U:=\overline{E}_+$, $V:=\overline{H}_+$, $\dot{U}:=r^+\dot{E}_+$, $\dot{V}:=r^+\dot{H}_+$, $\calu:=\cale$, $\calv:=\calh$, $\Tr:=\Tr_E$, $\opE:=\Ext_E$, and $\opT:=\opA_+^+$.
    Then Theorem \ref{js007} shows that $S_\opB=(\opA_+^+,\opB\Tr_E)\in \call_{\mathrm{iso}}(U,\Comp_{\opB,\opC})$, where we recall that
    \begin{align*}
      \Comp_{\opB,\opC}:=\{(h,g)&\in \overline{H}_+\times \mathcal{G}\mid (g,\mathsf V h)\in \mathsf{ran}(\opB,\opC)\}
      \end{align*}
    is endowed with the initial topology of the map
  \begin{align*}
      \Comp_{\opB,\opC}\ni (h,g)\mapsto (h,(\opB,\opC)^{-1}(g,\mathsf V h))\in \overline{H}_+\times \mathcal{E}.
  \end{align*}
    Now consider the map $R:\Comp_{\opB,\opC}\to \overline{F}_+\times\calg$ defined by $R(h,g):=(\opA_+^-h,g)$.
    Since $\opA_+^-\in \call_{\mathrm{iso}}(\overline{H}_+,\overline{F}_+)$, it holds that $R\in \call_{\mathrm{iso}}(\Comp_{\opB,\opC},\Comp)$ by virtue of the chosen initial topologies on $\Comp_{\opB,\opC}$ and $\Comp$.
    Therefore, the result follows in light of $S=R\circ S_\opB$.
\end{proof}
Let us emphasize that in the previous theorem, the space $\overline{F}_+$ does not need to carry any additional boundary structure, while the requirements on $\overline{H}_+$ are considerably weaker than those on $\overline{E}_+$.
Indeed, we do not require an abstract trace space for $\overline{H}_+$, but only a space $\calh$ and an operator $\mathsf V\in\call(\overline{H}_+,\calh)$ such that $\ker\mathsf V=r^+\dot H_+$.
The existence of such an operator may seem restrictive at first if one has in mind an elliptic equation of the form $\lambda u-\Delta u=f$ with $f\in\overline{F}_+:=W^{-1,p}(\R^d_+)$.
In this situation, the solution space is $\overline{E}_+=W^{1,p}(\R^d_+)$, which admits Sobolev traces, while the intermediate space is $\overline{H}_+=L^p(\R^d_+)$, which does not.
Nevertheless, as discussed in Example \ref{js009}, the space $L^p(\R^d_+)$ fits perfectly into the abstract framework: in this case $r^+\dot H_+=\overline{H}_+$, so we may simply take $\calh=\{0\}$ and $\mathsf V=0$.

\medskip

Finally, in analogy to Lemma \ref{js012}, no compatibility conditions remain in the case that $(\opB,\opC)\in \call(\cale,\calg\times\calh)$ is an isomorphism.
\begin{lemma}\label{js013}
If in the situation of Theorem \ref{thm: abstract_complementing_condition} it holds $(\opB,\opC)\in \call_{\mathrm{iso}}(\cale,\calg\times\calh)$, then $\Comp=\overline{F}_+\times\calg$, and the initial topology on $\Comp$ is the product topology of $\overline{F}_+\times\calg$.
\end{lemma}

\section{Newton Polygon Analysis}\label{sec:newt}
\subsection{Newton Polygons and Order Functions}\label{sec:newt_poly}
In this work, we are particularly interested in differential operators $\opL(\partial_t,\partial_x)$ whose symbol $L(\tau,\xi)$, $(\tau,\xi)\in\R\times\R^n$, is neither homogeneous nor quasi-homogeneous, where quasi-homogeneity means that there exist $N\ge 0$ and $\rho>0$ such that $L(\eta^\rho\tau,\eta\xi)=\eta^N L(\tau,\xi)$ for all $(\tau,\xi)\in\R\times\R^n$ and $\eta>0$.\footnote{cf.\ Example \ref{js004}.}
The analysis of such mixed-order operators is considerably more involved, since many of the standard techniques for elliptic and parabolic equations rely on the presence of a single (quasi)-homogeneity.

To describe the interaction of several competing homogeneities, the notion of a Newton polygon was introduced, see \cite{denk1998newton,denk2013general}.
Rather than encoding the symbol by a single homogeneity, a Newton polygon records the different homogeneities present in the symbol simultaneously.
In the following, we develop the part of this theory needed for the present work from a viewpoint that emphasizes the convex-geometric and algebraic structure underlying Newton polygons.
\begin{definition}
    Let $E=\{(x_1, y_1),...,(x_M, y_M)\}\subseteq [0, \infty)^2$ be a finite set. The Newton polygon $\mathcal{N}(E)$ of $E$ is the convex envelope of the elements of $E$ and of their projections to the coordinate axes, i.e.,
    \begin{equation*}
        \mathcal{N}(E):=\conv\left (E\cup \bigcup_{m=1}^M \{(x_m, 0), (0, y_m)\} \cup\{(0,0)\}\right). 
    \end{equation*}
    Given a Newton polygon $\mathcal{N}$, we denote by $\mathcal{N}_V:=\{(r_0,s_0), (r_1, s_1),..., (r_{J+1}, s_{J+1})\}$ the set of its vertices, ordered counterclockwise starting from $(r_0, s_0)=(0,0)$.
    If $J\ge 1$, we say that $\mathcal{N}$ is \emph{regular in time} if $r_2\neq r_1$.
    It is \emph{regular in space} if $s_{J+1}\neq s_J$.
    In particular, we say that $\mathcal{N}$ is \emph{regular} if it is regular in both time and space.
\end{definition}
The rectangular Newton polygon\footnote{\label{footnote002}Note the flip in notation: $\caln_{(s,r)}$ corresponds to the point $(r,s)$ in the plane, in accordance with footnote \ref{footnote001}.} $\caln_{(s,r)}:=\caln(\set{r,s})$, $r,s\in [0,\infty)$, is a particular example of a Newton polygon that is neither regular in space nor regular in time, see Figure \ref{figure002}.

\medskip

We first record the elementary structure of the class of Newton polygons. For $y\in[0,\infty)$, set
\begin{align}\label{def:elem_polyg}
\mathcal{N}_\infty:=\conv\set{(0,0),(0,1)},
\qquad
\mathcal{N}_y:=\conv\set{(0,0),(1,0),(0,y)}.
\end{align}
\begin{figure}[t]
\centering

\begin{subfigure}[t]{0.23\textwidth}
\centering
\begin{tikzpicture}
  
  \filldraw[draw=Sepia,fill=Sepia!30, thick]
  (0,0)--(1.7,0)--(0, 2) -- cycle;
  
  \draw[->] (-0.5,0) -- (2.5,0) node[below] {$\vert \xi\vert$};
  \draw[->] (0,-0.5) -- (0,2.5) node[left] {$\tau$};

  \fill[Sepia] (0,0) circle (1.5pt);
  
  \fill[Sepia] (1.7,0) circle (1.5pt);
  \node[below] at (1.7,0) {$(1,0)$};

  \fill[Sepia] (0,2) circle (1.5pt);
  \node[left] at (0,2) {$(0,y)$};

  \node at (.6,.7) {$\mathcal{N}_y$};
\end{tikzpicture}
\end{subfigure}
\hfill
\begin{subfigure}[t]{0.35\textwidth}
\centering
\begin{tikzpicture}
  
  \filldraw[draw=Sepia,fill=Sepia!30, thick]
  (0,0)--(1.7,0)--(1.7, 2)-- (0,2) -- cycle;
  
  \draw[->] (-0.5,0) -- (2.5,0) node[below] {$\vert \xi\vert$};
  \draw[->] (0,-0.5) -- (0,2.5) node[left] {$\tau$};

  \fill[Sepia] (0,0) circle (1.5pt);
  
  \fill[Sepia] (1.7,0) circle (1.5pt);
  \node[below] at (1.7,0) {$(r,0)$};

  \fill[Sepia] (1.7,2) circle (1.5pt);
  \node[above right] at (1.7,2) {$(r,s)$};

  \fill[Sepia] (0, 2) circle (1.5pt);
  \node[left] at (0,2) {$(0,s)$};

  \node at (.9,1) {$\mathcal{N}_{(s,r)}$};
\end{tikzpicture}
\end{subfigure}
\hfill
\begin{subfigure}[t]{0.36\textwidth}
\centering
\begin{tikzpicture}
  \fill[Sepia!30, thick]
  (0,0)--(3,0)--(2, 1.5)-- (1,2)-- (0,2.25) -- cycle;
  \fill[Sepia!40] (3,0) -- (2,1.5) -- (2,0) -- cycle;
  \fill[Sepia!40] (2,1.5) -- (1,2) -- (1,1.5) -- cycle;
  \fill[Sepia!40] (1,2) -- (0,2.25) -- (0,2) -- cycle;
  \draw[draw=Sepia, thick]
  (0,0)--(3,0)--(2, 1.5)-- (1,2)-- (0,2.25) -- cycle;
  
  \draw[->] (-0.5,0) -- (3.7,0) node[below] {$\vert \xi\vert$};
  \draw[->] (0,-0.5) -- (0,2.5) node[left] {$\tau$};

  \fill[Sepia] (0,0) circle (1.5pt);
  \node[below left] at (0,0) {$(r_0,s_0)$};
  
  \fill[Sepia] (3,0) circle (1.5pt);
  \node[below] at (2.8,0) {$(r_1,s_1)$};

  \fill[Sepia] (2,1.5) circle (1.5pt);
  \node[right] at (2,1.55) {$(r_2,s_2)$};

  \fill[Sepia] (1, 2) circle (1.5pt);
  \node[above right] at (1,2) {$(r_3,s_3)$};

  \fill[Sepia] (0, 2.25) circle (1.5pt);
   \node[left] at (0,2.2) {$(r_4,s_4)$};

  \node at (1.3,1) {$\mathcal{N}$};
\end{tikzpicture}
\end{subfigure}
\caption{Visualization of an elementary Newton polygon $\caln_y$, a rectangular Newton polygon $\caln_{(s,r)}$, and a generic Newton polygon $\caln$ generated by the three shaded elementary Newton polygons.}
\label{figure002}
\end{figure}
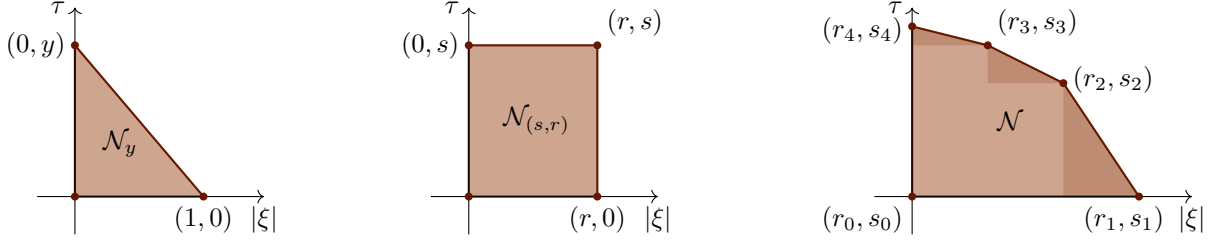
The following proposition shows that the set of all Newton polygons is the convex cone\footnote{Here, (Minkowski) addition and scalar multiplication are defined by $K+L:=\set{x+y\mid x\in K, y\in L}$ and $\alpha K:=\set{\alpha x\mid x\in K}$ for $\alpha\in \R$ and $K,L\subseteq\R^2$. For convex $K$ and $m\in\N$, $m$-fold repeated addition equals $mK$.} generated by these elementary Newton polygons.

\begin{proposition}\label{prop: elementary_decomposition_Newton_polygon}
The set of Newton polygons is a convex cone generated by the elementary Newton polygons $\mathcal{N}_y$, $y\in [0,\infty]$.
\end{proposition}
\begin{proof}
It follows directly from the definition that a nonnegative scalar multiple of a Newton polygon is again a Newton polygon.
The same holds for Minkowski sums: if $\mathcal{N}_1=\mathcal{N}(E_1)$ and $\mathcal{N}_2=\mathcal{N}(E_2)$, then $\mathcal{N}_1+\mathcal{N}_2$ is a compact convex subset of $[0,\infty)^2$ containing the origin, and together with each of its points it contains its projections onto the coordinate axes.
Since it is a polygon, it is therefore a Newton polygon.

\medskip

It remains to show that every Newton polygon is generated by the elementary ones. Let $\mathcal{N}$ be a Newton polygon with vertices $\mathcal{N}_V=\set{(r_0,s_0),(r_1,s_1),\ldots,(r_{J+1},s_{J+1})}$, ordered counterclockwise as in the definition.
For $j\in\set{1,\ldots,J}$, set
\begin{align*}
\alpha_j:=r_j-r_{j+1},
\qquad
\beta_j:=s_{j+1}-s_j.
\end{align*}
Thus the edge joining $(r_j,s_j)$ and $(r_{j+1},s_{j+1})$ has edge vector $(-\alpha_j,\beta_j)$. If $\alpha_j>0$, set $y_j:=\beta_j/\alpha_j$ and $\alpha_j':=\alpha_j$.
If $\alpha_j=0$, set $y_j:=\infty$ and $\alpha_j':=\beta_j$.
Then $\alpha_j'\mathcal{N}_{y_j}$ is the triangle $\conv\set{(0,0),(\alpha_j,0),(0,\beta_j)}$ if $\alpha_j,\beta_j>0$, and a segment if $\alpha_j=0$ or $\beta_j=0$.

\medskip

Set $\widetilde{\mathcal{N}}:=\sum_{j=1}^J\alpha_j'\mathcal{N}_{y_j}$.
We show that $\widetilde{\mathcal{N}}=\mathcal{N}$ by comparing their horizontal sections.
Fix $s\in[0,s_{J+1}]$.
A point of $\widetilde{\mathcal{N}}$ at height $s$ is obtained by choosing numbers $t_j\in[0,\beta_j]$ such that $\sum_{j=1}^Jt_j=s$.
At height $t_j$, the right endpoint of the horizontal section of $\alpha_j'\mathcal{N}_{y_j}$ has first coordinate $\alpha_j-t_j/y_j$ if $y_j>0$, and $\alpha_j$ if $y_j=0$.
Consequently, finding the right endpoint of the horizontal section of $\widetilde{\mathcal{N}}$ amounts to minimizing $\sum_{j=1}^J t_j/y_j$ under the constraint $\sum_{j=1}^Jt_j=s$.
The convexity of $\mathcal{N}$ implies that the quantities $y_j$ are nonincreasing, so that this minimum is attained by using the available height first in the summands with the largest $y_j$.
Let $k\in\set{1,\ldots,J}$ be such that
\begin{align*}
s=s_k+\theta\beta_k,
\qquad
\theta\in[0,1].
\end{align*}
Since $s_k=\sum_{j=1}^{k-1}\beta_j$, the minimizing choice is
$t_j=\beta_j$ for $j<k$, $t_k=\theta\beta_k$, and $t_j=0$ for $j>k$.
Hence the right endpoint of the horizontal section of $\widetilde{\mathcal{N}}$ at height $s$ has first coordinate
\begin{align*}
(1-\theta)\alpha_k+\sum_{j=k+1}^J\alpha_j =(1-\theta)r_k+\theta r_{k+1},
\end{align*}
where we used $r_{J+1}=0$.
This is precisely the first coordinate of the point at height $s$ on the edge joining $(r_k,s_k)$ and $(r_{k+1},s_{k+1})$.

Therefore every horizontal section of $\widetilde{\mathcal N}$ coincides with the corresponding horizontal section of $\mathcal N$.
Since both are Newton polygons, it follows that $\mathcal{N}=\widetilde{\mathcal{N}}$.
Consequently, the elementary Newton polygons generate the cone of all Newton polygons.
\end{proof}
It will be convenient to enlarge this cone to a linear space in which differences of Newton polygons can also be considered.
Abstractly, this can be achieved by taking the Grothendieck completion of the commutative monoid of Newton polygons under Minkowski addition, whose elements are virtual polygons as introduced in \cite{PuK92}; see also \cite{Sch18}.
We shall instead work with a more concrete analytic realization in terms of support functions.
The support function of a convex body $K\subseteq \R^2$ is defined by
\begin{align*}
    \mu(K)(x) := \max\set{\langle v,x\rangle\mid v\in K}, \qquad x\in\R^2.
\end{align*}
Support functions satisfy $\mu(K+L)=\mu(K)+\mu(L)$ and $\mu(\alpha K)=\alpha\mu(K)$ for $\alpha>0$, and they determine the body uniquely.
Thus, the cone of Newton polygons is identified with a cone of functions whose linear span is a concrete realization of the corresponding Grothendieck completion.
This leads to the following definition, for which we recall the definition of the elementary Newton polygons $\caln_y$ in \eqref{def:elem_polyg}.
\begin{definition}\label{def_order_function}
\begin{enumerate}
    \item For $y\in [0,\infty]$, we define the \emph{elementary order function} $o_y:\R^2\to \R$ by the support function of $\mathcal{N}_y$, i.e., $o_y(x):=\max\set{\langle v,x\rangle\mid v\in \mathcal{N}_y}$.
    \item An \emph{order function} is a finite linear combination of elementary order functions.
    \item A \emph{strictly positive} order function is a finite conical linear combination of elementary order functions, i.e., all coefficients are non-negative.
    \item For $(r,s)\in \R^2$, we define the order function $\mu_{(s,r)}:=so_\infty+ro_0$.
\end{enumerate}
\end{definition}
In particular, every order function is the difference of two strictly positive order functions.
Let us record that strictly positive order functions are in one-to-one correspondence with Newton polygons.
\begin{proposition}\label{prop: correspondence_order_functions}
Let $\mu=\sum_{\ell=1}^L a_\ell o_{y_\ell}$, $a_\ell\ge 0$, $y_\ell\in[0,\infty]$,
be a strictly positive order function, and define $\mathcal{N}(\mu):=\sum_{\ell=1}^L a_\ell\mathcal{N}_{y_\ell}$.
Then $\mathcal{N}(\mu)$ is a Newton polygon and
\begin{align*}
\mu(x)=\max\set{\langle v,x\rangle\mid v\in \mathcal{N}(\mu)}, \qquad x\in \R^2.
\end{align*}
Conversely, let $\mathcal{N}$ be a Newton polygon and define $\mu(\mathcal{N}):\R^2\to\R$ via
\begin{align*}
\mu(\mathcal{N})(x):=\max\set{\langle v,x\rangle\mid v\in\mathcal{N}}.
\end{align*}
Then $\mu(\mathcal{N})$ is a strictly positive order function.
Moreover, $\mathcal{N}(\mu(\mathcal{N}))=\mathcal{N}$ and  $\mu(\mathcal{N}(\mu))=\mu$.
\end{proposition}
\begin{proof}
Follows directly from the fact that the cone of Newton polygons is generated by the elementary Newton polygons and that it is identified with the cone of support functions.
\end{proof}
In particular, it holds $\mu_{(s,r)}=\mu(\caln_{(s,r)})$ for $r,s\in [0,\infty)$.
\begin{remark}\label{js015}
Our definition of strictly positive order functions differs from that of \cite{denk2013general}, but the two notions are equivalent because by Proposition \ref{prop: correspondence_order_functions} above and Remark 2.22 in \cite{denk2013general}, both are in one-to-one correspondence with Newton polygons.
To be more precise, the special geometry of Newton polygons implies that $\mu(x)$ depends only on the positive part $x_+:=(\max\{x_1,0\},\max\{x_2,0\})$:
The estimate $\mu(\mathcal N)(x)\leq\mu(\mathcal N)(x_+)$ follows since $\mathcal N\subseteq[0,\infty)^2$ and thus $\langle v,x\rangle\leq\langle v,x_+\rangle$ for every $v\in\mathcal N$.
For the converse inequality, let $v\in\mathcal N$ be such that $\mu(\mathcal N)(x_+)=\langle v,x_+\rangle$.
Since Newton polygons are closed under coordinatewise decrease, the point
$\widetilde v\in[0,\infty)^2$ defined by
\begin{align*}
    \widetilde v_j:=
    \begin{cases}
        v_j,& x_j\geq0,\\
        0,& x_j<0,
    \end{cases}
    \qquad j\in\{1,2\},
\end{align*}
belongs to $\mathcal N$.
Hence $\mu(\mathcal N)(x)\geq\langle\widetilde v,x\rangle=\langle v,x_+\rangle=\mu(\mathcal N)(x_+)$, so that in summary
\begin{align*}
    \mu(\mathcal N)(x) =\mu(\mathcal N)(x_+).
\end{align*}
Parametrizing vectors in $(0,\infty)^2$ by $x=r(1,\gamma)$, $r,\gamma\in (0,\infty)$, and using the homogeneity $\mu(rx)=r\mu(x)$, we see that all information contained in $\mu$ is tested by the single number $\gamma\in (0,\infty)$.
This is exactly the point of view taken in \cite{denk2013general}.
\end{remark}

\subsection{Weight Functions}

For homogeneous symbols $p(\xi)$ depending only on $\xi$, a classical way to define ellipticity is to ask for a lower estimate of the form $\vert p\vert \geq C \vert \xi\vert^d$ with $d$ the degree of homogeneity.
The weight functions associated with Newton polygons generalize this idea by providing two-sided estimates for more complex behaviors.

Let $\mathcal{N}$ be a Newton polygon and $\mathcal{N}_V$ the set of its vertices.
We use the coordinatewise partial order on $\mathbb R^2$ given by
\begin{align*}
    (r,s)\preceq(\widetilde r,\widetilde s)
    \quad\Longleftrightarrow\quad
    r\leq \widetilde r
    \ \text{and}\ 
    s\leq \widetilde s.
\end{align*}
Then we define the set $\mathcal{N}_{SV}$ of \emph{significant vertices} as
\begin{align*}
    \mathcal{N}_{SV}:=\set{v\in \mathcal{N}_V\mid (\forall \widetilde v \in \mathcal{N}) \ (v\preceq \widetilde v \Rightarrow v=\widetilde v)},
\end{align*}
see Figure \ref{figure001}.
Specifically, for a nontrivial Newton polygon $\caln\ne \set{(0,0)}$ we have
\begin{equation*}
    \mathcal{N}_{SV}= \left\{\begin{array}{ll}
    \mathcal{N}_V\setminus\{(0,0)\}     &\text{if $\mathcal{N}$ is regular,}  \\
    \mathcal{N}_V\setminus \{(0,0), (r_1, 0)\}     & \text{if $\mathcal{N}$ is regular in space but not in time,}\\
    \mathcal{N}_V \setminus \{(0,0), (0,s_{J+1})\}  & \text{if $\mathcal{N}$ is regular in time but not in space,}\\
    \mathcal{N}_V \setminus \{(0,0), (r_1, 0), (0, s_{J+1})\}  & \text{if $\mathcal{N}$ is neither regular in space nor in time}.
    \end{array}\right.
\end{equation*}
\begin{figure}[ht]
\centering

\begin{subfigure}[t]{0.48\textwidth}
\centering
\begin{tikzpicture}
  \begin{scope}
  \clip (3,1) -- (5,1) -- (5,3) -- (3,3) -- cycle;
  \shade[inner color=Sepia!40,outer color=white] (3,1) circle[radius=2];
  \end{scope}
  \node[Sepia] at (4,1.75) {$\{v\preceq\widetilde{v}\}$};
  
  \filldraw[draw=Sepia,fill=Sepia!30, thick]
  (0,0)--(3,0)--(3, 1)-- (0,2.5)--cycle;
  
  \draw[->] (-0.5,0) -- (5,0) node[below] {$\vert \xi\vert$};
  \draw[->] (0,-0.5) -- (0,4) node[left] {$\tau$};

  \fill[Sepia] (0,0) circle (1.5pt);
  
  \fill[blue] (3,1) circle (1.5pt);
  \node[below left, font=\small] at (3,1) {$v$};

  \fill[Sepia] (0,2.5) circle (1.5pt);

  \fill[Sepia] (3, 0) circle (1.5pt);
\end{tikzpicture}
\end{subfigure}
\hfill
\begin{subfigure}[t]{0.48\textwidth}
\centering
\begin{tikzpicture}
  \begin{scope}
  \clip (3,0) -- (5,0) -- (5,2) -- (3,2) -- cycle;
  \shade[inner color=Sepia!40,outer color=white] (3,0) circle[radius=2];
  \end{scope}
  \node[Sepia] at (4,.75) {$\{w\preceq\widetilde{v}\}$};
  
  \filldraw[draw=Sepia,fill=Sepia!30, thick]
  (0,0)--(3,0)--(3, 1)-- (0,2.5)--cycle;
  
  \draw[->] (-0.5,0) -- (5,0) node[below] {$\vert \xi\vert$};
  \draw[->] (0,-0.5) -- (0,4) node[left] {$\tau$};

  \fill[Sepia] (0,0) circle (1.5pt);
  
  \fill[blue] (3,1) circle (1.5pt);
  \node[below left, font=\small] at (3,1) {$v$};

  \fill[Sepia] (0,2.5) circle (1.5pt);

  \fill[red] (3, 0) circle (1.5pt);
  \node[above left, font=\small] at (3, 0) {$w$};
\end{tikzpicture}
\end{subfigure}
\caption{Visualization of a significant vertex $v$ (left) and a non-significant vertex $w$ (right).
Observe that $w\preceq v$, so that the vertex $v$ obstructs the significance of the vertex $w$.}
\label{figure001}
\end{figure}
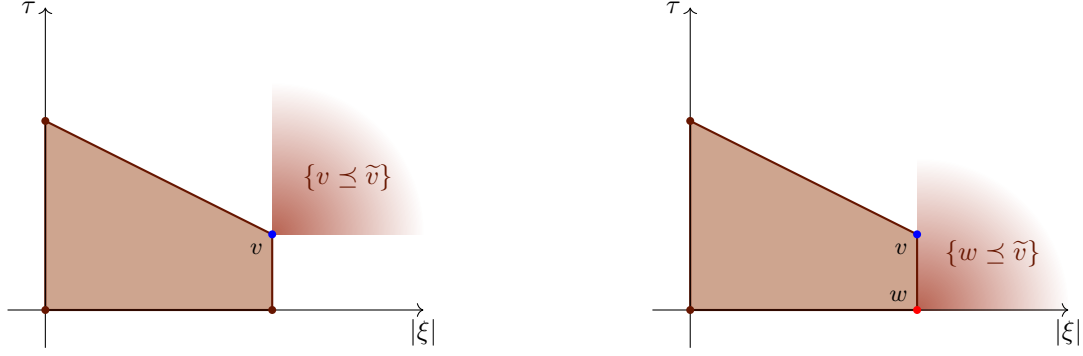
\begin{definition}\label{def: gen_weight_fct}
Let $\mu$ be an order function.
    \begin{enumerate}
        \item If $\mu$ is strictly positive, a symbol $w$ of the form $w= \sum_{i=1}^I a_i\langle \tau\rangle^{s_i}\langle \xi \rangle ^{r_i}$ with $a_i>0$, $i\in\set{1,\ldots,I}$, is called a \emph{weight function for} $\mu$, if $\mathcal{N}_{SV}\subseteq \{(r_i, s_i) \mid i\in\set{1,..., I}\}\subseteq \caln$, where $\caln:=\caln(\mu)$ is the corresponding Newton polygon.
        \item In the general case, choose strictly positive order functions $\mu_1$, $\mu_2$ such that $\mu=\mu_1-\mu_2$.
        If $w_i$ is a weight function for $\mu_i$, $i\in\set{1,2}$, we call $w:=\frac{w_1}{w_2}$ a weight function for $\mu$.
    \end{enumerate}
\end{definition}
A prominent example of a weight function for a strictly positive order function $\mu$ is $w_\mu:=\sum \langle\tau\rangle^{s}\langle\xi\rangle^{r}$, where the sum runs over all $(r,s)\in \caln(\mu)_{SV}$.
For elementary order functions this gives
    \begin{align*}
        w_{o_0}(\tau, \xi)=\langle\xi\rangle, \quad w_{o_\infty}(\tau, \xi)=\langle\tau\rangle \quad \text{and} \quad w_{o_{y}}(\tau, \xi) =\langle\tau\rangle^y+\langle \xi\rangle, \quad y\in (0,\infty).
    \end{align*}
For $s,r\in\R$, the monomial $w_{(s,r)}(\tau,\xi):= \langle\tau\rangle^s\langle\xi\rangle^r$ is a weight function for $\mu_{(s,r)}$.

We next show that for two order functions $\mu$ and $\nu$, the product of a weight function for $\mu$ and a weight function for $\nu$ is a weight function for $\mu+\nu$.
The following geometric lemma will be useful.
\begin{lemma}\label{lem:significant_vertices_sum}
\begin{enumerate}
    \item\label{lem:significant_vertices_sumi} Let $P,Q\subseteq\mathbb{R}^2$ be convex polygons. Then every vertex of the sum polygon $P+Q$ is the sum of a vertex of $P$ and a vertex of $Q$.
    \item\label{lem:significant_vertices_sumii} Let $\mathcal N_1$ and $\mathcal N_2$ be Newton polygons. Then every significant vertex of $\mathcal N_1+\mathcal N_2$ can be written as $v=v_1+v_2$ with $v_k\in(\mathcal N_k)_{SV}$ for $k\in\{1,2\}$.
    \item\label{lem:significant_vertices_sumiii} Let $\mathcal{N}$ be a Newton polygon.
    For all $x\in[0,\infty)^2$ it holds $\mu(\caln)(x)=\max\set{\langle v,x\rangle\mid v\in \caln_{SV}}$.
\end{enumerate}
\end{lemma}
\begin{proof}
\begin{enumerate}
    \item Let $z$ be a vertex of $P+Q$, and choose $x\in P$, $y\in Q$ such that $z=x+y$.
    Suppose that $x_1,x_2\in P$ are such that $x\in [x_1,x_2]$.
    Then $z\in [x_1+y,x_2+y]$, and since $z$ is an extreme point of $P+Q$, it holds $z\in \set{x_1+y,x_2+y}$.
    Consequently, $x\in \set{x_1,x_2}$, and thus $x$ is an extreme point of $P$, that is, a vertex.
    By the same argument, $y$ is a vertex of $Q$.
    \item Let $v\in(\mathcal N_1+\mathcal N_2)_{SV}$.
    By part \ref{lem:significant_vertices_sumi}, there exist vertices $v_k\in(\mathcal N_k)_V$, $k\in\{1,2\}$, such that $v=v_1+v_2$.
    Suppose that $v_1\notin(\mathcal N_1)_{SV}$, i.e., there exists $\widetilde v_1\in\mathcal N_1$ such that $v_1\preceq\widetilde v_1$ and $v_1\neq\widetilde v_1$.
    Consequently,
    \begin{align*}
        v=v_1+v_2 \preceq \widetilde v_1+v_2,
        \qquad
        v\neq\widetilde v_1+v_2.
    \end{align*}
    Since $\widetilde v_1+v_2\in\mathcal N_1+\mathcal N_2$, this contradicts $v\in(\mathcal N_1+\mathcal N_2)_{SV}$.
    Thus $v_1\in(\mathcal N_1)_{SV}$.
    The same argument shows that $v_2\in(\mathcal N_2)_{SV}$.
    \item For $x=(0,0)$ the assertion is clear.
    Let $(0,0)\ne x\in[0,\infty)^2$ and consider the affine space
    \begin{align*}
        U_x:=\set{v\in \R^2\mid \langle v,x\rangle=\mu(x)}.
    \end{align*}
    Since $U_x$ is one-dimensional and $\caln$ is convex and compact, $F_x:=U_x\cap \caln$ is a nonempty line segment.
    Choose $v\in F_x$ maximizing $v_1+v_2$.
    By linearity, we may assume that $v$ is an endpoint of $F_x$, and hence a vertex of $\mathcal N(\mu)$.
    If $v$ were not significant, there would exist $\widetilde v\in\mathcal N(\mu)$ with $v\preceq\widetilde v$ and $v\neq\widetilde v$.
    Since $x\in[0,\infty)^2$, we would have $\mu(\caln)(x)=\langle v,x\rangle\leq\langle\widetilde v,x\rangle\leq\mu(\caln)(x)$, so that $\widetilde v\in F_x$.
    But $\widetilde v_1+\widetilde v_2>v_1+v_2$, contradicting the choice of $v$.
    Thus, $F_x$ contains a significant vertex of $\caln(\mu)$, which shows the assertion.
\end{enumerate}
\end{proof}
\begin{corollary}\label{cor: product_weight_functions}
    Let $\mu_1, \mu_2$ be order functions, and $w_k$ a weight function for $\mu_k$, $k\in\set{1,2}$.
    Then $w_1 w_2$ is a weight function for $\mu_1+\mu_2$, and $w_1^{-1}$ is a weight function for $-\mu_1$.
    Moreover, for every $\alpha\in \R$ there exists a weight function $w_\alpha$ for $\alpha\mu_1$ such that $w_1^\alpha\simeq w_\alpha$.
\end{corollary}
\begin{proof}
	We first show that $w_1w_2$ is a weight function for $\mu_1+\mu_2$.
    Assume first that $\mu_1$ and $\mu_2$ are strictly positive.
    Write $w_k=\sum_{i_k=1}^{I_k} a_{k,i_k}\langle \tau\rangle^{s_{i_k}}\langle \xi\rangle^{r_{i_k}}$, where ${\mathcal{N}}_{SV}(\mu_k) \subseteq \{(r_{i_k}, s_{i_k}) \mid i_k\in\set{1,..., I_k}\}\subseteq \mathcal{N}(\mu_k)$.
    Then
    \begin{align*}
        w_1w_2=\sum_{i_1=1}^{I_1}\sum_{i_2=1}^{I_2} a_{1,i_1}a_{2,i_2}\langle \tau\rangle^{s_{i_1}+s_{i_2}}\langle \xi\rangle^{r_{i_1}+r_{i_2}}.
    \end{align*}
    Every coefficient remains positive, and since $\caln(\mu_1+\mu_2)=\caln(\mu_1)+\caln(\mu_2)$, it holds $(r_{i_1}+r_{i_2},s_{i_1}+s_{i_2})\in \caln(\mu_1+\mu_2)$ for every $i_1\in \set{1,\ldots,I_1}$ and $i_2\in \set{1,\ldots,I_2}$.

    \medskip

    Since $\mathcal{N}_{SV}(\mu_1+\mu_2)\subseteq \mathcal{N}_{SV}(\mu_1)+\mathcal{N}_{SV}(\mu_2)$ by Lemma \ref{lem:significant_vertices_sum}, every significant vertex of $\mathcal{N}(\mu_1+\mu_2)$ is of the form $(r_{i_1}+r_{i_2},s_{i_1}+s_{i_2})$ for some $i_1\in\set{1,\ldots,I_1}$ and $i_2\in \set{1,\ldots,I_2}$.
    This shows that $w_1w_2$ is a weight function for $\mu_1+\mu_2$ in the case of strictly positive order functions.

    \medskip

    If $\mu_1$ and $\mu_2$ are arbitrary, choose strictly positive order functions $\nu_1$, $\nu_1'$, $\nu_2$, $\nu_2'$, such that $\mu_k=\nu_k-\nu_k'$ for $k\in\set{1,2}$, and corresponding weight functions $v_1$, $v_1'$, $v_2$, $v_2'$ such that $w_k=\frac{v_k}{v_k'}$.
    By the previous point, $v_1v_2$ is a weight function for $\nu_1+\nu_2$, and $v_1'v_2'$ is a weight function for $\nu_1'+\nu_2'$.
    Hence, $w_1w_2=\frac{v_1v_2}{v_1'v_2'}$ is a weight function for $\nu_1+\nu_2-(\nu_1'+\nu_2')=\mu_1+\mu_2$.
    Moreover, $w_1^{-1}=\frac{v_1'}{v_1}$ is a weight function for $\nu_1'-\nu_1=-\mu_1$.
    
    \medskip
    
    Finally, observe that for weight function $w$ for a strictly positive order function $\mu$ of the form $w=\sum_{i=1}^{I} a_{i}\langle \tau\rangle^{s_{i}}\langle \xi\rangle^{r_{i}}$ we have $w^{|\alpha|}(\tau,\xi)\simeq w_{|\alpha|}(\tau,\xi):=\sum_{i=1}^{I} a_{i}^{|\alpha|} \langle \tau\rangle^{|\alpha| s_{i}}\langle \xi\rangle^{|\alpha| r_{i}}$.
    Since we have both $\caln(|\alpha| \mu)=|\alpha|\caln(\mu)$ and $\caln_{SV}(|\alpha| \mu)=|\alpha|\caln_{SV}(\mu)$, the function $w_{|\alpha|}$ is a weight function for $|\alpha|\mu$.
    Therefore, we may use $\alpha\mu_1=|\alpha|\nu_1-|\alpha|\nu_1'$ for $\alpha\ge 0$ and $\alpha\mu_1=|\alpha|\nu_1'-|\alpha|\nu_1$ for $\alpha<0$ to obtain that $w_1^\alpha=(v_1/v_1')^\alpha\simeq v_{1,\alpha}/v_{1,\alpha}'=:w_\alpha$.
\end{proof}

The vertices of a Newton polygon can be thought of as the points of 'maximal weights', as explained in the following lemma, cf.\ \cite[Remark 2.17]{denk2013general}.
Recall that $w_{(s,r)}(\tau,\xi)=\langle\tau\rangle^s\langle\xi\rangle^r$.
\begin{lemma}\label{lemma: point_inside}
    Let $w$ be a weight function for a strictly positive order function $\mu$ and $(r,s)\in [0,\infty)^2$.
    Then $(r, s)\in \mathcal{N}(\mu)$ if and only if there exists a constant $c>0$ such that
    \begin{equation*}
        w_{(s,r)} \leq c w. 
    \end{equation*}
    In particular, if $v$ is a weight function for a strictly positive order function $\nu$ such that $\mathcal{N}(\nu)\subseteq \mathcal{N}(\mu)$, then there exists $c>0$ such that $v\leq c w$.
\end{lemma}
\begin{proof}
    Write $w(\tau,\xi)= \sum_{i=1}^I a_i \langle\tau\rangle^{s_i} \langle\xi\rangle^{r_i}$ with $a_i>0$ and $\mathcal{N}(\mu)_{SV}\subseteq\set{(r_i,s_i)\mid i\in\set{1,\ldots,I}}\subseteq\mathcal N(\mu)$.
    By Lemma \ref{lem:significant_vertices_sum}\ref{lem:significant_vertices_sumiii} it holds $\mu(x) = \max \set{\langle(r_i,s_i),x\rangle\mid i\in\set{1,\ldots,I}}$ for all $x\in [0,\infty)^2$.
    For $(\tau,\xi)\in \R\times\R^n$, we therefore obtain
    \begin{align}\label{eqn: weight_support_equivalence}
        ae^{\mu(x)} \leq w(\tau,\xi) \leq Ae^{\mu(x)},
    \end{align}
    where $a:=\min\set{a_i\mid i\in \set{1,\ldots,I}}$, $A:=\sum_{i=1}^I a_i$, and  $x:=\bigl(\log\langle\xi\rangle,\log\langle\tau\rangle\bigr)\in[0,\infty)^2$.
    Let $(r,s)\in\mathcal N(\mu)$.
    Then $rx_1+sx_2\le \mu(x)$ and \eqref{eqn: weight_support_equivalence} yield
    \begin{align*}
        w_{(s,r)}(\tau,\xi)
        =e^{r\log\langle\xi\rangle+s\log\langle\tau\rangle}
        \leq e^{\mu(x)}
        \leq a^{-1}w(\tau,\xi).
    \end{align*}

    Conversely, suppose that $w_{(s,r)}(\tau,\xi)\leq cw(\tau,\xi)$ for all $(\tau,\xi)\in\R\times\R^n$.
    Let $x=(x_1,x_2)\in[0,\infty)^2$ and $t>0$.
    Choose $(\tau_t,\xi_t)\in \R\times\R^n$ such that
    \begin{align*}
        \log\langle\xi_t\rangle=tx_1,
        \qquad
        \log\langle\tau_t\rangle=tx_2.
    \end{align*}
    In particular, $e^{t(rx_1+sx_2)}=w_{(s,r)}(\tau_t,\xi_t)\le cw(\tau_t,\xi_t)\le cAe^{t\mu(x)}$, where we have used \eqref{eqn: weight_support_equivalence} again.
    Taking logarithms, dividing by $t$, and letting $t\to \infty$ yields $rx_1+sx_2\leq \mu(x)=\max\set{\langle x,v\rangle\mid v\in \mathcal N(\mu)}$ for all $x\in [0,\infty)^2$, and thus for all $x\in \R^2$ since $\mu(x)=\mu(x_+)$ by Remark \ref{js015}.
    Thus, $(r,s)\in\mathcal N(\mu)$.
\end{proof}
Two weight functions for the same order function are equivalent, as we will show next.
\begin{lemma}\label{lemma: additivity}
    Let $\mu, \nu$ be order functions and $\alpha,\beta\in \R$.
    \begin{enumerate}
    \item\label{lemma: additivityi} If $w$ and $v$ are weight functions for $\mu$, then $w\simeq v$.
    \item\label{lemma: additivityii} If $w,v,u$ are weight functions for $\mu$, $\nu$, and $\alpha\mu+\beta\nu$, respectively, then $u\simeq w^\alpha v^\beta$.
    \end{enumerate}
\end{lemma}
\begin{proof}
\begin{enumerate}
	\item If $w_1',w_2'$ are weight functions for a strictly positive order function $\mu'$, then Lemma \ref{lemma: point_inside} applied in both directions yields $w_1'\simeq w_2'$.
	Let now $\mu=\mu_1-\mu_2=\nu_1-\nu_2$ for strictly positive order functions $\mu_1$, $\mu_2$, $\nu_1$, $\nu_2$, and let $w_1$, $w_2$, $v_1$, $v_2$ be corresponding weight functions with $w=\frac{w_1}{w_2}$ and $v=\frac{v_1}{v_2}$.
	By Corollary \ref{cor: product_weight_functions}, $w_1v_2$ and $w_2v_1$ are weight functions for the strictly positive order function $\mu_1+\nu_2=\mu_2+\nu_1$, and thus $w_1v_2\simeq w_2v_1$ by the already established result.
	This implies the assertion.
	\item By Corollary \ref{cor: product_weight_functions}, there exists a weight function $\widetilde u$ for $\alpha\mu+\beta\nu$ such that $\widetilde u\simeq w^\alpha v^\beta$.
	Since $u$ and $\widetilde u$ are weight functions for the same order function, part \ref{lemma: additivityi} gives $u\simeq\widetilde u$, and hence $u\simeq w^\alpha v^\beta$.
\end{enumerate}
\end{proof}
Knowing if a symbol $m$ and its derivatives can be estimated from above and from below by a given weight function will turn out to be crucial to establish maximal regularity estimates in the next chapter. We thus introduce the following definitions. We will use the multi-index notations 
\begin{equation*}
   (\tau, \xi)^{(\alpha,\beta)}:=\tau^\alpha \xi_1^{\beta_1}...\xi_n^{\beta_n} \qquad  \partial_{(\tau,\xi)} ^{(\alpha, \beta)}:= \partial_\tau^\alpha \partial_{\xi_1}^{\beta_1}... \partial_{\xi_n}^{\beta_n}
\end{equation*}
for $\alpha\in \mathbb{N}_0, \beta\in \mathbb{N}_0^n$.

\begin{definition}\label{def: Newton_ellipticity}Let $m : (\mathbb{R}\setminus \{0\})\times \mathbb{R}^n\to \mathbb{C}$ be smooth and let $\mu$ be an order function. 
    \begin{enumerate}
        \item\label{def: Newton_ellipticityi} $\mu$ is an \emph{upper order function} for $m$ if for all $\lambda>0$ and any weight function $w$ for $\mu$ there exists a constant $C_\lambda>0$ such that $|m(\tau,\xi)|\le C_\lambda w(\tau,\xi)$ for all $(\tau, \xi)$ with $|\tau| \ge \lambda$.
        \item\label{def: Newton_ellipticityii} $\mu$ is a \emph{strong upper order function} if for all $\alpha \in \{0, 1\}$, $\beta \in \{0,1\}^n$, $\mu$ is an upper order function for the symbol $(\tau, \xi)^{(\alpha,\beta)} \partial_{(\tau,\xi)} ^{(\alpha, \beta)} m(\tau, \xi)$.
        \item\label{def: Newton_ellipticityiii} $m$ is \emph{$\mathcal{N}$-elliptic with order function} $\mu$, if $\mu$ is an upper order function for $m$ and if $-\mu$ is an upper order function for $1/m$.
        \item\label{def: Newton_ellipticityiv} We say that $m$ is \emph{strongly $\mathcal{N}$-elliptic} with order function $\mu$ if it is $\mathcal{N}$-elliptic and  $\mu$ is a strong upper order function for $m$.
    \end{enumerate}
\end{definition}
By Lemma \ref{lemma: additivity}, replacing ``any weight function'' by ``some weight function'' in the above points \ref{def: Newton_ellipticityi} and \ref{def: Newton_ellipticityiii} yields an equivalent definition.
Moreover, by Corollary \ref{cor: product_weight_functions}, $m$ is $\caln$-elliptic with order function $\mu$ if and only if for all $\lambda>0$ and any weight function $w$ for $\mu$ there exist constants $c_\lambda,C_\lambda>0$ such that
\begin{align*}
    c_\lambda w(\tau,\xi)\le |m(\tau,\xi)|\le C_\lambda w(\tau,\xi) \quad \text{for all $(\tau,\xi)$ with $|\tau|\ge \lambda$}.
\end{align*}
\begin{remark}\label{remark: different_def_ellipticity}
	Our definitions differ slightly from those in \cite{denk2013general} (see Definition 2.39) in that both the upper and lower estimates are required only away from $\tau=0$, but uniformly on \emph{every} strip $|\tau|\geq\lambda$, $\lambda>0$, not just \emph{some} strip $|\tau|\ge \lambda_0$.
This formulation is adapted to the purely oscillatory time-periodic setting considered below, where the frequency $\tau=0$ is removed.
\end{remark}
\begin{example}[Symbols with parabolic scaling]\label{js004}
    Let $N\geq 0, \rho>0$ and $P\in C^\infty ((\mathbb{R}\times \mathbb{R}^n)\setminus\{(0,0)\})$.
    We say that $P$ is \emph{quasihomogeneous} or more precisely \emph{$\rho$-homogeneous of degree $N$} if it holds 
    \begin{equation*}
        P(\eta^\rho \tau, \eta\xi)= \eta^N P(\tau, \xi)
    \end{equation*}
    for all $(\tau, \xi)\in (\mathbb{R}\times \mathbb{R}^n)\setminus\{(0,0)\}$ and $\eta>0$.
    Such a symbol always admits a strong upper order function $No_{\frac{1}{\rho}}$. 
    If, furthermore, $P(\tau, \xi) \neq 0$ for all $(\tau, \xi)\in (\mathbb{R}\times \mathbb{R}^n)\setminus\{(0,0)\}$, then $P$ is strongly $\mathcal{N}$-elliptic with order function $No_{\frac{1}{\rho}}$.
    Thus, the notion of $\mathcal{N}$-ellipticity generalizes the usual definition of parabolicity for quasihomogeneous symbols.
    Observe that the Newton polygon in that case has a simple triangular shape.
    \medskip 

    In Section \ref{sec:heat}, we will be interested in the symbol of the heat equation 
    \begin{equation*}
        H(\tau, \xi):= \ic\tau+\vert\xi\vert^2. 
    \end{equation*}
    Clearly, $H$ is $2$-homogeneous of degree 2, and thus strongly $\mathcal{N}$-elliptic with order function $\mu_H:= 2o_{\frac{1}{2}}$ by the preceding observation.
    The associated triangular Newton polygon $\caln(\mu_H)=2\caln_\frac12$ is described by the vertices $\mathcal{N}(\mu_H)_V= \{(0,0), (2, 0), (0, 1)\}$, and so $\mathcal{N}(\mu_H)_{SV}= \{(2, 0), (0, 1)\}$.
    Hence, a weight function for $\mu_H$ is
    \begin{equation*}
        w_{\mu_H}(\tau,\xi) = \langle\tau\rangle + \langle\xi\rangle^2. 
    \end{equation*}
\end{example}

\begin{example}[The Cahn--Hilliard--Gurtin determinant]\label{example: CHG_ellipticity}
    Let $c\in\R^n$.
    The following symbol is an example of a strongly $\mathcal{N}$-elliptic symbol with a non-triangular Newton polygon:
    \begin{equation*}
        D(\tau,\xi):=|\xi|^4+\tau c\cdot\xi+\ic\tau(1+|\xi|^2), \qquad (\tau,\xi)\in\R\times\R^n.
    \end{equation*}
    We will see later that this symbol is associated with the determinant of the Cahn--Hilliard--Gurtin (CHG) system.
    Note that $D$ is not quasihomogeneous in the above sense.
    
    \medskip

    We define the order function $\mu_D:=2o_0+2o_{\frac12}$ and claim that $D$ is strongly $\mathcal{N}$-elliptic with order function $\mu_D$.
    Here, the Newton polygon $\caln(\mu_D)=2\caln_0+2\caln_{\frac12}$ is described by the vertices $\mathcal{N}(\mu_D)_V=\{(0,0),(4,0),(2,1),(0,1)\}$, and so $\mathcal{N}(\mu_D)_{SV}=\{(4,0),(2,1),(0,1)\}$.
    Hence, a weight function for $\mu_D$ is
    \begin{equation*}
        w_{\mu_D}(\tau,\xi)=\langle\tau\rangle+\langle\tau\rangle\langle\xi\rangle^2+\langle\xi\rangle^4.
    \end{equation*}
    Since $D$ is a polynomial, the upper estimates
    \begin{equation*}
        |(\tau,\xi)^{(\alpha,\beta)}\partial_{(\tau,\xi)}^{(\alpha,\beta)}D(\tau,\xi)|\lesssim w_{\mu_D}(\tau,\xi), \qquad (\tau,\xi)\in\R\times\R^n,
    \end{equation*}
    follow from the triangle inequality and the Leibniz rule.
    For the lower estimates, we first observe that
    $|D(\tau,\xi)|\ge|\tau|(1+|\xi|^2)$.
    Moreover,
    \begin{align*}
        |\xi|^4 \leq \big||\xi|^4+\tau c\cdot\xi\big|+|\tau||c||\xi|
        \leq |D(\tau,\xi)|+\frac{|c|}{2}|\tau|(1+|\xi|^2)
        \leq \left(1+\frac{|c|}{2}\right)|D(\tau,\xi)|.
    \end{align*}
    Consequently,
    \begin{align*}
        |D(\tau,\xi)|\gtrsim_c |\tau|+|\xi|^2|\tau|+|\xi|^4.
    \end{align*}
    Let $\lambda>0$.
    For $|\tau|\ge\lambda$ we obtain due to $|\tau|\ge\frac12\min\{1,\lambda\}(1+|\tau|)$ a constant $C_{\lambda}>0$ such that
    \begin{align*}
        |D(\tau,\xi)|\gtrsim_{\lambda,c}1+|\tau|+|\xi|^2|\tau|+|\xi|^4\gtrsim_{\lambda}w_{\mu_D}(\tau,\xi).
    \end{align*}
    Thus there is $C_{\lambda}>0$ such that $|D(\tau,\xi)|\ge C_{\lambda}w_{\mu_D}(\tau,\xi)$ for all $(\tau,\xi)$ with $|\tau|\ge\lambda$.
\end{example}
We next provide the following permanence properties of upper order functions and $\mathcal{N}$-ellipticity.
\begin{proposition}\label{prop: arithmetics_order_functions}
    Let $m, m_1, m_2 : (\mathbb{R}\setminus \{0\}) \times\mathbb{R}^n\to \mathbb{C}$ be smooth symbols and let $\mu, \mu_1, \mu_2$ be order functions. 
    \begin{enumerate}[leftmargin=*]
        \item\label{prop: arithmetics_order_functionsi} If $\mu$ is a (strong) upper order  function for $m_1$ and $m_2$ then it is a (strong) upper order function for $m_1+m_2$. 
        \item\label{prop: arithmetics_order_functionsii} If $\mu_i$ is a (strong) upper order function for $m_i$, $i\in\{1,2\}$, then $\mu_1+\mu_2$ is a (strong) upper order function for $m_1m_2$. Furthermore, if $m_i$ is (strongly) $\mathcal{N}$-elliptic with order function $\mu_i$, $i\in\{1,2\}$, then $m_1m_2$ is (strongly) $\mathcal{N}$-elliptic with order function $\mu_1+\mu_2$.
        \item\label{prop: arithmetics_order_functionsiii} If $m$ is (strongly) $\mathcal{N}$-elliptic with order function $\mu$, then $\frac{1}{m}$ is (strongly) $\mathcal{N}$-elliptic with order function $-\mu$. 
        \item\label{prop: arithmetics_order_functionsiv} If $m$ is (strongly) $\mathcal{N}$-elliptic with order function $\mu$ and $m(\tau, \xi)> 0 $ for all $(\tau, \xi)\in (\mathbb{R}\setminus\{0\}) \times \mathbb{R}^n$ then for all $s\in \mathbb{R}, $ $m^s$ is (strongly) $\mathcal{N}$-elliptic with order function $s\mu$.
    \end{enumerate}
\end{proposition}
\begin{proof}
    Part \ref{prop: arithmetics_order_functionsi} is trivial, and part \ref{prop: arithmetics_order_functionsii} follows from Corollary \ref{cor: product_weight_functions} and the Leibniz formula for the derivatives of a product.
    For part \ref{prop: arithmetics_order_functionsiii}, the definition implies directly that $\frac1m$ is $\mathcal N$-elliptic with order function $-\mu$.
    It remains to verify the upper estimates for the derivatives.
    By Faà di Bruno's formula, for $\alpha+|\beta|\ge 1$ the function $\partial_{(\tau, \xi)}^{(\alpha, \beta)}  \frac{1}{m}$ can be written as 
    \begin{equation*}
        \partial_{(\tau,\xi)}^{(\alpha,\beta)} \frac{1}{m}(\tau, \xi)=  \sum_{k=1}^{\alpha+\vert\beta\vert} \sum_{\substack{\delta_1+\ldots +\delta_k=(\alpha,\beta)\\\delta_j\in \{0,1\}\times \{0,1\}^n}}  \frac{C_{k, \delta_1,\ldots,\delta_k}}{m^{k+1}(\tau, \xi)}\prod_{r=1}^k \partial_{(\tau, \xi)}^{\delta_r}m(\tau, \xi).
    \end{equation*}
    Hence, using \ref{prop: arithmetics_order_functionsi}, it is enough to show that $-\mu$ is an upper order function for each of the expressions $(\tau, \xi)^{(\alpha, \beta)} \frac{1}{m^{k+1}(\tau, \xi)}\prod_{r=1}^k \partial_{(\tau, \xi)}^{\delta_r}m(\tau, \xi)$.
    Choose a weight function $w$ for $\mu$ and $\lambda>0$.
    Using $\delta_1+\ldots +\delta_k=(\alpha,\beta)$, we obtain
    \begin{align*}
        \vert (\tau,\xi)^{(\alpha,\beta)} \frac{1}{m^{k+1}(\tau, \xi)}\prod_{r=1}^k \partial_{(\tau, \xi)}^{\delta_r}m(\tau, \xi)\vert
        &= \frac{1}{|m^{k+1}(\tau, \xi)|}\prod_{r=1}^k \vert (\tau, \xi)^{\delta_r}\vert \vert\partial_{(\tau, \xi)}^{\delta_r}m(\tau, \xi)\vert \lesssim_\lambda  \frac{w(\tau,\xi)^k}{w(\tau,\xi)^{k+1}}
    \end{align*}
    for all $(\tau,\xi)$ with $|\tau|>\lambda$, where the last step is a consequence of the strong $\mathcal{N}$-ellipticity of $m$.
    Since $w^k/w^{k+1}=w^{-1}$ is a weight function for $-\mu$ in light of Corollary \ref{cor: product_weight_functions}, we have proved \ref{prop: arithmetics_order_functionsiii}.
    For part \ref{prop: arithmetics_order_functionsiv}, we observe that on every strip $|\tau|\geq\lambda$, $\mathcal N$-ellipticity gives $m\simeq_\lambda w$ for any weight function $w$ for $\mu$.
    Hence $m^s\simeq_\lambda w^s$, and Lemma \ref{lemma: additivity} shows that $w^s$ is equivalent to a weight function for $s\mu$.
    The derivative estimates follow by a similar application of Faà di Bruno's formula as above, where one writes 
    \begin{equation*}
       \partial_{(\tau, \xi)}^{(\alpha, \beta)} m^s(\tau, \xi)= \sum_{k=1}^{\alpha+\vert\beta\vert} \sum_{\substack{\delta_1+\ldots +\delta_k=(\alpha,\beta)\\\delta_j\in \{0,1\}\times \{0,1\}^n}} C_{k, \delta_1,\ldots ,\delta_k}m(\tau, \xi)^{s-k}\prod_{r=1}^k \partial_{(\tau, \xi)}^{\delta_r}m(\tau, \xi).
       \qedhere
    \end{equation*}
\end{proof}
\begin{corollary}\label{cor: ellipt_weight_funct}
 \begin{enumerate}
     \item\label{cor: ellipt_weight_functi} Let $\mu$ be a strictly positive order function and $(s,r)\in \mathcal{N}(\mu)$. Then $\mu$ is a strong upper order function for $w_{(s,r)}$. 
     \item\label{cor: ellipt_weight_functii} If $\mu$ and $\nu$ are strictly positive order functions with $\mathcal{N}(\mu)\subseteq \mathcal{N}(\nu)$ then $\nu$ is a strong upper order function for any weight function of $\mu$.
     \item\label{cor: ellipt_weight_functiii} If $\mu$ is an order function and $w$ a weight function for $\mu$, then $w$ is strongly $\mathcal{N}$-elliptic with order function $\mu$. 
 \end{enumerate}
 \end{corollary}
 \begin{proof}
     Assume first that $\mu$ is strictly positive. Then Lemma \ref{lemma: point_inside} along with the observation 
    \begin{equation*}
        (\tau, \xi)^{(\alpha, \beta) }\partial_{(\tau, \xi)}^{(\alpha, \beta) } \langle \tau\rangle^s\langle \xi\rangle^r \lesssim \langle \tau \rangle^s \langle \xi\rangle^r
    \end{equation*}
    for each $\alpha\in \{0,1\}$, $\beta\in \{0,1\}^n$ shows that $\mu$ is a strong upper order function for $w_{(s,r)}(\tau, \xi)= \langle \tau\rangle^s \langle \xi \rangle^r$.
    The second point follows from point \ref{prop: arithmetics_order_functionsi} in Proposition \ref{prop: arithmetics_order_functions}.
    In particular, if $w$ is a weight function for $\mu$, then $\mu$ is a strong upper order function for $w$.
    Since $v\simeq w$ for any two weight functions of $\mu$ by Lemma \ref{lemma: additivity}, we find that $w$ is strongly $\mathcal{N}$-elliptic for $\mu$. 
    
    If $\mu$ is any order function and $w$ a corresponding weight function, choose strictly positive order functions $\mu_1$, $\mu_2$ and corresponding weight functions $w_1$, $w_2$ such that $\mu=\mu_1-\mu_2$ and $w=\frac{w_1}{w_2}$.
    Applying parts \ref{prop: arithmetics_order_functionsii}, \ref{prop: arithmetics_order_functionsiii} of Proposition \ref{prop: arithmetics_order_functions}, the conclusion follows.
 \end{proof}
\subsection{Mixed-Order Systems}\label{sec:mix_sys}

The notion of a mixed-order system has been introduced in \cite{agmon1964estimates} in order to study the regularity of solutions of systems of elliptic equations. The definition was later adapted to the Newton polygon setting in \cite{denk2013general}.
We recall this framework here and adapt it to our time-periodic setting in order to extend the scalar results developed above to systems whose components have different orders of regularity.
The passage to systems is essential even for scalar boundary value problems, since general boundary conditions naturally give rise to mixed-order systems on the boundary.
In the following, we consider a matrix-valued symbol  $L:(\mathbb{R}\setminus \{0\})\times \mathbb{R}^n\to \mathbb{C}^{m\times m}$ for some $m\in \mathbb{N}$. We denote by $\mathcal{D}: (\mathbb{R}\setminus\{0\})\times \mathbb{R}^n\to \mathbb{C}$, $\mathcal{D}(\tau, \xi):=\det L(\tau, \xi)$ its determinant.
\begin{definition}\label{def: mixed_order_systems}
    The symbol $L$ is associated with a \emph{(strong) mixed-order system} if there exist  order functions $s_1,...,s_m, t_1,...,t_m$ such that the following is true.
    \begin{enumerate}
        \item The order function $t_j-s_i$ is a (strong) upper order function of the symbol $L_{ij}: \mathbb{R}\times \mathbb{R}^n\to \mathbb{C}$ for all $i,j\in\{1,..., m\}$.
        \item The determinant $\mathcal{D}$ is (strongly) $\mathcal{N}$-elliptic for the order function $\delta:=\sum_{l=1}^m t_l-s_l$.
    \end{enumerate}
\end{definition}
While solving systems, we will make regular use of the adjugate symbol of $L$, denoted $\ad L$. 
Here 
\begin{equation*}
    (\ad L)_{ij}:= (-1)^{i+j}\det L^{(ji)}, \qquad i,j\in\{1,\ldots, m\}, 
\end{equation*}
where $L^{(ji)}$ denotes the $(m-1)\times (m-1)$ matrix obtained by removing the $j$-th row and the $i$-th column of $L$.
\begin{proposition}\label{prop: ad}
    If $L:(\mathbb{R}\setminus \{0\})\times \mathbb{R}^n\to \mathbb{C}^{m\times m}$ is a (strong) mixed-order system for the order functions $s_1,..., s_m$, $t_1,...,t_m$ then $\ad L$ is a (strong) mixed-order system for the order functions $S_1,...,S_m$, and $T_1,...,T_m$, with 
    \begin{equation*}
        S_i=\sum_{\substack{k=1\\k\neq i}}^m -t_k, \qquad T_j:=\sum_{\substack{k=1\\k\neq j}}^m -s_k,\qquad i, j\in\{1,...,m\}.
    \end{equation*}
\end{proposition}
\begin{proof}
    Let $i, j=1,..., m$ and denote $\mathcal{S}^{(ij)}$ the set of all bijections $\{1,..., m\}\setminus \{i\} \to \{1,..., m\}\setminus \{j\}$. Then we have 
    \begin{equation*}
        (\ad L)_{ij}= (-1)^{i+j}\det L^{(ji)}= (-1)^{i+j}\sum_{\sigma \in \mathcal{S}^{(ji)}}\sgn(\sigma) \prod_{\substack{k=1\\k\neq j}}L_{k\sigma(k)}.
    \end{equation*}
    Since $t_{\sigma(k)}-s_k$ is a (strong) upper order function for $L_{k\sigma(k)}$ for all $k\in\{1,\ldots, m\}$ with $k\neq j$ by definition, Proposition \ref{prop: arithmetics_order_functions} yields that $T_j-S_i= -\sum_{l\neq j} s_l + \sum_{l\neq i}t_l= \sum_{l\neq j}t_{\sigma(l)}-s_l$ is a (strong) upper order function for $(\ad L)_{ij}$. 

    To conclude, we note that $\det(\ad L)= \mathcal{D}^{m-1}$. Since $\mathcal{D}$ is (strongly) $\mathcal{N}$-elliptic with order function $\delta= \sum_{l=1}^m t_l-s_l$, a further application of Proposition \ref{prop: arithmetics_order_functions} yields that $\mathcal{D}^{m-1}$ is (strongly) $\mathcal{N}$-elliptic for the order function $(m-1)\delta = \sum_{l=1}^m T_l-S_l$.
\end{proof}
\begin{example}\label{example: CHG_system}
    Let $c\in\R^n$.
    The Cahn--Hilliard--Gurtin system 
    \begin{equation}\label{eqn: CHG_system}
        \left\{\begin{array}{rcl}
        \partial_tu_1 -\Delta u_2 &=&f_1,       \\
        \Delta u_1-\partial_tu_1+u_2-c\cdot\nabla u_2 &=&f_2,
        \end{array}\right.
    \end{equation}
    admits a structure of strong mixed-order system. Indeed, the associated operator and symbol matrix are
    \begin{equation}\label{eqn: CHG_symbol}
        \opL=
        \begin{pmatrix}
        \partial_t & -\Delta\\
        -(\partial_t-\Delta) & 1-c\cdot\nabla
        \end{pmatrix},
        \qquad
        L(\tau, \xi)=
        \begin{pmatrix}
        \ic\tau & \vert \xi\vert^2\\
        -\vert \xi\vert^2-\ic\tau & 1-\ic c\cdot\xi
        \end{pmatrix}, \qquad (\tau, \xi)\in \mathbb{R}\times \mathbb{R}^n.
    \end{equation}
    Choosing the order functions associated with the columns $t_1:=2o_{\frac{1}{2}}+ o_0$, $t_2:= 2o_0$ and the order functions associated with the rows $s_1:=0$, $s_2:= o_0$, one easily verifies that $t_j-s_i$ is a strong upper order function for the element $L_{ij}$, $i,j\in\set{1,2}$. 
    Furthermore, $\det L=D$ as defined in Example \ref{example: CHG_ellipticity}.
    This example establishes that $\det L$ is strongly $\mathcal{N}$-elliptic with order function $\mu_D=2o_{\frac{1}{2}}+2o_0$.
    Since  $t_1+t_2-s_1-s_2= 2o_{\frac{1}{2}}+2o_0=\mu_D$, we conclude that $L$ is a strong mixed-order system. 
\end{example}

\subsection{Newton Polygon Scales}\label{sec:newt_poly_scale}
In this section, we introduce a flexible framework of scales of spaces, which is stable under interpolation.
This will prove useful since interpolation spaces will naturally appear in the trace theory established in Section \ref{sec:trc_spc}.
The idea is to identify which minimal properties are required for a family of topological vector spaces $\{E(s, r)\mid s, r \in \mathbb{R}\}$ to be compatible with a Newton polygon (or more generally order function) representation.
We then show that those properties are preserved under interpolation.

Let $V$ be a topological vector space and let
\begin{align*}
    \ell:\mathbb R^2\longrightarrow\mathcal L_{\mathrm{iso}}(V),
    \qquad
    (r,s)\longmapsto\ell_{(s,r)},
\end{align*}
be a group homomorphism.\footnote{Here the flip of the components $s,r$ in the notation reflects our convention that the point $(r,s)$ of a Newton polygon represents temporal regularity $s$ and spatial regularity $r$, see footnotes \ref{footnote001} and \ref{footnote002}.}

\begin{definition}\label{def: lift_repr}
\begin{enumerate}
    \item Let $\mu$ be a strictly positive order function.
    We call $\ell \in \mathcal{L}(V)$ a \emph{lift operator} for $\mu$ if it is of the form $\ell= \sum_{i=1}^I a_i \ell_{(s_i, r_i)}$ with $a_i>0$ and $\mathcal{N}(\mu)_{SV}\subseteq \{(r_i,s_i)\mid i\in\{1,\ldots,I\}\}\subseteq \mathcal{N}(\mu)$.
    \item We call $\{\ell_{(s,r)}\mid (r,s)\in\mathbb{R}^2\}$ a \emph{lifting representation} on $V$, if for any lift operator $\ell$ associated with a strictly positive order function $\mu$ it holds $\ell\in \mathcal{L}_{\mathrm{iso}}(V)$.
    \item If $\{\ell_{(s,r)}\mid (r,s)\in\mathbb{R}^2\}$ is a lifting representation on $V$, $\mu=\mu_1-\mu_2$ an order function with $\mu_1$, $\mu_2$ strictly positive, then a lift operator for $\mu$ is an operator of the form $\ell:=\ell_{_1}\circ \ell_{_2}^{-1}$, where $\ell_i$ is a lift operator for $\mu_i$, $i\in\{1,2\}$.
\end{enumerate}
\end{definition}
 Clearly, for all $r,s\in\R$, a lift operator for $\mu_{(s,r)}$ is given by $\ell_{(s,r)}$.
 From the commutativity of the $\ell_{(s,r)}$, we get that if $\ell_1, \ell_2$ are lift operators for the order functions $\mu_1, \mu_2$, respectively, then $\ell_1, \ell_2$ commute and so do their inverses.
 Moreover, if $\ell_1:=\sum_{i=1}^{I_1} a_i\ell_{(s_i,r_i)}$ and $\ell_2:=\sum_{j=1}^{I_2} b_j\ell_{(s_j',r_j')}$ are lift operators for the strictly positive order functions $\mu_1$ and $\mu_2$, then $\ell_1\ell_2 = \sum_{i=1}^{I_1}\sum_{j=1}^{I_2}a_ib_j \ell_{(s_i+s_j',r_i+r_j')}$ is a lift operator for $\mu_1+\mu_2$ by Lemma \ref{lem:significant_vertices_sum}.
 By the definition of lifts for general order functions, the same conclusion extends immediately to arbitrary order functions $\mu_1$ and $\mu_2$.

 \medskip
 
 For the rest of this section, we will assume that $\{\ell_{(s,r)}\mid (r,s)\in\mathbb{R}^2\}$ is a lifting representation on $V$.\footnote{In this work, we will exclusively use the lifting representation given by Fourier operators associated to the weights $w_{(s,r)}(\tau,\xi)=\langle\tau\rangle^s\langle\xi\rangle^r$, see Section \ref{sec:Bessel}.
 However, the results in the present section are true for general lifting representations.}
 \begin{definition}\label{def: Newton_scales}
     A topological vector space $E \hookrightarrow V$
     satisfies the \emph{convexity condition} if
        for any lift operator $\ell$ associated to a strictly positive order function $\mu$, and any $(r, s)\in \mathcal{N}(\mu)$, it holds
        \begin{equation*}
             \ell_{(s, r)}\circ \ell^{-1}\in \mathcal{L}(E). 
        \end{equation*}
    In that case, we define $E(\mu):= \ell^{-1} E$ for any order function $\mu$, where $\ell$ is a lift operator associated to $\mu$.
    The family of all $E(\mu)$, where $\mu$ runs through all order functions, is called the \emph{Newton polygon scale} generated by $E$.
    \end{definition}
    \begin{remark}\label{remark: equiv_lifts}
    \begin{enumerate}
        \item\label{remark: equiv_liftsii} If $\mu$ is a strictly positive order function and $\ell_1$, $\ell_2$ are lifts associated to $\mu$, then linearity and the convexity property yield $\ell_{1} \ell_{2}^{-1} \in \mathcal{L}_{\mathrm{iso}}(E)$.
        \item\label{remark: equiv_liftsiv}
        The definition of $E(\mu)$ does not depend on the choice of the lift associated to $\mu$.
        Indeed, let $\ell=\ell_1\ell_2^{-1}$ and $\ell'=\ell_1'(\ell_2')^{-1}$ be two lifts for $\mu$, where $\ell_i$ and $\ell_i'$ are lifts for strictly positive order functions $\mu_i$ and $\mu_i'$ with $\mu=\mu_1-\mu_2=\mu_1'-\mu_2'$, respectively.
        Then $\ell_1\ell_2'$ and $\ell_2\ell_1'$ are lifts for the same strictly positive order function $\mu_1+\mu_2'=\mu_2+\mu_1'$.
        Hence, by point \ref{remark: equiv_liftsii},
        \begin{align*}
            \ell' \ell^{-1}=(\ell_2\ell_1')(\ell_1\ell_2')^{-1} \in\mathcal L_{\mathrm{iso}}(E),
        \end{align*}
        and therefore the identity map between the two realizations of $E(\mu)$ is a topological isomorphism.
        \item\label{remark: equiv_weight_functionsv}
        In particular, we have $\ell\in \call_{\mathrm{iso}}(E(\mu),E)$ for any lift $\ell$ associated to $\mu$.
        \item\label{remark: equiv_liftsi} For $s,r\in\R$, we write $E(s,r):=E(\mu_{(s,r)})$.
        Observe that $E(0,0)=E$ and that $\ell_{(s', r')}\in \mathcal{L}_{\mathrm{iso}}(E(s, r), E(s-s', r-r'))$ for all $s,r, s', r'\in \mathbb{R}$.
    \end{enumerate}
    \end{remark}
    \begin{proposition}\label{js014}
    Let $E$ generate a Newton polygon scale, and let $\mu$ and $\nu$ be order functions.
    Then it holds $\ell_\mu\in \call_{\mathrm{iso}}(E(\mu+\nu),E(\nu))$ for any lift $\ell_\mu$ for $\mu$.
\end{proposition}
\begin{proof}
    Let $\ell_\mu$ and $\ell_\nu$ be lifts for $\mu$ and $\nu$, respectively.
    Then $\ell_\mu\ell_\nu$ is a lift for $\mu+\nu$.
    Since $\ell_\mu = \ell_\nu^{-1}(\ell_\mu\ell_\nu)$, point \ref{remark: equiv_weight_functionsv} of Remark \ref{remark: equiv_lifts} yields $\ell_\mu\in \mathcal L_{\mathrm{iso}}(E(\mu+\nu),E(\nu))$.
\end{proof}

Furthermore, we can show that if $\mu$ is strictly positive, then the significant vertices of the Newton polygon $\mathcal{N}(\mu)$ encode all the information necessary to describe  the Newton polygon space $E(\mu)$ as an intersection of the elementary spaces $E(s,r)$.
\begin{proposition}\label{prop: properties_Newton_scales}
Let $E$ generate a Newton polygon scale and let $\mu, \nu$ be strictly positive order functions.
Let $\mathcal{N}(\mu)_{SV}\subseteq \{(r_i, s_i)\mid i\in\{1,..., I\}\}\subseteq \mathcal{N}(\mu)$.
~\begin{enumerate}
    \item\label{prop: properties_Newton_scalesi} We have $E(\mu)= \bigcap_{i=1}^{I}E(s_i, r_i)$ algebraically and topologically, where the intersection space is equipped with the canonical initial topology.
    \item\label{prop: properties_Newton_scalesii}  For all $(r, s)\in \mathcal{N}(\mu)$ it holds $ E(\mu) \hookrightarrow E(s, r)$. Therefore if $\mathcal{N}(\nu)\subseteq \mathcal{N}(\mu)$, then $E(\mu)\hookrightarrow E(\nu)$.
    \item\label{prop: properties_Newton_scalesiii} Let $\eta:=\max\{\mu,\nu\}$.
    Then $E(\mu)\cap E(\nu)=E(\eta)$ algebraically and topologically.
\end{enumerate}
\end{proposition}
\begin{proof}
    \begin{enumerate}[leftmargin=*]
    	\item Let us write $X:=\bigcap_{i=1}^{I }E(s_i, r_i)$.
	Observe that it suffices to show that $\ell\in \call_{\mathrm{iso}}(X,E)$ for the lift $\ell=\sum_{i=1}^{I}\ell_{(s_i,r_i)}$.
	Since $\ell_{(s_i,r_i)}\in \call(E(s_i,r_i),E)$ for all $i\in \set{1,\ldots,I}$ by definition of the $E(s_i,r_i)$, the defining property of the initial topology on $X$ implies $\ell \in \call(X, E)$.

	\medskip
	
	For the inverse, observe that for each $i\in \set{1,\ldots,I}$,
	\begin{align*}
		\ell^{-1}=\ell_{(s_i,r_i)}^{-1}(\ell^{-1}\ell_{(s_i,r_i)})\in \call(E,E(s_i,r_i))
	\end{align*}
	by the convexity condition and the definition of $E(s_i,r_i)$.
	Therefore, the defining property of the initial topology on $X$ yields $\ell^{-1}\in \call (E,X)$.
	\item Let $(r,s)\in\mathcal N(\mu)$ and $\ell$ as above.
	Since $\ell\in\mathcal L(E(\mu),E)$, the convexity condition yields
	\begin{align*}
		\ell_{(s,r)}^{-1} (\ell^{-1}\ell_{(s,r)})\ell \in \mathcal L(E(\mu),E(s,r)).
	\end{align*}
	Since the operator on the left is the canonical inclusion, it follows that $E(\mu)\hookrightarrow E(s,r)$.
	If $\mathcal N(\nu)\subseteq\mathcal N(\mu)$, then $E(\mu)\hookrightarrow E(\sigma,\rho)$ for every vertex $(\rho,\sigma)$ of $\mathcal N(\nu)$, and thus $E(\mu)\hookrightarrow E(\nu)$ by the defining property of the initial topology on $E(\nu)=\bigcap_{(\rho,\sigma)\in\mathcal{N}(\nu)_{SV}}E(\sigma,\rho)$.
    \item Let $\mathcal N(\mu)_{\mathrm{SV}}$ and $\mathcal N(\nu)_{\mathrm{SV}}$ denote the respective sets of significant vertices.
    Since $\mathcal N(\eta)=\conv(\mathcal N(\mu)\cup\mathcal N(\nu))$, every significant vertex of $\mathcal N(\eta)$ belongs to
    $\mathcal N(\mu)_{\mathrm{SV}}\cup\mathcal N(\nu)_{\mathrm{SV}}$, while every point of the latter set belongs to $\mathcal N(\eta)$.
    Hence, applying \ref{prop: properties_Newton_scalesi} to these sets of vertices yields
    \begin{align*}
        E(\eta) &= \bigcap_{(r,s)\in\mathcal N(\mu)_{\mathrm{SV}}\cup\mathcal N(\nu)_{\mathrm{SV}}}E(s,r) = E(\mu)\cap E(\nu).
        \qedhere
    \end{align*}
    \end{enumerate}
\end{proof}
\begin{definition}\label{def: ellipticity_scales}
    Let $E$ generate a Newton polygon scale $\mathcal{E}$ and let $\mu$ be an order function.
    An operator $A\in\mathcal{L}(V)$ is \emph{$\mathcal{N}$-elliptic} for $\mu$ in the scale $\mathcal{E}$ if for every order function $\nu$, it holds $A\in \mathcal{L}_{\mathrm{iso}}(E(\mu+\nu), E(\nu))$. 
\end{definition}
In particular, if $A$ commutes with all the $\ell_{(s,r)}$, then $A$ is $\mathcal{N}$-elliptic for the order function $\mu$ in the scale $\mathcal{E}$ if and only if 
\begin{equation*}
    A \ell_\mu^{-1} \in \mathcal{L}_{\mathrm{iso}}(E)
\end{equation*}
for any lift $\ell_\mu$ associated to $\mu$.
Indeed, if $ \ell_\nu$ is a lift for $\nu$, then $\ell_\mu\ell_\nu$ is a lift for $\mu+\nu$ and $A=\ell_\nu^{-1}(A\ell_\mu^{-1})(\ell_\mu\ell_\nu)\in \mathcal{L}_{\mathrm{iso}}( E(\mu+\nu), E(\nu))$.
We now investigate how interpolation within a Newton polygon scale yields again a Newton polygon scale.
This will prove particularly useful to transfer results from Bessel spaces to Besov spaces and to study trace spaces associated to a Newton polygon scale. 

Let $\mathcal{I}$ be an interpolation functor on the category of Banach spaces.
For an interpolation couple of Banach spaces $(A_0, A_1)$, we will write $\mathcal{I}(A_0, A_1)$ for the associated interpolation space. 
Let $E\hookrightarrow V$ be a Banach space generating a Newton polygon scale.
Then this scale is a scale of Banach spaces, and for fixed order functions $\mu_0, \mu_1$ we may define
\begin{align*}
    E_{\mathcal{I}, \mu_0, \mu_1}:=\mathcal{I}(E(\mu_0),E(\mu_1)).
\end{align*}

\begin{proposition}\label{prop: ellipticity_interpolation_a}
Let $\mathcal I$ be an interpolation functor on the category of Banach spaces and let $\mu_0,\mu_1$ be order functions.
If a Banach space $E\hookrightarrow V$ generates a Newton polygon scale, then so does $E_{\mathcal I,\mu_0,\mu_1} := \mathcal I(E(\mu_0),E(\mu_1))$.
Moreover, for every order function $\nu$,
\begin{align*}
    E_{\mathcal I,\mu_0,\mu_1}(\nu) = \mathcal I(E(\nu+\mu_0),E(\nu+\mu_1))
\end{align*}
algebraically and topologically.

Furthermore, if $A\in\mathcal L(V)$ is $\mathcal N$-elliptic with order function $\mu$ in the scale generated by $E$, then it is also $\mathcal N$-elliptic with order function $\mu$ in the scale generated by $E_{\mathcal I,\mu_0,\mu_1}$.
\end{proposition}
\begin{proof}
    Fix $\mu_0$ and $\mu_1$.
    Let $\mu$ be a strictly positive order function, and let $\ell_0$, $\ell_1$, and $\ell$ be lifts for the order functions $\mu_0$, $\mu_1$, and $\mu$, respectively, and $(r,s)\in \mathcal{N}(\mu)$.
    By the convexity property of $E$, we know $\ell_{(s,r)}\ell^{-1}\in \mathcal{L}(E)$.
    Writing $\ell_{(s,r)}\ell^{-1}=\ell_{k}^{-1}(\ell_{(s,r)}\ell^{-1}) \ell_{k}$, we obtain $\ell_{(s,r)}\ell^{-1} \in \mathcal{L}(E(\mu_k))$ for $k\in\{0,1\}$.
    Thus, by interpolation, we obtain the desired convexity property
    \begin{align*}
       \ell_{(s,r)}\ell^{-1}&\in \mathcal{L}(\mathcal{I}(E(\mu_0), E(\mu_1)))=\mathcal{L}(E_{\mathcal{I}, \mu_0, \mu_1}).
    \end{align*}
    Moreover, for any order function $\nu$ and a corresponding lift $\ell_\nu$ we have by Proposition \ref{js014}
    \begin{align*}
    \ell_{\nu}^{-1}
    &\in
    \mathcal L_{\mathrm{iso}} \bigl( \{E(\mu_0),E(\mu_1)\}, \{E(\nu+\mu_0),E(\nu+\mu_1)\} \bigr).
    \end{align*}
    Hence, by interpolation,
    \begin{align*}
        \ell_{\nu}^{-1} E_{\mathcal I,\mu_0,\mu_1}
        &= \mathcal I(E(\nu+\mu_0),E(\nu+\mu_1)),
    \end{align*}
    which proves the asserted identity.
    Finally, let $A\in\call(V)$ be $\mathcal{N}$-elliptic for an order function $\mu$.
    Then in particular $A\in \mathcal{L}_{\mathrm{iso}}(E(\mu+\mu_k+\nu),E(\mu_k+\nu))$ for all $k\in\{0,1\}$ and all order functions $\nu$.
    Interpolating these isomorphisms yields
    \begin{align*}
        A \in\mathcal L_{\mathrm{iso}}(E_{\mathcal I,\mu_0,\mu_1}(\mu+\nu),E_{\mathcal I,\mu_0,\mu_1}(\nu)),
    \end{align*}
    which is $\mathcal{N}$-ellipticity for $\mu$ in the scale generated by $E_{\mathcal{I}, \mu_0, \mu_1}$.
    This completes the proof.
\end{proof}
Although the generating space $E_{\mathcal I,\mu_0,\mu_1}$ depends on both
$\mu_0$ and $\mu_1$, the resulting interpolated Newton polygon scale itself depends only on the difference $\mu_1-\mu_0$.
Indeed, for all order functions $\mu$ and $\nu$, Proposition \ref{prop: ellipticity_interpolation_a} yields the covariance structure
\begin{align*}
    E_{\mathcal I,\mu_0+\nu,\mu_1+\nu}(\mu)
    =E_{\mathcal I,\mu_0,\mu_1}(\mu+\nu).
\end{align*}
Changing $\mu_0$ and $\mu_1$ by the same order function merely changes
the choice of origin within the same scale.

\section{The Whole-Space Problem}\label{sec:whole_space}

This chapter is devoted to well-posedness results for time-periodic equations and mixed-order systems.
We begin by showing how the Newton polygon theory developed in Section \ref{sec:newt} can be applied to the time-periodic framework introduced in the preceding chapter, yielding general and straightforward well-posedness results in purely oscillatory Bessel potential spaces.
We then establish a collection of interpolation results, which allow us to extend these results to the scale of Besov spaces.

\subsection{The Time-Periodic Setting}\label{sec:time_per}
We recall the time-periodic setting as presented e.g.~in \cite{EiK17,KyS17,kyed2019time}.
Since any time-periodic function $f:\mathbb{R}\to \mathbb{C}$ of period $T>0$ can be considered as a function defined on the torus $\mathbb{T}:=\mathbb{R}/ T\mathbb{Z}$, we define $G:=\mathbb{T}\times \mathbb{R}^n$ and aim at solving our equations in some spaces of functions defined on $G$. Here the differential structure on $G$ is  inherited from $\mathbb{R}\times \mathbb{R}^n$ through the quotient projection $\pi: \mathbb{R}\times \mathbb{R}^n\to G$, that is, 
\begin{equation*}
    C^\infty(G)=\{f: G\to \mathbb{C} \mid f\circ \pi \in C^\infty(\mathbb{R}\times \mathbb{R}^n)\}
\end{equation*}
and differentiation is defined by
\begin{equation*}
    (\partial_t^k\partial_x^\alpha f)(\pi(t, x)):=\partial_t^k\partial^\alpha_x (f\circ \pi)(t, x).
\end{equation*}
for all $f\in C^\infty (G)$, $k\in \mathbb{N}, \alpha\in \mathbb{N}_0^n$.

We define the space of test functions on $G$ in the usual way. For $K\subseteq \R^n$, define $\CRi_K(G):=\set{f\in \CRi(G)\mid \supp f\subseteq \mathbb{T}\times K}$, which is a Fréchet space for the seminorms
\begin{align*}
 \rho_{m,K}(f):=\sup\set{\sum_{k+|\alpha|\le m} |\partial_t^k\partial_x^\alpha f(t,x)| \mid (t,x)\in \mathbb{T}\times K}.
\end{align*}
For a fixed,  exhaustive sequence $K_1\subseteq K_2\ldots\subseteq \R^n$ of compact sets  we can define the test function space $\CRci(G):=\bigcup_{j\in\N} \CRi_{K_j}(G)$ endowed with the canonical LF topology. We also introduce the set of distributions $\cald'(G)$ as the strong dual of $\CRci(G)$.

Observe that since derivatives are continuous operators on $\CRci(G)$, distributional derivatives can be introduced by transposition in the usual way.
Integration on $G$ is done with respect to the product of the (normalized) Haar measure on $\mathbb{T}$ and of the Lebesgue measure on $\mathbb{R}^n$.
In particular, it holds  $f\in L^p(G)$ if and only if $f\circ \pi_{\vert [0, T]\times \mathbb{R}^n}\in L^p([0, T]\times \mathbb{R}^n)$. 

Once weak derivatives and integration have been defined, we can introduce Sobolev spaces.
Since our focus is on evolution equations, we distinguish between temporal and spatial regularity.
We therefore define for $s,r\in\N_0$ the Sobolev space of dominating mixed smoothness
\begin{equation}\label{eqn: def_sobolev}
    H^{(s, r), p}(G):=\{f\in L^p(G)\mid \partial_t^s\partial_x^\alpha f, \partial_x^\alpha f, \partial_t^sf\in L^p(G) \text{ for all } \alpha\in \mathbb{N}_0^n \text{ with } \vert \alpha\vert=r\}.
\end{equation}
\subsubsection{Fourier Transform and Tempered Distributions on $G$}\label{sec:fourier}
The space $G$ defined above is a locally compact abelian group and thus admits a Pontryagin dual 
\begin{equation*}
    \widehat{G}:=\frac{2\pi}{T}\mathbb{Z}\times \mathbb{R}^n, 
\end{equation*}
where one identifies $\widehat{G}$ by associating to any $(k, \xi)\in \frac{2\pi}{T}\mathbb{Z}\times \mathbb{R}^n $ the character $\chi: G\to \mathbb{C}$, $\chi(t, x):=\mathrm{e}^{\ic x\cdot \xi+\ic tk}$, see e.g. \cite[Section 1.2]{Rud62}.
The measure on $\widehat{G}$ will be the product of the counting measure on $\frac{2\pi}{T}\mathbb{Z}$ and the Lebesgue measure on $\mathbb{R}^n$.\footnote{In fact, we  normalize the Lebesgue measure of $\R^n$ so that the Fourier transform is an isometry between $L^2(G)$ and $L^2(\widehat{G})$.
We suppress this normalization factor in the sequel.}
We can and will identify the Pontryagin dual of $\widehat{G}$ with $G$.

\medskip

One defines a natural differential structure on $\widehat{G}$ by 
\begin{equation*}
    C^\infty(\widehat{G}):=\{h\in C(\widehat{G})\mid h(k, \cdot)\in C^\infty(\mathbb{R}^n) \text{ for all } k\in \mathbb{Z}\}
\end{equation*}
 where for all $f\in C^\infty(\widehat{G}), (k, \xi)\in \widehat{G}$ and $\alpha\in \mathbb{N}^n$, we set 
 \begin{equation*}
     \partial_x^\alpha f(k, \xi):=(\partial_x^\alpha f(k, \cdot))(\xi).
 \end{equation*}

The Schwartz-Bruhat space on $G$ is
\begin{equation*}
    \mathcal{S}(G):=\{f\in C^\infty(G): \rho_{\alpha, \beta, \gamma}(f)<\infty \text{ for all } (\alpha, \beta, \gamma)\in \mathbb{N}_0\times \mathbb{N}_0^n\times \mathbb{N}_0^n\}, 
\end{equation*}
where the topology is induced by the seminorms
\begin{equation*}
    \rho_{\alpha, \beta, \gamma}(f):=\sup\set{\vert x^\gamma \partial_t^\alpha \partial_x^\beta f(t, x)\vert \mid (t, x)\in G}.
\end{equation*}
Similarly, we introduce its counterpart on the dual group $\hat{G}$
\begin{equation*}
    \mathcal{S}(\widehat{G}):=\{h\in C^\infty(\widehat{G}): \hat{\rho}_{\alpha, \beta, \gamma}(h)<\infty \text{ for all } (\alpha, \beta, \gamma)\in \mathbb{N}_0\times \mathbb{N}_0^n\times \mathbb{N}_0^n\},
\end{equation*}
with the seminorms 
\begin{equation*}
    \hat{\rho}_{\alpha, \beta, \gamma}(h):=\sup\set{\vert k^\alpha \xi^\gamma \partial_\xi^\beta h(k, \xi)\vert\mid (k, \xi)\in \widehat{G}}.
\end{equation*}
The tempered distribution spaces  $\mathcal{S}'(G)$ and $\mathcal{S}'(\widehat{G})$ are then introduced as the strong dual of $\mathcal{S}(G)$ and $\mathcal{S}(\widehat{G})$, respectively.

\medskip

As a locally compact abelian group, $G$ admits a Fourier transform  $\mathcal{F}_G: L^1(G)\to C(\widehat{G})$ defined by 
\begin{equation*}
    \mathcal{F}_Gf (k, \xi)= \int_{\mathbb{T}}\int_{\mathbb{R}^n}f(t, x)\exp(-\ic x\cdot \xi - \ic tk)dxdt, \qquad (k, \xi)\in \widehat{G}.
\end{equation*}
and similarly, $\widehat G$ admits a Fourier transform $\mathcal{F}_{\widehat{G}}: L^1(\widehat{G})\to C(G)$
\begin{equation*}
    \mathcal{F}_{\widehat G}h(t, x):=\sum_{k\in \mathbb{Z}}\int_{\mathbb{R}^n}h(k, \xi)\exp(-\ic x\cdot \xi -\ic tk)d\xi, \qquad (t, x)\in G.
\end{equation*}
It holds $\mathcal{F}_G\in \call_{\mathrm{iso}}(\cals(G),\cals(\widehat{G}))$ and $\mathcal{F}_{\widehat{G}}\in \call_{\mathrm{iso}}(\cals(\widehat{G}),\cals(G))$, see \cite[Theorem 3.2]{Waw68}.
By transposition,\footnote{Recall from the beginning of Section \ref{sec:pre} that ${}^t \varphi:Y'\to X'$ for $\varphi:X\to Y$.} the Fourier transform also defines an isomorphism $\calf_G:={}^t\mathcal{F}_{\widehat{G}}\in \call_{\mathrm{iso}}(\mathcal{S}'(G), \mathcal{S}'(\widehat{G}))$.
We will write $\mathcal{F}:=\mathcal{F}_G$.

\medskip
Armed with the Fourier transform on $G$, we can define Fourier multipliers in the usual way. 
We denote by $\mathcal{O}(\mathbb{R}\times \mathbb{R}^n)$ the set of symbols $\mathfrak{m}\in C^\infty (\mathbb{R}\times \mathbb{R}^n)$ such that $\mathfrak{m}$ and all its derivatives are polynomially bounded.
For example, any weight function $w$ for an order function $\mu$ is contained in $\mathcal{O}(\R\times\R^n)$.
We will also set 
\begin{equation*}
    \mathcal{O}(\widehat{G}):=\{ \mathfrak{m}_{\vert \widehat{G}} \mid \mathfrak{m}\in \mathcal{O}(\mathbb{R}\times \mathbb{R}^n)\}. 
\end{equation*}
For a symbol $m\in \mathcal{O}(\widehat{G})$, the multiplication $M_m:\cals(\widehat{G})\to\cals(\widehat{G})$, $M_m(\varphi):=m\varphi$ is linear and bounded, and we can write $mT:={}^t M_m(T)$ for $T\in \cals'(G)$.
Then for $m\in \mathcal{O}(\widehat{G})$, we define $\op[m]\in\call(\mathcal{S}'(G),\mathcal{S}'(G))$ via 
\begin{equation}\label{eqn: def_multiplier}
    \op[m]f :=\mathcal{F}^{-1}m\mathcal{F}f.
\end{equation}
By definition it holds $\op[m_1]\op[m_2]=\op[m_1m_2]=\op[m_2]\op[m_1]$ for all $m_1,m_2\in \calo(\widehat{G})$.

\medskip

A remark on the case $p=2$ is in order:
Since the Fourier transform $\mathcal{F}$ restricted to $L^2(G)$ is an isometric isomorphism $\mathcal{F}\in \call_{\mathrm{iso}}(L^2(G),L^2(\widehat{G}))$ by \cite[Theorem 1.6.1]{Rud62}, the definition in \eqref{eqn: def_multiplier} stays well-defined on $L^2(G)$ if $m\in L^\infty(\widehat{G})$, in which case one has $\op[m]\in \call(L^2(G),L^2(G))$.

\medskip 

In the case $p\neq 2$, the situation is more subtle, as no satisfactory necessary and sufficient condition is currently known for a symbol to define a bounded Fourier multiplier on $L^p(\mathbb{R}^n)$.
For our purposes, however, the following sufficient condition due to Marcinkiewicz (see, for example, Corollary 5.2.5 in \cite{grafakos2008classical}) will suffice.

\begin{theorem}[Marcinkiewicz]\label{thm: Marcinkiewicz}
    Let $p\in(1,\infty)$.
    Let $m$ be a bounded function on $\mathbb{R}^n$ that is of class $C^n$ away from the coordinate axes.
    If for every multi-index $\alpha\in\{0,1\}^n$ the function 
\begin{equation*}
    x\mapsto x^\alpha \partial^\alpha m(x)
\end{equation*}
is bounded, then $m$ is a Fourier multiplier in $L^p(\mathbb{R}^n)$.
\end{theorem}

In order to apply this result to our time-periodic setting, we will use the following transfer theorem, see for example \cite{EdG77} or \cite[Theorem 2.15]{EiK17} for a more modern exposition.
    \begin{theorem}[Transfer principle]\label{thm: transfer}
    Let $G$ and $H$ be locally compact abelian groups with respective Pontryagin dual $\widehat{G}$ and $\widehat{H}$.
    Let $\Phi: \widehat{G}\to \widehat{H}$ be a continuous homomorphism.
    If $m\in C_b(\widehat{H})$ is a Fourier multiplier in $L^p(H)$, then $m\circ \Phi\in C_b (\widehat{G})$ is a Fourier multiplier in $L^p(G)$.
\end{theorem}
In our case, the homomorphism $\Phi$ is given by the inclusion $\Phi: \widehat{G}=\frac{2\pi}{T}\mathbb{Z}\times \mathbb{R}^n\hookrightarrow \mathbb{R}\times \mathbb{R}^n$.  Thus, if $\mathfrak{m}\in \mathcal{O}(\mathbb{R}\times \mathbb{R}^n)$ is a Fourier multiplier in $L^p(\mathbb{R}\times \mathbb{R}^n)$, then its restriction $m:=\mathfrak{m}_{\vert \widehat{G}}\in \mathcal{O}(\widehat{G})$ is a Fourier multiplier in $L^p(G)$.
\subsubsection{Purely Oscillatory Spaces}\label{sec:oscill}
As detailed, for example, in \cite{KyS17}, an efficient way to deal with time-periodic problems is to decompose them into one time-independent problem and one time-dependent problem, called 'purely oscillatory', where one can ignore the eventual singularities of the Fourier multipliers at the origin $k=0$.
This is done by introducing the time-averaging projections $\mathcal{P}\in \call(\mathcal{D}(G))$ and $\mathcal{P}_\perp \in\call(\mathcal{D}(G))$ by
    \begin{equation*}
        \mathcal{P}f(t, x)=\int_{\mathbb{T}}f(s,x)\dd s, \qquad \mathcal{P}_\perp=\id_G-\mathcal{P}.
    \end{equation*}
Note that $\mathcal{P}$ projects onto time-independent functions, while $\mathcal{P}_\perp$ removes the temporal mean.
It is clear that $\mathcal{P}$ and $\mathcal{P}_\perp$ are complementary continuous projections, and we continue to write $\mathcal{P}$ and $\mathcal{P}_\perp$ for their transposes  ${}^t\mathcal{P},{}^t\mathcal{P}_\perp\in \call(\mathcal{D}'(G))$.
For any locally convex space $E(G)\subseteq \mathcal{D}'(G)$, we will write $E_\perp(G):= \mathcal{P}_\perp E(G)$. 
Introducing $\delta_{\mathbb{Z}}:\frac{2\pi}{T}\mathbb{Z}\to \R$ as 
\begin{equation*}
    \delta_{\mathbb{Z}}(k):=\left\{\begin{array}{ll}
    1     & \text{if } k=0,  \\
    0     & \text{if } k\in \frac{2\pi}{T}\mathbb{Z}\setminus \{0\},
    \end{array}\right.
\end{equation*}
we can write $\mathcal{P}=\op[\delta_{\mathbb{Z}}]$ and $\mathcal{P}_\perp= \op[1-\delta_{\mathbb{Z}}]$ as operators on $\mathcal{S}'(G)$.

\medskip

Those projections allow us to decompose differential equations. Given an equation of the form $\op[m]u=f$ for some $u, f\in \mathcal{S}'(G)$ and $m\in \mathcal{O}(\widehat{G})$, one has 
\begin{equation*}
    \op[m]u=f 
    \quad \Leftrightarrow \quad 
    \left\{\begin{array}{rcl}
    \mathcal{P}\op[m] u&=&\mathcal{P}f,      \\
    \mathcal{P}_\perp \op[m]u&=&\mathcal{P}_\perp f,  
    \end{array}\right. 
    \quad \Leftrightarrow \quad 
    \left\{\begin{array}{rcl}
    \op[\delta_{\mathbb{Z}}m] u&=&\mathcal{P}f    ,  \\
     \op[(1-\delta_{\mathbb{Z}})m]u&=&\mathcal{P}_\perp f.  
    \end{array}\right. 
\end{equation*}
Here the line associated with $\mathcal{P}$ is no longer time-dependent and often falls into the scope of classical elliptic theory.
We will thus focus mainly on the part associated with $\mathcal{P}_\perp$, often called \emph{purely oscillatory} problem.
Indeed, reducing the time-periodic problem to the purely oscillatory setting allows us to consider operators $\op[m]$ for which the symbol $m$ may exhibit singularities in $k=0$. 

\medskip

To this end, we introduce a cut-off function $\eta\in C^\infty(\mathbb{R})$ with $\mathbbm{1}_{[-\frac{\pi}{T},\frac{\pi}{T}]}\le 1-\eta\le \mathbbm{1}_{[-\frac{2\pi}{T},\frac{2\pi}{T}]}$.
In particular
\begin{equation}\label{eqn: def_cutoff}
        \eta(\tau)=\left\{\begin{array}{ll}
        0     & \text{if } \vert \tau\vert \leq \frac{\pi}{T}  \\
        1     & \text{if } \vert \tau \vert \geq \frac{2\pi}{T}.
        \end{array}\right.
    \end{equation}
Then $\eta_{\vert \frac{2\pi}{T}\mathbb{Z}}=1-\delta_{\mathbb{Z}}$. 
We set
\begin{align*}
    \mathcal{O}^\perp(\R\times\R^n)&:=\{\mathfrak{m}\in\CRi((\R\setminus\{0\})\times\R^n) \mid \eta \mathfrak{m} \in \mathcal{O}(\mathbb{R}\times \mathbb{R}^n) \},\\
    \mathcal{O}^\perp(\widehat{G})&:=\{(\eta \mathfrak{m})_{\vert \widehat{G}}\mid \mathfrak{m}\in\mathcal{O}^\perp(\R\times\R^n)\}.
\end{align*}
Then $m=(1-\delta_{\Z})m$ for $m\in \mathcal{O}^\perp(\widehat{G})$.
If there is no danger of confusion, for $\mathfrak{m}\in \mathcal{O}^\perp(\R\times\R^n)$ we will often write
\begin{equation}\label{eqn: def_multiplier_perp}
    \op[\mathfrak{m}]:= \op[(\eta\mathfrak{m})_{|\widehat{G}}]:\mathcal{S}'_\perp(G)\to \mathcal{S}'_\perp(G). 
\end{equation}

\subsection{Results in Bessel-Potential Spaces}\label{sec:Bessel}

We recall the definition of the weight functions $w_{(s, r)}(\tau, \xi):=\langle \tau\rangle^s\langle \xi\rangle^r$ for the order function $\mu_{(s,r)}=so_\infty+ ro_0$, where $r,s\in\R$.
\begin{proposition}\label{js016}
   The family $\{\op[w_{(s,r)}] \mid s,r\in \mathbb{R}\}$ is a lifting representation on $\mathcal{S}'_\perp(G)$ in the sense of Definition \ref{def: lift_repr}, and any lift operator of this representation is of the form $\ell=\op[w]$, where $w$ is a weight function for an order function $\mu$.
   Moreover, $E:=L^p_\perp(G)$ generates a Newton polygon scale $\mathcal{H}_\perp(G)$ with respect to this lifting representation.
\end{proposition}
\begin{proof}
    Since $w_{(s,r)},w_{(s,r)}^{-1}\in \mathcal{O}^\perp(\R\times\R^n)$, it holds $\op[w_{(s,r)}]\in\mathcal{L}_{\mathrm{iso}}(\mathcal{S}'_\perp(G))$.
    By definition, it is clear that the homomorphism properties $\op[w_{(s,r)}]\op[w_{(s', r')}]= \op[w_{(s+s', r+r')}] $, $\op[w_{(0,0)}]= \op[1]= \id_{\mathcal{S}_\perp'(G)}$ and $\op[w_{(-s,-r)}]=\op[w_{(s,r)}]^{-1}$ hold for $s, r, s',  r'\in \mathbb{R}$.
    The definitions of weight functions and of lift operators directly show that any lift $\ell$ is of the form $\op[w]$ for some weight function.

    Let $\mu$ be a strictly positive order function and $\ell$ a lift for $\mu$. By the above, it holds $\ell= \op[w]$ for some weight function $w$ associated to $\mu$. Since $w, w^{-1} \in \mathcal{O}^\perp(\mathbb{R}\times \mathbb{R}^n)$, we have $\op[w]\in \mathcal{L}_{\mathrm{iso}}(\mathcal{S}'(G))$ with inverse $\op[w]^{-1}:= \op[\frac{1}{w}]$. Thus the lifting property is verified. 

    We now show the convexity property: let $w$ be as above, $(r, s)\in \mathcal{N}(\mu)$  and set
    \begin{equation*}
        \mathfrak{m}:= \eta\frac{w_{(s,r)}}{ w}
    \end{equation*}
    with $\eta$ the cutoff function defined in (\ref{eqn: def_cutoff}). 
    It is clear that the symbol $\eta$ has strong upper order function $\chi\equiv 0$.
    Furthermore, by Corollary \ref{cor: ellipt_weight_funct}, $w_{(s,r)}$ has strong upper order function $\mu$ and $w$ is $\mathcal{N}$-elliptic with order function $\mu$. By Proposition \ref{prop: arithmetics_order_functionsiii}, we get that $\mathfrak{m}$ has strong upper order function $\chi\equiv 0$. Thus it is a Fourier multiplier in $L^p(\mathbb{R}\times \mathbb{R}^n)$ by Marcinkiewicz theorem \ref{thm: Marcinkiewicz}.
    By the transfer principle (Theorem \ref{thm: transfer}), we conclude that its restriction 
    \begin{equation*}
        m:= (1-\delta_{\mathbb{Z}})\Big(\frac{w_{(s,r)}}{ w}\Big)_{|\widehat{G}}
    \end{equation*}
    to $\widehat{G}$ is a Fourier multiplier in $L^p_\perp(G)$. Consequently, $\op[m]\in \call(L^p_\perp(G))$, which shows the convexity property.
\end{proof}
The space $H^{\mu,p}_\perp(G)$ is defined via
\begin{align}\label{def: polygon_space}
    H^{\mu,p}_\perp(G):=E(\mu)=\op[w^{-1}]L^p_\perp(G),    
\end{align}
and its topology is realized by the norm $\|f\|_{\mu, p}:=\|\op[w]f\|_{L^p_\perp(G)}$, where $w$ is any weight function for $\mu$.
For $s,r\in \R$, we also define $H^{(s,r),p}_\perp(G):=E(s,r)=\op[w_{(s,r)}]^{-1}L^p_\perp(G)$.
In particular, for $s,r \in \mathbb{N}$ we may use a standard cut-off method (see for example \cite[Theorem 6.2.3]{BeL76}) to deduce that the above definition of $H^{(s,r),p}_\perp(G)$ is consistent with \eqref{eqn: def_sobolev} for $s,r\in \mathbb{N}_0$.
Moreover, for $\mu$ strictly positive, Proposition \ref{prop: properties_Newton_scales} yields the representation as intersection
\begin{equation}\label{eqn: Bessel_intersection}
    H^{\mu, p}_\perp(G) = \bigcap_{i=1}^{I} H^{(s_i,r_i), p}_\perp(G)
\end{equation}
with equivalent norms, whenever $\caln(\mu)_{SV}\subseteq\set{(r_i, s_i)\mid i\in\set{1,\ldots,I}}\subseteq \caln(\mu)$.

\medskip 

We next state well-posedness results for time-periodic mixed-order scalar equations on the whole space, see Theorem \ref{thm: whole space} below. 

\begin{proposition}\label{prop_boundedness}
    Let $\mu$ and $\nu$ be order functions and $P\in \mathcal{O}^\perp(\R\times\R^n)$.
    Assume that $\mu$ is a strong upper order function for $P$.
    Then $\op[P]\in \mathcal{L}(H^{\mu+\nu, p}_\perp(G),H^{\nu, p}_\perp(G))$. If $p=2$ then one can replace ``strong upper order function'' by ``upper order function'' in the above statement. 
\end{proposition}
\begin{proof}
     As in the discussion after Definition \ref{def: ellipticity_scales}, it suffices to show $\op[\frac{P}{w}]\in \call(L^p_\perp(G))$ for $w$ a weight function associated to $\mu$.
     Therefore, replacing $\mathfrak{m}$ in the proof of Proposition \ref{js016} by $\eta\frac{P}{ w}$ yields the result.
    If $p=2$ and $\mu$ is only an upper order function for $P$, the same argument shows that $\eta\frac{P}{ w}$ has upper order function $\chi\equiv 0$, i.e., it is bounded. Therefore, $\op[\frac{P}{w}]\in \call(L^2_\perp(G))$ by the theorem of Plancherel.
\end{proof}
\begin{theorem}\label{thm: whole space}
    Let $\mu$ and $\nu$ be order functions and $P\in \mathcal{O}^\perp(\R\times\R^n)$.
    Assume furthermore that $P$ is strongly $\mathcal{N}$-elliptic for $\mu$.
    Then $\op[P]\in \mathcal{L}_{\mathrm{iso}}(H^{\mu+\nu, p}_\perp(G),H^{\nu, p}_\perp(G))$ with $\op[P]^{-1}=\op[\frac{1}{P}]$.
    In particular, for all $f\in H^{\nu, p}_\perp(G)$, the equation $\op[P]u=f$ has a unique solution $u\in H^{\nu+\mu, p}_\perp(G)$, and it holds 
    \begin{equation*}
        \|u\|_{\mu+\nu, p}\simeq \|f\|_{\nu, p},
    \end{equation*}
    where the implicit constant does not depend on $f$.
    In the case $p=2$, one can replace ``strongly $\mathcal{N}$-elliptic'' by ``$\mathcal{N}$-elliptic'' in the above statement. 
\end{theorem}
\begin{proof}
    By Proposition \ref{prop_boundedness} we have $\op[P]\in \mathcal{L}(H^{\mu+\nu,p}_\perp(G),H^{\nu,p}_\perp(G))$. 
    Since $P$ is strongly $\mathcal{N}$-elliptic for $\mu$, Proposition~\ref{prop: arithmetics_order_functions} yields that $\frac1P$ is $\mathcal N$-elliptic for the order function $-\mu$, so that the conclusion follows again from Proposition \ref{prop_boundedness}.
\end{proof}
We next prove the counterpart of Theorem \ref{thm: whole space} for mixed-order systems. Here we consider a smooth  matrix-valued  symbol $L: (\mathbb{R}\setminus \{0\})\times \mathbb{R}^n\to \mathbb{C}^{m\times m}$ as well as $\mathcal{D}:=\det L: (\mathbb{R}\setminus\{0\}) \times \mathbb{R}^n\to \mathbb{C}$.
We assume that the symbols $L_{ij}$ belong to the class $\mathcal{O}^\perp(\R\times\R^n)$.
\begin{theorem}\label{thm: whole_space_systems}
    Let $L$ be as above and assume that $L$ is a strong mixed-order system in the sense of Definition \ref{def: mixed_order_systems} with order functions $s_1, ..., s_m$, and $t_1,..., t_m$.
   Let $\nu$ be any order function.
    We define 
    \begin{equation*}
        \mathbb{E}:=\prod_{j=1}^m H^{\nu+t_j,p}_\perp(G) \qquad \text{and} \qquad \mathbb{F}:=\prod_{i=1}^m H^{\nu+s_i,p}_\perp(G).
    \end{equation*}
    Then $\op[L]\in\mathcal{L}_{\mathrm{iso}}(\mathbb{E},\mathbb{F})$ with $\op[L]^{-1}=\op[\ad L]\op[\mathcal{D}^{-1}]$, where $\mathcal{D}^{-1}$ here has to be understood as the diagonal symbol matrix with all diagonal entries equal to $\frac{1}{\mathcal{D}}$.
    In particular, for all $f=(f_1, ..., f_m)\in \mathbb{F}$, the system $\op[L]u=f$ admits a unique solution $u=(u_1,...,u_m)\in \mathbb{E}$, and it holds
    \begin{equation*}
        \sum_{j=1}^m \| u_j\|_{\nu+ t_j,p}\leq C \sum_{i=1}^m \| f_i\|_{\nu+s_i,p}
    \end{equation*}
    for some constant $C>0$ not depending on $f_1,..., f_m$. 
\end{theorem}
\begin{proof}
    By definition of a (strong) mixed-order system, $t_j-s_i$ is a (strong) upper order function for $L_{ij}$.
    Thus by Proposition \ref{prop_boundedness}, $\op[L_{ij}]\in \mathcal{L}(H^{\nu+t_j,p}_\perp(G),H^{\nu+s_i,p}_\perp(G))$ for all $i, j\in\{1,...,m\}$. 

\medskip

    In a similar fashion,  the determinant $\mathcal{D}$ is strongly $\mathcal{N}$-elliptic for the order function $\delta:=\sum_{l=1}^m (t_l-s_l)$.
    Thus, by Theorem \ref{thm: whole space}, $\op[\mathcal{D}]^{-1}\in \mathcal{L}(H^{\nu+s_j, p}_\perp(G),H^{\nu+s_j+\delta, p}_\perp(G))$ for all $j\in\{1,...,m\}$.
    Furthermore, by Proposition \ref{prop: ad} and \ref{prop_boundedness}, we have 
    \begin{equation*}
        \op[(\ad L)_{ij}]\in \mathcal{L}(H^{\nu+s_j+\delta, p}_\perp(G),H^{\nu+s_j +\delta +S_i-T_j, p}_\perp(G))
    \end{equation*}
    for all $i,j\in\{1,\ldots,m\}$, where $S_i$ and $T_j$ are defined in Proposition \ref{prop: ad}.
    Since $\nu+s_j +\delta +S_i-T_j=\nu+t_i$, the boundedness of $\op[\ad L]\op[\mathcal{D}^{-1}]$ follows.
    Since $\op[L]$ and $\op[\ad L]\op[\mathcal{D}^{-1}]$ are mutually inverse, this completes the proof.
\end{proof}
\subsection{Interpolation Results}
This section develops results related to real and complex interpolation of Newton polygon spaces.
Those will be indispensable for the characterization of trace spaces in Section \ref{sec:trc_spc}.
We recall the following result about retractions and coretractions.
\begin{lemma}\label{lem:interp_retraction}
    Let $A_k$, $B_k$, $k\in \{0,1\}$, be Banach spaces.
    If $r_k\in \mathcal{L}(A_k,B_k)$ is a retraction with corresponding coretraction $e_k\in \mathcal{L}(B_k,A_k)$ such that $r_0x=r_1x$ for all $x\in A_0\cap A_1$ and $e_0y=e_1y$ for all $y\in B_0\cap B_1$, then for all $\theta\in (0,1)$ and $q\in (1,\infty)$ it holds
    \begin{align*}
        r[A_0,A_1]_\theta=[B_0,B_1]_\theta \quad \text{and} \quad r(A_0,A_1)_{\theta,q}=(B_0,B_1)_{\theta,q},
    \end{align*}
    with equivalent norms, where $r:A_0+A_1\to B_0+B_1$ is given by $r(a_0+a_1)=r_0a_0+r_1a_1$.
\end{lemma}
\begin{proof}
    Follows from Theorem 1.2.4 in \cite{triebel1995interpolation}, see e.g. Lemma 1.51 in \cite{denk2013general}.
\end{proof}

\begin{lemma}\label{lemma: bdd_H_infty}
    Let $\mu,\nu$ be order functions and let $w$ be a weight function for $\mu$.
    Assume that $H^{\mu+\nu,p}_\perp(G)\hookrightarrow H^{\nu,p}_\perp(G)$.
    Then the operator $A:=\op[w]$ with domain $D(A):=H^{\mu+\nu,p}_\perp(G)$ admits a bounded $H^\infty$-calculus of angle zero on $H^{\nu,p}_\perp(G)$.
    This calculus is given by $f(A)=\op[f\circ w]$ for $f\in H^\infty(\Sigma_\phi)$, $\phi\in (0,\pi)$.
\end{lemma}
\begin{proof}
    Let $\phi\in(0,\pi)$ and $f\in H^\infty(\Sigma_\phi)$.
    Since $w>0$, the disc $B(w,w\sin(\phi/2))$ is contained in $\Sigma_\phi$, and therefore Cauchy's formula gives $|w^kf^{(k)}(w)|\lesssim_{k,\phi}\|f\|_\infty$ for every $k\in\mathbb N_0$.
    For every $(\alpha,\beta)\in\{0,1\}\times\{0,1\}^n$ we have $|(\tau,\xi)^{(\alpha,\beta)}\partial_{(\tau,\xi)}^{(\alpha,\beta)} w|\lesssim w$ by Corollary \ref{cor: ellipt_weight_funct}\ref{cor: ellipt_weight_functiii}, so that Faà di Bruno's formula yields
    \begin{align*}
        \big|(\tau,\xi)^{(\alpha,\beta)}\partial_{(\tau,\xi)}^{(\alpha,\beta)}(f\circ w)(\tau,\xi)\big|
        \lesssim_{\alpha,\beta,\phi}\|f\|_{\infty}.
    \end{align*}
    Theorems \ref{thm: Marcinkiewicz} and \ref{thm: transfer} therefore imply $\|\op[f\circ w]\|_{\mathcal L(L^p_\perp(G))}\lesssim_\phi\|f\|_{H^\infty(\Sigma_\phi)}$.
    As in \cite[Example 10.2]{KuW04}, $A$ admits a bounded $H^\infty$-calculus on $L^p_\perp(G)$ for any angle $\phi>0$ given by $f(A)=\op[f\circ w]$.
    Since $\op[f\circ w]$ commutes with any lift $\ell_\nu\in\call_{\mathrm{iso}}(H_\perp^{\nu,p}(G),L_\perp^p(G))$, the same estimate holds on $H_\perp^{\nu,p}(G)$.
\end{proof}

\begin{proposition}\label{prop: complex_interpolation}
    Let $\mu,\nu$ be order functions and $\theta\in (0,1)$.
    Then
    \begin{align*}
        [H^{\nu,p}_\perp(G),H^{\mu,p}_\perp(G)]_\theta
        =
        H^{(1-\theta)\nu+\theta\mu,p}_\perp(G)
    \end{align*}
    with equivalent norms.
    In particular, the Newton polygon scale $\mathcal{H}_\perp(G)$ is closed under complex interpolation.
\end{proposition}
\begin{proof}
    Since every order function is the difference of two strictly positive order functions, we may choose an order function $\sigma$ such that both $\nu-\sigma$ and $\mu-\sigma$ are strictly positive.
    Let $w_0$ and $w_1$ be weight functions for $\nu-\sigma$ and $\mu-\sigma$, respectively, and set
    \begin{align*}
        X:=H_\perp^{\sigma,p}(G), \qquad
        A_0:=\op[w_0], \qquad
        A_1:=\op[w_1].
    \end{align*}
    Thus $D(A_0)=H_\perp^{\nu,p}(G)$ and $D(A_1)=H_\perp^{\mu,p}(G)$.
    By Lemma~\ref{lemma: bdd_H_infty}, the operators $A_0$ and $A_1$ are invertible and admit bounded $H^\infty$-calculi of angle zero on $X$.
    In particular, they have bounded imaginary powers and, since they are Fourier multipliers, they commute.

    We first identify their real fractional powers.
    Let $j\in\{0,1\}$ and $\alpha\in(0,1)$.
    Observe that $e(z):=z(1+z)^{-2}$ is a regularizer for $f(z):=z^\alpha$ in the sense of \cite[Chapter~1.2]{Haa06}, since both $e$ and $ef$ belong to $H^\infty(\Sigma_\phi)$ for every $\phi\in(0,\pi)$.
    Hence
    $A_j^\alpha=e(A_j)^{-1}(ef)(A_j)$.
    The Fourier multiplier realization of the $H^\infty$-calculus from Lemma~\ref{lemma: bdd_H_infty} gives
    $e(A_j)=\op[e(w_j)]$ and $(ef)(A_j)=\op[e(w_j)w_j^\alpha]$.
    Consequently, $A_j^\alpha=\op[w_j^\alpha]$ as closed operators on $X$.
    By Corollary~\ref{cor: product_weight_functions}, $w_\theta:=w_0^{1-\theta}w_1^\theta$ is equivalent to a weight function for
    $(1-\theta)(\nu-\sigma)+\theta(\mu-\sigma)$, so that $A_\theta:=\op[w_\theta]$ on $X$ satisfies $D(A_\theta) = H_\perp^{(1-\theta)\nu+\theta\mu,p}(G)$ and $A_\theta=A_0^{1-\theta}A_1^\theta$.
    
    Thus, the result follows as soon as we establish $D(A_\theta)=[D(A_0),D(A_1)]_\theta$ with equivalent norms.
    Since \cite[Theorem 6.6.9]{Haa06} concerns fractional domains of a single operator and does not directly apply to the two commuting operators $A_0$ and $A_1$, we include the argument here.
    For the embedding $D(A_\theta)\hookrightarrow[D(A_0),D(A_1)]_\theta$, let $u\in D(A_\theta)$ and define
    \begin{align*}
        F(z):=
        e^{(z-\theta)^2}
        A_0^{z-1}A_1^{-z}A_\theta u, \qquad \Re z\in [0,1].
    \end{align*}
    Then $F(\theta)=u$, and since $A_0$ and $A_1$ are invertible and have bounded imaginary powers, $F$ is continuous on the closed strip and holomorphic in its interior, see \cite[Propositions~3.2.1 and 3.5.5]{Haa06}.
    Furthermore, for $s\in\R$,
    \begin{align*}
        A_0F(\ic s) &= e^{(\ic s-\theta)^2} A_0^{\ic s}A_1^{-\ic s}A_\theta u,
        \qquad A_1F(1+\ic s) = e^{(1+\ic s-\theta)^2} A_0^{\ic s}A_1^{-\ic s}A_\theta u.
    \end{align*}
    Since $A_0$ and $A_1$ have bounded imaginary powers, there are $M\geq1$ and $\eta\geq0$ such that
    $\|A_0^{\ic s}A_1^{-\ic s}\|_{\call(X)}\leq Me^{\eta|s|}$ for all $s\in\R$.
    The Gaussian factor therefore ensures that both $\|F(\ic s)\|_{D(A_0)}$ and $\|F(1+\ic s)\|_{D(A_1)}$ are controlled by $\|A_\theta u\|_X$ uniformly in $s\in\R$, so that
    \begin{align}\label{eq: interpolation_first_inclusion}
        \|u\|_{[D(A_0),D(A_1)]_\theta}
        \lesssim
        \|A_\theta u\|_X.
    \end{align}

    For the converse estimate, let first $u\in D(A_0)\cap D(A_1)$ and consider functions of the form $F(z)=\sum_{l=1}^N\varphi_l(z)u_l$ with $F(\theta)=u$, where $N\in\N$, $u_l\in D(A_0)\cap D(A_1)$, and $\varphi_l\in\calf_0(\C,\C)$, in the notation for complex interpolation from \cite[Appendix~B.3]{Haa06}.
    Define
    \begin{align*}
        G(z):= e^{(z-\theta)^2} A_0^{1-z}A_1^zF(z), \qquad \Re z\in [0,1].
    \end{align*}
    We argue that $G(z)$ is well-defined, i.e., that the operator $A_0^{1-z}A_1^z$ can be applied to every $u_l$ occurring in $F$.
    For $z=\alpha+\ic s$ with $\alpha\in[0,1]$, set
$\lambda_\alpha:=(1-\alpha)(\nu-\sigma)+\alpha(\mu-\sigma)$.
Then $\mathcal N(\lambda_\alpha)$ is contained in the convex hull of
$\mathcal N(\nu-\sigma)\cup\mathcal N(\mu-\sigma)$, whose order function is
$\eta-\sigma:=\max\{\nu-\sigma,\mu-\sigma\}$.
Hence Proposition~\ref{js014} and Proposition~\ref{prop: properties_Newton_scales} give
\begin{align*}
    D(A_0)\cap D(A_1) = H_\perp^{\nu,p}(G)\cap H_\perp^{\mu,p}(G) = H_\perp^{\eta,p}(G)\hookrightarrow H_\perp^{\sigma+\lambda_\alpha,p}(G) = D(A_0^{1-\alpha}A_1^\alpha).
\end{align*}
Together with the bounded imaginary powers and
$A_0^{1-z}A_1^z=A_0^{-\ic s}A_1^{\ic s}A_0^{1-\alpha}A_1^\alpha$, this shows that $G(z)$ is well defined.
    The holomorphy of $G$ in the open strip follows from \cite[Proposition~3.2.1]{Haa06}.
    Moreover, it holds $G(\theta)=A_0^{1-\theta}A_1^\theta F(\theta)=A_\theta u$ and
    \begin{align*}
        G(\ic s) &= e^{(\ic s-\theta)^2} A_0^{-\ic s}A_1^{\ic s}A_0F(\ic s),
        \qquad G(1+\ic s) = e^{(1+\ic s-\theta)^2} A_0^{-\ic s}A_1^{\ic s}A_1F(1+\ic s).
    \end{align*}
    The bounded imaginary powers and the three-lines lemma therefore yield
    \begin{align*}
        \|A_\theta u\|_X
        \lesssim
        \sup_{s\in\R}
        \max\{
        \|F(\ic s)\|_{D(A_0)},
        \|F(1+\ic s)\|_{D(A_1)}
        \}.
    \end{align*}
    Taking the infimum over all such functions $F$ gives $\|A_\theta u\|_X \lesssim \|u\|_{[D(A_0),D(A_1)]_\theta}$ for all $u\in D(A_0)\cap D(A_1)$ by \cite[Proposition B.3.4]{Haa06}.
    Since the embedding $D(A_0)\cap D(A_1)\subseteq [D(A_0),D(A_1)]_\theta$ is dense and $A_\theta$ is closed, this shows $[D(A_0),D(A_1)]_\theta\hookrightarrow D(A_\theta)$.
    This completes the proof.
\end{proof}
In line with results on real interpolation of Bessel potential and Besov spaces, see \cite[Proposition 6.1]{MeV12} or \cite[Theorem VII.4.1.3]{amann2019linear} for real interpolation results of the Bessel-potential scale in the Banach-space valued setting, we define purely oscillatory directional Besov-Bessel potential spaces as interpolation of purely oscillatory Bessel spaces. 
\begin{definition}\label{def: purely_osc_besov}
     Let $\eta$ be an order function.
     For $\theta\in (0,1)$ and $a<0<b$ with $0=(1-\theta)a+ \theta b$ we define
     \begin{align*}
         B^{(0,0),p}_{\eta,\perp}(G):= (H^{a\eta, p}_\perp (G), H^{b\eta, p}_\perp(G))_{\theta,p}.
     \end{align*}
     We write $BH^{(0,0),p}_\perp(G):=B^{(0,0),p}_{o_\infty,\perp}(G)$ and $HB^{(0,0),p}_\perp(G):=B^{(0,0),p}_{o_0,\perp}(G)$.
\end{definition}
By Proposition \ref{prop: complex_interpolation} and the reiteration theorem, this definition is independent, up to equivalence of norms, of the choice of $a$ and $b$.
In particular, replacing $\eta$ by $r\eta$, $r>0$, does not change the space.
By Proposition~\ref{prop: ellipticity_interpolation_a}, the space $E:=B^{(0,0),p}_{\eta,\perp}(G)$ generates a Newton polygon scale $\calb_\eta$, whose spaces we denote by $B^{\mu,p}_{\eta,\perp}(G):=E(\mu)$ and $B^{(s,r),p}_{\eta,\perp}(G):=E(s,r)$.
For the distinguished directions $o_\infty$ and $o_0$, we write $\calb\calh:=\calb_{o_\infty}$ and $\calh\calb:=\calb_{o_0}$, with corresponding notation $BH^{\mu,p}_\perp(G):=B^{\mu,p}_{o_\infty,\perp}(G)$ and $HB^{\mu,p}_\perp(G):=B^{\mu,p}_{o_0,\perp}(G)$.
We record a direct corollary of Proposition \ref{prop: ellipticity_interpolation_a} and Theorem \ref{thm: whole space}.
\begin{corollary}\label{cor: ellipticity_HB_BH}
    If $P\in \mathcal{O}^\perp(\R\times\R^n)$ is strongly $\mathcal{N}$-elliptic with order function $\mu$ in the sense of Definition \ref{def: Newton_ellipticity}, then for all order functions $\nu,\eta$ it holds 
    \begin{equation*}
        \op[P]\in \mathcal{L}_{\mathrm{iso}}(B^{\mu+\nu, p}_{\eta,\perp}(G), B^{\nu, p}_{\eta,\perp}(G)).
    \end{equation*}
    If $p=2$, one can replace ``strongly $\mathcal{N}$-elliptic'' by ``$\mathcal{N}$-elliptic'' in the above.
\end{corollary}
For general $\mu_0,\mu_1$, the spaces just introduced interpolate along the direction $\mu_1-\mu_0$.
Indeed, setting $\mu_\theta:=(1-\theta)\mu_0+\theta\mu_1$ and choosing $\eta:=\mu_1-\mu_0$, a common lift $\op[w_{\mu_\theta}]$ gives
\begin{align}\label{eq: real_interp_ray}
    (H^{\mu_0,p}_\perp(G),H^{\mu_1,p}_\perp(G))_{\theta,p} =B^{\mu_\theta,p}_{\mu_1-\mu_0,\perp}(G).
\end{align}
\begin{lemma}\label{lemma: interpolation_intersection}
     Let $\mu$ be an order function, $r>0$ and $y\in (0,\infty)$.
     With $s:=ry$ it holds
      \begin{align*}
          B^{\mu,p}_{o_y,\perp}(G)
          &=BH_\perp^{\mu- r o_y+ s o_\infty, p}(G) \cap HB_\perp^{\mu- r o_y+ r o_0, p}(G).
      \end{align*}
    In particular, for $\theta\in (0,1)$ and $s_0,s_1,r_0,r_1\in\R$ with $s_0<s_1$ and $r_0<r_1$, it holds
    \begin{align*}
        (H^{(s_0,r_0),p}_\perp(G),H^{(s_1,r_0),p}_\perp(G)\cap H^{(s_0,r_1),p}_\perp(G))_{\theta,p} = BH^{(s_\theta,r_0),p}_\perp(G) \cap HB^{(s_0,r_\theta),p}_\perp(G).
    \end{align*}
    where $s_\theta:=(1-\theta)s_0+\theta s_1$ and $r_\theta:=(1-\theta)r_0+\theta r_1$.
\end{lemma}
\begin{proof}
    By a common lift, we may assume $\mu=ro_y$.
    Let $w$ and $w'$ be weight functions for $2so_\infty$ and $2ro_0$, respectively.
    Then $A: =\op[w]$ and $B:=\op[w']$ acting in $X:=L^p_\perp(G)$ with respective domains $D(A):=H^{2so_\infty,p}_\perp(G)$ and $D(B):=H^{2ro_0,p}_\perp(G)$ admit a bounded $H^\infty$-calculus of angle zero by Lemma \ref{lemma: bdd_H_infty}.
    In particular, Theorem 1.15.5 in \cite{triebel1995interpolation} is applicable, which yields $(X,D(A)\cap D(B))_{\theta,p}=(X,D(A))_{\theta,p}\cap (X,D(B))_{\theta,p}$.
    Since $H^{2ro_y,p}_\perp(G)=H^{2so_\infty,p}_\perp(G)\cap H^{2ro_0,p}_\perp(G)=D(A)\cap D(B)$, we get 
    \begin{align*}
         (L^p_\perp(G), H^{ro_y,p}_\perp(G))_{\frac12,p}
         &=\big(L^p_\perp(G), H^{2so_\infty,p}_\perp(G)\big)_{\frac12, p}\cap\big(L^p_\perp(G), H^{2ro_0,p}_\perp(G)\big)_{\frac12, p},
    \end{align*}
    i.e., $B^{ro_y,p}_{o_y,\perp}(G)=B_{o_\infty,\perp}^{s o_\infty, p}(G) \cap B_{o_0,\perp}^{r o_0, p}(G)$ by \eqref{eq: real_interp_ray} and the invariance of $B_{\eta,\perp}^{\mu,p}(G)$ under positive rescaling of $\eta$.
    This is the assertion for $\mu=ro_y$.
    The additional interpolation statement follows by setting $r:=\theta(r_1-r_0)$, $y:=\frac{s_1-s_0}{r_1-r_0}$ and $\mu:=\mu_0+\theta (r_1-r_0)o_y$ with $\mu_0:=s_0o_\infty+r_0o_0$. 
\end{proof}
\begin{lemma}\label{lemma: geometric_embedding_B}
    Let $\mu,\nu,\eta,\zeta$ be order functions.
    Assume that there exists $\varepsilon>0$ such that $H^{\mu\pm\varepsilon\eta,p}_\perp(G) \hookrightarrow H^{\nu\pm\varepsilon\zeta,p}_\perp(G)$.
    Then $B^{\mu,p}_{\eta,\perp}(G) \hookrightarrow B^{\nu,p}_{\zeta,\perp}(G)$.
\end{lemma}
\begin{proof}
    By \eqref{eq: real_interp_ray} and invariance under positive rescaling of $\eta$ and $\zeta$, respectively, we obtain
    \begin{align*}
        B^{\mu,p}_{\eta,\perp}(G) = (H^{\mu-\varepsilon\eta,p}_\perp(G), H^{\mu+\varepsilon\eta,p}_\perp(G))_{\frac12,p},
        \qquad
        B^{\nu,p}_{\zeta,\perp}(G) = (H^{\nu-\varepsilon\zeta,p}_\perp(G), H^{\nu+\varepsilon\zeta,p}_\perp(G))_{\frac12,p}.
    \end{align*}
    Since $H^{\mu\pm\varepsilon\eta,p}_\perp(G) \hookrightarrow H^{\nu\pm\varepsilon\zeta,p}_\perp(G)$, the assertion follows.
\end{proof}

\section{The Half-Space Problem with General Boundary Conditions}\label{sec:half_space}

\subsection{Quotient Spaces and Equivalent Norms}\label{sec:quot_spc}
In the following we write $\mathbb{R}^n_\pm:=\{ x\in \mathbb{R}^n \mid x_n\gtrless 0\}$ as well as $G_\pm:=\T\times \R^n_\pm$. As we aim at treating the half-space problem in Newton polygon spaces, we consider the quotient spaces $\overline{H}^{\mu, p}(G_+)$ as defined in Definition \ref{def:restr_and_supp_space}, where we recall 
\begin{equation*}
    \overline{H}_\perp^{\mu, p}(G_+)=\{ r^+U: U\in H^{\mu, p}_\perp(G)\}.
\end{equation*}
Here, the name quotient space is justified by the assertion $\overline{H}^{\mu, p}_\perp(G_+)\cong H^{\mu,p}_\perp(G)/\dot{H}^{\mu,p}_\perp(\overline{G_-})$ from Lemma \ref{lem:restr_supp_space}.
Moreover, with $H^{\mu,p}_\perp(G)$, also $\dot H^{\mu,p}_\perp(G_\pm)$ and $\overline{H}_\perp^{\mu, p}(G_\pm)$ are Banach spaces.
In view of the definition in \eqref{def: polygon_space}, it is natural to ask whether the quotient space $ \overline{H}_\perp^{\mu, p}(G_+)$ can be obtained as the image of $L^p_\perp(G_+)$ under some Fourier multiplier. This will be done using the results of Section \ref{sec:half_supp} concerning operators preserving half-space support. Namely, we will define a weight function $\omega_\mu^-$ that is strongly $\mathcal{N}$-elliptic and  preserves lower support.

We start by establishing a useful necessary condition on a Fourier multiplier $m$ for the associated operator $\op[m]$ to preserve half-space support. This will rely on the following Paley-Wiener theorem.
\begin{theorem}[Paley-Wiener-Stein]\label{thm_Paley_Wiener}
    Let $f\in \cals(\R)$.
    Then $\supp f\subseteq \overline{\mathbb{R}_-}$ if and only if $\mathcal{F}_\R f$ admits an extension $\widetilde{\mathcal{F}_\R f}\in C_b(\overline{\mathbb{C}_+})$ that is holomorphic on $\mathbb{C}_+$.
    In that case, such an extension is given by 
    \begin{equation*}
        \widetilde{\mathcal{F}_\R f}(z):=\int_{-\infty}^\infty f(x)\mathrm{e}^{-\ic x z}d x,\qquad z\in \overline{\mathbb{C}_+}.
    \end{equation*}
\end{theorem}
\begin{proof}
See \cite[Theorem 3.5]{stein2010complex}.
\end{proof}
\begin{proposition}\label{prop_support_preserving}
    Let $P\in \mathcal{O}(\mathbb{R})$ and
    suppose that $P$ admits an extension $\widetilde{P}\in C(\overline{\mathbb{C}_+})$ that is holomorphic on $\mathbb{C}_+$, and such that there exists a constant $C>0$ and a polynomial $Q$ verifying 
    \begin{equation*}
        \vert \widetilde{P}(z)\vert \leq C\vert Q(z)\vert
    \end{equation*}
    for all $z\in \overline{\mathbb{C}_+}$.
    Then $\op[P]:\cals(\R)\to\mathcal{S}(\mathbb{R})$ preserves lower support. 
\end{proposition}
\begin{proof}
    Let $f\in \dot{\mathcal{S}}(\overline{\mathbb{R}_-})$.
    Since $\calf_\R f\in \cals(\R)$ and $P\in \calo(\R)$, we have $P\calf_\R f \in \cals(\R)$, so that $\op[P]f\in \cals(\R)$.
    Therefore we only need to show that $\supp \op[P]f\subseteq \overline{\R_-}$.

    \medskip
    
    By Theorem \ref{thm_Paley_Wiener}, $\mathcal{F}_\R f$ has an extension $\widetilde{\mathcal{F}_\R f}\in C_b(\overline{\mathbb{C}_+})$ that is holomorphic on $\mathbb{C}_+$.
    Write $Q(z)=\sum_{k=0}^K c_k z^k$.
    Since $\op[Q]=\sum_{k=0}^K c_k (-\ic \partial_x)^k$ is a differential operator, it holds $\op[Q]f\in \dot{\cals}(\overline{\R_-})$.
    Thus, the function $h:=\mathcal{F}_\R\op[Q]f$ admits an extension $\widetilde{h}\in C(\overline{\mathbb{C}_+})$ that is holomorphic on $\mathbb{C}_+$, and given by 
    \begin{align*}
        \widetilde{h}(z)&=\int_{-\infty}^\infty (\op[Q]f)(x)\mathrm{e}^{-\ic x z}\dd x=\sum_{k=0}^K c_k (-\ic)^k\int_{-\infty}^\infty \partial_x^k f(x)\mathrm{e}^{-\ic x z}\dd x\\
        &=\sum_{k=0}^K c_kz^k\int_{-\infty}^\infty f(x)\mathrm{e}^{-\ic x z}\dd x
        =Q(z)\widetilde{\mathcal{F}_\R f}(z).
    \end{align*}
    Consider now $P\mathcal{F}_\R f$.
    By hypothesis, this function admits an extension $\widetilde{P}\widetilde{\mathcal{F}_\R f}$.
    This extension is continuous on $\overline{\mathbb{C}_+}$ and holomorphic on $\mathbb{C}_+$.
    Furthermore, it is bounded on $\overline{\mathbb{C}_+}$ since $\widetilde{h}$ is bounded on $ \overline{\mathbb{C}_+}$ and there it holds
\begin{equation*}
    \vert \widetilde{P}\widetilde{\mathcal{F}_\R f}\vert \leq C\vert Q\widetilde{\mathcal{F}_\R f}\vert =\vert \widetilde{h}\vert.
\end{equation*}
Thus, using once more Theorem \ref{thm_Paley_Wiener}, we deduce $\supp \op[P]f \subseteq \overline{\mathbb{R}_-}$.
\end{proof}
\begin{remark}
    Let $m\in \calo^\perp(\widehat{G})$ be such that for all $k\neq 0$ and $\xi'\in \mathbb{R}^{n-1}$, the operator $\op[m(k, \xi', \cdot)]: \mathcal{S}(\mathbb{R})\to \mathcal{S}(\mathbb{R})$ preserves lower support. Then $\op[m] : \mathcal{S}_\perp(G)\to \mathcal{S}_\perp(G)$ also preserves lower support. Indeed, one has 
    \begin{equation*}
        \op[m]f=\mathcal{F}^{-1}_{G^{n-1}}\op[m(k, \xi', \cdot)]\mathcal{F}_{G^{n-1}},
    \end{equation*}
    and both $\mathcal{F}^{-1}_{G^{n-1}}:\cals(\widehat{G}^{n-1})\to\cals(G^{n-1})$ and $\mathcal{F}_{G^{n-1}}: \mathcal{S}(G^{n-1})\to \mathcal{S}(\widehat{G}^{n-1})$ preserve lower support.
\end{remark}
We now build support-preserving weight functions. 
\begin{proposition}\label{prop:weight_fct_supp}
    Let $\mu$ be an order function.
    Then there exist  smooth functions $\omega_\mu^\pm : \mathbb{R}\times \mathbb{R}^n\to \mathbb{C}$ which are strongly $\mathcal{N}$-elliptic for the order function $\mu$, and such that $\op[\omega_\mu^\pm]$ and $\op[\frac{1}{\omega_\mu^\pm}]$ preserve support in $G_\pm$.
     In particular, one has $\op[\omega_\mu^-]_+\in \call_{\mathrm{iso}}(\overline{H}^{\mu,p}_\perp(G_+),L^p_\perp(G_+))$.
    The functions can be chosen such that $\omega_\mu^\pm\omega_\nu^\pm=\omega_{\mu+\nu}^\pm$ for all order functions $\mu$ and $\nu$.
     If $\mu$ is strictly positive and the first component of every vertex of the associated Newton polygon is an integer, then $\op[\omega_\mu^\pm]$ preserves both upper and lower support.
\end{proposition}
\begin{proof}
    In view of Proposition \ref{prop: arithmetics_order_functions}, we start by showing the result for an elementary order function $o_y$ with $y\in [0,\infty]$. For $y=\infty$, one can simply take $\omega_{o_\infty}^\pm(\tau, \xi', \xi_n):=w_{o_\infty}(\tau, \xi', \xi_n)= \langle \tau\rangle$ because the symbol does not depend on the normal component $\xi_n$. 
    For $y\geq 0$ we choose
    \begin{equation*}
        \omega _{o_y}^\pm(\tau, \xi', \xi_n):=\langle\tau\rangle^y +\langle\xi'\rangle \pm \ic\xi_n.
    \end{equation*}
    Then $\omega_{o_y}^\pm$ is smooth.
    For every fixed $(\tau, \xi')\in \mathbb{R}\times  \mathbb{R}^{n-1} $, the operator $\op[\omega_{o_y}^\pm (\tau, \xi', \cdot)]$ is a differential operator and thus preserves both upper and lower support.
    For every $a\in\R$, we define $(\omega_{o_y}^\pm(\tau,\xi',\cdot))^a$ using the holomorphic branch of the logarithm in the corresponding half-plane.
    Since $\omega_{o_y}^\pm(\tau,\xi',\cdot)$ has no zeros in $\overline{\C_\mp}$, the function $\omega_{o_y}^\pm(\tau,\xi',\cdot)^a$ extends continuously and polynomially boundedly to $\overline{\C_\mp}$ and is holomorphic in $\C_\mp$.
    Hence $\op[(\omega_{o_y}^\pm)^a]$ preserves support in $G_\pm$ for every $a\in\R$ by Proposition~\ref{prop_support_preserving}.

    We now show the strong $\mathcal{N}$-ellipticity.
    The upper estimates for $(\tau,\xi)^{\alpha,\beta}\partial_{(\tau,\xi)}^{(\alpha,\beta)}\omega_{o_y}^\pm$ for $(\alpha,\beta)\in \{0,1\}\times \{0,1\}^{n}$ are immediate.
    Moreover,
    \begin{equation*}
        \vert \omega_{o_y}^\pm(\tau, \xi', \xi_n)\vert^2= \xi_n^2 +w_{o_y}(\tau, \xi', 0)^2 \gtrsim \xi_n^2 +w_{2o_y}(\tau, \xi',0)=w_{2o_y}(\tau, \xi', \xi_n),
    \end{equation*}
    so that the lower estimate follows.
    Hence $\omega_{o_y}^\pm$ is strongly $\mathcal{N}$-elliptic for all $y\in [0,\infty]$ with order function $o_y$.
    If $\mu=\sigma o_\infty+\sum_{k=1}^K a_k o_{y_k}$ with $\sigma,a_k\in \R$ and $y_k\in [0,\infty)$, then by Proposition \ref{prop: arithmetics_order_functions}, we obtain that 
    \begin{equation*}
        \omega_\mu^\pm:= \omega_{o_\infty}^\sigma\prod_{k=1}^{K} (\omega_{o_{y_k}}^\pm)^{a_k}
    \end{equation*}
    is strongly $\mathcal{N}$-elliptic with order function $\mu$.
    In particular, $\op[\omega_\mu^\pm]\in\call_{\mathrm{iso}}(H^{\mu,p}_\perp(G),L^p_\perp(G))$ by Theorem \ref{thm: whole space}.
    Since $\op[\omega_\mu]$ is a composition of the operators $\op[\omega_{o_{y_k}}^{a_k}]$, $k\in\set{1,..., K}$, the support preserving properties for $\op[\omega_\mu]$ and $\op[\frac1{\omega_\mu}]=\op[\frac1{\omega_\mu}]^{-1}$ follow.
    Proposition \ref{prop: boundedness_lower_support} gives $\op[\omega_\mu^-]_+\in \call_{\mathrm{iso}}(\overline{H}^{\mu,p}_\perp(G_+),L^p_\perp(G_+))$.
    
    If $\mu$ is strictly positive and the first component of every vertex of its associated Newton polygon is an integer, we obtain a representation as above with $a_k\in\mathbb N_0$.
    Consequently, $\omega_\mu^\pm$ is polynomial in $\xi_n$, and hence $\op[\omega_\mu^\pm]$ preserves both upper and lower support.
\end{proof}
In contrast to the whole-space setting, the family $\{\overline H^{\mu,p}_\perp(G_+)\mid \mu\text{ order function}\}$ does not inherit the Newton scale structure from the natural lifts.
Indeed, the group homomorphism $(r,s)\mapsto\ell_{(s,r)}:=\op[\omega_{(s,r)}^-]_+$ does not define a lifting representation on $\overline{\cals}'_\perp(G_+)$, since positive sums of the operators $\ell_{(s,r)}$ need not be invertible.
Consequently, Proposition~\ref{prop: properties_Newton_scales} is not directly available, and the following intersection property has to be established separately.

\begin{corollary}\label{cor: equiv_norm_half_space}
    Let $\mu$ be a strictly positive order function and $\caln(\mu)_{SV}\subseteq\set{(r_i,s_i)\mid i\in\set{1,\ldots,I}}\subseteq\caln(\mu)$.
    Suppose that all $r_i\in\N_0$ for all $i\in\{1,\ldots,I\}$.
    Then it holds 
    \begin{equation*}
        \overline{H}^{\mu, p}_\perp(G_+) = \bigcap_{i=1}^{I} \overline{H}^{(s_i, r_i), p}_\perp(G_+).
    \end{equation*}
    with equivalent norms.
\end{corollary}
\begin{proof}
The embedding $\overline{H}^{\mu, p}_\perp(G_+) \hookrightarrow \bigcap_{i=1}^{I} \overline{H}^{(s_i, r_i), p}_\perp(G_+)$ follows from the whole-space result in \eqref{eqn: Bessel_intersection}.
For the other direction, we need a common extension operator $e^+\in \bigcap_{i=1}^{I} \call(\overline{H}^{(s_i,r_i),p}_\perp(G_+),H^{(s_i,r_i),p}_\perp(G))$ with $r^+e^+=\id$.
Since $\overline{H}^{(s_i, r_i), p}_\perp(G_+)=\overline{H}^{r_i,p}(\R^n_+;H^{s_i,p}_\perp(\mathbb T))$ and $r_i\in \N_0$, such a common extension operator is provided in \cite[Theorem VII.1.3.1]{amann2019linear}. Indeed, it does not depend on the Banach space $H^{s_i,p}_\perp(\mathbb T)$ by Remark VI.1.2.4 in \cite{amann2019linear}.
Thus for $u\in \bigcap_{i=1}^{I} \overline{H}^{(s_i, r_i), p}_\perp(G_+)$, we have $e^+u\in \bigcap_{i=1}^{I} H^{(s_i, r_i), p}_\perp(G)=H^{\mu,p}_\perp(G)$, and since $r^+e^+u=u$, it holds
\begin{align*}
\|u\|_{\overline{H}^{\mu,p}_\perp(G_+)}\le \|e^+u\|_{H^{\mu,p}_\perp(G)}\lesssim \sum_{i=1}^I\|e^+u\|_{H^{(s_i,r_i),p}_\perp(G)}\lesssim  \sum_{i=1}^I\|u\|_{\overline{H}^{(s_i,r_i),p}_\perp(G_+)}. 
\end{align*}
This shows the converse embedding $\bigcap_{i=1}^{I} \overline{H}^{(s_i, r_i), p}_\perp(G_+) \hookrightarrow \overline{H}^{\mu, p}_\perp(G_+)$ and completes the proof.
\end{proof}
\subsection{Trace Theory}\label{sec:trc_spc}
In this section, we identify the trace spaces associated with the restricted Newton polygon spaces studied in Section \ref{sec:quot_spc}.
As in the quasihomogeneous, anisotropic setting considered, for example, in \cite{amann2019linear}, \cite{kyed2019time}, and \cite{denk2008parabolic}, the trace spaces are described in terms of intersections of spaces belonging to the $\calb\calh$ and $\calh\calb$ scale.
The more general scales $\calb_\eta$ introduced after Definition \ref{def: purely_osc_besov} provide a uniform description of the traces.
Each trace component belongs to a distinct Newton polygon scale $\calb_{\eta_i}$ associated with the local slope of the shifted polygon.
For $p=2$ all of these scales coincide with the Bessel-potential scale $\calh_\perp$, whereas for $p\neq2$ the trace space generally does not belong to any of the scales $\calh_\perp$, $\calb\calh$ and $\calh\calb$ separately.
Its representation as an intersection of spaces from the latter two scales is dictated by the geometry of the shifted Newton polygon.

\medskip 

Throughout this section, we consider a strictly positive order function of the form $\mu=\sum_{j=1}^{J} (r_j-r_{j+1})o_{y_j}$, where the numbers $r_j \in \mathbb{N}_0$ and the numbers $y_j\in [0,\infty]$ are strictly decreasing in $j$.
Then $\mathcal{N}_V=\{(r_0,s_0),..., (r_{J+1}, s_{J+1})\}$ is the set of vertices of its Newton polygon $\mathcal{N}:=\mathcal{N}(\mu)$, where $r_0=s_0=0$ and the numbers $s_j$ are chosen such that $s_{j+1}-s_{j}=(r_j-r_{j+1})y_j$ for all $j\in \set{1,\ldots, J}$.
We consider the usual trace operators of $i$-th order: for $i,j\in \mathbb{N}_0$, we set
\begin{equation*}
    \Tr_j: \CRci(\overline{G_+})\to \CRci(G^{n-1}), \qquad \Tr_j\phi(t, x'):=\partial_{n}^j\phi(t, x', 0)
\end{equation*}
and
\begin{equation*}
    \Tr^{(i)}: \CRci(\overline{G_+})\to \CRci(G^{n-1})^{i}, \qquad \Tr^{(i)}\phi:=(\Tr_0\phi, \Tr_1\phi,..., \Tr_{i-1}\phi).
\end{equation*}

The aim of this section is to identify the unique function space $T^{\mu, p}_\perp(G^{n-1})$ such that the trace operator $\Tr^{(r_1)}$ extends to a bounded and surjective operator $\Tr^{(r_1)}: \overline{H}^{\mu, p}(G_+)\to T^{\mu, p}_\perp(G^{n-1})$ whose kernel is given by $r^+\dot{H}^{\mu,p}(G_+)$.

The structure of the trace space $T^{\mu, p}_\perp(G^{n-1})$ will depend on the geometric features of the Newton polygon $\mathcal{N}(\mu)$.
We introduce a few more notations to ease its description. For all $\delta\in[0,r_1)$, let $\iota (\delta)\in \{1, ..., J\}$ be the unique number such that
\begin{equation}\label{eqn: def_iota}
    \delta\in [r_{\iota(\delta)+1}, r_{\iota(\delta)}).
\end{equation}
In particular $\iota(0)=J$.
For $i\in\{0,\ldots,r_1-1\}$, set $\eta_i:=o_{y_{\iota(i)}}$.
Geometrically, the interval $[i,i+1]$ lies beneath a unique nonvertical edge of the upper boundary of $\mathcal N$.
This edge has slope $-y_{\iota(i)}$, and hence determines the elementary order function $\eta_i=o_{y_{\iota(i)}}$, see Figure \ref{fig: local_trace_direction}.
\begin{figure}[t]
    \centering
    \begin{tikzpicture}
        \draw[->] (-0.5,0) -- (7,0) node[below] {$|\xi|$};
        \draw[->] (0,-0.5) -- (0,3.5) node[left] {$\tau$};

        \filldraw[draw=Sepia,fill=Sepia!20,thick]
        (0,0)--(6,0)--(5,1)--(2,2.2)--(0,2.6)--cycle;

        \fill[Sepia] (0,0) circle (1.5pt);
        \fill[Sepia] (6,0) circle (1.5pt);
        \fill[Sepia] (5,1) circle (1.5pt);
        \fill[Sepia] (2,2.2) circle (1.5pt);
        \fill[Sepia] (0,2.6) circle (1.5pt);

        \draw[dashed,Sepia!80] (3,0) -- (3,1.8);
        \draw[dashed,Sepia!80] (4,0) -- (4,1.4);
        \draw[Sepia!80] (2,0) -- (2,.1);
        \draw[Sepia!80] (5,0) -- (5,.1);
        \node[below,font=\small] at (3,0) {$i$};
        \node[below,font=\small] at (4,0) {$i+1$};
        \node[below,font=\small] at (2,0) {$r_{\iota(i)+1}$};
        \node[below,font=\small] at (5,0) {$r_{\iota(i)}$};
        \node[below,font=\small] at (6,0) {$r_1$};

        \draw[NavyBlue,very thick] (3,1.8) -- node[midway,above,sloped, yshift=-2pt] {$\eta_i$} (4,1.4);

        \draw[<->,Sepia!80] (3,0.25) -- (4,0.25);
    \end{tikzpicture}
    \hfill
    \begin{tikzpicture}
        \draw[->] (-0.5,0) -- (7,0) node[below] {$|\xi|$};
        \draw[->] (0,-0.5) -- (0,3.5) node[left] {$\tau$};

        \filldraw[draw=Sepia,fill=Sepia!20,thick]
        (0,0)--(6,0)--(5,1)--(2,2.2)--(0,2.6)--cycle;
        \filldraw[draw=Sepia,fill=Sepia!30,thick]
        (0,0)--(3.8,0)--(2.8,1)--(0,2.12)--cycle;

        \fill[Sepia] (0,0) circle (1.5pt);
        \fill[Sepia] (6,0) circle (1.5pt);
        \fill[Sepia] (5,1) circle (1.5pt);
        \fill[Sepia] (2,2.2) circle (1.5pt);
        \fill[Sepia] (0,2.6) circle (1.5pt);

        \fill[Sepia] (3.8,0) circle (1.5pt);
        \fill[Sepia] (2.8,1) circle (1.5pt);
        \fill[Sepia] (0,2.12) circle (1.5pt);

        \draw[Sepia!80] (2,0) -- (2,.1);
        \draw[Sepia!80] (5,0) -- (5,.1);
        \node[below,font=\small] at (3.8,0) {$r_1-\delta$};
        \node[below,font=\small] at (2,0) {$r_{\iota(\delta)+1}$};
        \node[below,font=\small] at (5,0) {$r_{\iota(\delta)}$};
        \node[below,font=\small] at (6,0) {$r_1$};

        \draw[NavyBlue,very thick] (2,2.2) -- node[midway,above,sloped, yshift=-2pt] {$(r_{\iota(\delta)}-r_{\iota(\delta)+1})\eta_i$} (5,1);
        \draw[NavyBlue,very thick] (0,2.12) -- node[midway,above,sloped, yshift=-2pt] {$(r_{\iota(\delta)}-\delta)\eta_i$} (2.8,1);


        \node[font=\small] at (4.3,.5) {$\mathcal{N}$};
        \node[font=\small] at (1.2,.5) {$\mathcal{N}^{(\delta)}$};
    \end{tikzpicture}
    \caption{Left: The interval $[i,i+1]$ lies beneath the edge of the upper boundary of $\mathcal N$ connecting $(r_{\iota(i)},s_{\iota(i)})$ with $(r_{\iota(i)+1},s_{\iota(i)+1})$.
    This edge has slope $-y_{\iota(i)}$ and determines the order function $\eta_i=o_{y_{\iota(i)}}$.
    Right: A visual representation of a shifted Newton polygon.
    The upper left vertex of $\caln$ is $(0,s_{J+1})$, whereas the upper left vertex of $\caln^{(\delta)}$ is given by $(0,s_{\iota(\delta)}+y_{\iota(\delta)}(r_{\iota(\delta)}-\delta))$.
    For $\delta:=i+\frac1p$ its second component is denoted by $s_i'$.}
    \label{fig: local_trace_direction}
\end{figure}
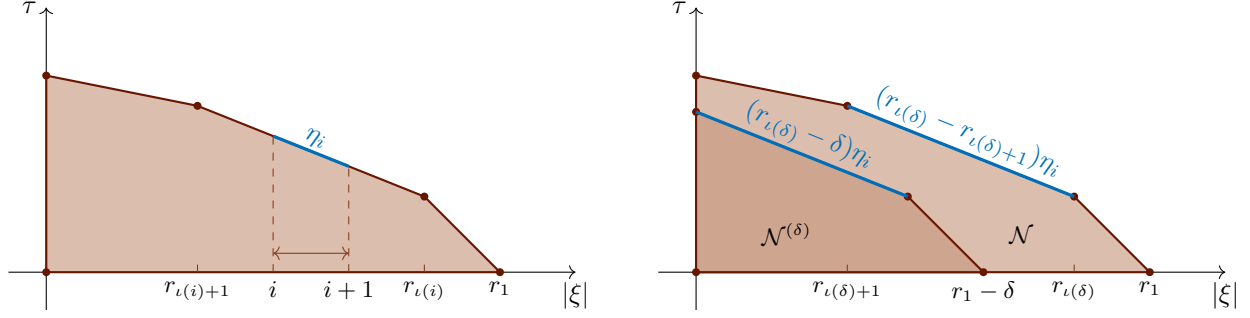
We introduce
\begin{align*}
    \mu^{(r_1)}=0 \quad\text{and}\quad
    \mu^{(\delta)}=(r_{\iota(\delta)}-\delta)o_{y_{\iota(\delta)}}+\sum_{j=1}^{\iota(\delta)-1} (r_j-r_{j+1})o_{y_j}, \quad \delta\in [0,r_1).
\end{align*}
Then $\mu^{(\delta)}$ remains strictly positive and its Newton polygon $\mathcal{N}^{(\delta)}:=\mathcal{N}(\mu^{(\delta)})$ is the Newton polygon $\mathcal{N}$ shifted to the left by $\delta$.
That is, $\caln^{(r_1)}=\set{(0,0)}$, while for $\delta\in [0,r_1)$, $\caln^{(\delta)}$ is given by its vertices
\begin{equation}\label{eqn: def_shifted_polygon}
    \mathcal{N}^ {(\delta)}_V=\{(r_0, s_0), (r_1-\delta, s_1),(r_2-\delta, s_2), ..., (r_{\iota(\delta)}-\delta, s_{\iota(\delta)}), (0,s_{\iota(\delta)}+y_{\iota(\delta)}(r_{\iota(\delta)}-\delta))\}.
\end{equation}
Since the first coordinates of the vertices are integers, the shifted order functions satisfy $\mu^{(i+\delta)}=\mu^{(i)}-\delta\eta_i$ for every $\delta\in[0,1]$.
We define
\begin{align}\label{eqn: def_trace_space_general}
    T^{\mu,p}_{i\perp}(G^{n-1}):=B^{\mu^{(i+\frac1p)},p}_{\eta_i,\perp}(G^{n-1}),
    \qquad
    T^{\mu,p}_{\perp}(G^{n-1}):=\prod_{i=0}^{r_1-1}T^{\mu,p}_{i\perp}(G^{n-1}).
\end{align}
By \eqref{eq: real_interp_ray}, it holds $T^{\mu,p}_{i\perp}(G^{n-1}) = (H^{\mu^{(i+\varepsilon_0)},p}_\perp(G^{n-1}), H^{\mu^{(i+\varepsilon_1)},p}_\perp(G^{n-1}))_{\vartheta,p}$, where $0\leq\varepsilon_1<\frac1p<\varepsilon_0\leq1$ and
$\vartheta:=(\varepsilon_0-\frac1p)/(\varepsilon_0-\varepsilon_1)$.
Lemma~\ref{lemma: interpolation_intersection} provides an explicit representation in the $\calb\calh$- and $\calh\calb$-scales as
\begin{align*}
    T^{\mu,p}_{i\perp}(G^{n-1})
    =
    \left\{\begin{array}{ll}
    \displaystyle\bigcap_{j=1}^{\iota(i)}HB^{(s_j,r_j-i-\frac1p),p}_\perp(G^{n-1}),
    & i<r_{\mathrm{reg}},\\[0.3cm]
    \displaystyle BH^{(s_i',0),p}_\perp(G^{n-1})\cap\bigcap_{j=1}^{\iota(i)}HB^{(s_j,r_j-i-\frac1p),p}_\perp(G^{n-1}),
    & i\geq r_{\mathrm{reg}},
    \end{array}\right.
\end{align*}
where
\begin{equation*}
    r_{reg}:=\left\{\begin{array}{ll}
     0    & \text{if } \mathcal{N} \text{ is regular in space}, \\
     r_J    & \text{if } \mathcal{N} \text{ is not regular in space},
    \end{array}\right.
\end{equation*}
and $s_i':=s_{\iota(i)}+y_{\iota(i)}(r_{\iota(i)}-i-\frac{1}{p})$.
To illustrate the spaces $T_{i\perp}^{\mu, p}$, we present a few examples before turning to the proofs.
\begin{example}[Triangular Case]\label{example : trace_triangular_case}
    We consider a triangular Newton polygon with vertices given by $\mathcal{N}_V=\{(0,0),  (r, 0), (0,s)\}$, where $r\in \mathbb{N}$ and $s>0$. Here the Newton polygon is regular in space and so $r_{reg}=0$. Furthermore, for all $i\in\set{0,..., r-1}$, it holds $\iota(i)= 1$ and thus $\eta_i=o_y$ with $y:=y_1= \frac{s}{r}$.
    Hence
    \begin{equation*}
       T^{\mu, p}_{i\perp}(G^{n-1})= B^{(r-(i+\frac1p))o_y,p}_{o_{y},\perp}(G^{n-1})=BH^{(s-\frac{s}{r}(i+\frac{1}{p}), 0), p}_\perp(G^{n-1})\cap HB^{(0, r-(i+\frac{1}{p})), p}_\perp(G^{n-1}).
    \end{equation*}
    This coincides with the usual trace space for quasihomogeneous spaces, see for example \cite[Chapter VIII 1.3]{amann2019linear} and \cite{kyed2019time}. 
\end{example}
\begin{example}[Cahn--Hilliard--Gurtin Determinant]
    Consider the Newton polygon from Example \ref{example: CHG_ellipticity} associated with the order function $\mu:=\mu_D=2o_\frac12+2o_0$, that is $\mathcal{N}_V=\{(0,0), (4, 0), (2, 1), (0, 1)\}$.
    Here $\iota(0)=\iota(1)=2$ and $\iota(2)=\iota(3)= 1$ while $y_2=0, y_1= \frac{1}{2}$, so that $\eta_0=\eta_1=o_0$ and $\eta_2=\eta_3=o_\frac12$.
    Hence $\mu^{(i+\frac1p)}=2o_\frac12+(2-(i+\frac1p))o_0$ for $i\in \set{0,1}$ and $\mu^{(i+\frac1p)}=(4-(i+\frac1p))o_\frac12$ for $i\in \set{2,3}$.
    Therefore, we find
    \begin{equation*}
        T^{\mu, p}_{i\perp}(G^{n-1})=\left\{ \begin{array}{ll}
         B^{2o_\frac12+(2-(i+\frac1p))o_0,p}_{o_0,\perp}(G^{n-1})=HB^{(1, 2-i-\frac{1}{p}), p}_\perp(G^{n-1})\cap HB^{(0, 4-i-\frac{1}{p}),p}_\perp(G^{n-1})     & i\in\{0,1\}, \\
        B^{(4-(i+\frac1p))o_\frac12,p}_{o_\frac12,\perp}(G^{n-1})=BH^{(1-\frac{1}{2}(i-2+\frac{1}{p}), 0), p}_\perp(G^{n-1})\cap   HB^{(0, 4-i-\frac{1}{p}),p}_\perp(G^{n-1})    & i\in\{2,3\}.
        \end{array}\right.
    \end{equation*}
    Trace spaces of Newton polygons with a similar shape have been treated in \cite{NeS25} in the special case $p=2$. 
    \begin{figure}[ht]
   
   \begin{subfigure}[b]{.49\textwidth}
       \begin{tikzpicture}
 \draw[->] (-.75,0) -- (5,0) node[below , font=\tiny] {$\vert \xi\vert$};
  \draw[->] (0,-.5) -- (0,2) node[left, font=\tiny] {$\tau$};

  \coordinate (A) at (0,0);
  \coordinate (B) at (3.75,0);
  \coordinate (C) at (1.75,1);
  \coordinate (D) at (0,1);

  \filldraw[color=Sepia, fill=Sepia!30,thick] (A)--(B)--(C)--(D)--cycle;

    \fill[color=Sepia] (3.75,0) circle (1.5pt);
  \node[ below, color=Sepia, font=\scriptsize] at ( 3.75,0) {$HB(0,4-\frac{1}{p})$};

 \fill[color=Sepia] (1.75,1) circle (1.5pt);
  \node[ above , color=Sepia, font=\scriptsize] at ( 1.75,1) {$HB(1,2-\frac{1}{p})$};
     
      \node[xshift=4pt, font=\small] at (7/2, 1) {$T^{\mu_D,p}_{0\perp}$};
\end{tikzpicture}
   \end{subfigure}
   \hfill
   \begin{subfigure}[b]{.49\textwidth}
       \begin{tikzpicture}
 \draw[->] (-.75,0) -- (5,0) node[below , font=\tiny] {$\vert \xi\vert$};
  \draw[->] (0,-.5) -- (0,2) node[left, font=\tiny] {$\tau$};

  \coordinate (A) at (0,0);
  \coordinate (B) at (2.75,0);
  \coordinate (C) at (0.75,1);
  \coordinate (D) at (0,1);

  \filldraw[color=Sepia, fill=Sepia!30,thick] (A)--(B)--(C)--(D)--cycle;

    \fill[color=Sepia] (2.75,0) circle (1.5pt);
  \node[ below, color=Sepia, font=\scriptsize] at ( 2.75,0) {$HB(0,3-\frac{1}{p})$};

 \fill[color=Sepia] (0.75,1) circle (1.5pt);
  \node[ above right, color=Sepia, font=\scriptsize] at ( 0.75,1) {$HB(1,1-\frac{1}{p})$};
     
      \node[xshift=4pt, font=\small] at (7/2, 1) {$T^{\mu_D,p}_{1\perp}$};
\end{tikzpicture}
\end{subfigure}

   \begin{subfigure}{.49\textwidth}
      
      \begin{tikzpicture}
 \draw[->] (-.75,0) -- (5,0) node[below , font=\tiny] {$\vert \xi\vert$};
  \draw[->] (0,-.5) -- (0,2) node[left, font=\tiny] {$\tau$};

  \coordinate (A) at (0,0);
  \coordinate (B) at (1.75,0);
  \coordinate (C) at (0, 0.875);

  \filldraw[color=Sepia, fill=Sepia!30,thick] (A)--(B)--(C)--cycle;

\fill[color=Sepia] (1.75,0) circle (1.5pt);
  \node[ below, color=Sepia, font=\scriptsize] at ( 1.75,0) {$HB(0,2-\frac{1}{p})$};

 \fill[color=Sepia] (0,0.875) circle (1.5pt);
  \node[ above right , color=Sepia, font=\scriptsize] at (0,0.875) {$BH(1-\frac{1}{2p},0)$};
     
      \node[xshift=4pt, font=\small] at (7/2, 1) {$T^{\mu_D,p}_{2\perp}$};
\end{tikzpicture}
   \end{subfigure}\hfill
   \begin{subfigure}{.49\textwidth}
       \begin{tikzpicture}
 \draw[->] (-.75,0) -- (5,0) node[below , font=\tiny] {$\vert \xi\vert$};
  \draw[->] (0,-.5) -- (0,2) node[left, font=\tiny] {$\tau$}; 
  
  \coordinate (A) at (0,0);
  \coordinate (B) at (0.75,0);
  \coordinate (C) at (0, 0.375);

  \filldraw[color=Sepia, fill=Sepia!30,thick] (A)--(B)--(C)--cycle;

    \fill[color=Sepia] (0.75,0) circle (1.5pt);
  \node[ below right, color=Sepia, font=\scriptsize,] at ( 0.75,0) {$HB(0,1-\frac{1}{p})$};

 \fill[color=Sepia] (0,0.375) circle (1.5pt);
  \node[ above right , color=Sepia, font=\scriptsize] at (0,0.375) {$BH(\frac12-\frac{1}{2p}, 0)$};
     
      \node[xshift=4pt, font=\small] at (7/2, 1) {$T^{\mu_D,p}_{3\perp}$};
\end{tikzpicture}
   \end{subfigure}
   \caption{The Newton polygons associated with the trace spaces of $\mu_D$.}
   \label{fig: trace_spc}
\end{figure}
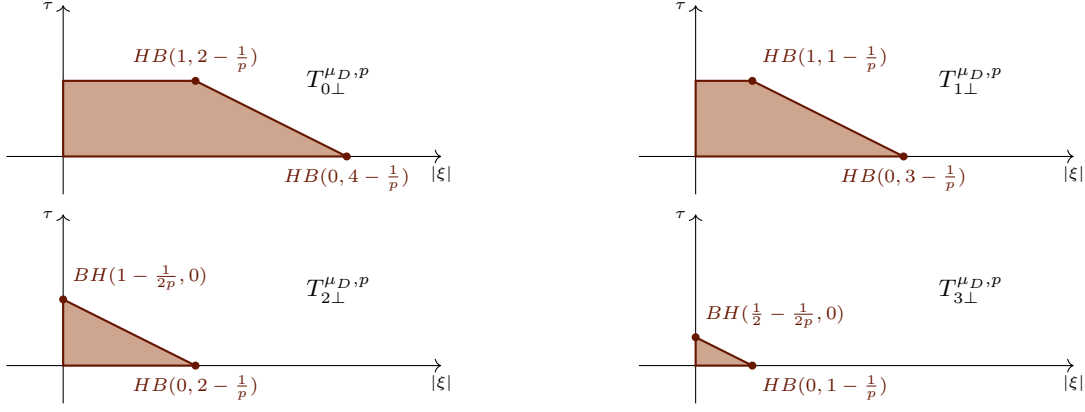
\end{example}
\begin{theorem}\label{thm: trace_space}
    Let $\mu$ be a strictly positive order function, $\mathcal{N}=\mathcal{N}(\mu)$ its Newton polygon and $\mathcal{N}_V=\{(r_0,s_0),..., (r_{J+1}, s_{J+1})\}$ the set of its vertices.
    Assume that $r_j\in \N_0$ for all $j\in\set{0,\ldots,J+1}$, and define $T^{\mu, p}_{\perp}(G^{n-1})$ as in (\ref{eqn: def_trace_space_general}).
    Then the trace operator $\Tr^{(r_1)}$ extends to a bounded linear operator $ \Tr^{(r_1)}\in \mathcal{L}( \overline{H}^{\mu, p}_\perp (G_+) ,T^{\mu, p}_{\perp}(G^{n-1}))$ with $\ker \Tr^{(r_1)}=r^+\dot{H}^{\mu,p}_\perp(G_+)$.
    Furthermore, there is a corresponding extension operator $E^\mu \in \mathcal{L}(T^{\mu, p}_{\perp}(G^{n-1}), \overline{H}^{\mu, p}_\perp (G_+))$ such that $\Tr^{(r_1)}\circ E^\mu= \id_{T^{\mu, p}_{\perp}(G^{n-1})}$.
    In particular, $r^+\dot H^{\mu,p}(G_+)$ is a closed subspace of $\overline H^{\mu,p}_\perp(G_+)$, its image topology induced by $r^+$ agrees with the corresponding subspace topology, and
    \begin{align*}
        (\overline{H}^{\mu, p}_\perp (G_+), r^+\dot{H}^{\mu, p}_\perp (G_+), T^{\mu, p}_{\perp}(G^{n-1}), \Tr^{(r_1)}, E^\mu)   
    \end{align*}
    is an abstract trace tuple in the sense of Definition \ref{def: abstract_trace}.
\end{theorem}
Before turning to the proof, we will need the following lemma on Poisson operators, which we will use in the construction of the extension operator $E^\mu$.
\begin{figure}[t]
  \centering
  \begin{tikzpicture}
  \draw[->] (-0.5,0) -- (7,0) node[below] {$\vert \xi\vert$};
  \draw[->] (0,-0.5) -- (0,3.5) node[left] {$\tau$};

  \filldraw[draw=Sepia,fill=Sepia!20,thick] (0,0)--(6,0)--(5,1)--(3,2.2)--(0,3.1)--cycle;
  \filldraw[draw=Sepia,fill=Sepia!30,thick] (0,0)--(6,0)--(5,1)--(3,2.2)--(2,2.5)--(1,2.7)--(0,2.75)--cycle;
  \fill[Sepia!35] (2.5,0)--(6,0)--(5,1)--(3,2.2)--(2.5,2.35)--cycle;
  \draw[Sepia,thick] (0,0)--(6,0)--(5,1)--(3,2.2)--(2,2.5)--(1,2.7)--(0,2.75)--cycle;
  \draw[NavyBlue,very thick] (2.5,2.35)-- node[sloped,midway,above] {$i\eta_i$}(0,3.1);
  \draw[Sepia!80,dashed] (2.5,0) node[below,font=\small] {$i$} --(2.5,2.35);

    \fill[Sepia] (0,0) circle (1.5pt);
    \fill[Sepia] (6,0) circle (1.5pt);
    \fill[Sepia] (5,1) circle (1.5pt);
    \fill[Sepia] (3,2.2) circle (1.5pt);
    \fill[Sepia] (0,3.1) circle (1.5pt);
    \fill[Sepia] (2, 2.5) circle (1.5pt);
    \fill[Sepia] (1, 2.7) circle (1.5pt);
    \fill[Sepia] (0, 2.75) circle (1.5pt);

    \node[Sepia!80] at (2.5,1.2) {$\caln(\mu)$};
    \node[Sepia!80] at (4,.7) {$\caln(\mu^{(i)})$};
    \node at (4.2,2) {$\caln(\nu_i)$};
  \end{tikzpicture}
  \hfill
  \begin{tikzpicture}
  \draw[->] (-0.5,0) -- (7,0) node[below] {$\vert \xi\vert$};
  \draw[->] (0,-0.5) -- (0,3.5) node[left] {$\tau$};

  \filldraw[draw=Sepia,fill=Sepia!20,thick] (0,0)--(6,0)--(5,1)--(3,2.2)--(0,3.1)--cycle;

  \fill[Sepia!30] (5,1) -- (0, 2.5) -- (0,0) -- (5,0) -- cycle;
  \draw[Sepia, thick]
  (0,0)--(6,0)--(5,1)--(3,2.2) --(0,3.1) --cycle;
  \draw[Sepia!80, dashed] (5, 0)--(5,1) -- (0,2.5);

    \fill[Sepia] (0,0) circle (1.5pt);
    \fill[Sepia] (6,0) circle (1.5pt);
    \fill[Sepia] (5,1) circle (1.5pt);
    \node[right, font=\small] at (5,1) {$(r,s)$};
    \fill[Sepia] (3,2.2) circle (1.5pt);
    \fill[Sepia] (0,3.1) circle (1.5pt);
    \fill[Sepia] (0, 2.5) circle (1.5pt);

    \node[Sepia!80] at (2,.7) {$\caln(so_\infty+r\eta_i)$};
    \node[Sepia!80] at (4.2,2) {$\caln(\nu_i)$};
  \end{tikzpicture}
  \caption{The situation of Lemma \ref{lemma: poisson_extension_newton}.
  Left: $\caln(\nu_i)$ is constructed by discarding the part of $\caln(\mu)$ to the left of $i$ and filling it up with the current slope, resulting in $\caln(\nu_i)=\caln(\mu^{(i)}+i\eta_i)$.
  Right: The upper edge of the polygon $\caln(so_\infty+r\eta_i)$ is parallel to the leftmost nonvertical edge of $\caln(\nu_i)$. Since both pass through $(r,s)$, we have $\caln(so_\infty+r\eta_i)\subseteq \caln(\nu_i)$ by convexity.}
  \label{figure005}  
\end{figure}
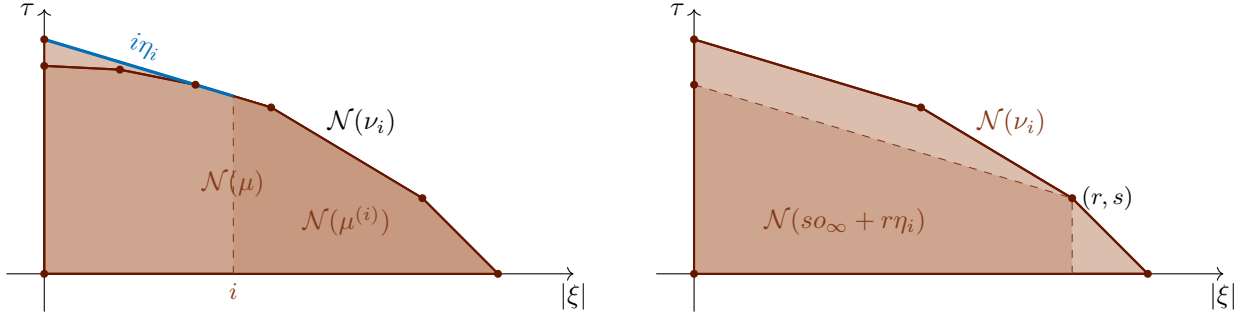
\begin{lemma}\label{lemma: poisson_extension_newton}
    Let $i\in\{0,\ldots,r_1-1\}$ and let $A_i:=\op[w_i]$, where $w_i$ is a weight function for $\eta_i$.
    Then $P_{ki}\in \call(T^{\mu,p}_{i\perp}(G^{n-1}),\overline H^{\mu,p}_\perp(G_+))$ for any $k>0$, where $P_{ki}\sigma(t,x',x_n):=(e^{-x_nkA_i}A_i^{-i}\sigma)(t,x')$.
\end{lemma}
\begin{proof}
    Let first $k=1$ and set $\nu_i:=\mu^{(i)}+i\eta_i$.
    Since $\mu^{(i+\delta)}=\mu^{(i)}-\delta\eta_i$ for $\delta\in[0,1]$, we have
    \begin{align*}
        T^{\mu,p}_{i\perp}(G^{n-1})=B^{\mu^{(i+\frac1p)},p}_{\eta_i,\perp}(G^{n-1})=B^{\nu_i-(i+\frac1p)\eta_i,p}_{\eta_i,\perp}(G^{n-1}).
    \end{align*}
    Corollary~\ref{cor: ellipticity_HB_BH} therefore shows that
    $A_i^{-i}\in\mathcal L_{\mathrm{iso}}(T^{\mu,p}_{i\perp}(G^{n-1}),B^{\nu_i-\frac1p\eta_i,p}_{\eta_i,\perp}(G^{n-1}))$.
    It is consequently enough to show that $S_i\in \mathcal L(B^{\nu_i-\frac1p\eta_i,p}_{\eta_i,\perp}(G^{n-1}),\overline H^{\nu_i,p}_\perp(G_+))$, where $S_i\sigma(t,x',x_n):=(e^{-x_nA_i}\sigma)(t,x')$.

    Let $(r,s)$ be a vertex of the Newton polygon associated with $\nu_i$.
    If $r>0$, the upper edge of $\mathcal N(so_\infty+r\eta_i)$ is parallel to the leftmost nonvertical edge of the upper boundary of $\mathcal N(\nu_i)$ and both polygons meet at $(r,s)$.
    Hence $\mathcal N(so_\infty+r\eta_i)\subseteq\mathcal N(\nu_i)$ by convexity, see Figure~\ref{figure005}.
    If $r=0$, then the same inclusion is immediate from the fact that $(0,s)$ is a vertex of $\mathcal N(\nu_i)$.
    Hence $H^{\nu_i,p}_\perp(G^{n-1})\hookrightarrow H^{so_\infty+r\eta_i,p}_\perp(G^{n-1})\hookrightarrow H^{(s,r),p}_\perp(G^{n-1})$.
    A common shift in the direction $\pm\eta_i$ and Lemma~\ref{lemma: geometric_embedding_B} therefore give
    \begin{align*}
        B^{\nu_i-\frac1p\eta_i,p}_{\eta_i,\perp}(G^{n-1})
        \hookrightarrow
        B^{so_\infty+(r-\frac1p)\eta_i,p}_{\eta_i,\perp}(G^{n-1})
        \hookrightarrow
        B^{so_\infty+ro_0-\frac1p\eta_i,p}_{\eta_i,\perp}(G^{n-1}).
    \end{align*}
    Consider first $A_i$ on $X:=H^{so_\infty+ro_0-\eta_i,p}_\perp(G^{n-1})$ with domain $D(A_i)=H^{so_\infty+ro_0,p}_\perp(G^{n-1})$, so that in particular $(X,D(A_i))_{1-\frac1p,p} = B^{so_\infty+ro_0-\frac1p\eta_i,p}_{\eta_i,\perp}(G^{n-1})$.
    By Lemma~\ref{lemma: bdd_H_infty}, $A_i$ admits bounded $H^\infty$-calculus on $X$ and hence maximal $L^p$-regularity, which yields
    \begin{align}\label{lemma: poisson_extension_newton_e1}
        \|S_i\sigma\|_{L^p(\R_+;H^{(s,r),p}_\perp(G^{n-1}))}\lesssim \|\sigma\|_{B^{si_\infty+ro_0-\frac1p\eta_i,p}_{\eta_i,\perp}} \lesssim \|\sigma\|_{B^{\nu_i-\frac1p\eta_i,p}_{\eta_i,\perp}}.
    \end{align}
    If $r\geq1$, we additionally consider $A_i$ on
    $X:=H^{so_\infty,p}_\perp(G^{n-1})$ with domain
    $D(A_i)=H^{so_\infty+\eta_i,p}_\perp(G^{n-1})$.
    Since $A_i$ admits a boundary $H^\infty$-calculus on $X$, the semigroup $S_i$ is analytic and hence
    $\partial_n^{r-1}S_i=(-A_i)^{r-1}S_i=S_i(-A_i)^{r-1}$.
    By Corollary~\ref{cor: ellipticity_HB_BH},
    $A_i^{r-1}\in \call(B^{so_\infty+(r-\frac1p)\eta_i,p}_{\eta_i,\perp}(G^{n-1}),B^{so_\infty+(1-\frac1p)\eta_i,p}_{\eta_i,\perp}(G^{n-1}))$.
    Since the target space of $A_i^{r-1}$ is precisely $(X,D(A))_{1-\frac1p,p}$, maximal regularity of $A_i$ on $X$ gives
    \begin{align*}
        \|\partial_n^rS_i\sigma\|_{L^p(\R_+;H^{(s,0),p}_\perp(G^{n-1}))}\lesssim \|(-A_i)^{r-1}\sigma\|_{B^{so_\infty+(1-\frac1p)\eta_i,p}_{\eta_i,\perp}(G^{n-1})}\lesssim \|\sigma\|_{B^{so_\infty+(r-\frac1p)\eta_i,p}_{\eta_i,\perp}(G^{n-1})}\lesssim \|\sigma\|_{B^{\nu_i-\frac1p\eta_i,p}_{\eta_i,\perp}}.
    \end{align*}
    Together with \eqref{lemma: poisson_extension_newton_e1} and $\overline H^{(s,r),p}_\perp(G_+)=H^{r,p}(\R_+;H^{(s,r),p}_\perp(G^{n-1})$, this yields
    $S_i\in \call(B^{\nu_i-\frac1p\eta_i,p}_{\eta_i,\perp},\overline H^{(s,r),p}_\perp(G_+))$.
    Since $(r,s)$ was arbitrary, Corollary~\ref{cor: equiv_norm_half_space} shows $S_i\in \call(B^{\nu_i-\frac1p\eta_i,p}_{\eta_i,\perp},\overline H^{\nu_i,p}_\perp(G_+))$
    Finally, the Newton polygon associated with $\nu_i$ contains the Newton polygon associated with $\mu$, and hence
    $\overline H^{\nu_i,p}_\perp(G_+)\hookrightarrow\overline H^{\mu,p}_\perp(G_+)$, which gives the assertion for $k=1$.

    Replacing $A_i$ by $kA_i$ does not change its domain or any of the interpolation spaces above, and $kA_i$ has the same maximal regularity properties.
    This proves the last assertion.
\end{proof}
\begin{proof}[Proof of Theorem \ref{thm: trace_space}]
        We first show that $\Tr^{(r_1)}$ extends to a well-defined and bounded operator $\overline{H}(\mu)\to T(\mu)$.
        Obviously, it is sufficient to prove that $\Tr_i$ extends to a bounded operator $\overline{H}(\mu)\to T_i(\mu)$ for all $i\in\{0,..., r_1-1\}$.
        
        \medskip
        
        For $i\in\{0,\ldots,r_1-1\}$, set
        \begin{align*}
            V_i:=\{v\in L^p(\R_+;H^{\mu^{(i)},p}_\perp(G^{n-1})):\partial_nv\in L^p(\R_+;H^{\mu^{(i+1)},p}_\perp(G^{n-1}))\}.
        \end{align*}
        The trace method (see, e.g., Section 1.8 of \cite{triebel1995interpolation}) yields $\Tr_0\in\mathcal L(V_i,T^{\mu,p}_{i\perp}(G^{n-1}))$ because
        \begin{align*}
            T^{\mu,p}_{i\perp}(G^{n-1}) = (H^{\mu^{(i+1)},p}_\perp(G^{n-1}),H^{\mu^{(i)},p}_\perp(G^{n-1}))_{1-\frac1p,p}.
        \end{align*}
        By $\partial_n^i\in\mathcal L(\overline H^{\mu,p}_\perp(G_+),\overline H^{\mu^{(i)},p}_\perp(G_+))$ and $\overline H^{\mu^{(i)},p}_\perp(G_+)\hookrightarrow V_i$, we get $\Tr_i=\Tr_0\partial_n^i\in\mathcal L(\overline H^{\mu,p}_\perp(G_+),T^{\mu,p}_{i\perp}(G^{n-1}))$.
        Hence $\Tr^{(r_1)}\in \mathcal L(\overline H^{\mu,p}_\perp(G_+),T^{\mu,p}_{\perp}(G^{n-1}))$.

          \medskip
          
          We now construct a bounded right-inverse. This part of the proof is adapted from \cite{denk2008inhomogeneous}. Pick $\sigma=(\sigma_0, ..., \sigma_{r_1-1}) \in T(\mu)$.
          For $i\in\{0,..., r_1-1\}$, let $A_i:=\op[w_i]$, where $w_i$ is a weight function for $\eta_i$.
          Use Lemma \ref{lemma: poisson_extension_newton} to define $E_i\in \call(T_i(\mu),\overline{H}(\mu))$ via
        \begin{equation*}
            E_i\sigma_i:=\sum_{k=1}^{r_1} c_{ki}P_{ki}\sigma_i,\qquad i\in\{0,..., r_1-1\},
        \end{equation*}
        where we can choose the coefficients $c_{ki}$ in such a way that for $i,m\in\{0,...,r_1-1\}$ it holds
        \begin{equation}\label{eqn: choice_coeff}
            \partial_{n}^{m}E_i\sigma|_{x_n=0}=\delta_{m, i}\sigma_i.
        \end{equation}
        Indeed, by the explicit definition of $P_{ki}$ in Lemma \ref{lemma: poisson_extension_newton}, we have $\partial_{n}^{m}E_i\sigma|_{x_n=0}=\sum_{k=1}^{r_1}(-k)^{m}c_{ki}A^{m-i}_{i}\sigma_i$.
        Therefore, (\ref{eqn: choice_coeff}) is verified if $VC_i=e_{i+1}$ for $i\in\{0,...,r_1-1\}$, where $C_i=(c_{1i},...,c_{r_1i})$ and $V$ is the Vandermonde matrix. Since $V$ is invertible, we can set $C_i=V^{-1}e_{i+1}$ to achieve this.
        We now set  $E^\mu\sigma:=\sum_{i=0}^{r_1-1}E_i\sigma_i\in \call(T(\mu),\overline{H}(\mu))$.
        By (\ref{eqn: choice_coeff}), we have $\partial_{n}^{m}E^\mu\sigma=\sigma_m$ for $m\in\{0,...,r_1-1\}$ and thus $\Tr^{(r_1)}E^\mu=\id$.

        \medskip
        
        Finally, we establish $\ker\Tr^{(r_1)}=r^+\dot{H}^{\mu,p}_\perp(G_+)$.
        Consider first $u\in \ker\Tr^{(r_1)}$, i.e., $u\in  \overline{H}^{\mu,p}_\perp(G_+)$ with $\Tr^{(r_1)} u=0$.
        Then $u\in \overline{H}^{r_j,p}(\mathbb{R}^n_+; H^{s_j,p}_\perp(\mathbb{T}))$ with $\Tr^{(r_j)}u=0$ for all $j\in\set{0,\ldots,J}$.
        By \cite[Theorem VIII.1.6.8.]{amann2019linear}, for every $j$ there is $\seqK{u_{j,k}}\subseteq\CRci(\R^n_+;H^{s_j,p}_\perp(\mathbb{T}))$ such that $u_{j,k}\to u$ in $\overline{H}^{r_j,p}(\mathbb{R}^n_+; H^{s_j,p}_\perp(\mathbb{T}))$.
        In particular, $e_0^+u_{j,k}$ converges in $H^{r_j,p}(\mathbb{R}^n; H^{s_j,p}_\perp(\mathbb{T}))=H^{(s_j,r_j),p}_\perp(G)$, and the limit is $U:=e_0^+u$.
        Therefore, $U\in \bigcap_{j=0}^{J+1} H^{(s_j,r_j),p}_\perp(G)=H^{\mu,p}_\perp(G)$ and $\supp U\subseteq \overline{G_+}$, that is $U\in \dot H^{\mu,p}_\perp(\overline{G_+})$.
        Thus $u=r^+U\in r^+\dot{H}^{\mu,p}_\perp(\overline{G_+})$, and consequently  $\ker\Tr^{(r_1)}\subseteq r^+\dot{H}^{\mu,p}_\perp(G_+)$.

        \medskip

        Conversely, let $U\in \dot{H}^{\mu,p}_\perp(G_+)$ and $u:=r^+U$.
        Then $u\in \overline{H}^{\mu,p}_\perp(G_+)$ by definition.
        Moreover $\Tr_j u=(-1)^j\Tr_j v$ for $r\in\set{0,\ldots,r_1-1}$, where $v:=r^+ V=0$ and $V(t,x',x_n):=U(t,x',-x_n)$.
        Therefore $\Tr^{(r_1)}u=0$, and consequently $r^+\dot{H}^{\mu,p}_\perp(G_+)\subseteq \ker\Tr^{(r_1)}$.
        This completes the proof.
\end{proof}
 The case $p=2$ admits a particularly transparent formulation since the scales $\mathcal{H}_\perp$, $\mathcal{HB}$, and $\mathcal{BH}$ coincide.
 We therefore obtain the following corollary.
 See also \cite{NeS25} for a proof in a simpler setting. 
\begin{corollary}\label{cor: trace_L2}
    Let $\mu$ be a strictly positive order function, $\mathcal{N}=\mathcal{N}(\mu)$ its Newton polygon and $\mathcal{N}_V=\{(r_0,s_0),..., (r_{J+1}, s_{J+1})\}$ the set of its vertices.
    Assume that $r_j\in \N_0$ for all $j\in\set{0,\ldots,J+1}$.
    Then
    \begin{align*}
     \Tr^{(r_1)}\in \mathcal{L}(\overline{H}^{\mu, 2}_\perp(G_+), T^{\mu, 2}_\perp(G^{n-1}))
     \qquad \text{for} \qquad
     T^{\mu, 2}_\perp(G^{n-1}):=\prod_{i=0}^{r_1-1} H^{\mu^{(i+\frac{1}{2})}, 2}_\perp(G^{n-1}).
    \end{align*}
    Furthermore, there exists $E^\mu\in \mathcal{L}(T^{\mu, 2}_\perp(G^{n-1}),\overline{H}^{\mu, 2}_\perp(G_+)) $ such that $\Tr^{(r_1)} \circ E^\mu= \id_{T^{\mu, 2}_\perp(G^{n-1})}$. 
\end{corollary}

\begin{proposition}\label{prop: embeddings_trace_space}
    Let $\mu$ be a strictly positive order function, $\mathcal{N}=\mathcal{N}(\mu)$ its Newton polygon and $\mathcal{N}_V=\{(r_0,s_0),..., (r_{J+1}, s_{J+1})\}$ the set of its vertices.
    Assume that $r_j\in \N_0$ for all $j\in\set{0,\ldots,J+1}$.
    \begin{enumerate}[leftmargin=*]
        \item If $\caln$ is not regular in space, then 
        \begin{enumerate}
            \item\label{prop: embeddings_trace_spacei} For all $(r, s)\in \mathcal{N}^{(\frac{1}{p})}$, it holds
            \begin{equation*}
                T_{0\perp}^{\mu, p}(G^{n-1}) \hookrightarrow HB^{(s, r), p}_\perp(G^{n-1})
            \end{equation*}
            \item\label{prop: embeddings_trace_spaceii} For all $(r, s) \in \mathcal{N}^{(\frac{1}{p})}$ such that $(r,s)$ is neither a vertex of $\mathcal{N}^{(\frac{1}{p})}$, nor belongs to the horizontal edge $[(0, s_{J+1}), (r_J-\frac1p, s_J)] $, it holds 
            \begin{equation*}
                T_{0\perp}^{\mu, p}(G^{n-1})\hookrightarrow BH^{(s, r), p}_\perp(G^{n-1})
            \end{equation*}
        \end{enumerate}
        \item If $\caln$ is regular in space, then 
        \begin{enumerate}
            \item\label{prop: embeddings_trace_spaceiii} For all $(r, s)\in \mathcal{N}^{(\frac{1}{p})}$ such that $(r,s)\neq (0, s_{J+1}-y_J/p)$, it holds 
            \begin{equation*}
                 T_{0\perp}^{\mu, p}(G^{n-1}) \hookrightarrow HB^{(s, r), p}_\perp(G^{n-1}).
            \end{equation*}
            \item\label{prop: embeddings_trace_spaceiv} For all $(r, s)\in \mathcal{N}^{(\frac{1}{p})}$ such that $(r, s) \notin \{(r_1-\frac1p, s_1),..., (r_J-\frac1p, s_J)\}$, it holds 
             \begin{equation*}
                 T_{0\perp}^{\mu, p}(G^{n-1}) \hookrightarrow BH^{(s, r), p}_\perp(G^{n-1}).
            \end{equation*}
        \end{enumerate}
    \end{enumerate}
\end{proposition}
\begin{proof}
    Let $(r,s)\in \caln^{(\frac1p)}$ and set $\nu:=so_\infty+ro_0$.
    We repeatedly use Lemma~\ref{lemma: geometric_embedding_B}: if $\zeta$ is an order function such that
    \begin{align}\label{prop: embeddings_trace_space_e1}
        H^{\mu^{(\frac1p)}\pm\varepsilon\eta_0,p}_\perp(G^{n-1})
        \hookrightarrow H^{\nu\pm\varepsilon\zeta,p}_\perp(G^{n-1})
    \end{align}
    for some $\varepsilon>0$, then
    $T_{0\perp}^{\mu,p}(G^{n-1})=B^{\mu^{(\frac1p)},p}_{\eta_0,\perp}(G^{n-1})\hookrightarrow B^{\nu,p}_{\zeta,\perp}(G^{n-1})=B^{(s,r),p}_{\zeta,\perp}(G^{n-1})$.

    For \ref{prop: embeddings_trace_spacei} and \ref{prop: embeddings_trace_spaceiii}, we choose $\zeta:=o_0$.
    Since $\mu^{(\delta)}=\mu-\delta\eta_0$ for $\delta\in[0,1]$, for every sufficiently small $\varepsilon>0$ we have $\mu^{(\frac1p)}\pm\varepsilon\eta_0=\mu^{(\frac1p\pm\varepsilon)}$.
    Consequently the polygons corresponding to $\mu^{(\frac1p)}\pm\varepsilon\eta_0$ are precisely $\mathcal N^{(\frac1p\mp\varepsilon)}$.
    On the other hand, observe that the restriction on $(r,s)$ in \ref{prop: embeddings_trace_spacei} and \ref{prop: embeddings_trace_spaceiii} implies that there is $r'>0$ with $r'\ge r$ and $(r',s)\in \caln^{(\frac1p)}$. Since $HB^{(s,r')}_\perp(G^{n-1})\hookrightarrow HB^{(s,r)}_\perp(G^{n-1})$, we may assume that $r>0$.
    Thus, replacing $\nu$ by $\nu\pm\varepsilon o_0$ replaces the point $(r,s)$ by $(r\pm\varepsilon,s)$, and consequently $\caln(\nu\pm\varepsilon o_0)\subseteq \caln^{\frac1p\mp\varepsilon}$.
    Therefore \eqref{prop: embeddings_trace_space_e1} is satisfied for $\zeta:=o_0$ in the cases \ref{prop: embeddings_trace_spacei} and \ref{prop: embeddings_trace_spaceiii}.

    For \ref{prop: embeddings_trace_spaceii} and \ref{prop: embeddings_trace_spaceiv}, we first treat the special case that $\mathcal{N}^{(\frac1p)}$ is not regular in time and $(r,s)$ is in the relative interior of the right horizontal edge of $\caln^{(\frac1p)}$.
    In this case, for sufficiently small $\varepsilon>0$, we still have $(r,s\pm\eps)\in \caln^{(\frac1p)}$, and thus Proposition \ref{prop: properties_Newton_scales}\ref{prop: properties_Newton_scalesii} gives
    \begin{align*}
        T_{0\perp}^{\mu,p}(G^{n-1})=B^{\mu^{(\frac1p)},p}_{\eta_0,\perp}(G^{n-1})
        \hookrightarrow
        B^{(s-\varepsilon,r),p}_{\eta_0,\perp}(G^{n-1})
        \cap
        B^{(s+\varepsilon,r),p}_{\eta_0,\perp}(G^{n-1}).
    \end{align*}
    By reiteration, it holds
    \begin{align*}
        \big(
        B^{(s-\varepsilon,r),p}_{\eta_0,\perp}(G^{n-1}),
        B^{(s+\varepsilon,r),p}_{\eta_0,\perp}(G^{n-1})
        \big)_{\frac12,p} = BH^{(s,r),p}_\perp(G^{n-1}).
    \end{align*}
    Since the intersection of the endpoint spaces embeds into their real interpolation space, the desired embedding follows.
    For the remaining cases, we choose $\zeta:=yo_\infty$ with a suitable $y>0$.
    If $(r,s)$ lies in the relative interior of a non-horizontal edge of $\mathcal N^{(\frac1p)}$ of strictly negative slope $-y$, then shifting the polygon by $\mp\varepsilon\eta_0$ changes the height of this edge at the fixed spatial coordinate $r$ by $\pm\varepsilon y$.
    Consequently, \eqref{prop: embeddings_trace_space_e1} holds for all sufficiently small $\varepsilon>0$.
    Lemma~\ref{lemma: geometric_embedding_B} and the invariance under positive rescaling of $\zeta$ therefore give
    $T_{0\perp}^{\mu,p}(G^{n-1})\hookrightarrow B^{\nu,p}_{yo_\infty,\perp}(G^{n-1})=BH^{(s,r),p}_\perp(G^{n-1})$.
    The same conclusion holds for points lying strictly below the upper boundary by choosing $y>0$ and $\varepsilon>0$ sufficiently small.
    Thus, in the non-regular case the only obstructions are the vertices and the horizontal upper edge, which proves \ref{prop: embeddings_trace_spaceii}, while in the regular case only the vertices have to be excluded, which proves \ref{prop: embeddings_trace_spaceiv}.
\end{proof}
    We note that one can obtain a similar result for the spaces $T_i(\mu)$ with $i\in\set{0,..., r_1-1}$ from the preceding proposition by recalling that $T_i(\mu)= T_0(\mu^{(i)})$. 
\subsection{$L^p$-Theory for Mixed-Order Equations under General Boundary Conditions}
We are now ready to apply our abstract Theorem \ref{thm: abstract_complementing_condition} to the setting of purely oscillatory Newton polygon spaces.
The resulting Theorem \ref{thm: main_result} below characterizes well-posedness in terms of the complemented boundary symbol and identifies the corresponding compatibility space.
We start by fixing the notations and definitions in that context. 
\subsubsection*{Admissible symbols}
We consider a symbol $P\in \mathcal{O}^\perp(\R\times\R^n)$ of the form 
\begin{equation}\label{eqn: notations_P}
   P(\tau, \xi', \xi_n)= \sum_{j=0}^N c_j(\tau, \xi') (\ic\xi_n)^j,
\end{equation}
where $N\in\N_0$, $c_j\in \mathcal{O}^\perp(\R\times \R^{n-1})$ for $j\in\{0,\ldots, N\}$, and where we recall the notations of Section \ref{sec:time_per} for multipliers in purely oscillatory spaces.
We want to consider $\op[P]$ as an admissible operator in the sense of Definition \ref{def: admissible}.
With Proposition \ref{prop_support_preserving} in mind, this will be done by studying the roots of the complex polynomial $z\mapsto P(\tau, \xi', z)$. 
\begin{definition}\label{def:admissible_symb}
   Let $P$ be as in (\ref{eqn: notations_P}) and assume that $P$ is strongly $\mathcal{N}$-elliptic with a strictly positive order function $\mu$. We say that the symbol $P$ is \emph{admissible} if there exist strictly positive order functions $\mu_+, \mu_-$ and symbols $P^+, P^-\in \mathcal{O}^\perp(\R\times\R^n)$ with $P^\pm(\tau, \xi', \xi_n):= \sum_{j=0}^{N_\pm} c_j^\pm(\tau, \xi') (\ic\xi_n)^j$ where $c_j^\pm\in \mathcal{O}^\perp(\R\times\R^{n-1})$ such that: 
   \begin{itemize}
       \item $\mu= \mu_+ +\mu_-$, $P= P^+P^-$
       \item $P_\pm$ is strongly $\mathcal{N}$-elliptic with order function $\mu_\pm$
   
   \item For all fixed $(\tau, \xi')\in \mathbb{R}\times \mathbb{R}^{n-1}$ with $\tau\neq 0$, the roots $\rho_1^\pm(\tau, \xi'), ..., \rho_{N_\pm}^\pm(\tau, \xi')$ of the complex polynomial $P^\pm(\tau, \xi', \cdot)$ lie in $\mathbb{C}_\pm$, that is, $\Im (\rho_j^\pm(\tau, \xi'))\gtrless 0$ for $j\in\set{1,..., N_\pm}$. 
   \end{itemize} 
\end{definition}
Observe that in that case we have $N=N_+ + N_-$, $N=\ord \mu$, and $N_\pm= \ord \mu_\pm$.
If $P$ is an admissible symbol, then one can consider the operator $\opA:= \op[P]$ as an admissible operator in the sense of Definition \ref{def: admissible} with $\opA^\pm = \op[P^\pm]$:
Let $\nu$ be a strictly positive order function whose Newton polygon has integer vertices and set
\begin{align*}
 F:= H^{\nu, p}_\perp(G),
 \qquad
 H:= H^{\nu+\mu_-, p}_\perp(G),
 \qquad
 E:=H^{\mu+\nu, p}_\perp(G).   
\end{align*}
Then Theorem \ref{thm: whole space} and the $\mathcal{N}$-ellipticity conditions yield $\op[P] \in \mathcal{L}_{\mathrm{iso}}(E,F)$, as well as $\op[P^+] \in \mathcal{L}_{\mathrm{iso}}(E,H)$ and $\op[P^-] \in \mathcal{L}_{\mathrm{iso}}(H,F)$.
We will write $M:=\ord(\nu)\in\N$ throughout.

The required support-preserving properties follow from Proposition \ref{prop_support_preserving}.
On the one hand, both $\op[P^+]$ and $\op[P^-]$ preserve both upper and lower support because they are differential operators in the variable $x_n$.
On the other hand, for any $(k, \xi')\in \mathbb{Z}\times \mathbb{R}^{n-1}$ fixed with $k\neq 0$, the condition on the roots implies that $\frac{1}{P^\pm(k, \xi', \xi_n)}$ admits a well-defined extension $z\mapsto \frac{1}{P^\pm(k, \xi', z)}$ to $\overline{\mathbb{C}_\pm}$ verifying the hypothesis of Proposition \ref{prop_support_preserving}.

\subsubsection*{Trace spaces} 
We set the trace spaces $\mathcal{E}:=T^{\mu+\nu,p}_\perp(G^{n-1})$ and $\mathcal{H}:= T^{\mu_-+\nu, p}_\perp(G^{n-1})$.
Then Theorem \ref{thm: trace_space} yields operators $\Tr_E:= \Tr^{(N+M)}$ and $\Tr_H:=\Tr^{(N_-+M)}$ together with extension operators $\Ext_E:=E^{\mu+\nu}$ and $\Ext_H:=E^{\mu_-+\nu}$ such that $(\overline{E}_+,r^+\dot{E}_+,\mathcal{E},\Tr_E,\Ext_E)$ and $(\overline{H}_+,r^+\dot{H}_+,\mathcal{H},\Tr_H,\Ext_H)$ are abstract trace tuples in the sense of Definition \ref{def: abstract_trace}. 

\subsubsection*{Complementing boundary matrix}
In order to define our boundary operator $\opB$, we consider $N_+$ operators $B_1,..., B_{N_+}$ of the form
\begin{equation*}
    B_i= \sum_{j=0}^{N-1} \op[b_{ij}]\Tr_j
\end{equation*}
with $b_{ij}\in \mathcal{O}^\perp(\R\times\R^{n-1})$. 
Setting $b_{ij}:=0$ for $j\in \set{N,\ldots,N+M-1}$, we define the $N_+\times (N+M)$ symbol matrix $\mathfrak{B}(\tau, \xi')$ by 
\begin{equation*}
    \mathfrak{B}_{ij}(\tau, \xi'):= b_{ij-1}(\tau, \xi')
\end{equation*}
and set $\opB:= \op[\mathfrak{B}]$.
Then for any $u\in \overline{E}_+$ and $g=(g_1,..., g_{N_+})\in \mathcal{G}$, it holds $B_i u=g_i$ for all $i\in\set{1,..., N_+}$ if and only if  $\opB \Tr_Eu= g$.

\medskip 

In a similar fashion, the complementary boundary operator $\opC= \Tr_H \op[P^+]_+ \Ext_E$ can be represented by the $(N_-+M)\times (N+M)$ symbol matrix $\mathfrak{C}$ with
\begin{equation*}
    \mathfrak{C}_{ij}=c^+_{j-i},
\end{equation*}
where one understands $c_k^+=0$ if $k<0$ or $k>N_+$.
Indeed, by $\Tr_j (\ic\xi_n)^k=\Tr_{j+k}$ we have
\begin{equation*}
    \opC= \Tr_H \op[P^+]_+ \Ext_E= \op[\mathfrak{C}]\Tr_E\Ext_E=\op[\mathfrak{C}]. 
\end{equation*}
Since we are interested in injectivity and invertibility properties of the complemented operator $(\opB, \opC)$, one finally defines the $(N+M)\times(N+M)$ square matrix of symbols $\mathfrak{BC}$ where 
\begin{equation*}
    \mathfrak{BC}_{ij}= \left\{\begin{array}{ll}
    \mathfrak{B}_{ij}     &\text{if } i\in\set{1,..., N_+},  \\
     \mathfrak{C}_{i-N_+,j}    & \text{if } i\in\set{N_++1,..., N+M}. 
    \end{array}\right.
\end{equation*}
The matrix $\mathfrak{BC}$ will be called \emph{complemented boundary matrix}, and we have $(\opB, \opC)= \op[\mathfrak{BC}]$. 
\subsubsection*{Mixed-order system setting in the trace spaces}
To apply our Theorem \ref{thm: abstract_complementing_condition}, we must study the invertibility of complemented boundary operator $(\opB, \opC)=\op[\mathfrak{BC}]: \mathcal{E}\to \mathcal{G}\times \mathcal{H}$.
To do so, we wish to make use of the notion of mixed-order system as defined in Definition \ref{def: mixed_order_systems}. However, a complication arises from the fact that the trace spaces $\mathcal{E}, \mathcal{H}$ are not $H^{\mu,p}_\perp$-spaces, but rather interpolation of two of such spaces, as we stated in Theorem \ref{thm: trace_space}.
Thus, we will need two ``sets'' of order functions for their description, one for each argument of the real interpolation. 
For the order functions associated with the trace space $\mathcal{E}$ i.e. the columns of the mixed-order systems, one has
\begin{equation*}
    \mathcal{E}=T^{\mu+\nu, p}_\perp(G^{n-1})=  \prod_{j=0}^{N+M-1} (H^{t_{j,1}, p}_\perp (G^{n-1}), H^{t_{j, 0}, p}_\perp(G^{n-1}))_{1-\frac{1}{p}, p},
\end{equation*}
where we set $t_{j+1,\mathfrak{k}}:= (\mu+\nu)^{(j+\mathfrak{k})}$ for $j\in\set{0,..., N+M-1}$, $\mathfrak{k}\in\set{0,1}$ in line with Theorem \ref{thm: trace_space}. 
\medskip
In order to retain sufficient flexibility, we will assume that the target space  $\mathcal{G}$ of $\opB$ can be written as an interpolation space of any two Newton polygon spaces, i.e. we set
\begin{equation*}
     \mathcal{G}:=\prod_{i=1}^{N_+} ( H^{s_{i,1},p}_\perp(G^{n-1}), H^{s_{i,0}, p}_\perp(G^{n-1}))_{1-\frac{1}{p}, p}
\end{equation*}
for some order functions $s_{i, \mathfrak{k}}$ $i\in\set{1,\ldots, N_+}$, $\mathfrak{k}\in\set{0,1}$. Those order functions will correspond to the rows $1,..., N_+$ of the mixed-order systems. 

\medskip

Finally, the order functions associated with the target space $\mathcal{H}$ of $\opC$, i.e. to the rows $N_++1,..., N+M$
we have
\begin{equation*}
    \mathcal{H}=T^{\mu_-+\nu, p}_\perp(G^{n-1})=:\prod_{i=0}^{N_-+M-1} (H^{s_{N_++1+i,1}, p}_\perp (G^{n-1}), H^{s_{N_++1+i,0}, p}_\perp(G^{n-1}))_{1-\frac{1}{p}, p}
\end{equation*}
by setting $s_{N_++1+i,\mathfrak{k}}:= (\mu_-+\nu)^{(i+\mathfrak{k})}$ in line with Theorem \ref{thm: trace_space} for $i\in\set{0,\ldots, N_-+M-1}$, $\mathfrak{k}\in\set{0,1}$.
We emphasize that the order functions $t_{1, \mathfrak{k}}, ..., t_{N+M, \mathfrak{k}}$ and $s_{N_++1, \mathfrak{k}}, ..., s_{N+M, \mathfrak{k}}$ are entirely determined by the admissible structure of the operator $\op[P]$, while we maintain flexibility for the choice of $s_{1, \mathfrak{k}}, ..., s_{N_+, \mathfrak{k}}$.
This allows us to choose the regularity of the boundary data $g\in \mathcal{G}$ and will prove to be essential for certain applications, see Section \ref{sec:app}.

\subsubsection*{Compatibility space for the data}
Assume that $\det(\mathfrak{BC}(\tau, \xi'))\neq 0$ for all $(\tau, \xi')\in \mathbb{R} \times \mathbb{R}^{n-1}$ with $\tau\neq 0$, that is, $(\opB,\opC)= \op[\mathfrak{BC}]$ is injective. 
Then the set $\Comp$ defined in Theorem \ref{thm: abstract_complementing_condition} can be written as
\begin{align*}
   \Comp&= \{ (f, g)\in \overline{F}_+ \times \mathcal{G} \mid (g, \Tr_H(\opA^-)_+^{-1}f)\in \ran(\opB, \opC)\} \\
   &=\{(f, g)\in \overline{H}^{\nu, p}_\perp(G_+) \times \mathcal{G}\mid (g, \Tr^{(N_-+M)}(\op[P^-]^{-1}_+f) \in \ran \op[\mathfrak{BC}]\}\\
   &= \{(f, g)\in \overline{H}^{\nu, p}_\perp(G_+) \times \mathcal{G}\mid b(f,g)\in T^{\mu+\nu,p}_\perp(G^{n-1}) \}
\end{align*}
with $b(f,g):=\op[\mathfrak{BC}^{-1}](g,\Tr^{(N_-+M)}(\op[P^-]_+^{-1} f))$. This will be the set of \emph{compatible} right-hand sides $(f, g)$ associated with $(P, B_1,..., B_{N_+}, \nu)$.
We endow it with the norm
\begin{align*}
    \|f,g\|_{\Comp}:=\|f\|_{\overline{H}^{\nu,p}_\perp(G_+)} + \|b(f,g)\|_{T^{\mu+\nu,p}_\perp(G^{n-1})}.
\end{align*}
We are now ready to formulate our main result. 
\begin{theorem}\label{thm: main_result}
    With the notations settled above, assume that for $\mathfrak{k}\in\{0,1\}$ and $i,j \in\{ 1,..., N+M\}$, the order function $t_{j, \mathfrak{k}}-s_{i, \mathfrak{k}}$ is a strong upper order function for the symbol $\mathfrak{BC}_{ij}$.
    Suppose that $\det(\mathfrak{BC}(\tau, \xi'))\neq 0$ for all $(\tau, \xi')\in \mathbb{R} \times \mathbb{R}^{n-1}$ with $\tau\neq 0$.
   
    \begin{enumerate}
        \item\label{thm: main_resulti} For $S:=(\op[P]_+,\opB \Tr^{(N+M)})$ we have $S\in \call_{\mathrm{iso}}(\overline{H}^{\mu+\nu,p}_\perp(G_+),\Comp)$.
        In particular,
        for all $(f, g) \in \Comp$, the problem 
        \begin{align}\label{thm: main_result_e1}
            \left\{ \begin{array}{rll}
            \op[P]_+u     &=&f  \\
             B_i u&=&g_i,    \qquad  i\in\{1,..., N_+\} 
            \end{array}\right.
        \end{align}
        admits a unique solution $u\in \overline{H}^{\mu+\nu,p}_\perp(G_+)$, and it holds
        \begin{equation*}
            \|u\|_{\mu+\nu, p, +} \lesssim \|f\|_{\nu, p, +} + \|b(f,g)\|_{T^{\mu+\nu,p}_\perp(G^{n-1})}
        \end{equation*}
        with $b(f,g):=(\opB,\opC)^{-1}(g,\Tr^{(N_-+M)}(\op[P^-]_+^{-1} f))$.
        \item\label{thm: main_resultii} If furthermore $ \sum_{i=1}^{N+M} (t_{i,0}-s_{i,0})=  \sum_{i=1}^{N+M} (t_{i,1}-s_{i,1})=:\delta$ and $\det\mathfrak{BC}$ is strongly $\mathcal{N}$-elliptic with order function $\delta$, then $\Comp=  \overline{H}^{\nu, p}_\perp(G_+) \times \mathcal{G}$ topologically.
        In particular, for all $(f, g) \in \overline{H}^{\nu, p}_\perp(G_+) \times \mathcal{G}$, the problem \eqref{thm: main_result_e1} admits a unique solution $u\in \overline{H}^{\mu+\nu,p}_\perp(G_+)$, and it holds
        \begin{equation*}
            \|u\|_{\mu+\nu, p, +} \lesssim \|f\|_{\nu, p, +} + \|g\|_{\mathcal{G}}. 
        \end{equation*} 
    \end{enumerate}
\end{theorem}
\begin{proof}
\begin{enumerate}
    \item We want to apply Theorem \ref{thm: abstract_complementing_condition} with $\opA=\opA^-\opA^+:= \op[P]$ and $E$, $F$, $H$, $\mathcal{E}$, $\mathcal{H}$, $\mathcal{G}$ defined as above and where  $(\opB, \opC)=\op[\mathfrak{BC}]$.
    Since for fixed $\mathfrak{k}\in\{0,1\}$, the order function $t_{j,\mathfrak{k}}-s_{i,\mathfrak{k}}$ is a strong upper order function for the component $(\mathfrak{BC})_{ij}$, Proposition \ref{prop_boundedness} yields
    \begin{equation*}
        \op[\mathfrak{BC}] \in \mathcal{L}\big(\prod_{j=1}^{N+M} H^{t_{j, \mathfrak{k}}, p}_\perp(G^{n-1}) , \prod_{i=1}^{N+M} H^{s_{i, \mathfrak{k}}, p}_\perp(G^{n-1})\big). 
    \end{equation*}
    Since we have by definition 
    \begin{equation*}
        \mathcal{E}=\big(\prod_{j=1}^{N+M} H^{t_{j, 1}, p}_\perp(G^{n-1}), \prod_{j=1}^{N+M} H^{t_{j, 0}, p}_\perp(G^{n-1})\big)_{1-\frac{1}{p}, p}
    \end{equation*}
     and 
     \begin{equation*}
         \mathcal{G} \times \mathcal{H}=\big(\prod_{i=1}^{N+M} H^{s_{i, 1}, p}_\perp(G^{n-1}), \prod_{i=1}^{N+M} H^{s_{i,0}, p}_\perp(G^{n-1})\big)_{1-\frac{1}{p}, p},
     \end{equation*}
     interpolation gives
   \begin{equation*}
       \op[\mathfrak{BC}] \in \mathcal{L}(\mathcal{E}, \mathcal{G} \times \mathcal{H}),
   \end{equation*}
   which in particular implies $\opB=\op[\mathfrak{B}]\in \mathcal{L}(\mathcal{E}, \mathcal{G})$. 
   As we mentioned above, if $\det \mathfrak{BC}\neq 0$, the operator  $(\opB, \opC)=\op[\mathfrak{BC}]$ is injective and hence satisfies the abstract complementing boundary condition in Definition \ref{def: abstract_complementing_condition_upper} for $\op[P^+]_+$.
   Thus, the first point follows immediately by Theorem \ref{thm: abstract_complementing_condition}.
   \item Since $\mathfrak{BC}$ is a mixed-order system with the order functions $(t_{1, \mathfrak{k}},..., t_{N+M, \mathfrak{k}}), (s_{1,\mathfrak{k}},..., s_{N+M,\mathfrak{k}})$ for both $\mathfrak{k}=0$ and $\mathfrak{k}=1$, we obtain
   \begin{equation*}
        \op[\mathfrak{BC}] \in \mathcal{L}_{\mathrm{iso}}(\prod_{j=1}^{N+M} H^{t_{j,\mathfrak{k}}, p}_\perp(G^{n-1}) , \prod_{i=1}^{N+M} H^{s_{i,\mathfrak{k}}, p}_\perp(G^{n-1}))
    \end{equation*}
    for $\mathfrak{k}\in\{0,1\}$ by Theorem \ref{thm: whole_space_systems}, and the same interpolation argument as above gives 
    \begin{equation*}
       \op[\mathfrak{BC}] \in \mathcal{L}_{\mathrm{iso}}(\mathcal{E}, \mathcal{G} \times \mathcal{H}). 
   \end{equation*}
   This completes the proof by Lemma \ref{js013}.
\end{enumerate}

\end{proof}
In view of the lower estimates defining $\mathcal{N}$-ellipticity, it is clear that if $P\in \mathcal{O}^\perp(\mathbb{R}\times \mathbb{R}^{n-1})$ is $\mathcal{N}$-elliptic, then for all fixed $(\tau, \xi')\in \mathbb{R}\times \mathbb{R}^{n-1}$ with $\tau\neq 0$, the complex polynomial $P(\tau, \xi', \cdot)$ has no real roots.
Thus, by selecting the roots according to the sign of their imaginary parts,  one can always factorise  $P=P^-P^+$  in such a way that the roots $\rho_1^\pm(\tau, \xi'), ..., \rho_{N_\pm}^\pm(\tau, \xi')$ of the complex polynomial $P(\tau, \xi', \cdot)$ lie in $\mathbb{C}_\pm$. 
However, knowing whether both $P^+$ and $P^-$ are (strongly) $\mathcal{N}$-elliptic is non-trivial in the general case.
We conclude this section by showing that at least for a parabolic symbol this is the case. 

\begin{proposition}\label{prop: admissible_parabolic}
    Let $P\in \mathcal{O}^\perp(\mathbb{R}\times \mathbb{R}^n)$ be such that $P(\tau, \xi)\neq 0$ for all $(\tau, \xi)\in (\mathbb{R}\times \mathbb{R}^n)\setminus\set{(0,0)}$.
    If $P$ is $\rho$-homogeneous of degree $N$ for some $\rho, N\in \mathbb{N}_0$, then $P$ admits an admissible factorisation $P=P^-P^+$ in the sense of Definition \ref{def:admissible_symb}, and $P^\pm$ is $\rho$-homogeneous of degree $N_\pm$, $N_++N_-=N$.
\end{proposition}
\begin{proof}
    In view of Example \ref{js004}, the $\rho$-homogeneity of $P$ implies that $P$ is strongly $\mathcal{N}$-elliptic with order function $\mu=No_{\frac{1}{\rho}}$.
    By assumption, it holds that $P(\tau,\xi',\xi_n)=\sum_{j=0}^N c_j(\tau,\xi')(\ic \xi_n)^j$ with $c_j\in \mathcal{O}^\perp(\R\times\R^{n-1})$, and $c_N(\tau,\xi')=c_N\ne 0$ is a non-zero constant.
    We assume without loss of generality $c_N=1$.
    Let us for each $(\tau,\xi')\in (\R\times \R^{n-1})\setminus\set{(0,0)}$ denote by $R(\tau,\xi')$ the multiset\footnote{A \emph{multiset} allows for multiple instances of each of its elements. Here, this is relevant since multiple roots may occur.} of roots of the complex polynomial $z\mapsto P(\tau,\xi',z)$.
    In particular, $|R(\tau,\xi')|=N$.
    Since there are no real roots by assumption, we may decompose $R(\tau,\xi')$ into the disjoint multisets $R_+(\tau,\xi')$ and $R_-(\tau,\xi')$ of roots with positive and negative imaginary part, respectively.
    The coefficients $c_j(\tau,\xi')$ of $P$ being smooth, the roots of $z\mapsto P(\tau,\xi,z)$ depend locally continuously on $\tau$ and $\xi$, see for instance \cite{Bri66,Coo08}.
    Thus, since $(\R\times\R^{n-1})\setminus\{(0,0)\}$ is connected, the number $N_\pm$ of roots in $\C_\pm$ is independent of $\tau$ and $\xi$, i.e. $|R_\pm(\tau,\xi')|=N_\pm$ for all $(\tau,\xi')\in (\R\times\R^{n-1})\setminus\set{0}$, and $N_++N_-=N$.
    For $(\tau,\xi')\in (\R\times\R^{n-1})\setminus\set{0}$ we define
    \begin{align*}
        P^{\pm}(\tau,\xi',\cdot)&:\C\to\C,
        \qquad
        P^\pm(\tau,\xi',z):=\prod_{\rho\in R_\pm(\tau,\xi')} (z-\rho).
    \end{align*}
    Then $P^\pm: [(\R\times\R^{n-1})\setminus\{(0,0)\}]\times\C\to\C$ is smooth and $\rho$-homogeneous of order $N_\pm$ by a classical argument, cf. Chapter 4.4 in \cite{tanabe2017functional}.
    In particular, $P^\pm$ belongs to $\mathcal{O}^\perp(\R\times\R^n)$.
    Since $P(\tau, \xi)\neq 0$ implies $P^\pm(\tau, \xi)\neq 0$ for all $(\tau, \xi)\in (\mathbb{R}\times \mathbb{R}^n)\setminus\set{(0,0)}$, we may use again Example \ref{js004} to conclude that $P^\pm$ is strongly $\mathcal{N}$-elliptic with order function $\mu_\pm= N_\pm o_{\frac{1}{\rho}}$.
\end{proof}
\section{Applications}\label{sec:app}
In this section we illustrate how the present framework applies beyond the scope of the classical Lopatinski\u{\i}--Shapiro theory by providing the applications mentioned in the introduction.
\subsection{Higher-Order Regularity for the Heat Equation with Dirichlet Boundary Conditions}\label{sec:heat}
As an illustration of the compatibility conditions appearing in
Theorem \ref{thm: main_result}, we consider the heat equation with Dirichlet boundary condition
    \begin{align}\label{eqn: heat_equation}
        \left\{\begin{array}{rlll}
        \partial_t u - \Delta u&=&f     & \text{on } G_+ \\
         \Tr_0u&=&g    & \text{on } G^{n-1}
        \end{array}\right.
    \end{align}
Theorem \ref{thm: main_result} allows us to choose a right-hand side $f$ in any function space of type $\overline{H}^{\nu, p}_\perp(G_+)$ and to determine the corresponding compatibility condition for $g$ which allows for the maximal regularity gain of the heat operator $\partial_t-\Delta$, that is a gain of two derivatives in space and one derivative in time.
We demonstrate this phenomenon with $f\in \overline{H}^{(0,1), p}$. The result reads as follows.
\begin{figure}[t]
   \begin{subfigure}[b]{.49\textwidth}
       \begin{tikzpicture}
 \draw[->] (-.75,0) -- (4,0) node[below , font=\tiny] {$\vert \xi\vert$};
  \draw[->] (0,-.5) -- (0,1.5) node[left, font=\tiny] {$\tau$}; 
  
   \draw (1,0.1) -- (1,-0.1);
   \draw (2,0.1) -- (2,-0.1);
   \draw (3,0.1) -- (3,-0.1);
    \node[below, yshift= -2pt, font=\tiny] at (1,0) {$1$};
    \node[below, yshift=-2pt, font=\tiny] at (2, 0) {$2$};
    \node[below, yshift=-2pt, font=\tiny] at (3, 0) {$3$};

    \draw (-0.1, .5) -- (0.1, .5);
    \draw (-0.1, 1) -- (0.1, 1);
    \node[left, xshift=-2pt, font= \tiny] at (0,.5) {$1/2$};
    \node[left, xshift=-2pt, font= \tiny] at (0,1) {$1$};
  \coordinate (A) at (0,0);
  \coordinate (B) at (3,0);
  \coordinate (C) at (1,1);
  \coordinate (D) at (0,1);

  \filldraw[color=Sepia, fill=Sepia!30,thick] (A)--(B)--(C)--(D)--cycle;

     
      \node[font=\small] at (3, 1) {$\mathbb{E}$};
\end{tikzpicture}
\end{subfigure}
\hfill
   \begin{subfigure}[b]{.49\textwidth}
      
      \begin{tikzpicture}
 \draw[->] (-.75,0) -- (4,0) node[below , font=\tiny] {$\vert \xi\vert$};
  \draw[->] (0,-.5) -- (0,1.5) node[left, font=\tiny] {$\tau$}; 
  
   \draw (.75,0.1) -- (.75,-0.1);
  
    \node[below, yshift= -2pt, font=\tiny] at (.75,0) {$3/4$};

    \draw (-0.1, .375) -- (0.1, .375);
    \node[left, xshift=-2pt, font= \tiny] at (0,0.375) {$3/8$};
  \coordinate (A) at (0,0);
  \coordinate (B) at (.75,0);
  \coordinate (C) at (0, 0.375);

  \filldraw[color=Sepia, fill=Sepia!30,thick] (A)--(B)--(C)--cycle;

     
      \node[font=\small] at (5/2, 1) {$\Tr_2\mathbb{E}$};
\end{tikzpicture}
   \end{subfigure}
   \caption{The Newton polygons associated with the solution space $\mathbb{E}$ and with $\Tr_0\Delta u=\partial_t g-\Tr_0 f$ in Theorem \ref{js011} for $p=4$.
   The individual contributions $\partial_t g$ and $\Tr_0 f$ are contained only in the $HB$-space corresponding to the lower right corner $(1-\frac1p,0)=(\frac34,0)$.
   The compatibility condition in the theorem asserts that their difference is additionally contained in the $BH$-space corresponding to the top left corner $(0,\frac12-\frac1{2p})=(0,\frac38)$.}
   \label{fig: trace_spc_2}
\end{figure}
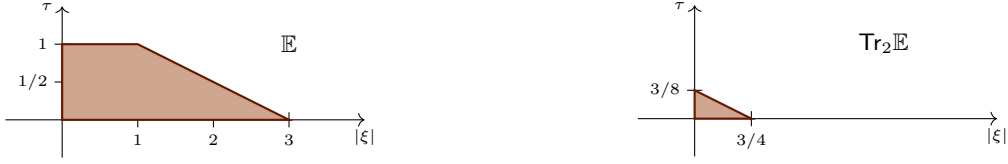
\begin{theorem}\label{js011}
     It holds $S:=(\partial_t-\Delta,\Tr_0)\in \call_{\mathrm{iso}}(\mathbb{E},\mathbb{D})$, where
    \begin{align*}
    \begin{aligned}
    \mathbb{E}&:= \overline{H}^{(1, 1), p}_\perp(G_+)\cap \overline{H}^{(0,3), p}_\perp(G_+), 
    &
    \mathbb{G}&:= HB^{(1,1-\frac1p), p}_\perp(G^{n-1})\cap HB_\perp^{(0,3-\frac1p), p}(G^{n-1}), \\
    \mathbb{F}&:=\overline{H}^{(0,1),p}_\perp(G_+),
    &
    \mathbb{D}&:=\{(f,g)\in \mathbb{F}\times\mathbb{G}\mid \partial_t g - \Tr_0 f \in BH_\perp^{(\frac12-\frac1{2p},0), p}(G^{n-1})\},
    \end{aligned}
    \end{align*}
    the norm on $\mathbb{D}$ is given by $\|f, g\|_{\mathbb{D}}:=\| f\|_{\mathbb{F}} +\|g\|_{\mathbb{G}}+\| \partial_tg-\Tr_0f\|_{\mathbb G'}$ and $\mathbb G':=BH_\perp^{(\frac12-\frac1{2p},0), p}(G^{n-1})$.

    In particular, for $f\in \mathbb{F}$ and $g\in \mathbb{G}$ there is a unique $u\in \mathbb{E}$ solving \eqref{eqn: heat_equation} if and only if $\partial_t g -\Tr_0 f \in \mathbb G'$.
    In this case it holds
    \begin{equation*}
        \|u\|_{\mathbb{E}}\simeq \|f\|_{\mathbb{F}} + \|g\|_{\mathbb{G}}+\|\partial_tg-\Tr_0f\|_{\mathbb G'}. 
    \end{equation*}
\end{theorem}
\begin{proof}
Let us fix the appropriate notations for the application of Theorem \ref{thm: main_result}.
    The symbol $H(\tau, \xi):= \ic \tau +\vert\xi\vert^2$ is $2$-homogeneous of degree 2. In view of Example \ref{js004} and Proposition \ref{prop: admissible_parabolic}, it is thus strongly $\mathcal{N}$-elliptic with order function $2o_{\frac{1}{2}}$ and  admissible. 
     Here, for all $(\tau, \xi')\in \mathbb{R}\times \mathbb{R}^{n-1}$ fixed with $\tau\neq 0$, the roots of the polynomial $z\mapsto H(\tau, \xi', z)$ are
    \begin{equation*}
        \rho_1^{\pm} (\tau, \xi') := \pm \ic\sqrt{\ic \tau + \vert \xi'\vert^2}, 
    \end{equation*} 
    where for any complex number $a\in \mathbb{C}\setminus (-\infty, 0)$, $\sqrt{a}$ denotes the principal square root of $a$.
    In particular $\sgn\Im\rho_1^\pm=\pm1$.
    
    Thus, the admissible factorisation is $H^\pm (\tau, \xi', \xi_n):= \xi_n-\rho_1^\pm(\tau, \xi')$ and we write in particular $H^+(\tau, \xi', \xi_n):=  c_1(\tau,\xi') (\ic\xi_n)+ c_0^+(\tau, \xi')$ with $c_1^+(\tau,\xi'):=-\ic$ and $c_0^+(\tau, \xi'): = -\rho_1^+(\tau, \xi')= -\ic\sqrt{\ic \tau + \vert\xi'\vert^2}$.
    
    \medskip 

    The space for the right-hand side $f$ is described by $\overline{H}^{(0,1),p}_\perp(G_+)= \overline{H}^{\nu, p}_\perp(G_+)$ with the order function $\nu:= o_0$, while the solution space is described by $\mathbb{E}=\overline{H}^{\mu+\nu}_\perp(G_+)$ with the order function $\mu:=2o\frac12$.
    Observe that $B=\Tr_0\in\mathcal L(\mathbb E,\mathbb G)$ by Theorem \ref{thm: trace_space}.
    The extended boundary matrix $\mathfrak{BC}$ is of dimension $N+M=2+1= 3$ and given by 
    \begin{equation*}
        \mathfrak{BC}= \begin{pmatrix}
            1 & 0 & 0 \\
            c_0^+ & -\ic & 0 \\
            0 & c_0^+ & -\ic
        \end{pmatrix}, 
    \end{equation*}
    where the first row corresponds to the boundary condition $\Tr_0$ and the two other rows are the complementing conditions $\Tr_0\op[H^+]_+= \op[c_0^+]\Tr_0 -\ic \Tr_1$ and $\Tr_1\op[H^+]_+= \op[c_0^+]\Tr_1-\ic\Tr_2 $. Here $\det(\mathcal{BC)}=-1 $ is non-vanishing, so that Theorem \ref{thm: main_result} yields a space $\Comp$ with $S\in \mathcal{L}_{\mathrm{iso}}(\mathbb{E}, \Comp)$.
It thus remains to show that $\mathbb{D}$ equals the compatibility space $\Comp$ with equivalent norms. 
Observe that 
\begin{equation*}
    \mathfrak{BC}^{-1} =  \begin{pmatrix}
            1 & 0 & 0 \\
            -\ic c_0^+ & \ic & 0 \\
            -{c_0^+}^2 & c_0^+ & \ic
        \end{pmatrix}
\end{equation*}
and so the norm on $\Comp$ is given by 
\begin{align*}
    \|f, g\|_{\Comp}= \| f\|_\mathbb{F}+ \| g\|_{\mathbb{G}} &+ \| \op[c_0^+]g -\Tr_0(\op[H^-]_+^{-1}f) \|_{T_{1\perp}^{\mu+\nu, p}}\\
    &+ \| \op[{c_0^+}]^2 g -\op[c_0^+] \Tr_0(\op[H^-]_+^{-1}f) -\ic\Tr_1 (\op[H^-]_+^{-1}f)\|_{T_{2\perp}^{\mu+\nu, p}}
\end{align*}
 We estimate the two last terms separately.
 On the one hand, the relations $\op[(c_0^+)^2]=-(\partial_t-\Delta_{x'})$ and $-\op[c_0^+]\Tr_0 - \ic\Tr_1 = \op[c_0^-]\Tr_0 - \ic\Tr_1=\Tr_0 \op[H^-]_+$ yield
\begin{align*}
    \| \op[{c_0^+}]^2 g &- \op[c_0^+] \Tr_0(\op[H^-]_+^{-1}f) -\ic\Tr_1 (\op[H^-]_+^{-1}f)\|_{T_{2\perp}^{\mu+\nu, p}} \\
    &= \|  \partial_t g - \Delta_{x'} g - \Tr_0 f\|_{T_{2\perp}^{\mu+\nu, p}}
    \le  \|  \partial_t g - \Tr_0 f\|_{T_{2\perp}^{\mu+\nu, p}}+\|   \Delta_{x'} g\|_{T_{2\perp}^{\mu+\nu, p}}.
\end{align*}
Since $\| \Delta_{x'}g\|_{T_{2\perp}^{\mu+\nu, p}} \lesssim \| g\|_{T_{0\perp}^{\mu+\nu, p}} = \| g\|_{\mathbb{G}}$ by the mapping properties of $\Delta_{x'}$, and since
\begin{align*}
    \|\partial_t g - \Tr_0 f\|_{T_{2\perp}^{\mu+\nu, p}}
    &= \|\partial_t g - \Tr_0 f\|_{HB_\perp^{(0, 1-\frac{1}{p}), p}} + \|\partial_t g - \Tr_0 f\|_{BH_\perp^{(\frac{1}{2}-\frac{1}{2p}, 0), p}} \\
    &\lesssim \|f\|_{\mathbb{F}} + \|g\|_{\mathbb{G}}+ \|\partial_t g - \Tr_0 f\|_{BH_\perp^{(\frac{1}{2}-\frac{1}{2p}, 0), p}},
\end{align*}
it follows
\begin{align}\label{eqn: estimates_second_term_heat}
\begin{split}
     \| \op[{c_0^+}^2] g - \op[c_0^+] &\Tr_0(\op[H^-]_+^{-1}f) -\ic\Tr_1 (\op[H^-]_+^{-1}f)\|_{T_{2\perp}^{\mu+\nu, p}} \\
     &\lesssim \| f\|_{\mathbb{F}} + \| g\|_{\mathbb{G}}+ \|  \partial_t g - \Tr_0 f\|_{BH_\perp^{(\frac{1}{2}-\frac{1}{2p}, 0), p}} = \| f, g\|_{\mathbb{D}}. 
\end{split}
\end{align}
On the other hand, writing the trace spaces in Theorem \ref{thm: trace_space} as interpolation spaces and using $(\mu+\nu)^{(k)}=(3-k)o_{\frac12}$ for $k\in\set{1,2,3}$, we see that $\op[c_0^+]\in \mathcal{L}_{\mathrm{iso}}(T^{\mu+\nu,p}_{1\perp}(G^{n-1}),T^{\mu+\nu,p}_{2\perp}(G^{n-1}))$, and so 
\begin{align*}
    \| &\op[c_0^+]g - \Tr_0(\op[H^-]_+^{-1}f) \|_{T_{1\perp}^{\mu+\nu, p}} \simeq \| \op[{c_0^+}]^2 g - \op[c_0^+] \Tr_0(\op[H^-]_+^{-1}f) \|_{T_{2\perp}^{\mu+\nu, p}}\\
    &\leq\|\op[c_0^+]^2 g - \op[c_0^+] \Tr_0(\op[H^-])_+^{-1}f -\ic \Tr_1(\op[H^-])_+^{-1}f\|_{T_{2\perp}^{\mu+\nu, p}}
    +\| \Tr_1 (\op[H^-]_+^{-1}f)\|_{T_{2\perp}^{\mu+\nu, p}}.
\end{align*}
Remarking that $T_{2\perp}^{\mu+\nu, p}= T_{1\perp}^{\mu_-+\nu, p}$ due to $(\mu+\nu)^{(2+k)}=(\mu_-+\nu)^{1+k}=(1-k)o_\frac12$ for $k\in\set{0,1}$, we obtain 
\begin{equation*}
    \| \Tr_1 (\op[H^-]_+^{-1}f)\|_{T_{2\perp}^{\mu+\nu, p}}=\| \Tr_1 (\op[H^-]_+^{-1}f)\|_{T_{1\perp}^{\mu_-+\nu, p}}\lesssim \| f\|_{\mathbb{F}}
\end{equation*}
and thus recalling (\ref{eqn: estimates_second_term_heat}), we have 
\begin{equation*}
     \| \op[c_0^+]g - \Tr_0(\op[H^-]_+^{-1}f) \|_{T_{1\perp}^{\mu+\nu, p}} \lesssim \|f, g\|_{\mathbb{D}}. 
\end{equation*}
We have proved $\mathbb{D}\hookrightarrow\Comp$. 

For the converse embedding, let $(f, g)\in \Comp$.
Then there exists a unique $u\in \mathbb{E}$ verifying  (\ref{eqn: heat_equation}). Applying $\Tr_0$ to $\op[H]_+ u= f$, we obtain $\partial_tg-\Tr_0f=\Tr_0\Delta u.$, and so
\begin{align*}
    \|  \partial_t g - \Tr_0 f\|_{BH_\perp^{(\frac{1}{2}-\frac{1}{2p}, 0), p}} &\leq\|  \partial_t g - \Tr_0 f\|_ {T_{2\perp}^{\mu+\nu, p}}
    =\| \Tr_0 \Delta u\|_{T_{2\perp}^{\mu+\nu, p}} \\&\lesssim \| \Delta u\|_{H^{(\frac{1}{2}, 0), p}_\perp \cap H^{(0, 1), p}_\perp}
    \lesssim \| u\| _{H^{(\frac{1}{2}, 2), p}_\perp \cap H^{(0, 3), p}_\perp}
    \lesssim \| u\|_{\mathbb{E}}\simeq \|f, g\|_{\Comp}.
\end{align*}
Thus $\Comp \hookrightarrow \mathbb{D}$.
This completes the proof.
\end{proof}
In contrast to the Dirichlet situation, the same gain in spatial regularity for Neumann or other higher-order boundary conditions preserves the full product structure of the data spaces.
\begin{lemma}\label{lemma: heat_Neumann_shifted}
    Let $\mathbb E:=\overline{H}^{(1, 1), p}_\perp(G_+)\cap \overline{H}^{(0,3), p}_\perp(G_+)$, $\mathbb F:=\overline{H}^{(0,1),p}_\perp(G_+)$, and set
    \begin{align}\label{lemma: heat_Neumann_shifted_e1}
    \begin{split}
        \mathbb G_N &:= BH_\perp^{(1-\frac1{2p},0),p}(G^{n-1}) \cap HB_\perp^{(0,2-\frac1p),p}(G^{n-1}), \\
        \mathbb G_H &:= BH_\perp^{(\frac12-\frac1{2p},0),p}(G^{n-1}) \cap HB_\perp^{(0,1-\frac1p),p}(G^{n-1}).
    \end{split}
    \end{align}
    Then $(\partial_t-\Delta,\Tr_1)\in \call_{\mathrm{iso}}(\mathbb E,\mathbb F\times\mathbb G_N)$ and $(\partial_t-\Delta,\Tr_2+a\Tr_0)\in \call_{\mathrm{iso}}(\mathbb E,\mathbb F\times\mathbb G_H)$ for any $a\ge 0$.
\end{lemma}
\begin{proof}
    With respect to the trace vector $(\Tr_0u,\Tr_1u,\Tr_2u)$, the complemented boundary matrices are
    \begin{align*}
        \mathfrak{BC}_N
        =
        \begin{pmatrix}
            0&1&0\\
            c_0^+&-\ic&0\\
            0&c_0^+&-\ic
        \end{pmatrix},
        \qquad
        \mathfrak{BC}_H
        =
        \begin{pmatrix}
            a&0&1\\
            c_0^+&-\ic&0\\
            0&c_0^+&-\ic
        \end{pmatrix},
    \end{align*}
    with $\det\mathfrak{BC}_N=\ic c_0^+$ and $\det\mathfrak{BC}_H=(c_0^+)^2-a$.
    Since $c_0^+$ is strongly $\mathcal N$-elliptic with order function $o_{\frac12}$, so is $\det\mathfrak{BC}_N$.
    Moreover, $(c_0^+)^2-a=-(\ic\tau+|\xi'|^2+a)$, so that $\det\mathfrak{BC}_H$ is strongly $\mathcal N$-elliptic with order function $2o_\frac12$.
    For the Neumann result, we set
    \begin{align*}
    \begin{array}{llllll}
        t_{1,0}:=2o_{\frac12}+o_0
        & t_{2,0}:=2o_{\frac12}
        & t_{3,0}:=o_{\frac12}
        & s_{1,0}:=2o_{\frac12}
        & s_{2,0}:=o_{\frac12}+o_0
        & s_{3,0}:=o_{\frac12}\\
        t_{1,1}:=2o_{\frac12}
        & t_{2,1}:=o_{\frac12}
        & t_{3,1}:=0,
        & s_{1,1}:=o_{\frac12}
        & s_{2,1}:=o_{\frac12}
        & s_{3,1}:=0,
    \end{array}
    \end{align*}
    where the choice of $s_{1,0}$ and $s_{1,1}$ is motivated by the description $\mathbb G_N=(H^{o_\frac12,p}_\perp(G^{n-1}),H^{2o_\frac12,p}_\perp(G^{n-1}))_{1-\frac1p,p}$.
    Since $t_{j,\mathfrak{k}}-s_{i,\mathfrak{k}}$ is a strong upper order function for $\mathfrak{BC}_{N,ij}$ for $i,j\in\set{1,2,3}$ and $\mathfrak{k}\in\set{0,1}$, and since it holds $\sum_{l=1}^3 (t_{l,\mathfrak k} - s_{l,\mathfrak k}) = o_{\frac12}$, the assertion $(\partial_t-\Delta,\Tr_1)\in \call_{\mathrm{iso}}(\mathbb E,\mathbb F\times\mathbb G_N)$ follows from the second part of Theorem \ref{thm: main_result}.

    For $\mathfrak{BC}_H$, choose the same $t_{j,\mathfrak k}$ and $s_{i,\mathfrak k}$ as above except for $s_{1,0}:=o_{\frac12}$ and $s_{1,1}:=0$.
    Then the result follows from Theorem \ref{thm: main_result} in the same way as above since $\mathbb G_H=(L^p_\perp(G^{n-1}),H^{o_\frac12,p}_\perp(G^{n-1}))_{1-\frac1p,p}$.
\end{proof}
\subsection{Parabolic Problems with Dynamic Boundary Conditions}
As seen in Proposition \ref{prop: admissible_parabolic}, parabolic operators admit a simple, triangular Newton polygon and a trivial admissible structure.
In particular, maximal time-periodic $L^p$-regularity for such operators with boundary conditions independent of the time variable was established in \cite{kyed2019time} without requiring the Newton polygon machinery, relying on Lopatinski\u{\i}-Shapiro conditions to guarantee well-posedness.
In contrast, dynamic boundary conditions fall outside the classical Lopatinski\u{\i}–Shapiro framework.
Nevertheless, they fit naturally into the approach developed in this article, as we demonstrate in the following examples.

\subsubsection*{The Cahn-Hilliard Equation with Dynamic Boundary Conditions and Surface Diffusion}

Let us consider the following Cahn-Hilliard equation with dynamic boundary condition and surface diffusion
 \begin{equation}\label{eqn: ch_dynamic_diff}
            \left\{\begin{array}{rcll}
            \partial_tu+\Delta^2u&=& f     & \text{ on }\mathbb{T}\times \mathbb{R}^n_+ \\
             \partial_n \Delta u&=&g_1    & \text{ on }\mathbb{T}\times \mathbb{R}^{n-1}\\
             \partial_t u +  \partial_n u-\Delta_{\mathbb{R}^{n-1}}u &=& g_2& \text{ on }\mathbb{T}\times \mathbb{R}^{n-1},
            \end{array}\right.
\end{equation}
see \cite{denk2008parabolic}.
We will realize the boundary conditions as $\opB\Tr^{(4)}u$ with the operator
\begin{align*}
    \opB=
    \begin{pmatrix}
        0 & \Delta_{\R^{n-1}} & 0 & 1 \\
        \partial_t-\Delta_{\R^{n-1}} & 1 & 0 & 0
    \end{pmatrix}.
\end{align*}
The symbol $P_{CH}(\tau, \xi):= \ic \tau +\vert\xi\vert^4$ is $4$-homogeneous of degree 4 and thus admissible by Proposition \ref{prop: admissible_parabolic}. 

\begin{theorem}\label{thm:CH}
\begin{enumerate}
    \item\label{thm:CHi} It holds $S=(\partial_t+\Delta^2,\opB\Tr^{(4)})\in \call_{\mathrm{iso}}(\mathbb{E},\mathbb{F}\times\mathbb{G})$, where $\mathbb{E}:= \overline{H}^{(1, 0),p}_\perp(G_+) \cap \overline{H}^{(0,4),p}_\perp(G_+)$, $\mathbb{F}:= L^p_\perp(G_+)$, and $\mathbb{G}:= \mathbb{G}_1 \times \mathbb{G}_2$ with
    \begin{align*}
        \mathbb{G}_1&:= T^{o_\frac14,p}_{0\perp}(G^{n-1})=BH^{(\frac{1}{4}-\frac{1}{4p}, 0), p}_\perp(G^{n-1}) \cap HB^{(0, {1-\frac{1}{p}}),p}_\perp(G^{n-1}),\\
        \mathbb{G}_2&:=(H^{3o_{\frac{1}{4}}-2o_{\frac{1}{2}}, p}_\perp(G^{n-1}),H^{4o_{\frac{1}{4}}-2o_{\frac{1}{2}},p}_\perp(G^{n-1}))_{1-\frac{1}{p}, p} .
    \end{align*}
    In particular, for all $f\in \mathbb{F}$ and $g=(g_1, g_2) \in \mathbb{G}$, problem (\ref{eqn: ch_dynamic_diff}) admits a unique solution $u\in \mathbb{E}$ and it holds 
    \begin{equation*}
        \| u\|_\mathbb{E}\simeq \| f\|_{\mathbb{F}}+\| g\|_{\mathbb{G}}. 
    \end{equation*}
    \item\label{thm:CHii} If we keep the notation from part \ref{thm:CHi}, and set 
    \begin{align*}
        \mathbb{E}'&:=\{u\in \mathbb{E}\mid \Tr_0 u \in \mathcal E_0 \},\\
        \|u\|_{\mathbb{E}'} &:= \|u\|_{\mathbb{E}}+\|\Tr_0u\|_{\mathcal E_0},
    \end{align*}
    where 
    \begin{equation*}
        \mathcal E_0:= T_{0\perp}^{3o_{\frac{1}{4}}+2o_{\frac{1}{2}}}(G^{n-1})= BH_\perp^{(\frac{7}{4}-\frac{1}{4p}, 0), p}(G^{n-1})\cap HB^{(1, 3-\frac{1}{p}), p}_\perp(G^{n-1})\cap HB^{(0, 5-\frac{1}{p}),p}_\perp(G^{n-1}),
    \end{equation*}
    then for $\mathbb{G}':=\mathbb{G}_1\times\mathbb{G}_2'$ with
    \begin{align*}
        \mathbb{G}_2'&:= T_{0\perp}^{3o_{\frac{1}{4}}, p}(G^{n-1})= BH^{(\frac{3}{4}-\frac{1}{4p}, 0), p}_\perp(G^{n-1}) \cap HB^{(0, 3-\frac{1}{p}),p}_\perp(G^{n-1}),
    \end{align*}
     it holds $S\in \call_{\mathrm{iso}}(\mathbb{E}',\mathbb{F}\times\mathbb{G}')$.
     In particular, for all $f\in \mathbb{F}$ and $g=(g_1, g_2) \in \mathbb{G}'$, problem (\ref{eqn: ch_dynamic_diff}) admits a unique solution $u\in \mathbb{E}'$ and it holds 
    \begin{equation*}
        \| u\|_\mathbb{E} + \|\Tr_0 u\|_{\mathcal E_0} \simeq \| f\|_{\mathbb{F}}+\| g\|_{\mathbb{G}'}. 
    \end{equation*}
\end{enumerate}
\end{theorem}
\begin{proof}
     Since $P_{CH}$ is $4$-homogeneous of degree $4$ and $P_{CH}(\tau, \xi)\neq 0$ for all $(\tau, \xi)\in \mathbb{R}\times \mathbb{R}^n$ with $\tau\neq 0$, it is strongly $\mathcal{N}$-elliptic with order function $\mu_{CH}:= 4o_{\frac{1}{4}}$. Let us write  $P_{CH}= P^-P^+$, where $P^\pm$ strongly $\mathcal{N}$-elliptic with order function $\mu_\pm := 2o_{\frac{1}{4}}$. Writing $P^+(\tau, \xi', \xi_n)= -(\ic \xi_n)^2+ c_1^+(\tau, \xi')(\ic \xi_n)+ c_0^+(\tau, \xi')$, the complementary boundary matrix associated with (\ref{eqn: ch_dynamic_diff}) is given by 
     \begin{equation*}
         \mathfrak{BC}(\tau, \xi'):= \begin{pmatrix}
             0 & -\vert \xi'\vert^2 &0 & 1\\
             \ic \tau +\vert \xi'\vert^2 & 1 & 0 &0\\
          c_0^+(\tau, \xi') &  c_1^+(\tau, \xi')    & -1 &0 \\
          0& c_0^+(\tau, \xi') &  c_1^+(\tau, \xi')    & -1.
         \end{pmatrix}
     \end{equation*}
     and we obtain $\det(\mathfrak{BC})(\tau, \xi')= (\ic \tau +\vert \xi'\vert^2)(\vert \xi'\vert^2-{c_1^+}(\tau, \xi')^2-c_0^+(\tau, \xi'))+ c_1^+(\tau, \xi')c_0^+(\tau, \xi')$. By Lemma \ref{lemma: ellipticity_CH} below, $\det(\mathfrak{BC})$ is strongly $\mathcal{N}$-elliptic with order function $2o_{\frac{1}{2}}+2o_{\frac{1}{4}}$. 
    \begin{enumerate}
        \item The order functions  associated with the columns of the complementary boundary matrix will be those of the trace space $T^{4o_{\frac{1}{4}, p}}_\perp(G^{n-1})$
	\begin{equation*}
	\begin{array}{llll}
	    t_{1,0}:=\mu_{CH}^{(0)}=4o_{\frac{1}{4}}
	    & t_{2,0}:=\mu_{CH}^{(1)}=3o_{\frac{1}{4}}
	    & t_{3,0}:=\mu_{CH}^{(2)}=2o_{\frac{1}{4}}
	    & t_{4,0}:=\mu_{CH}^{(3)}=o_{\frac{1}{4}} \\
	    t_{1,1}:=\mu_{CH}^{(1)}=3o_{\frac{1}{4}}
	    & t_{2,1}:=\mu_{CH}^{(2)}=2o_{\frac{1}{4}}
	    & t_{3,1}:=\mu_{CH}^{(3)}=o_{\frac{1}{4}}
	    & t_{4,1}:=\mu_{CH}^{(4)}=0.
	\end{array}
	\end{equation*}
   For the rows of the complementary boundary matrix, we set 
\begin{equation*}
\begin{array}{llll}
    s_{1,0}:=o_{\frac{1}{4}}
    & s_{2,0}:=4o_{\frac{1}{4}}-2o_{\frac{1}{2}}
    & s_{3,0}:=\mu_+^{(0)}=2o_{\frac{1}{4}}
    & s_{4,0}:=\mu_+^{(1)}=o_{\frac{1}{4}} \\
    s_{1,1}:=0
    & s_{2,1}:=3o_{\frac{1}{4}}-2o_{\frac{1}{2}}
    & s_{3,1}:=\mu_+^{(1)}=o_{\frac{1}{4}}
    & s_{4,1}:=\mu_+^{(2)}=0.
\end{array}
\end{equation*}
     the choice of $s_{i,\mathfrak{k}}$, $\mathfrak{k}\in\{0,1\}$, $i\in\{1,2\}$, being motivated by the equality 
     \begin{equation*}
         \mathbb{G}_i:= (H^{s_{i,1}, p}_\perp(G^{n-1}), H^{s_{i,0}, p}_\perp(G^{n-1}))_{1-\frac{1}{p}, p}, \qquad i\in\{1,2\}.
     \end{equation*}
    Then $t_{j,\mathfrak{k}}-s_{i,\mathfrak{k}}$ is a strong upper order function for $\mathfrak{BC}_{ij}$ for all $i,j\in\{1,2,3,4\}$ and $\mathfrak{k}\in\{0,1\}$, and additionally $\delta=2o_{\frac{1}{4}}+2o_{\frac{1}{2}}= \sum_{l=1}^4 t_{l, \mathfrak{k}}-s_{l, \mathfrak{k}}$.
    Thus, \ref{thm:CHi} follows by the second point of Theorem \ref{thm: main_result}.
        \item Since $\mathbb{E}'$ embeds continuously in $\mathbb{E}$, we immediately get $\partial_t+\Delta^2= \op[P_{CH}] \in \mathcal{L}(\mathbb{E}', \mathbb{F})$ from the previous point. Let us set 
    \begin{equation*}
        \mathcal E':= \Tr^{(4)}(\mathbb{E}')= \mathcal E_0 \times T_{1\perp}^{4o_{\frac{1}{4}}, p}(G^{n-1}) \times  T_{2\perp}^{4o_{\frac{1}{4}}, p}(G^{n-1}) \times  T_{3\perp}^{4o_{\frac{1}{4}}, p}(G^{n-1})
    \end{equation*}
    and 
    \begin{equation*}
        \widetilde{\mathbb{G}}:=\mathbb{G}_1 \times \mathbb{G}_2' \times T_{0\perp}^{\mu_-}(G^{n-1}) \times T_{1\perp}^{\mu_-}(G^{n-1}).
    \end{equation*}
    Setting all the order functions as above except 
     \begin{equation*}
        \begin{array}{ll}
         t_{1,0}:=  3o_{\frac{1}{4}}+2o_{\frac{1}{2}} & t_{1,1} : =2o_{\frac{1}{4}}+2o_{\frac{1}{2}}
    \end{array}
    \end{equation*}
    and
    \begin{equation*}
        \begin{array}{ll}
         s_{2, 0}:=  3o_{\frac{1}{4}} & s_{2,1} : = 2o_{\frac{1}{4}}, 
    \end{array}
    \end{equation*}
    the fact that $t_{j,\mathfrak{k}}-s_{i,\mathfrak{k}}$ is a strong upper order function for $\mathfrak{BC}_{ij}$ for all $i,j\in\{1,2,3,4\}$ and $\mathfrak{k}\in\{0,1\}$ remains true.
    Writing the spaces as interpolation spaces and arguing as in the proof of Theorem \ref{thm: main_result}, this yields $\op[\mathfrak{BC}] \in \mathcal{L}(\mathcal E', \widetilde{\mathbb{G}})$.
    Considering only the first two rows of $\mathfrak{BC}$, this means $\opB \Tr^{(4)}\in \mathcal{L}(\mathbb{E}', \mathbb{G}')$ and thus $S\in \mathcal{L}(\mathbb{E}', \mathbb{F}\times\mathbb{G}')$.

    \medskip

    On the other hand, we also find that the equality $\delta= 2o_{\frac{1}{4}}+2o_{\frac{1}{2}}= \sum_{l=1}^4 t_{l, \mathfrak{k}}-s_{l, \mathfrak{k}}$ still holds for  $\mathfrak{k}=0$ and $\mathfrak{k}=1$, and so $\mathfrak{BC}$ is a strong mixed-order system, i.e., $\op[\mathfrak{BC}] \in \mathcal{L}_{\mathrm{iso}}(\mathcal E', \widetilde{\mathbb{G}})$.
    Now fix $(f, g)\in \mathbb{F}\times \mathbb{G}'$. Since $\mathbb{G}'\hookrightarrow \mathbb{G}$, the equation (\ref{eqn: ch_dynamic_diff}) admits a unique solution $u\in \mathbb{E}$ by \ref{thm:CHi} with 
    \begin{equation*}
        \|u\|_{\mathbb{E}}\lesssim \|f\|_{\mathbb{F}}+\|g\|_{\mathbb{G}}\lesssim \|f\|_{\mathbb{F}}+\|g\|_{\mathbb{G}'}.
    \end{equation*}
    Furthermore, using the explicit formula for $S^{-1}$ given in Theorem \ref{thm: main_result}, we get that the solution $u$ must satisfy 
    \begin{equation*}
      \op[\mathfrak{BC}] \Tr^{(4)}u= (g_1, g_2, \Tr_0(\op[P^-]_+^{-1} f), \Tr_1(\op[P^-]^{-1}_+f))\in \widetilde{\mathbb{G}}.
    \end{equation*}
    Therefore, $\op[\mathfrak{BC}] \in \mathcal{L}_{\mathrm{iso}}(\mathcal E', \widetilde{\mathbb{G}})$ yields
    \begin{equation*}
        \|\Tr_0 u \|_{\mathcal E_0} \leq \|\Tr^{(4)}u\|_\mathbb{H}\lesssim\|f\|_{\mathbb{F}}+ \|g\|_{\mathbb{G}'}, 
    \end{equation*}
    where we additionally used the boundedness of the trace operators $\Tr_0\in \mathcal{L}(\overline{H}^{\mu_+, p}_\perp(G_+), \mathbb{G}_3)$, $\Tr_1\in \mathcal{L}(\overline{H}^{\mu_+, p}_\perp(G_+), \mathbb{G}_4)$ and the bulk operator $\op[P^-]_+ \in \mathcal{L}_{\mathrm{iso}}(\overline{H}^{\mu_+, p}_\perp(G_+), \mathbb{F})$, which holds due to Proposition \ref{prop: boundedness_lower_support} since $\op[P^-]$ and $\op[P^-]^{-1}$ preserve lower support.
    This completes the proof.
    \end{enumerate}
	\end{proof}
\begin{lemma}\label{lemma: ellipticity_CH}
    The symbol 
    \begin{equation*}
        d(\tau, \xi'):=  (\ic \tau +\vert \xi'\vert^2)(|\xi'|^2-{c_1^+}^2-c_0^+)+ c_1^+c_0^+
    \end{equation*}
    is strongly $\mathcal{N}$-elliptic with order function $2o_{\frac{1}{4}}+2o_{\frac{1}{2}}$.
   
\end{lemma}
\begin{proof}
	For the sake of the proof, let us write $c_0:=c_0^+$ and $c_1:=c_1^+$.
	For all  $(\tau, \xi')\in \mathbb{R}\times \mathbb{R}^n$ with $\tau \neq 0$, let us denote by $\rho_1(\tau, \xi'), \rho_2(\tau, \xi')$ the roots of $z\mapsto P_{CH}(\tau, \xi', z)$ with $\Im(\rho_i(\tau, \xi'))>0$.
	These roots are explicitly given by $\eps_1\sqrt{-|\xi'|^2+\eps_2\sqrt{-\ic\tau}}$ with $\eps_1,\eps_2\in\set{-1,1}$ depending on the sign of $\tau$.
	They are $4$-parabolic of degree $1$, and strongly $\mathcal{N}$-elliptic with order function $o_{\frac{1}{4}}$ (cf.\@ Example \ref{js004}).
    
	Since $c_0=\rho_1\rho_2$ and $c_1= \ic(\rho_1 + \rho_2)$, an application of Proposition \ref{prop: arithmetics_order_functions} yields that $(|\xi'|^2-c_1^2-c_0)$ has strong upper order function $2o_{\frac{1}{4}}$.
	Since the symbol $\ic \tau + \vert \xi'\vert^2$ is $2$-parabolic of degree 2, it admits a strong upper order function $2o_{\frac{1}{2}}$.
	Thus the symbol $(\ic \tau +\vert \xi'\vert^2)(|\xi'|^2-c_1^2-c_0)$  has strong upper order function $2o_{\frac{1}{4} }+ 2o_{\frac{1}{2}}$.
	A similar argument also shows that $c_1c_0$ admits a strong upper order function $3o_{\frac{1}{4}}$.
	Since the Newton polygon of $3o_{\frac{1}{4}}$ is included in the one of $2o_{\frac{1}{4} }+ 2o_{\frac{1}{2}}$, we conclude that $c_1c_0$ also admits $2o_{\frac{1}{4} }+ 2o_{\frac{1}{2}}$ as strong upper order function and so does $d$.
    
    \medskip
    
    For the lower bounds, \cite[Lemma 3.5]{denk2008parabolic} shows that $||\xi'|^2-c_1^2-c_0|\simeq |\tau|^\frac12+|\xi'|^2$.
    In particular, for $\lambda>0$ we have $||\xi'|^2-c_1^2-c_0|\gtrsim_\lambda \langle\tau\rangle^\frac12+\langle\xi'\rangle^2$ for all $(\tau,\xi')$ with $|\tau|\ge \lambda$.
    Since the heat symbol $\ic\tau+|\xi'|^2$ is $\caln$-elliptic for the order function $2o_{\frac12}$, and $c_1c_0$ has upper order function $3o_\frac14$, we obtain $c_\lambda',c_\lambda''>0$ such that
    \begin{align*}
		|d(\tau,\xi')|&\ge c_\lambda'\bigl(\langle\tau\rangle+\langle\xi'\rangle^2\bigr)\bigl(\langle\tau\rangle^\frac12+\langle\xi'\rangle^2\bigr)-c_\lambda''(\langle\tau\rangle^\frac34+\langle\xi'\rangle^3)\\
        &\ge c_\lambda'\bigl(\langle\tau\rangle^\frac32+\langle\tau\rangle\langle\xi'\rangle^2+\langle\xi'\rangle^4\bigr)-c_\lambda''(\langle\tau\rangle^\frac34+\langle\xi'\rangle^3).
	\end{align*}
	Since
	\begin{align*}
	\langle\tau\rangle^{3/4}+\langle\xi'\rangle^3
	\ll
	\langle\tau\rangle^{3/2}
	+\langle\tau\rangle\langle\xi'\rangle^2
	+\langle\xi'\rangle^4
	\end{align*}
	as $|\tau|+|\xi'|\to\infty$, we may choose $R_\lambda>0$ so large that there is $c_\lambda>0$ with
	\begin{align*}
	c_\lambda' \bigl(\langle\tau\rangle^\frac32 + \langle\tau\rangle\langle\xi'\rangle^2 +\langle\xi'\rangle^4\bigr)-c_\lambda''(\langle\tau\rangle^\frac34+\langle\xi'\rangle^3)  \ge c_\lambda(\langle\tau\rangle^\frac32 +\langle\tau\rangle\langle\xi'\rangle^2+\langle\xi'\rangle^4) = c_\lambda w_{2o_\frac14+2o_\frac12}(\tau,\xi')
    	\end{align*}
    	whenever $|\tau|+|\xi'|\ge R_\lambda$.
    	Here, in the last step we have used that the Newton polygon $\caln(2o_\frac14+2o_\frac12)$ has the significant vertices $\set{(4,0),(2,1),(0,\frac32)}$.
    	It therefore remains to establish the lower estimate $|d|\ge c_\lambda w_{2o_\frac14+2o_\frac12}$ on the compact set $K_{\lambda}:=\set{(\tau,\xi')\in\R\times\R^{n-1}\mid |\tau|\ge \lambda, |\tau|+|\xi'|\le R_\lambda}$.
	By continuity, it suffices to show that $d(\tau,\xi')\ne 0$ whenever $\tau\ne 0$.
    
    \medskip

    To show this, let first $\tau>0$.
    The case $\tau<0$ then follows by complex conjugation: indeed, the roots in the upper half-plane satisfy $\rho_j(-\tau,\xi')=-\overline{\rho_j(\tau,\xi')}$, up to their numbering, and hence $c_0(-\tau,\xi')=\overline{c_0(\tau,\xi')}$ and $c_1(-\tau,\xi')=\overline{c_1(\tau,\xi')}$.
    Consequently, $d(-\tau,\xi')=\overline{d(\tau,\xi')}$.

    \medskip
    
    We observe that $\rho_1^2$ and $\rho_2^2$ are the two roots of $z\mapsto (z+|\xi'|^2)^2+\ic \tau=z^2+2|\xi'|^2z+|\xi'|^4+\ic\tau$, and thus $\rho_1^2+\rho_2^2=-2|\xi'|^2$ and $c_0^2=\rho_1^2\rho_2^2=|\xi'|^4+\ic \tau$.
    In particular, $-c_1^2=\rho_1^2+\rho_2^2+2\rho_1\rho_2=2(c_0-|\xi'|^2)$, and we may rewrite
    \begin{align*}
    	d(\tau,\xi')=\frac12 c_1(2c_0-(\ic\tau+|\xi'|^2)c_1).
    \end{align*}
    Since $c_1=\ic(\rho_1+\rho_2)$ has negative real part, it is nonzero.
    It remains to show that $(\ic\tau+|\xi'|^2)c_1\ne 2 c_0$.

    Observe that $c_0=\rho_1\rho_2=-\sqrt{\ic\tau+|\xi'|^4}$ lies in the third quadrant.
    Write $c_0=:-x-\ic y$ with $x,y>0$.
    By $c_0^2=\ic\tau+|\xi'|^4$ we have $x^2-y^2=|\xi'|^4$ and $2xy=\tau$.
    We record in particular that $x>|\xi'|^2$.
    Since $c_1^2=2(|\xi'|^2-c_0)=2(|\xi'|^2+x+\ic y)$ lies in the first quadrant and $\Re c_1<0$, also $\Im c_1<0$.
    We write $c_1=-\alpha-\ic\beta$ with $\alpha,\beta>0$.
    
    To derive a contradiction, assume that $(\ic\tau+|\xi'|^2)c_1= 2 c_0$.
    Comparing real and imaginary parts gives
    \begin{align*}
    |\xi'|^2\alpha-\tau\beta=2x, \qquad \tau\alpha+|\xi'|^2\beta=2y.
    \end{align*}
    Solving this system for $\beta$, we find $\beta = \frac{2(|\xi'|^2y-\tau x)}{|\xi'|^4+\tau^2}$.
    Since $\beta>0$, it follows from $2xy=\tau$ that $|\xi'|^2y>\tau x=2x^2y$ and hence $|\xi'|^2>2x^2$.
    By our earlier observation $x>|\xi'|^2$, this shows $0<|\xi'|^2<x<\frac12$.
    
    We now square the assumed identity $(\ic\tau+|\xi'|^2)c_1= 2 c_0$. Using $c_1^2=2(|\xi'|^2-c_0)$ and $c_0^2=\ic\tau+|\xi'|^4$, we obtain $(\ic\tau+|\xi'|^2)^2(|\xi'|^2-c_0)= 2 (\ic\tau + |\xi'|^4)$.
    Taking imaginary parts gives $(|\xi'|^4-\tau^2)y+2|\xi'|^2\tau(|\xi'|^2+x)=2\tau$.
    Using $x^2-y^2=|\xi'|^4$ and $2xy=\tau$, we obtain
    \begin{align}\label{lemma: ellipticity_CH_e1}
    |\xi'|^4-4x^4+4|\xi'|^4x^2+4|\xi'|^4x+4|\xi'|^2x^2-4x=0.
    \end{align}
    However, by $0<|\xi'|^2<x<\frac12$,
    \begin{align*}
    |\xi'|^4-4x^4+4|\xi'|^4x^2+4|\xi'|^4x+4|\xi'|^2x^2-4x &< x^2+8x^3-4x =x(8x^2+x-4)<0.
    \end{align*}
    This contradicts \eqref{lemma: ellipticity_CH_e1}.
    Therefore  $d(\tau,\xi')\ne 0$ for $\tau\ne 0$.
    This completes the proof.
\end{proof}
\subsubsection*{The time-periodic Cahn-Hilliard operator with Liu-Wu dynamic boundary conditions}
We consider the following Cahn-Hilliard equation with Liu-Wu dynamic boundary conditions
 \begin{equation}\label{eqn: ch_lw}
            \left\{\begin{array}{rcll}
            \partial_tu+\Delta^2u&=& f     & \text{ on }\mathbb{T}\times \mathbb{R}^n_+ \\
             \partial_n \Delta u&=&g_1    & \text{ on }\mathbb{T}\times \mathbb{R}^{n-1}\\
             \partial_t u  - \partial_n \Delta_{\mathbb{R}^{n-1}}u &=& g_2& \text{ on }\mathbb{T}\times \mathbb{R}^{n-1},
            \end{array}\right.
\end{equation}
which is a linearization of the model derived in \cite{LiW19} in the absence of surface diffusion.
Here the boundary conditions can be realized as $\opB\Tr^{(4)}u$ with the operator
\begin{align*}
    \opB=
    \begin{pmatrix}
        0 & \Delta_{\R^{n-1}} & 0 & 1 \\
        \partial_t & -\Delta_{\R^{n-1}} & 0 & 0
    \end{pmatrix}.
\end{align*}
As in the previous example, $P_{CH}(\tau, \xi)= \ic \tau +\vert\xi\vert^4$ is $4$-homogeneous of degree 4 and thus admissible by Proposition \ref{prop: admissible_parabolic}.
Our result reads as follows. 
\begin{theorem}\label{thm:LW}
    \begin{enumerate}
    \item\label{thm:LWi} It holds $S=(\partial_t+\Delta^2,\opB\Tr^{(4)})\in \call_{\mathrm{iso}}(\mathbb{E},\mathbb{F}\times\mathbb{G})$, where 
    \begin{align*}
        \mathbb{E}&:= \overline{H}^{(1, 0),p}_\perp(G_+) \cap \overline{H}^{(0,4),p}_\perp(G_+), \\
        \mathbb{F}&:= L^p_\perp(G_+), \quad \mathbb{G}:= \mathbb{G}_1 \times \mathbb{G}_2,\\
        \mathbb{G}_1&:= T^{o_\frac14,p}_{0\perp}(G^{n-1})=BH^{(\frac{1}{4}-\frac{1}{4p}, 0), p}_\perp(G^{n-1}) \cap HB^{(0, {1-\frac{1}{p}}),p}_\perp(G^{n-1}),\\
        \mathbb{G}_2&:=(H^{3o_{\frac{1}{4}}-3o_{\frac{1}{3}}, p}_\perp(G^{n-1}),H^{4o_{\frac{1}{4}}-3o_{\frac{1}{3}},p}_\perp(G^{n-1}))_{1-\frac{1}{p}, p} .
    \end{align*}
    In particular, for all $f\in \mathbb{F}$ and $g=(g_1, g_2) \in \mathbb{G}$, problem (\ref{eqn: ch_lw}) admits a unique solution $u\in \mathbb{E}$ and it holds 
    \begin{equation*}
        \| u\|_\mathbb{E}\simeq \| f\|_{\mathbb{F}}+\| g\|_{\mathbb{G}}. 
    \end{equation*}
    \item\label{thm:LWii} If we keep the notation from part \ref{thm:CHi}, and set 
    \begin{align*}
        \mathbb{E}'&:=\{u\in \mathbb{E}\mid \Tr_0 u \in \mathcal E_0 \},\\
        \|u\|_{\mathbb{E}'} &:= \|u\|_{\mathbb{E}}+\|\Tr_0u\|_{\mathcal E_0},
    \end{align*}
    where 
    \begin{equation*}
        \mathcal E_0:= T_{0\perp}^{o_{\frac{1}{4}}+3o_{\frac{1}{3}}}(G^{n-1})= BH_\perp^{(\frac{5}{4}-\frac{1}{4p}, 0), p}(G^{n-1})\cap HB^{(1, 1-\frac{1}{p}), p}_\perp(G^{n-1})\cap HB^{(0, 4-\frac{1}{p}),p}_\perp(G^{n-1}),
    \end{equation*}
    then for $\mathbb{G}':=\mathbb{G}_1\times\mathbb{G}_2'$ with $\mathbb{G}_2':=\mathbb{G}_1=T^{o_\frac14,p}_{0\perp}(G^{n-1})$, it holds $S\in \call_{\mathrm{iso}}(\mathbb{E}',\mathbb{F}\times\mathbb{G}')$.
     In particular, for all $f\in \mathbb{F}$ and $g=(g_1, g_2) \in \mathbb{G}'$, problem (\ref{eqn: ch_lw}) admits a unique solution $u\in \mathbb{E}'$ and it holds 
    \begin{equation*}
        \| u\|_\mathbb{E} + \|\Tr_0 u\|_{\mathcal E_0} \simeq \| f\|_{\mathbb{F}}+\| g\|_{\mathbb{G}'}. 
    \end{equation*}
\end{enumerate}
    \end{theorem}
    \begin{proof}
    With the same notation as in the proof of Theorem \ref{thm:CH}, the complementary boundary matrix associated to (\ref{eqn: ch_lw}) is given by 
     \begin{equation*}
         \mathfrak{BC}(\tau, \xi'):= \begin{pmatrix}
             0 & -\vert \xi'\vert^2 &0 & 1\\
             \ic \tau & \vert \xi'\vert^2 &  0 &0 \\
          c_0^+(\tau, \xi') &  c_1^+(\tau, \xi')    & -1 &0 \\
          0& c_0^+(\tau, \xi') &  c_1^+(\tau, \xi')    & -1
         \end{pmatrix}
     \end{equation*}
     and we obtain $\det(\mathfrak{BC})(\tau, \xi')= \ic \tau (\vert \xi'\vert^2-c_0^+(\tau, \xi')-{c_1^+}(\tau, \xi')^2)+|\xi'|^2 c_1^+(\tau, \xi')c_0^+(\tau, \xi')$.
     By Lemma \ref{lemma: ellipticity_LW} below, $\det(\mathfrak{BC})$ is strongly $\mathcal{N}$-elliptic with order function $3o_{\frac{1}{3}}+2o_{\frac{1}{4}}$. 

    \begin{enumerate}
        \item Compared to the proof of Theorem \ref{thm:CH}\ref{thm:CHi}, we keep all order functions $t_{j,\mathfrak{k}}$ and $s_{j,\mathfrak{k}}$ unchanged except for
     \begin{equation*}
         \begin{array}{ll}
         s_{2,0}:=  4o_{\frac{1}{4}}-3o_{\frac{1}{3}} & s_{2,1} : = 3o_{\frac{1}{4}}-3o_{\frac{1}{3}},
    \end{array}
     \end{equation*}
     which is motivated by the definition of  $\mathbb{G}_2$.
    Then $t_{j,\mathfrak{k}}-s_{i,\mathfrak{k}}$ is a strong upper order function for $\mathfrak{BC}_{ij}$ for all $i,j\in\{1,2,3,4\}$ and $\mathfrak{k}\in\{0,1\}$, and additionally $\delta=2o_{\frac{1}{4}}+3o_{\frac{1}{3}}= \sum_{l=1}^4 t_{l, \mathfrak{k}}-s_{l, \mathfrak{k}}$.
    Thus, \ref{thm:CHi} follows by the second point of Theorem \ref{thm: main_result}.
        \item Since $\mathbb{E}'$ embeds continuously in $\mathbb{E}$, we immediately get $\partial_t+\Delta^2= \op[P_{CH}] \in \mathcal{L}(\mathbb{E}', \mathbb{F})$ from the previous point. Let us set 
    \begin{equation*}
        \mathcal E':= \Tr^{(4)}(\mathbb{E}')= \mathcal E_0 \times T_{1\perp}^{4o_{\frac{1}{4}}, p}(G^{n-1}) \times  T_{2\perp}^{4o_{\frac{1}{4}}, p}(G^{n-1}) \times  T_{3\perp}^{4o_{\frac{1}{4}}, p}(G^{n-1})
    \end{equation*}
    and 
    \begin{equation*}
        \widetilde{\mathbb{G}}:=\mathbb{G}_1 \times \mathbb{G}_2' \times T_{0\perp}^{\mu_-}(G^{n-1}) \times T_{1\perp}^{\mu_-}(G^{n-1}).
    \end{equation*}
    Setting all the order functions as above except 
     \begin{align*}
         t_{1,0}:=  o_{\frac{1}{4}}+3o_{\frac{1}{3}}, \quad t_{1,1} : =3o_{\frac{1}{3}},
         \quad s_{2, 0}:=  o_{\frac{1}{4}} \quad  s_{2,1} : = 0, 
    \end{align*}
    the fact that $t_{j,\mathfrak{k}}-s_{i,\mathfrak{k}}$ is a strong upper order function for $\mathfrak{BC}_{ij}$ for all $i,j\in\{1,2,3,4\}$ and $\mathfrak{k}\in\{0,1\}$ remains true.
    Writing the spaces as interpolation spaces and arguing as in the proof of Theorem \ref{thm: main_result}, this yields $\op[\mathfrak{BC}] \in \mathcal{L}(\mathcal E', \widetilde{\mathbb{G}})$.
    Considering only the first two rows of $\mathfrak{BC}$, this means $\opB \Tr^{(4)}\in \mathcal{L}(\mathbb{E}', \mathbb{G}')$ and thus $S\in \mathcal{L}(\mathbb{E}', \mathbb{F}\times\mathbb{G}')$.

    \medskip

    On the other hand, we also find that the equality $\delta= 2o_{\frac{1}{4}}+3o_{\frac{1}{3}}= \sum_{l=1}^4 t_{l, \mathfrak{k}}-s_{l, \mathfrak{k}}$ still holds for  $\mathfrak{k}=0$ and $\mathfrak{k}=1$, and so $\mathfrak{BC}$ is a strong mixed-order system, i.e., $\op[\mathfrak{BC}] \in \mathcal{L}_{\mathrm{iso}}(\mathcal E', \widetilde{\mathbb{G}})$.
    Now fix $(f, g)\in \mathbb{F}\times \mathbb{G}'$. Since $\mathbb{G}'\hookrightarrow \mathbb{G}$, the equation (\ref{eqn: ch_lw}) admits a unique solution $u\in \mathbb{E}$ by \ref{thm:CHi} with 
    \begin{equation*}
        \|u\|_{\mathbb{E}}\lesssim \|f\|_{\mathbb{F}}+\|g\|_{\mathbb{G}}\lesssim \|f\|_{\mathbb{F}}+\|g\|_{\mathbb{G}'}.
    \end{equation*}
    Furthermore, using the explicit formula for $S^{-1}$ given in Theorem \ref{thm: main_result}, we get that the solution $u$ must satisfy 
    \begin{equation*}
      \op[\mathfrak{BC}] \Tr^{(4)}u= (g_1, g_2, \Tr_0(\op[P^-]_+^{-1} f), \Tr_1(\op[P^-]^{-1}_+f))\in \widetilde{\mathbb{G}}.
    \end{equation*}
    Therefore, $\op[\mathfrak{BC}] \in \mathcal{L}_{\mathrm{iso}}(\mathcal E', \widetilde{\mathbb{G}})$ yields
    \begin{equation*}
        \|\Tr_0 u \|_{\mathbb{H}_0} \leq \|\Tr^{(4)}u\|_{\mathcal E'}\lesssim\|f\|_{\mathbb{F}}+ \|g\|_{\mathbb{G}'}, 
    \end{equation*}
    where we additionally used the boundedness of the trace operators $\Tr_0\in \mathcal{L}(\overline{H}^{\mu_+, p}_\perp(G_+), \mathbb{G}_3)$, $\Tr_1\in \mathcal{L}(\overline{H}^{\mu_+, p}_\perp(G_+), \mathbb{G}_4)$ and the bulk operator $\op[P^-]_+ \in \mathcal{L}_{\mathrm{iso}}(\overline{H}^{\mu_+, p}_\perp(G_+), \mathbb{F})$, which holds due to Proposition \ref{prop: boundedness_lower_support} since $\op[P^-]$ and $\op[P^-]^{-1}$ preserve lower support.
    This completes the proof.
    \end{enumerate}
    \end{proof}
We recall that
\begin{align*}
P^+(\tau,\xi',\xi_n)
&=(\xi_n-\rho_1^+)(\xi_n-\rho_2^+)
=c_0^++c_1^+(i\xi_n)-(i\xi_n)^2,
\end{align*}
where $\Im\rho_j^+>0$ for $j\in\set{1,2}$.
Thus $c_0^+=\rho_1^+\rho_2^+$ and $c_1^+=\ic(\rho_1^++\rho_2^+)$.
\begin{lemma}\label{lemma: ellipticity_LW}
The symbol
\begin{align*}
	d_{\mathrm{LW}}(\tau,\xi'):=\ic \tau (\vert \xi'\vert^2-c_0^+(\tau, \xi')-{c_1^+}(\tau, \xi')^2)+|\xi'|^2 c_1^+(\tau, \xi')c_0^+(\tau, \xi')
\end{align*}
is strongly $\mathcal N$-elliptic with order function $\mu=2o_{\frac14}+3o_{\frac13}$.
\end{lemma}
\begin{proof}
Since the roots of $P_{\mathrm{CH}}(\tau,\xi',\cdot)$ satisfy $(|\xi'|^2+\rho^2)^2=-\ic\tau$, we have $(\rho_1^+)^2+(\rho_2^+)^2=-2|\xi'|^2$ and $(\rho_1^+\rho_2^+)^2=\ic\tau+|\xi'|^4$.
Consequently, $(c_0^+)^2=\ic\tau+|\xi'|^4$ and $(c_1^+)^2=2(|\xi'|^2-c_0^+)$.
Moreover, $\Re c_1^+=-\Im(\rho_1^++\rho_2^+)<0$.
The branch of the first identity determined by the roots $\rho_1^+,\rho_2^+$ satisfies $\Re c_0^+<0$, since one of the roots is situated in the first quadrant, while the other is situated in the second quadrant.

Thus, by direct calculation of the complemented boundary determinant,
\begin{align*}
d_{\mathrm{LW}}
&=\ic\tau\bigl(r^2-c_0^+-(c_1^+)^2\bigr)+r^2c_0^+c_1^+
=c_1^+\left(r^2c_0^+-\frac{\ic\tau}{2}c_1^+\right).
\end{align*}
Here we used $(c_1^+)^2=2(r^2-c_0^+)$.
We first show that the two terms cannot cancel, i.e., that the angle between $z_1$ and $z_2$ is uniformly bounded away from $\pi$, where $z_1:=|\xi'|^2c_0^+$ and $z_2:=-\frac{\ic\tau}{2}c_1^+$.

Assume first that $\tau>0$.
Since $(c_0^+)^2=\ic\tau+|\xi'|^4$, $\Re c_0^+<0$ and $\Im c_0^+<0$, we have $\arg c_0^+\in[\pi,5\pi/4]$.
Indeed, writing $c_0^+=a+\ic b$, the identity $(c_0^+)^2=|\xi'|^4+\ic\tau$ yields $a^2-b^2=|\xi'|^4\geq0$, and hence $|a|\geq|b|$.
Thus $\arg z_1\in[\pi,5\pi/4]$.

Moreover, $|\xi'|^2-c_0^+$ belongs to the first quadrant and satisfies $\arg(|\xi'|^2-c_0^+)\in[0,\pi/4]$.
Since $(c_1^+)^2=2(|\xi'|^2-c_0^+)$ and $\Re c_1^+<0$, it follows that $\arg c_1^+\in[\pi,9\pi/8]$.
Consequently, $\arg(-\ic c_1^+)\in[\pi/2,5\pi/8]$, and therefore the angle between $z_1$ and $z_2$ is at most $3\pi/4$.

If $\tau<0$, the corresponding picture is obtained by complex conjugation, and the same bound on the angle between $z_1$ and $z_2$ follows.
Hence there exists a constant $c>0$ such that
\begin{align*}
\left|
|\xi'|^2c_0^+ - \frac{\ic\tau}{2}c_1^+\right| \geq c\left(|\xi'|^2|c_0^+| + |\tau||c_1^+|\right),
\end{align*}
see, e.g., \cite[Lemma C.2]{bravin2025well}.
Multiplying by $|c_1^+|$, we conclude that
\begin{align*}
|d_{\mathrm{LW}}| \gtrsim |\xi'|^2|c_0^+||c_1^+| + |\tau||c_1^+|^2.
\end{align*}
Since $c_0^+=\rho_1^+\rho_2^+$, we have $|c_0^+|\gtrsim |\tau|^\frac12+|\xi'|^2$.
Moreover, from $\Re c_0^+<0$ it follows $|c_1^+|^2=2||\xi'|^2-c_0^+|\gtrsim |\tau|^\frac12+|\xi'|^2$ and thus $|c_1^+|\gtrsim |\tau|^\frac14 +|\xi'|$.
Together, this yields for $\lambda>0$
\begin{align*}
|d_{\mathrm{LW}}| \gtrsim |\tau|^\frac32 + |\tau||\xi'|^2 + |\xi'|^5\gtrsim_\lambda \langle\tau\rangle^\frac32 + \langle\tau\rangle\langle\xi'\rangle^2 + \langle\xi'\rangle^5 = w_{2o_\frac14+3o_\frac13}(\tau,\xi'),
\end{align*}
whenever $|\tau|\ge \lambda$.
Since the corresponding strong upper estimates follow as in the previous section, $d_{\mathrm{LW}}$ is strongly $\mathcal N$-elliptic for the order function $2o_\frac14+3o_\frac13$.
\end{proof}
\subsection{The Cahn--Hilliard--Gurtin System with Static and Dynamic Boundary Conditions}\label{sec: CHG}
In \cite{NeS25} we showed maximal regularity for the Cahn--Hilliard--Gurtin system
   \begin{align}\label{eqn: CHG_c_system_1}
        \left\{\begin{array}{rcll}
        \partial_tu_1 -\Delta u_2 &=&f_1 & \text{in } G_+,      \\
        \Delta u_1-\partial_tu_1+u_2 - c\cdot \nabla u_2 &=&f_2  & \text{in } G_+,\\
        \partial_n u_1&=&g_1  & \text{on } \mathbb{T}\times \mathbb{R}^{n-1},\\
        \partial_n u_2&=&g_2 &\text{on } \mathbb{T}\times \mathbb{R}^{n-1}.
    \end{array}\right.
    \end{align}
for $c=0$ and in the Hilbert space case, for which the complexity of the trace theory is significantly reduced.
The theory developed in the present work allows us to treat the general case $p\in(1,\infty)$ and arbitrary anisotropy $c\in \R^n$ of the material.
This system has been treated by Wilke \cite{Wil12} in the setting of initial-value problems by means of a reflection technique, which led Wilke to impose the normal component $c_n$ to vanish.
Our approach does not need to impose such a restriction on $c\in \R^n$.
Moreover, it also naturally treats the Cahn--Hilliard--Gurtin system with dynamic boundary conditions
\begin{align}\label{eqn: CHG_dynamic}
    \left\{
    \begin{array}{rcll}
        \partial_tu_1-\Delta u_2&=&f_1&\text{in }G_+,\\
        \Delta u_1-\partial_tu_1+u_2-c\cdot\nabla u_2&=&f_2&\text{in }G_+,\\
        \partial_tu_1-\Delta_{\mathbb R^{n-1}}u_2-\partial_nu_2&=&g_1&\text{on }\mathbb T\times\mathbb R^{n-1},\\
        (1-c_n)u_2-\partial_tu_1+\Delta_{\mathbb R^{n-1}}u_1+\partial_nu_1-u_1&=&g_2&\text{on }\mathbb T\times\mathbb R^{n-1},
    \end{array}
    \right.
\end{align}
which is a linearized version of the model introduced in \cite{GoM13}.
There, the parameter $c\in\R^n$ is assumed to fulfill $x^2+x\,c\cdot y+|y|^2\gtrsim x^2+|y|^2$ for all $x\in\R$ and $y\in\R^n$.
Since $x^2+x\,c\cdot y+|y|^2 = |y+\frac{x}{2}c|^2 + (1-\frac{|c|^2}{4})x^2$, this condition is equivalent to $|c|<2$ in our setup.
Beyond this restriction, we do not impose any further conditions on $c$.
In particular, as in the static case, the normal component $c_n$ may be non-zero.
System \eqref{eqn: CHG_dynamic} combines two features which were treated separately in the preceding applications: the operator in the interior is a system of genuinely mixed order, while the dynamic boundary conditions give rise to a separate mixed-order structure.
In particular, system \eqref{eqn: CHG_dynamic} falls outside the scope of the Newton-polygon approach to dynamic boundary conditions developed in \cite{denk2008parabolic}, where the interior operator is assumed to be parabolic in the usual sense.
To the best of our knowledge, Theorem \ref{thm: CHG_c_dyn} below is the first maximal regularity result for a genuinely mixed-order system coupled to dynamic boundary conditions with a distinct mixed-order structure.

\medskip

The Cahn--Hilliard--Gurtin system was introduced in \cite{Gur96} and arises naturally in the study of phase separation processes and interface dynamics when microscopic relaxation effects are taken into account.
The model augments the classical Cahn–Hilliard equation by coupling it to an additional evolution equation for the chemical potential, which leads to a symbol that is neither homogeneous nor quasi-homogeneous.
In the dynamic model \eqref{eqn: CHG_dynamic}, the first boundary equation expresses conservation of the total mass in the bulk and on the boundary, while the second one describes the interaction of the chemical potential with the boundary microforces and the surface free energy.

\medskip

The interior operator $\opL$ is the same in \eqref{eqn: CHG_c_system_1} and \eqref{eqn: CHG_dynamic} and given by \eqref{eqn: CHG_symbol} along with its symbol matrix $L(\tau,\xi)$.
    Recall that Example \ref{example: CHG_system} and Theorem \ref{thm: whole_space_systems} yield $\opL\in \call_{\mathrm{iso}}(E,F)$, where $E:=E_1\times E_2$ and $F:=F_1\times F_2$ with
    \begin{align*}
    E_1 &:= H^{(1,1),p}_\perp(G) \cap H^{(0,3),p}_\perp(G), &
    E_2 &:= H^{(0,2),p}_\perp(G),
    \\
    F_1 &:= L^p_\perp(G), &
    F_2 &:= H^{(0,1),p}_\perp(G).
\end{align*}
    In order to use our abstract theory, we would like to show that $\opL$ has an admissible factorization in the sense of Definition \ref{def: admissible}.
    For this, we first investigate its determinant symbol $D$ from Example \ref{example: CHG_ellipticity}.
\begin{lemma}\label{lemma: CHG_c_root_geometry}
    Let $c\in\R^n$.
    For $(\tau,\xi')\in(\R\setminus\set{0})\times\R^{n-1}$, the polynomial $D(\tau,\xi',\cdot)$ has two roots in each open half-plane, counted with multiplicity.
    Moreover, for all $\lambda>0$ the two roots in each half-plane can locally be labelled $\rho_{1}^\pm$ and $\rho_{2}^\pm$ such that $|\rho_{1}^\pm|\simeq_\lambda w_{o_{\frac12}}$ and $|\rho_{2}^\pm|\simeq_\lambda w_{o_0}$ on $|\tau|\ge\lambda$.
    There exists $\varepsilon>0$ such that $\pm\Im\rho_{j}^\pm\ge\varepsilon|\rho_{j}^\pm|$ for $j\in\set{1,2}$.
    The upper constants and $\varepsilon$ may be chosen independently of $\lambda$, whereas the lower constants may depend on $\lambda$.
\end{lemma}
\begin{proof}
    Since $D$ has no real roots for $\tau\neq0$ by the $\caln$-ellipticity in Example \ref{example: CHG_ellipticity}, and the roots are continuous in $c\in\R^n$ by \cite{Coo08}, the number of roots in either half-plane is independent of $c\in \R^n$.
    For $c=0$, $D(\tau,\xi',\cdot)$ is biquadratic and thus has two roots $\rho_j^\pm(\tau,\xi')$ in each half-plane.
    Hence the same is true for every $c\in\R^n$.

    It remains to establish the uniform size and sector estimates.
    Suppose first that $(\tau_k,\xi_k')$ is a sequence with $w_{o_{\frac12}}(\tau_k,\xi_k')\to\infty$ as $k\to\infty$.
    For brevity, we write $R_k:=w_{o_{\frac12}}(\tau_k,\xi_k')$, $D_k:=D(\tau_k,\xi_k',\cdot)$ and $\rho_{j,k}^\pm:=\rho_j^\pm(\tau_k,\xi_k')$.
    After passing to a subsequence, we may assume that $\langle\xi_k'\rangle/R_k\to\alpha$ and $\tau_k/R_k^2\to\beta$ as $k\to\infty$, where $\alpha\geq0$ and $\beta\in\R$ satisfy $\alpha+|\beta|^{1/2}=1$.
    Thus $P_k\to P$ locally uniformly as $k\to\infty$, where
    \begin{align*}
        P_k(z):=R_k^{-4}D_k(R_kz) \quad \text{and} \quad P(z):=(z^2+\alpha^2)(z^2+\alpha^2+\ic\beta).
    \end{align*}
    We use that Rouché's theorem implies that the roots of $P_k$, counted with multiplicity, converge to the roots of $P$.
    The roots of $z^2+\alpha^2+\ic\beta$ have modulus bounded from above and below uniformly under the constraint $\alpha+|\beta|^{1/2}=1$, and their upper and lower roots belong to the fixed closed sector $\set{z\in \C\mid \pm\Im z\ge |\Re z|}$.
    Thus, one pair of roots of $D_k$ satisfies $|\rho_{1,k}^\pm|\simeq R_k=w_{o_\frac12}(\tau_k,\xi'_k)$ and $\pm\Im\rho_{1,k}^\pm\gtrsim|\rho_{1,k}^\pm|$.

    Suppose now that $\alpha>0$.
    Then the roots $\pm\ic\alpha$ of the first factor of $P$ have modulus comparable to one and lie on the imaginary axis.
    Rouché's theorem therefore shows that the remaining roots of $D_k$ have modulus comparable to $R_k$ and satisfy the same uniform sector estimate.
    Since $\alpha>0$ also implies $\langle\xi_k'\rangle\simeq R_k$, this gives $|\rho_{2,k}^\pm|\simeq w_{o_0}(\tau_k,\xi_k')$.

    It remains to consider $\alpha=0$, in which case the second pair of roots collapses to zero under the preceding normalization.
    Set $\eta_k:=\xi_k'\langle\xi'_k\rangle^{-1}$.
    Then $\langle\xi'_k\rangle^{-1},|\eta_k|\in [0,1]$, while $\langle\xi'_k\rangle/R_k\to0$ and hence $\langle\xi'_k\rangle^2/|\tau_k|\to0$.
    After passing to a subsequence, we may assume that $\langle\xi'_k\rangle^{-1}\to\delta\in[0,1]$ and $\eta_k\to\eta$ with $|\eta|^2+\delta^2=1$.
    A direct calculation gives
    \begin{align*}
        \frac{D_k(\langle\xi'_k\rangle z)}{\tau_k\langle\xi'_k\rangle^2}\to P_{\delta,\eta}(z):=\ic(z^2+|\eta|^2+\delta^2)+\delta(c_nz+c'\cdot\eta)
    \end{align*}
    locally uniformly in $z$.
    Since $P_{\delta,\eta}$ clearly has no real roots, and for $c=0$ precisely one root in each open half-plane, the same holds for every $c\in\R^n$.
    The admissible set of parameters $(\delta,\eta)$ is compact, and hence its upper and lower roots remain bounded away from the real axis and have uniformly bounded modulus.
    Rouché's theorem now gives $|\rho_{2,k}^\pm|\simeq \langle\xi'_k\rangle$ and $\pm\Im\rho_{2,k}^\pm\gtrsim|\rho_{2,k}^\pm|$.
    This proves the asserted estimates at high frequencies.

    We next consider bounded sequences of $(\tau_k,\xi_k')$ in order to obtain upper and sector estimates independently of $\lambda$.
    If $\tau_k\to0$ while $\xi_k'\to\xi'\neq0$, then $D_{k}$ converges locally uniformly to $(z^2+|\xi'|^2)^2$, and Rouché's theorem gives uniformly bounded roots contained in fixed sectors around the imaginary axis.
    It remains to consider $(\tau_k,\xi_k')\to(0,0)$.
    Set $L_k:=(|\tau_k|+|\xi_k'|^4)^{1/4}$ and, after passing to a subsequence, assume that $\tau_k/L_k^4\to\theta$ and $\xi_k'/L_k\to\eta$, where $|\theta|+|\eta|^4=1$.
    Then
    \begin{align*}
        L_k^{-4}D_{k}(L_kz)\to (z^2+|\eta|^2)^2+\ic\theta
    \end{align*}
    locally uniformly in $z$.
    The roots of this limiting polynomial are nonzero and remain in fixed closed sectors contained in the respective open half-planes, uniformly for $|\theta|+|\eta|^4=1$.
    Rouché's theorem therefore gives $|\rho_{j,k}^\pm|\lesssim L_k$ and $\pm\Im\rho_{j,k}^\pm\gtrsim|\rho_{j,k}^\pm|$.
    Since $w_{o_{\frac12}}$ and $w_{o_0}$ are bounded from below by one, this proves the asserted upper estimates and the sector estimate also near $(\tau,\xi')=(0,0)$.
    Together with compactness on bounded subsets separated from $(0,0)$, the upper constants and the sector constant may therefore be chosen independently of $\lambda$.

    Finally, fix $\lambda>0$.
    If either lower estimate failed on $|\tau|\geq\lambda$, a corresponding sequence either satisfies $w_{o_{\frac12}}(\tau_k,\xi_k')\to\infty$, contradicting the high-frequency estimates above, or remains in a compact subset of $\{|\tau|\geq\lambda\}$.
    In the latter case no root can vanish, since zero is real and $D$ has no real roots.
    Compactness therefore gives the required lower estimates, with constants which may depend on $\lambda$.
\end{proof}
In view of Lemma \ref{lemma: CHG_c_root_geometry}, it is natural to introduce $D^\pm:=(\xi_n-\rho_{1}^\pm)(\xi_n-\rho_{2}^\pm)$.
We write
    \begin{align*}
        D^\pm(\tau,\xi',\xi_n) =: -(\ic\xi_n)^2 + d_{1}^\pm(\tau,\xi')(\ic\xi_n) + d_{0}^\pm(\tau,\xi').
    \end{align*}
We note that $D(-\tau,-\xi',-\overline z)=\overline{D(\tau,\xi',z)}$.
Hence the upper and lower root sets satisfy $R^\pm(-\tau,-\xi')=-\overline{R^\pm(\tau,\xi')}$, and consequently $d_j^\pm(-\tau,-\xi')=\overline{d_j^\pm(\tau,\xi')}$ for $j\in\{0,1\}$.
We will use this symmetry below without further comment.
For later reference, we record that for $u\in \overline{\cals}'_\perp(G_+)$ we have that the constraint $\op[D^+]_+u=0$ implies
\begin{align}\label{eq: d0_d1_useful}
\begin{split}
 \partial_n^2 u&=\op[d_0^+]u+\op[d_1^+]\partial_n u, \\
 \partial_n^3 u&=\op[d_1^+]\partial_n^2 u + \op[d_0^+]\partial_n u=\op[d_0^+d_1^+]u+\op[d_0^++(d_1^+)^2]\partial_nu.
\end{split}
\end{align}
\begin{lemma}\label{lemma: CHG_c_factor_symbols}
    The symbols $d_{1}^\pm$ and $d_{0}^\pm$ are smooth on $(\R\setminus\set{0})\times\R^{n-1}$.
    Moreover, $o_{\frac12}$ is a strong upper order function for $d_{1}^\pm$, and $o_{\frac12}+o_0$ is a strong upper order function for both $d_{0}^\pm$ and $D^\pm$.
\end{lemma}
\begin{proof}
    We first note that the coefficients of $D^{\pm}$ are smooth by the classical argument in Chapter 4.4 in \cite{tanabe2017functional}.
It remains to prove the strong upper estimates.
Fix $\varepsilon\in(0,1)$ sufficiently small and set $R:=w_{o_{\frac12}}$ and $S:=w_{o_0}$.
We first consider $M:=\{(\tau,\xi')\in\R\times\R^{n-1}:S>\varepsilon R\}$.
Since $S\le R$, it holds $S\simeq R$ on $M$.
Define $\mathsf S(\tau,\xi'):=(\tau/R(\tau,\xi')^2,\xi'/R(\tau,\xi'),1/R(\tau,\xi'))$.
Moreover, define $P: \R\times\R^{n-1}\times[0,1]\times\C\to \C$ via
\begin{align*}
    P(\theta,\eta,\delta,\zeta):=(\zeta^2+|\eta|^2)^2+\ic\theta(\zeta^2+|\eta|^2+\delta^2)+\delta\theta(c_n\zeta+c'\cdot\eta).
\end{align*}
A direct calculation gives
\begin{align*}
    P(\mathsf S(\tau,\xi'),z)=R(\tau,\xi')^{-4}D(\tau,\xi',R(\tau,\xi')z).
\end{align*}
Observe that $\mathsf S(\tau,\xi')$ ranges, for $(\tau,\xi')\in M$, in a bounded subset of $\R\times\R^{n-1}\times(0,1]$.
We denote its closure by $K_M$.
If $\rho_{j}^\pm(\tau,\xi')$ is a root of $D^\pm(\tau,\xi',\cdot)$, then $\rho_{j}^\pm(\tau,\xi')/R(\tau,\xi')$ is a root of $P(\mathsf{S}(\tau,\xi'),\cdot)$.
By Lemma~\ref{lemma: CHG_c_root_geometry} and $R\simeq S$ on $M$, we have $|\rho_{j}^\pm|\simeq R$ and $\pm\Im\rho_{j}^\pm\gtrsim R$.
Hence the normalized roots $\rho_{j}^\pm/R$ remain in a fixed compact subset of the respective open half-plane.

Consequently, for every point of $K_M$ one may choose a contour in the corresponding open half-plane which encloses precisely the two roots associated with $D^\pm$.
By compactness, finitely many such neighborhoods and contours suffice.
Applying the contour construction from Chapter 4.4 in \cite{tanabe2017functional} to the polynomial $P$ shows that there are smooth functions $\widetilde d_{0}^\pm$ and $\widetilde d_{1}^\pm$ on a neighborhood of $K_M$, with all derivatives bounded, such that $d_{1}^\pm/R=\widetilde d_{1}^\pm\circ\mathsf S$ and $d_{0}^\pm/R^2=\widetilde d_{0}^\pm\circ\mathsf S$.
Since $R\simeq S$ on $M$, the same conclusion holds with $d_{0}^\pm/(RS)$ in place of $d_{0}^\pm/R^2$.

We next consider $N:=\{(\tau,\xi')\in\R\times\R^{n-1}:S<2\varepsilon R\}$.
For $\varepsilon$ sufficiently small, Lemma~\ref{lemma: CHG_c_root_geometry} gives, after relabelling, $|\rho_{1}^\pm|\simeq R$ and $|\rho_{2}^\pm|\simeq S$.
Since $S/R<2\varepsilon$ on $N$, the two roots have uniformly separated moduli and therefore cannot meet.

We first consider $\rho_{1}^\pm$ and use the rescaling map $\mathsf S$ introduced above.
The normalized root $\rho_{1}^\pm/R$ is a root of $P(\mathsf S(\tau,\xi'),\cdot)$.
By Lemma~\ref{lemma: CHG_c_root_geometry}, it satisfies $|\rho_{1}^\pm|/R\simeq1$ and $\pm\Im\rho_{1}^\pm/R\gtrsim1$.
Moreover, the second root in the same half-plane has modulus comparable to $S/R$, and hence is uniformly separated from $\rho_{1}^\pm/R$ if $\varepsilon$ is sufficiently small.
Thus $\rho_{1}^\pm/R$ is a uniformly simple root.
Since $\mathsf S(N)$ is bounded, the implicit function theorem and a finite covering argument show that there is a smooth function $\widetilde\rho_{1}^\pm$, with all derivatives bounded on a neighborhood of the closure of $\mathsf S(N)$, such that $\rho_{1}^\pm/R=\widetilde\rho_{1}^\pm\circ\mathsf S$.

For $\rho_{2}^\pm$, define $\mathsf T(\tau,\xi'):=(S(\tau,\xi')^2/\tau,\xi'/S(\tau,\xi'),1/S(\tau,\xi'))$ and define $Q:\R\times\R^{n-1}\times[0,1]\times\C\to\C$ by
\begin{align*}
    Q(\kappa,y,\sigma,\zeta):=\kappa(\zeta^2+|y|^2)^2+\ic(\zeta^2+|y|^2+\sigma^2)+\sigma(c_n\zeta+c'\cdot y).
\end{align*}
A direct calculation gives
\begin{align*}
    Q(\mathsf T(\tau,\xi'),z)=\frac{1}{\tau S(\tau,\xi')^2}D(\tau,\xi',S(\tau,\xi')z).
\end{align*}
The map $\mathsf T$ has bounded image on $N$.
Indeed, $\xi'/S$ and $1/S$ are bounded by definition of $S$, while $S<2\varepsilon R$ implies $S\lesssim\varepsilon\langle\tau\rangle^{1/2}$ and therefore $|S^2/\tau|\lesssim\varepsilon^2$ for $\varepsilon$ sufficiently small.
We denote the closure of $\mathsf T(N)$ by $K_N$.

The normalized root $\rho_{2}^\pm/S$ is a root of $Q(\mathsf T(\tau,\xi'),\cdot)$.
By Lemma~\ref{lemma: CHG_c_root_geometry}, it satisfies $|\rho_{2}^\pm|/S\simeq1$ and $\pm\Im\rho_{2}^\pm/S\gtrsim1$.
The other root in the same half-plane has modulus comparable to $R/S$ and is therefore uniformly separated from $\rho_{2}^\pm/S$.
Thus $\rho_{2}^\pm/S$ is a uniformly simple root.
The implicit function theorem and compactness of $K_N$ consequently give a smooth function $\widetilde\rho_{2}^\pm$, with all derivatives bounded on a neighborhood of $K_N$, such that $\rho_{2}^\pm/S=\widetilde\rho_{2}^\pm\circ\mathsf T$.

The chain rule and the symbol estimates for $R$ and $S$ now show that $o_{\frac12}$ and $o_0$ are strong upper order functions for $\rho_{1}^\pm$ and $\rho_{2}^\pm$, respectively.
Since $d_{1}^\pm=\ic(\rho_{1}^\pm+\rho_{2}^\pm)$ and $d_{0}^\pm=\rho_{1}^\pm\rho_{2}^\pm$, Proposition \ref{prop: arithmetics_order_functions} shows that $o_{\frac12}$ is a strong upper order function for $d_{1}^\pm$ and $o_{\frac12}+o_0$ is a strong upper order function for $d_{0}^\pm$ on $N$.

Finally, $D^\pm=-(\ic\xi_n)^2+d_{1}^\pm(\ic\xi_n)+d_{0}^\pm$, and each of the three terms has strong upper order function $o_{\frac12}+o_0$.
Hence the same is true for $D^\pm$ by Proposition \ref{prop: arithmetics_order_functions}.
\end{proof}

\begin{corollary}\label{cor: CHG_c_factor_ellipticity}
    The symbols $d_{1}^\pm$ and $d_{0}^\pm$ are strongly $\mathcal N$-elliptic with order functions $o_{\frac12}$ and $o_{\frac12}+o_0$, respectively.
    Moreover, $D^\pm$ is strongly $\mathcal N$-elliptic with order function
    $o_{\frac12}+o_0$.
\end{corollary}
\begin{proof}
    Lemma \ref{lemma: CHG_c_factor_symbols} gives the strong upper estimates.
    For every $\lambda>0$, the lower estimate for $d_{0}^\pm=\rho_1^\pm\rho_2^\pm$ for $|\tau|\ge \lambda$ follows from Lemma \ref{lemma: CHG_c_root_geometry}.
    The lower estimates for $d_{1}^\pm=\ic(\rho_1^\pm+\rho_2^\pm)$ and $D^\pm=(\xi_n-\rho_1^\pm)(\xi_n-\rho_2^\pm)$ follow, since the angle between $\rho_1^\pm$ and $\rho_2^\pm$ is at most $\frac\pi 2$, and the angle between  $\xi_n$ and $-\rho_j^\pm$ is at most $\frac{3\pi} 4$ by Lemma \ref{lemma: CHG_c_root_geometry}, and hence $|d_{1}^\pm|\gtrsim |\rho_1^\pm|+|\rho_2^\pm|$ and $|\xi_n-\rho_j^\pm|\gtrsim |\xi_n|+|\rho_j^\pm|$, $j\in \set{1,2}$, see e.g.\ \cite[Lemma C.2]{bravin2025well}.
    Proposition \ref{prop: arithmetics_order_functions} completes the proof.
\end{proof}

We now demonstrate that $\opL$ is admissible in the sense of Definition \ref{def: admissible}.
Observe that $\partial_t$ is invertible on the purely oscillatory spaces, so that we may define
\begin{align*}
    \opA^- &:=
    \begin{pmatrix}
        1&0\\
        -\partial_t^{-1}(\partial_t-\Delta)&\partial_t^{-1}\op[D^-]
    \end{pmatrix},
    \qquad \opA^+ :=
    \begin{pmatrix}
        \partial_t&-\Delta\\
        0&\op[D^+]
    \end{pmatrix}.
\end{align*}
Since $\op[D^-]\op[D^+]=\op[D]$, a short computation implies $\opL=\opA^-\opA^+$.
The operator $\partial_t^{-1}$ acts only in the tangential direction and hence preserves half space support.
Together with the support properties of $\op[D^\pm]$ and their inverses, this shows that $\opA^\pm$ satisfy the support requirements of Definition \ref{def: admissible}.
Consider
\begin{equation}\label{eq: CHG_spaces}
\begin{aligned}
    \mathbb E_1
    &:=\overline{E}_{1+}=
    \overline H^{(1,1),p}_\perp(G_+)
    \cap
    \overline H^{(0,3),p}_\perp(G_+),
    &\qquad
    \mathbb E_2
    &:=\overline{E}_{2+}=
    \overline H^{(0,2),p}_\perp(G_+),
    \\
    \mathbb F_1
    &:=\overline{F}_{1+}=
    L^p_\perp(G_+),
    &\qquad
    \mathbb F_2
    &:=\overline{F}_{2+}=
    \overline H^{(0,1),p}_\perp(G_+).
\end{aligned}
\end{equation}
We write $\mathbb E:=\mathbb E_1\times\mathbb E_2$ and $\mathbb F:=\mathbb F_1\times\mathbb F_2$.
Let $H:=\opA^+ E$ be the intermediate space associated with the factorization $\opL=\opA^-\opA^+$.
Then $\opA^+\in \call_{\mathrm{iso}}(E,H)$ by definition, and since $\opL\in \call_{\mathrm{iso}}(E,F)$ as observed above, we have that $\opA^-=\opL(\opA^+)^{-1}: H\to F$ is an isomorphism as well.
For $\mathbb H:=\overline{H}_+$, Proposition \ref{prop: boundedness_lower_support} yields $\opA^+_+\in \call(\mathbb E,\mathbb H)$ as well as $\dot{\opA}^+_+\in \call_{\mathrm{iso}}(r^+\dot E,r^+ \dot H)$ and $\opA^-_+\in \call_{\mathrm{iso}}(\mathbb H,\mathbb F)$.
Moreover, $\mathbb H=\opA_+^+\mathbb E$ topologically, since $r^+\opA^+=\opA^+_+r^+$.

Since $\mathbb E$ is the product of the Newton polygon spaces $\mathbb E_1$ and $\mathbb E_2$, the trace space $\mathcal E:=\Tr_{\mathbb E}\mathbb E$ is characterized by Theorem \ref{thm: trace_space}.
In contrast, $\mathbb H$ is not a product of Newton polygon spaces, but couples both variables.
We must therefore construct the corresponding trace structure for $\mathbb H$.
For this purpose, we equip
\begin{align*}
    \mathcal K := \Tr_{\mathbb E}\bigl(\ker \opA_+^+\bigr) \subseteq\mathcal E
\end{align*}
with the subspace topology of $\mathcal{E}$.
Note that $\Tr_{\mathbb E}|_{\ker \opA_+^+}$ is injective, since $\dot \opA_+^+$ is invertible on $r^+\dot E_+$.
Thus, $\opA_+^+\Ext_{\mathbb E}|_\calk=\opA_+^+(\Ext_{\mathbb E}|_\calk-(\Tr_{\mathbb E}|_{\ker\opA_+^+})^{-1})\in \call(\mathcal K,r^+\dot H_+)$.
In particular
\begin{align*}
    (\Tr_{\mathbb E}|_{\ker \opA_+^+})^{-1}=\Ext_{\mathbb E}-(\dot \opA_+^+)^{-1}\opA_+^+\Ext_{\mathbb E}\in \call(\mathcal K,\ker \opA_+^+),   
\end{align*}
and hence $\Tr_{\mathbb E}\in \call_{\mathrm{iso}}(\ker \opA_+^+,\mathcal K)$.
Consequently, since $\ker\opA^+_+$ is complete, so is $\mathcal K$.

We define $\mathcal H := \mathcal E/\mathcal K$ and let $q_{\mathcal K}:\mathcal E\to \mathcal H$ be the canonical quotient map.
Since $q_{\mathcal K}\Tr_{\mathbb E}\in\mathcal L(\mathbb E,\mathcal H)$ vanishes on $\ker\opA_+^+$, there is a unique $\mathsf V\in \call(\mathbb H,\mathcal H)$ with $\mathsf V \opA_+^+=q_{\mathcal K}\Tr_{\mathbb E}$ by the universal property of quotients.
We claim that $\ker\mathsf V=r^+\dot H$.
Indeed, if $h\in r^+\dot H$, then $\opA_+^+u=h$ for $u:=(\dot \opA_+^+)^{-1} h\in r^+\dot E$, and therefore $\mathsf Vh=0$.
Conversely, let $h=\opA_+^+u\in \ker\mathsf V$ with $u\in \mathbb E$.
Then $q_{\mathcal K}\Tr_{\mathbb E}u=\mathsf V\opA_+^+u=0$, and hence $\Tr_{\mathbb E}u\in\mathcal K$.
Thus, there exists $v\in\ker \opA_+^+$ with $\Tr_{\mathbb E}v=\Tr_{\mathbb E}u$.
Consequently, $u-v\in r^+\dot E$ and $h=\opA_+^+(u-v)=\dot \opA_+^+(u-v)\in r^+\dot H$.
In particular, for the corresponding complementary operator we have
\begin{align*}
  \mathsf C=\mathsf V \opA_+^+\Ext_{\mathbb E}=q_{\mathcal K}.   
\end{align*}
Having identified the intermediate trace space, we are left with analyzing the boundary operators on this space.
We first do this for the static boundary conditions, whose kernel restriction will subsequently serve as the reference isomorphism for the dynamic problem.

\subsubsection*{The Cahn--Hilliard--Gurtin System with Static Boundary Conditions}
In this section we consider the well-posedness of the system \eqref{eqn: CHG_c_system_1}.
Theorem \ref{thm: trace_space} shows that $\opB\in \call(\cale, \mathbb G)$, where $\mathbb G:=\mathbb G_1\times\mathbb G_2$ with
\begin{align}\label{eq: CHG_stat_bdry_spc}
    \mathbb G_1
    &:=
    BH_\perp^{(1-\frac1{2p},0),p}(G^{n-1})
    \cap
    HB_\perp^{(0,2-\frac1p),p}(G^{n-1}),
    &
    \mathbb G_2
    &:=
    HB_\perp^{(0,1-\frac1p),p}(G^{n-1}),
\end{align}
and where
$\opB\Tr_{\mathbb E}u:=(\Tr_1u_1,\Tr_1u_2)$ realizes the boundary conditions.
With the notation of the preceding construction of the trace structure of $\mathbb H$, we therefore have $(\opB,\mathsf C) =(\opB,q_{\mathcal K})\in \call(\mathcal E,\mathbb G\times\mathcal H)$ for the complemented boundary operator.
We now establish that  $(\opB,\mathsf C)\in\call_{\mathrm{iso}}(\mathcal E,\mathbb G\times \mathcal H)$, which is equivalent to $\opB|_{\mathcal K}\in\call_{\mathrm{iso}}(\mathcal K,\mathbb G)$.
To this end, we need two preparational lemmas.
\begin{lemma}\label{lemma: CHG_c_u2_boundary}
    Let $\widetilde{\mathbb G}_1:=\bigl(H_\perp^{-2o_{\frac12},p}(G^{n-1}),H_\perp^{-o_{\frac12},p}(G^{n-1})\bigr)_{1-\frac1p,p}$ and define the symbol matrix
    \begin{align}\label{lemma: CHG_c_u2_boundary_e1}
        \mathfrak R:=
        \begin{pmatrix}
            d_0^+d_1^+ & d_0^++(d_1^+)^2-|\xi'|^2 \\
            0 & 1
        \end{pmatrix}
    \end{align}
    Then $\mathfrak R\Tr_{\mathbb E_2}\in\call_{\mathrm{iso}}(\mathbb E_2\cap\ker\op[D^+]_+,\widetilde{\mathbb G}_1\times\mathbb G_2)$, where $\mathbb E_2\cap\ker\op[D^+]_+$ is equipped with the subspace topology of $\mathbb E_2$.
\end{lemma}
\begin{proof}
    With the support-preserving weight functions from Proposition \ref{prop:weight_fct_supp}, set
    \begin{align*}
        \omega(\tau,\xi',\xi_n):=\frac{\omega_{o_\frac12}^-(\tau,\xi',\xi_n)}{\omega_{o_0}^-(\tau,\xi',\xi_n)}=\frac{w_{o_\frac12}(\tau,\xi')-\ic\xi_n}{w_{o_0}(\tau,\xi')-\ic\xi_n}=\frac{\langle\tau\rangle^{\frac12}+\langle\xi'\rangle-\ic\xi_n}{\langle\xi'\rangle-\ic\xi_n}.
    \end{align*}
    The support-preserving lifting property gives $\op[\omega]_+\in\call_{\mathrm{iso}}(\overline H_\perp^{o_{\frac12}+o_0,p}(G_+),\mathbb E_2)$.
    Since $\op[\omega]_+$ commutes with $\op[D^+]_+$, it restricts to $\op[\omega]_+\in\call_{\mathrm{iso}}(\overline H_\perp^{o_{\frac12}+o_0,p}(G_+)\cap\ker\op[D^+]_+,\mathbb E_2\cap\ker\op[D^+]_+)$.

    Let $u\in\mathbb E_2\cap\ker\op[D^+]_+$ and define $w:=\op[\omega_{o_\frac12}^-]_+^{-1}u$ and $v:=\op[\omega_{o_0}^-]_+w$, so that $u=\op[\omega]_+v$.
    By \eqref{eq: d0_d1_useful} and $\Tr_{\mathbb E_2}=(\Tr_1,\Tr_2)$, we have $\Tr_{\mathbb E_2} u=\op[\mathfrak M_{\frac12}]\Tr_{\mathbb E_2}w$ and $\Tr_{\mathbb E_2} v=\op[\mathfrak M_{0}]\Tr_{\mathbb E_2}w$, where
    \begin{align*}
        \mathfrak M_y:=
        \begin{pmatrix}
            w_{o_y}&-1\\
            -d_0^+&w_{o_y}-d_1^+
        \end{pmatrix}, \qquad y\in \set{0,\frac12}.
    \end{align*}
    Note that $\mathfrak m_y(\tau,\xi'):=\det\mathfrak M_y(\tau,\xi')=w_{o_y}(\tau,\xi')^2-w_{o_y}(\tau,\xi')d_1^+(\tau,\xi')-d_0^+(\tau,\xi')=-D^+(\tau,\xi',-\ic w_{o_y}(\tau,\xi'))$.
    By Lemma~\ref{lemma: CHG_c_factor_evaluation}, $\mathfrak m_0$ is strongly $\mathcal N$-elliptic with order function $o_{\frac12}+o_0$.
    Consequently, we deduce that $\Tr_{\mathbb E_2} u=\op[\mathfrak M_{\frac12}\mathfrak M_{0}^{-1}] _{\mathbb E_2} v$ and thus $\op[\mathfrak R]\Tr_{\mathbb E_2}=\op[\mathfrak B_\omega]\Tr_{\mathbb E_2}\circ\op[\omega]_+^{-1}$ on $\mathbb E_2\cap \ker\op[D^+]_+$, where
    \begin{align*}
        \mathfrak B_\omega
        &:=
        \mathfrak R
        \mathfrak M_{\frac12}\mathfrak M_{0}^{-1}
        =\frac1{\mathfrak m_0}
        \begin{pmatrix}
            d_0^+d_1^+\mathfrak m_0+\langle\tau\rangle^{\frac12}d_0^+(d_0^++\langle\xi'\rangle d_1^+-|\xi'|^2)
            &
            \mathfrak R_{12}\mathfrak m_0+\langle\tau\rangle^{\frac12}(\langle\xi'\rangle \mathfrak R_{12}+d_0^+d_1^+)\\
            \langle\tau\rangle^{\frac12}d_0^+
            &
            \mathfrak m_0+\langle\tau\rangle^{\frac12}\langle\xi'\rangle
        \end{pmatrix}.
    \end{align*}
    We choose the column and row orders
    \begin{align*}
        t_{1,0} &= o_{\frac12}+o_0, \qquad & t_{2,0} &= o_{\frac12}, \qquad & s_{1,0} &= -o_{\frac12}, \qquad & s_{2,0} &= o_0,\\
        t_{1,1} &= o_{\frac12}, & t_{2,1} &= 0, & s_{1,1} &= -2o_{\frac12}, & s_{2,1} &= 0.
    \end{align*}
    Since $d_0^+$ and $d_1^+$ have strong upper order functions $o_{\frac12}+o_0$ and $o_{\frac12}$, respectively, and $\mathfrak m_0$ is strongly $\mathcal N$-elliptic with order function $o_\frac12+o_0$, the entries $\mathfrak B_{\omega,ij}$ have strong upper order functions $t_{j,\mathfrak k}-s_{i,\mathfrak k}$ for all $i,j\in\set{1,2}$ and $\mathfrak k\in\set{0,1}$.
    Furthermore, $\det\mathfrak B_\omega=d_0^+d_1^+\mathfrak m_\frac12\mathfrak m_0^{-1}$.
    By Corollary~\ref{cor: CHG_c_factor_ellipticity} and Lemma~\ref{lemma: CHG_c_factor_evaluation} below, this determinant is strongly $\mathcal N$-elliptic with order function $3o_{\frac12}$.
    Since $\sum_{l=1}^2(t_{l,\mathfrak{k}}-s_{l,\mathfrak{k}})=3o_{\frac12}$ for $\mathfrak{k}\in\{0,1\}$, Theorem \ref{thm: main_result} with $N_-=0$ and thus $\mathfrak{B}_\omega\mathfrak{C}=\mathfrak{B}_\omega$ yields
    \begin{align*}
        S_\omega:=(\op[D^+]_+,\op[\mathfrak B_\omega]\Tr_{\mathbb E_2})\in\call_{\mathrm{iso}}\bigl(\overline H_\perp^{o_{\frac12}+o_0,p}(G_+),L^p_\perp(G_+)\times\widetilde{\mathbb G}_1\times\mathbb G_2\bigr),
    \end{align*}
    and in particular $\op[\mathfrak B_\omega]\Tr^{(2)}\in \call_{\mathrm{iso}}(\overline H_\perp^{o_{\frac12}+o_0,p}(G_+)\cap \ker\op[D^+]_+,\widetilde{\mathbb G}_1\times\mathbb G_2)$.
    Since $\op[\mathfrak B_\omega]\Tr_{\mathbb E_2}\circ\op[\omega]_+^{-1}=\op[\mathfrak R]\Tr_{\mathbb E_2}$ on $\mathbb E_2\cap \ker\op[D^+]_+$, this gives the assertion.
\end{proof}
\begin{lemma}\label{lemma: CHG_c_factor_evaluation}
    The symbols $\mathfrak m_0(\tau,\xi'):=-D^+(\tau,\xi',-\ic\langle\xi'\rangle)$ and $\mathfrak m_\frac12(\tau,\xi'):=-D^+(\tau,\xi',-\ic(\langle\tau\rangle^{\frac12}+\langle\xi'\rangle))$ are strongly $\mathcal N$-elliptic with order functions $o_{\frac12}+o_0$ and $2o_{\frac12}$, respectively.
\end{lemma}

\begin{proof}
    Let $\rho_1^+$ and $\rho_2^+$ be the two roots of $D^+(\tau,\xi',\cdot)$.
    By Lemma~\ref{lemma: CHG_c_root_geometry}, for every $\lambda>0$ and $|\tau|\geq\lambda$ we may label them such that
    $|\rho_1^+|\simeq_\lambda\langle\tau\rangle^{\frac12}+\langle\xi'\rangle$ and
    $|\rho_2^+|\simeq_\lambda\langle\xi'\rangle$, while
    $\Im\rho_j^+\geq\varepsilon|\rho_j^+|$ for $j\in\{1,2\}$ with $\varepsilon>0$ independent of $\lambda$.
    Consequently, for every $r>0$,
    $|-\ic r-\rho_j^+|\simeq r+|\rho_j^+|$ uniformly in $r$, since $-\ic r$ lies on the negative imaginary axis whereas $\rho_j^+$ remains in a fixed closed sector of the upper half-plane.

    Since $D^+(\tau,\xi',\xi_n)=(\xi_n-\rho_1^+)(\xi_n-\rho_2^+)$, it follows that
    \begin{align*}
        |D^+(\tau,\xi',-\ic\langle\xi'\rangle)|
        &\simeq_\lambda
        (\langle\xi'\rangle+|\rho_1^+|)(\langle\xi'\rangle+|\rho_2^+|)
        \simeq_\lambda
        \langle\xi'\rangle(\langle\tau\rangle^{\frac12}+\langle\xi'\rangle),\\
        |D^+(\tau,\xi',-\ic(\langle\tau\rangle^{\frac12}+\langle\xi'\rangle))|
        &\simeq_\lambda
        (\langle\tau\rangle^{\frac12}+\langle\xi'\rangle+|\rho_1^+|)
        (\langle\tau\rangle^{\frac12}+\langle\xi'\rangle+|\rho_2^+|)
        \simeq_\lambda
        (\langle\tau\rangle^{\frac12}+\langle\xi'\rangle)^2.
    \end{align*}
    The expressions on the right-hand side are weight functions for
    $o_{\frac12}+o_0$ and $2o_{\frac12}$, respectively, and hence give the required lower estimates.
    For the strong upper estimates, we use
    \begin{align*}
        D^+(\tau,\xi',-\ic\langle\xi'\rangle)
        &=-\langle\xi'\rangle^2+d_1^+\langle\xi'\rangle+d_0^+,\\
        D^+(\tau,\xi',-\ic(\langle\tau\rangle^{\frac12}+\langle\xi'\rangle))
        &=-(\langle\tau\rangle^{\frac12}+\langle\xi'\rangle)^2
        +d_1^+(\langle\tau\rangle^{\frac12}+\langle\xi'\rangle)+d_0^+,
    \end{align*}
    together with Proposition \ref{prop: arithmetics_order_functions} and the fact that by Lemma~\ref{lemma: CHG_c_factor_symbols}, $d_1^+$ and $d_0^+$ have strong upper order functions
    $o_{\frac12}$ and $o_{\frac12}+o_0$, respectively.
\end{proof}
\begin{lemma}\label{lemma: CHG_c_kernel_surjective}
Define the symbol matrix $\mathfrak{R}$ by \eqref{lemma: CHG_c_u2_boundary_e1}.
\begin{enumerate}
    \item\label{lemma: CHG_c_kernel_surjectivei} For $u\in \mathbb E$ with $\opA_+^+u=0$ it holds $\op[\mathfrak J]\opB\Tr_{\mathbb E}=\op[\mathfrak R]\Tr_{\mathbb E_2}u_2$, where $\mathfrak J:=\diag(i\tau,1)$.
    \item\label{lemma: CHG_c_kernel_surjectiveii} It holds $\opB|_{\mathcal K}\in\call_{\mathrm{iso}}(\mathcal K,\mathbb G)$ and $(\opB,\opC)\in\call_{\mathrm{iso}}(\mathcal E,\mathbb G\times\mathcal H)$.
\end{enumerate}
\end{lemma}
\begin{proof}
    Let $u\in \mathbb E$ be in $\ker\opA_+^+$.
    Observe that $\opB\Tr_{\mathbb E}u=(\Tr_1u_1,\Tr_1u_2)$ and thus the second component of the claimed identity is the trivial $\Tr_1 u_2=\Tr_1 u_2$.
    To show the identity also in the first component, note that $\partial_tu_1-\Delta u_2=0$ and $\op[D^+]_+u_2=0$, and so the definition of $\mathfrak R$ and \eqref{eq: d0_d1_useful} gives
    \begin{align*}
        0=\lim_{x_n\to0}\partial_n(\partial_tu_1-\Delta u_2)(\cdot,x_n)=\partial_t\Tr_1u_1-\op[\mathfrak R_{11}]\Tr_0u_2-\op[\mathfrak R_{12}]\Tr_1u_2,
    \end{align*}
    which proves \ref{lemma: CHG_c_kernel_surjectivei}.

    \medskip
    
    If additionally $\opB\Tr_\E u=0$, we obtain $\op[\mathfrak R]\Tr_{\mathbb E_2}u_2=0$ from \ref{lemma: CHG_c_kernel_surjectivei} and hence $u_2=0$ by Lemma \ref{lemma: CHG_c_u2_boundary}.
    Thus $\partial_tu_1=0$, and since $u_1$ is purely oscillatory, $u_1=0$.
    Therefore $\opB|_{\mathcal K}$ is injective.

    \medskip
    
    We next prove surjectivity of $\opB|_{\mathcal K}$.
    Let $g=(g_1,g_2)\in\mathbb G$.
    Since $2o_{\frac12}$ is a strong upper order function for $\partial_t$, the lifting and interpolation properties give $\partial_tg_1\in\widetilde{\mathbb G}_1$, where $\widetilde{\mathbb G}_1$ is as in Lemma~\ref{lemma: CHG_c_u2_boundary}.
    This lemma therefore yields a unique $u_2\in\mathbb E_2$ such that
    \begin{align*}
        \op[D^+]_+u_2=0, \qquad \op[\mathfrak R]\Tr_{\mathbb E_2}u_2=(\partial_tg_1,g_2).
    \end{align*}
    Since $(1-c\cdot\nabla)u_2\in\mathbb F_2$, Lemma~\ref{lemma: heat_Neumann_shifted} yields a unique $u_1\in\mathbb E_1$ satisfying $(\partial_t-\Delta)u_1=(1-c\cdot\nabla)u_2$ and $\Tr_1u_1=g_1$.
    Set $w:=\partial_tu_1-\Delta u_2$.
    Since the differential operators commute and $D=D^-D^+$, we have
    \begin{align*}
        (\partial_t-\Delta)w=\partial_t(1-c\cdot\nabla)u_2-\Delta(\partial_t-\Delta)u_2=\op[D]_+u_2=0.
    \end{align*}
    The definition of $\mathfrak R$ and \eqref{eq: d0_d1_useful} give $\partial_n\Delta u_2=\op[\mathfrak R_{11}]u_2+\op[\mathfrak R_{12}]\partial_nu_2$ in $\overline{\cals}'_\perp(G_+)$.
    Therefore, as $x_n\to0$,
    \begin{align*}
        \partial_nw(\cdot,x_n)\to\partial_t\Tr_1u_1-\partial_t g_1=0
    \end{align*}
    in $\mathcal S'_\perp(G^{n-1})$.
    Lemma~\ref{lemma: caloric_uniqueness} therefore gives $w=0$.
    Hence $\partial_tu_1-\Delta u_2=0$ and $\op[D^+]_+u_2=0$, so $u:=(u_1,u_2)\in\ker\opA_+^+$.
    Thus $\Tr_{\mathbb E}u\in\mathcal K$ and $\opB\Tr_{\mathbb E}u=(g_1,g_2)=g$, which proves surjectivity.

    The estimates in Lemma~\ref{lemma: CHG_c_u2_boundary}, Lemma~\ref{lemma: heat_Neumann_shifted}, and the trace theorem show that the right inverse constructed above is continuous.
    Hence $\opB|_{\mathcal K}\in\call_{\mathrm{iso}}(\mathcal K,\mathbb G)$.
    Since $\opC=q_{\mathcal K}$ is the quotient map by $\mathcal K$, the two assertions $\opB\in \call(\cale,\mathbb G)$ and $\opB|_{\mathcal K}\in\call_{\mathrm{iso}}(\mathcal K,\mathbb G)$ immediately imply  $(\opB,q_{\mathcal K})\in\call_{\mathrm{iso}}(\mathcal E,\mathbb G\times\mathcal H)$, which proves \ref{lemma: CHG_c_kernel_surjectiveii}.
\end{proof}
\begin{lemma}\label{lemma: caloric_uniqueness}
    Let $w\in L^p_\perp(G_+)$ satisfy $(\partial_t-\Delta)w=0$ in $\mathcal D'(G_+)$ and suppose that $\partial_nw(\cdot,x_n)\to 0$ in $\mathcal S'_\perp(G^{n-1})$ as $x_n\to 0$.
    Then $w=0$.
\end{lemma}
\begin{proof}
    By interior regularity, $x_n\mapsto w(\cdot,x_n)$ is smooth on $(0,\infty)$ with values in $\mathcal S'_\perp(G^{n-1})$.
    Since $\partial_nw(\cdot,x_n)$ converges in $\mathcal S'_\perp(G^{n-1})$ as $x_n\to0$, it is bounded there for sufficiently small $x_n$.
    From
    \begin{align*}
        w(\cdot,x_n)-w(\cdot,y_n)=\int_{y_n}^{x_n}\partial_nw(\cdot,r)\,dr
    \end{align*}
    we see that also $w(\cdot,x_n)$ is bounded in $\mathcal S'_\perp(G^{n-1})$ for sufficiently small $x_n$.

Let $W\in L^p_\perp(G)$ be the even extension of $w$.
For $\varphi\in\mathcal D(G)$ set
$\psi(t,x',x_n):=\varphi(t,x',x_n)+\varphi(t,x',-x_n)$ for $x_n>0$.
Since the contribution of $\langle W,(-\partial_t-\Delta)\varphi\rangle$ on $\{|x_n|<\varepsilon\}$ converges to zero as $\varepsilon\to0$, evenness of $W$ gives
\begin{align*}
    \langle(\partial_t-\Delta)W,\varphi\rangle
    =
    \lim_{\varepsilon\to0}
    \int_{G^{n-1}}\int_\varepsilon^\infty
    w(-\partial_t-\Delta)\psi\,dx_n.
\end{align*}
Using $(\partial_t-\Delta)w=0$ in the interior and integrating by parts yields
\begin{align*}
    \langle(\partial_t-\Delta)W,\varphi\rangle
    =
    \lim_{\varepsilon\to0}
    \left(
        \langle w(\cdot,\varepsilon),\partial_n\psi(\cdot,\varepsilon)\rangle
        -
        \langle\partial_nw(\cdot,\varepsilon),\psi(\cdot,\varepsilon)\rangle
    \right).
\end{align*}
Since $w(\cdot,\varepsilon)$ remains bounded and
$\partial_n\psi(\cdot,\varepsilon)\to0$ in $\mathcal S_\perp(G^{n-1})$, the first term vanishes as $\eps\to 0$.
The second term vanishes since
$\partial_nw(\cdot,\varepsilon)\to0$ and
$\psi(\cdot,\varepsilon)\to 2\varphi(\cdot,0)$ in $\mathcal S_\perp(G^{n-1})$.
Thus $(\partial_t-\Delta)W=0$ in $\mathcal D'(G)$.

    Taking the Fourier transform on $G$ gives $(\ic\tau+|\xi|^2)\widehat W(\tau,\xi)=0$.
    Since $W$ is purely oscillatory, $\tau\neq0$ on the support of $\widehat W$, and hence $\widehat W=0$.
    Consequently $w=0$.
\end{proof}
\begin{theorem}\label{thm: CHG_c_stat}
Let $c\in \R^n$, define $\mathbb E$ and $\mathbb F$ as in \eqref{eq: CHG_spaces}, and $\mathbb G$ as in \eqref{eq: CHG_stat_bdry_spc}.
Then $S:=(\opL_+,\opB\Tr_\mathbb E)\in \call_{\mathrm{iso}}(\mathbb E,\mathbb F\times \mathbb G)$.
In particular, for all $f=(f_1,f_2)\in \mathbb F$ and $g=(g_1,g_2)\in \mathbb G$, there is a unique $u=(u_1,u_2)\in \mathbb E$ satisfying \eqref{eqn: CHG_c_system_1}, and it holds
\begin{align*}
\|u\|_{\mathbb E}\simeq \|f\|_{\mathbb F}+\|g\|_{\mathbb G}.
\end{align*}
\end{theorem}
\begin{proof}
Since $\opL=\opA^-\opA^+$ is admissible in the sense of Definition \ref{def: admissible}, and $(\opB,\opC)\in\call_{\mathrm{iso}}(\mathcal E,\mathbb G\times\mathcal H)$ by Lemma \ref{lemma: CHG_c_kernel_surjective}, the result follows from Theorem \ref{thm: abstract_complementing_condition} and Lemma \ref{js013}.
\end{proof}

\subsubsection*{The Cahn--Hilliard--Gurtin System with Dynamic Boundary Conditions}
We now turn our attention to the system \eqref{eqn: CHG_dynamic}, and recall that the physically relevant regime is $|c|<2$.
To realize its boundary conditions in our framework, we recall that a generic $e\in\cale=\Tr_{\mathbb E}\mathbb E$ is of the form $\Tr_{\mathbb E} u = (\Tr_0u_1,\Tr_1u_1,\Tr_2u_1,\Tr_0u_2,\Tr_1u_2)$, and consider
\begin{align}\label{eqn: CHG_dyn_bdry_spc}
    \mathbb G_{\mathrm{dyn}}:=\mathbb G_{1,\mathrm{dyn}}\times \mathbb G_{2,\mathrm{dyn}}:=HB^{(0,-\frac1p),p}_\perp(G^{n-1})\times HB^{(0,1-\frac1p),p}_\perp(G^{n-1})
\end{align}
and $\opB_{\mathrm{dyn}}\in \call(\cale,\mathbb G_{\mathrm{dyn}})$ defined by
\begin{align*}
    \opB_{\mathrm{dyn}}
    :=
    \begin{pmatrix}
        \partial_t & 0 & 0& -\Delta_{\R^{n-1}} & -1
        \\
        -(\partial_t-\Delta_{\R^{n-1}}+1) & 1 & 0& 1-c_n & 0
    \end{pmatrix}.
\end{align*}
Recall that $\mathcal K := \Tr_{\mathbb E}(\ker \opA_+^+)$.
Since $\opC=q_{\mathcal K}$ is unaffected by the boundary conditions, the abstract complementing boundary condition for $\opB_{\mathrm{dyn}}$ is equivalent to the injectivity of $\opB_{\mathrm{dyn}}|_{\mathcal K}$.

We reduce this condition to a scalar determinant.
    To this end, note that comparison of coefficients in $D=D^-D^+$ gives $d_1^-=-d_1^+$ and
    \begin{align}
        d_0^-+d_0^++(d_1^+)^2 &= \ic\tau+2|\xi'|^2, \label{eq: dynamic_c_coeff_1}\\
        d_1^+(d_0^--d_0^+) &= -\ic\tau c_n, \label{eq: dynamic_c_coeff_2}\\
        d_0^+d_0^- &= |\xi'|^4+\tau c'\cdot\xi' +\ic\tau(1+|\xi'|^2). \label{eq: dynamic_c_coeff_3}
    \end{align}
\begin{lemma}\label{lemma: quadrant_CHG_c}
    Assume $|c|<2$ and let $\tau>0$.
    Then $d_{0}^+$, $d_{1}^+$, and $d_{0}^-$ lie in the open third quadrant of the complex plane.
\end{lemma}

\begin{proof}
    We first verify the assertion for $c=0$.
    Then since roots of $D^+$ lie in the upper half-plane and $d_{1}^+=\ic(\rho_{1}^++\rho_{2}^+)$, we have $\Re d_{1}^+<0$ and in particular $d_{1}^+\ne 0$.
    Thus, \eqref{eq: dynamic_c_coeff_2} gives $d_{0}^-=d_{0}^+$, and \eqref{eq: dynamic_c_coeff_3} yields $(d_{0}^+)^2=|\xi'|^4+\ic\tau(1+|\xi'|^2)$.
    Hence $d_{0}^+$ lies either in the open first or in the open third quadrant.
    The coefficients depend continuously on $(\tau,\xi')$ by \cite{Bri66}, and for fixed $\xi'\ne 0$ the two upper roots converge to $\ic|\xi'|$ as $\tau\downarrow0$, so that $d_{0}^+\to-|\xi'|^2$.
    Since $(0,\infty)\times\R^{n-1}$ is connected and $d_{0}^+$ never meets either coordinate axis, it follows that $d_{0}^+=d_{0}^-$ lies in the open third quadrant.
    Finally, \eqref{eq: dynamic_c_coeff_1} gives $(d_{1}^+)^2=\ic\tau+2|\xi'|^2-2d_{0}^+$, whose right-hand side lies in the open first quadrant.
    Together with $\Re d_{1}^+<0$, this shows that $d_{1}^+$ lies in the open third quadrant as well.

    Let now $c\ne 0$.
    Since the coefficients depend continuously on $c$, it remains to exclude that any of the coefficients is located at the negative real or imaginary axis.
    We first consider $d_{1}$.
    It cannot be located on the imaginary axis, since the argument for $\Re d_{1}<0$ remains valid.
    We therefore suppose $d_{1}^+=-u$ with $u>0$.
    From \eqref{eq: dynamic_c_coeff_2} it follows that $\Re d_{0}^-=\Re d_{0}^+=:a$.
    The imaginary part of \eqref{eq: dynamic_c_coeff_3} together with \eqref{eq: dynamic_c_coeff_1} gives $a=1+|\xi'|^2$.
    Taking the real part of the latter identity then yields the contradiction $u^2 = 2|\xi'|^2-2a = -2$.

It remains to consider $d_{0}^+$ and $d_{0}^-$.
Suppose first that $d_{0}^+$ lies on the negative real axis.
Write
\begin{align*}
    d_{0}^+=-x, \qquad d_{1}^+=-(u+\ic v), \qquad d_{0}^-=-(X+\ic Y),
\end{align*}
with $x,u,v>0$ and $X,Y\in\R$.
Taking real and imaginary parts in \eqref{eq: dynamic_c_coeff_1}--\eqref{eq: dynamic_c_coeff_3} gives
\begin{align*}
\begin{array}{c|ccc}
    & \eqref{eq: dynamic_c_coeff_1} & \eqref{eq: dynamic_c_coeff_2} & \eqref{eq: dynamic_c_coeff_3} \\[1mm]\hline
    \rule{0pt}{3ex}\Re
    & u^2-v^2=2|\xi'|^2+X+x
    & u(X-x)=vY
    & xX=|\xi'|^4+\tau c'\cdot\xi' \\[1mm]
    \Im
    & 2uv=\tau+Y
    & -\tau c_n=\dfrac{u^2+v^2}{u}Y
    & xY=\tau(1+|\xi'|^2).
\end{array}
\end{align*}
The imaginary part of \eqref{eq: dynamic_c_coeff_3} gives $Y>0$, and then the real part of \eqref{eq: dynamic_c_coeff_2} gives $X>0$.
Moreover, the imaginary part of \eqref{eq: dynamic_c_coeff_2} gives $c_n<0$, and the real part of \eqref{eq: dynamic_c_coeff_3} shows $|\xi'|>0$.

Set $t:=v/u$ and $\gamma:=1-t^2\frac{x+3(1+|\xi'|^2)}{x+1+|\xi'|^2}$.
From the imaginary part of \eqref{eq: dynamic_c_coeff_1}, the real part of \eqref{eq: dynamic_c_coeff_2}, and the imaginary part of \eqref{eq: dynamic_c_coeff_3} we obtain
\begin{align*}
    \tau=\frac{2tu^2x}{x+1+|\xi'|^2}, \qquad Y=\frac{2tu^2(1+|\xi'|^2)}{x+1+|\xi'|^2}, \qquad X=x+\frac{2t^2u^2(1+|\xi'|^2)}{x+1+|\xi'|^2}.
\end{align*}
Substitution into the real part of \eqref{eq: dynamic_c_coeff_1} yields $u^2\gamma=2(|\xi'|^2+x)$, and hence $0<\gamma<1$.
Using these identities in the imaginary part of \eqref{eq: dynamic_c_coeff_2} and the real and imaginary parts of \eqref{eq: dynamic_c_coeff_3} gives
\begin{align*}
    c_n^2=\frac{2(|\xi'|^2+x)(1+|\xi'|^2)^2(1+t^2)^2}{\gamma x^2}, \qquad c'\cdot\xi'=t(1+|\xi'|^2)+\frac{(x-|\xi'|^2)(x+1+|\xi'|^2)\gamma}{4tx}.
\end{align*}
Lemma~\ref{lemma: elementary_CHG_crossing}, applied with $r=|\xi'|$, $x$, $t$, and $\gamma$, gives $c_n^2+(c'\cdot\xi')^2/|\xi'|^2>4$, which contradicts $|c|<2$.
Hence $d_{0}^+$ cannot lie on the negative real axis.

Suppose now that $d_{0}^+$ lies on the negative imaginary axis.
Write
\begin{align*}
    d_{0}^+=-\ic y, \qquad d_{1}^+=-(u+\ic v), \qquad d_{0}^-=-(X+\ic Y),
\end{align*}
with $y,u,v>0$ and $X,Y\in\R$.
Taking real and imaginary parts in \eqref{eq: dynamic_c_coeff_1}--\eqref{eq: dynamic_c_coeff_3} gives
\begin{align*}
\begin{array}{c|ccc}
    & \eqref{eq: dynamic_c_coeff_1} & \eqref{eq: dynamic_c_coeff_2} & \eqref{eq: dynamic_c_coeff_3} \\[1mm]\hline
    \rule{0pt}{3ex}\Re
    & u^2-v^2=2|\xi'|^2+X
    & v(Y-y)=uX
    & -yY=|\xi'|^4+\tau c'\cdot\xi' \\[1mm]
    \Im
    & 2uv=\tau+Y+y
    & -\tau c_n=\dfrac{u^2+v^2}{v}X
    & yX=\tau(1+|\xi'|^2).
\end{array}
\end{align*}
The imaginary part of \eqref{eq: dynamic_c_coeff_3} gives $X>0$, and then the real part of \eqref{eq: dynamic_c_coeff_2} gives $Y>0$.
The real part of \eqref{eq: dynamic_c_coeff_1} yields $u>v$, so that $t:=v/u\in(0,1)$, and the real part of \eqref{eq: dynamic_c_coeff_3} gives $|\xi'|>0$.
Moreover, using the real part of \eqref{eq: dynamic_c_coeff_2} and the imaginary part of \eqref{eq: dynamic_c_coeff_3} in the real part of \eqref{eq: dynamic_c_coeff_3} gives
\begin{align*}
    c'\cdot\xi'=-\frac{1+|\xi'|^2}{t}-\frac{y^2+|\xi'|^4}{\tau}<-\frac{1+|\xi'|^2}{t}.
\end{align*}
Consequently, $|c|^2\geq |c'|^2\geq\frac{(c'\cdot\xi')^2}{|\xi'|^2}>\frac{(1+|\xi'|^2)^2}{|\xi'|^2}\ge 4$, again contradicting $|c|<2$.
Thus $d_{0}^+$ cannot lie on the negative imaginary axis and therefore remains in the open third quadrant.

Finally, interchanging $d_{0}^+$ and $d_{0}^-$ in \eqref{eq: dynamic_c_coeff_1}--\eqref{eq: dynamic_c_coeff_3} leaves the first and third identities unchanged and only replaces $c_n$ by $-c_n$ in the second.
The preceding arguments therefore apply verbatim to $d_{0}^-$.
\end{proof}
\begin{lemma}\label{lemma: elementary_CHG_crossing}
    Let $r>0$, $x>0$, $t>0$, and $\gamma\in(0,1)$.
    Then
    \begin{align*}
        \frac{2(r^2+x)(1+r^2)^2(1+t^2)^2}{\gamma x^2}+\frac{1}{r^2}\left(t(1+r^2)+\frac{(x-r^2)(x+1+r^2)\gamma}{4tx}\right)^2 > 4.
    \end{align*}
\end{lemma}
\begin{proof}
    Set $\alpha:=2(r^2+x)(1+r^2)^2/x^2$.
    If $x\le r^2$, then the first term on the left-hand side is larger than $\alpha$, while
    $\alpha\ge4(1+r^2)^2/r^2\ge16$.
    Hence we may assume $x>r^2$.

    Both summands inside the square are then positive, and the arithmetic-geometric mean inequality shows that the left-hand side is bounded from below by
    $\alpha/\gamma+\beta\gamma$, where
    $\beta:=(1+r^2)(x-r^2)(x+1+r^2)/(xr^2)$.
    If $\alpha\ge4$, this is larger than $4$ since $\gamma<1$.
    We may therefore assume $\alpha<4$.
    Set $z:=x/r^2>1$.
    Then
    \begin{align*}
        \alpha=\frac{2(1+r^2)^2(1+z)}{r^2z^2}\ge\frac{8(1+z)}{z^2}.
    \end{align*}
    Thus $\alpha<4$ implies $z^2>2(1+z)$, and hence $z>2$.
    Moreover,
    \begin{align*}
        \alpha\beta=\frac{2(1+r^2)^3}{r^4}\left(1-\frac{1}{z^2}\right)\left(1+\frac{1+r^2}{r^2z}\right) >2\cdot\frac{27}{4}\cdot\frac34>4,
    \end{align*}
    where we used $(1+r^2)^3/r^4\ge27/4$ and $1-z^{-2}>3/4$.
    Consequently, $\frac{\alpha}{\gamma}+\beta\gamma \ge 2\sqrt{\alpha\beta}>4$, which proves the assertion.
\end{proof}
\begin{lemma}\label{lemma: dynamic_CHG_reduction_c}
    Define the symbol matrices
      \begin{align*}
        \mathfrak T:=
        \begin{pmatrix}
            1&0\\
            \ic\tau+|\xi'|^2+1&\ic\tau
        \end{pmatrix},
        \qquad
        \mathfrak N:=
        \begin{pmatrix}
            d_0^+&d_1^+-1\\
            q_1&-(1+d_0^-)
        \end{pmatrix},
        \qquad
        q_1:=(1+|\xi'|^2)(\ic\tau+|\xi'|^2)+d_1^+d_0^-.
    \end{align*}
    \begin{enumerate}
        \item\label{lemma: dynamic_CHG_reduction_ci} For any $u\in \mathbb E$ with $\opA_+^+u=0$ it holds $\op[\mathfrak T]\opB_{\mathrm{dyn}}\Tr_{\mathbb E}u=\op[\mathfrak N]\Tr_{\mathbb E_2}u_2$.
        \item\label{lemma: dynamic_CHG_reduction_cii} The symbol $q_1$ has strong upper order function $2o_\frac12+2o_0$, and $d:=\det\mathfrak N$ is strongly $\mathcal N$-elliptic with order function $3o_{\frac12}+2o_0$.
    \end{enumerate}
\end{lemma}

\begin{proof}
    Since $\op[D^+]u_2=0$, we may use \eqref{eq: d0_d1_useful} to rewrite the first line $\partial_tu_1-\Delta u_2=0$ of $A_+^+u=0$ as
    \begin{align*}
	\partial_tu_1-(\Delta_{\R^{n-1}}+\op[d_0])u_2-\op[d_1]\partial_nu_2&=0, \\
	\partial_t\partial_n u_1-\op[d_0d_1]u_2 - (\Delta_{\R^{n-1}}+\op[d_0]+\op[d_1]^2)\partial_nu_2&=0.
    \end{align*}
    Sending $x_n\to 0$, we obtain
    \begin{align*}
        \partial_t\Tr_0u_1-(\Delta_{\R^{n-1}}+\op[d_0])\Tr_0u_2-\op[d_1]\Tr_1u_2&=0,\\
        \partial_t\Tr_1u_1-\op[d_0d_1]\Tr_0u_2-(\Delta_{\R^{n-1}}+\op[d_0]+\op[d_1]^2)\Tr_1u_2&=0.
    \end{align*}
    The first component of $\opB_{\mathrm{dyn}}e=(g_1,g_2)$ therefore gives
    $g_1=\op[d_0^+]\Tr_0u_2+(\op[d_1^+]-1)\Tr_1u_2$.
    From the definition of $\opB_{\mathrm{dyn}}$ we obtain with $A:=\partial_t-\Delta_{\R^{n-1}}+1$
    \begin{align*}
        \partial_tg_2+Ag_1 = \partial_t\Tr_1u_1+\partial_t(1-c_n)\Tr_0u_2 - A\Delta_{\R^{n-1}}\Tr_0u_2-A\Tr_1u_2,
    \end{align*}
    since the terms containing $\Tr_0u_1$ cancel.
    Using the second of the preceding trace relations, this becomes
    \begin{align*}
        \partial_tg_2+Ag_1 = \bigl(\op[d_0^+d_1^+]+\partial_t(1-c_n)-A\Delta_{\R^{n-1}}\bigr)\Tr_0u_2 + \bigl(\Delta_{\R^{n-1}}+\op[d_0^+]+\op[d_1^+]^2-A\bigr)\Tr_1u_2.
    \end{align*}
    The identities \eqref{eq: dynamic_c_coeff_1} and \eqref{eq: dynamic_c_coeff_2} reduce the two coefficients to $(1-\Delta_{\R^{n-1}})(\partial_t-\Delta_{\R^{n-1}})+\op[d_1^+d_0^-]$ and $-(1+\op[d_0^-])$, respectively.
    This completes the proof of \ref{lemma: dynamic_CHG_reduction_ci}.

    \medskip

    For \ref{lemma: dynamic_CHG_reduction_cii}, the strong upper estimates for $q_1(\tau,\xi')$ and $d(\tau,\xi')=(1-d_1^+)q_1-d_0^+(1+d_0^-)$ follow from Corollary~\ref{cor: CHG_c_factor_ellipticity} and Proposition~\ref{prop: arithmetics_order_functions}.
    For the lower estimate, assume first that $|\tau|+|\xi'|\ge R$ for a sufficiently large $R>0$.
    On that set we show
    \begin{align}\label{lemma: dynamic_CHG_reduction_c_e1}
        |q_1(\tau,\xi')|\gtrsim (\langle\tau\rangle+\langle\xi'\rangle^2)\langle\xi'\rangle^2
    \end{align}
    It suffices to consider $\tau>0$, since the case $\tau<0$ follows by complex conjugation upon replacing $\xi'$ by $-\xi'$.
    Lemma~\ref{lemma: quadrant_CHG_c} gives $\Im(d_1^+d_0^-)>0$, and hence $\Im q_1\ge\tau\langle\xi'\rangle^2$.
    If $\langle\tau\rangle\ge \langle\xi'\rangle^2$, this already yields \eqref{lemma: dynamic_CHG_reduction_c_e1}.
    Suppose therefore that $\langle\tau\rangle<\langle\xi'\rangle^2$.
    In that regime Corollary~\ref{cor: CHG_c_factor_ellipticity} gives
    $|d_1^+d_0^-|\lesssim \langle\xi'\rangle^3$, and consequently
    \begin{align*}
        |q_1| \ge (1+|\xi'|^2)|\xi'|^2-|d_1^+d_0^-|
        \ge \langle\xi'\rangle^2(\langle\xi'\rangle^2-1)-C\langle\xi'\rangle^3.
    \end{align*}
    If $R$ is sufficiently large, this yields $|q_1(\tau,\xi')|\gtrsim \langle\xi'\rangle^4$, which proves \eqref{lemma: dynamic_CHG_reduction_c_e1} in the remaining regime $\langle\tau\rangle<\langle\xi'\rangle^2$.

    \medskip

    Armed with \eqref{lemma: dynamic_CHG_reduction_c_e1}, we now address the full $d=\det\mathfrak N$ on the set $|\tau|+|\xi'|\ge R$.
    Corollary~\ref{cor: CHG_c_factor_ellipticity} gives $|d_1^+|\gtrsim \langle\tau\rangle^\frac12+\langle\xi'\rangle$ and
    $|d_0^+(1+d_0^-)|\lesssim (\langle\tau\rangle+\langle\xi'\rangle^2)\langle\xi'\rangle^2$.
    Since $d_1^+$ has negative real part by Lemma \ref{lemma: quadrant_CHG_c}, we also have $|1-d_1^+|\ge |d_1^+|\gtrsim \langle\tau\rangle^\frac12+\langle\xi'\rangle$.
    Therefore
    \begin{align*}
        |d| \ge |1-d_1^+|\,|q_1|-|d_0^+(1+d_0^-)| \gtrsim (\langle\tau\rangle^\frac32+\langle\xi'\rangle^3)\langle\xi'\rangle^2-C(\langle\tau\rangle+\langle\xi'\rangle^2)\langle\xi'\rangle^2 \gtrsim (\langle\tau\rangle^\frac32+\langle\xi'\rangle^3)\langle\xi'\rangle^2.
    \end{align*}
    This proves the required lower estimate outside a compact set.

    For $\lambda>0$, on the remaining compact subset $\{(\tau,\xi')\in \R\times\R^{n-1}:|\tau|+|\xi'|\le R, |\tau|\ge\lambda\}$, the lower estimate follows from continuity as soon as we have established that $d$ does not vanish for $(\tau,\xi')\in (\R\setminus\{0\})\times \R^{n-1}$.

    To this end fix $(\tau,\xi')\in (\R\setminus\{0\})\times \R^{n-1}$.
    We may assume $\tau>0$, since the case $\tau<0$ follows by complex conjugation upon replacing $\xi'$ by $-\xi'$.
    By Lemma~\ref{lemma: quadrant_CHG_c} we have
    \begin{align*}
        d_0^+=-(x+\ic y),
        \qquad
        d_1^+=-(u+\ic v),
        \qquad
        d_0^-=-(X+\ic Y),
    \end{align*}
    for some $x,y,u,v,X,Y>0$.
    Taking real and imaginary parts in \eqref{eq: dynamic_c_coeff_1} gives
    \begin{align} \label{eq: dynamic_c_real_imag_identity}
        u^2-v^2 &= 2|\xi'|^2+X+x,
        \qquad 2uv= \tau+Y+y.
    \end{align}
    Expanding the imaginary part of $d$ and using
    \eqref{eq: dynamic_c_real_imag_identity} yields
    \begin{align*}
        \Im d &= \tau X +2XY +Xv +2|\xi'|^2Y +Yu +|\xi'|^4v +|\xi'|^2\tau u +|\xi'|^2\tau +|\xi'|^2v +\tau u +\tau +y.
    \end{align*}
    Every term on the right-hand side is strictly positive, and thus $d(\tau,\xi')\ne 0$.
    This completes the proof.
\end{proof}
We now use Lemma~\ref{lemma: dynamic_CHG_reduction_c} to show that $\opB_{\mathrm{dyn}}|_{\mathcal K}$ is an isomorphism onto its image space.
We recall that
$\mathcal K:=\Tr_E(\ker \opA_{+}^+)$, $\mathcal H:=\mathcal E/\mathcal K$, and $\opC=q_{\mathcal K}$.
\begin{lemma}\label{lemma: dynamic_CHG_range_c}
    Assume that $|c|<2$ and let $\mathbb G_H$ be as in \eqref{lemma: heat_Neumann_shifted_e1}.
    Set
    \begin{align}
        \mathfrak L&:=
        \begin{pmatrix}
         1 & 0 \\ -\ell & 1
        \end{pmatrix},
        \quad \ell:=\frac{(\ic\tau+|\xi'|^2+1)\left(1-\ic c'\cdot\xi'+c_n\frac{d_0^+-|\xi'|^2}{d_1^+}\right)+(1-c_n)(d_0^--\ic\tau-|\xi'|^2)}{d_0^+(|\xi'|^2-d_0^+-d_1^+)},  \label{eq: dynamic_CHG_range_c_ell}\\
        \mathbb G_{\mathrm{dyn}}' &:=\set{g\in\mathbb G_{\mathrm{dyn}} \mid \op[\mathfrak L] g\in \mathbb G_{1,\mathrm{dyn}}\times \mathbb G_H},
        \quad \|g\|_{\mathbb G_{\mathrm{dyn}}'} :=\|g_1\|_{\mathbb G_{1,\mathrm{dyn}}} +\|g_2-\op[\ell]g_1\|_{\mathbb G_H}.\notag
    \end{align}
    Then $\opB_{\mathrm{dyn}}|_{\mathcal K}\in\call_{\mathrm{iso}}(\mathcal K,\mathbb G_{\mathrm{dyn}}')$.
    Moreover, $-o_0$ is a strong upper order function for $\ell$.
    For $c=0$, the order function $-o_\frac12-o_0$ is a strong upper order function for $\ell$, and $\mathbb G_{\mathrm{dyn}}'=\mathbb G_{1,\mathrm{dyn}}\times\mathbb G_{H}$ with equivalent norms.
\end{lemma}

\begin{proof}
    We will show $\opB_{\mathrm{dyn}}|_{\mathcal K} =\op[\mathfrak M_{\mathrm{dyn}}]\opB|_{\mathcal K}$, where $\op[\mathfrak M_{\mathrm{dyn}}] \in \call_{\mathrm{iso}}(\mathbb G, \mathbb G_{\mathrm{dyn}}')$.
    Since $\opB|_{\mathcal K}\in\call_{\mathrm{iso}}(\mathcal K,\mathbb G)$ by Lemma~\ref{lemma: CHG_c_kernel_surjective}, this is enough to conclude.
    Let $u\in \mathbb E$ be such that $\opA_+^+u=0$.
    With $\mathfrak R$ from \eqref{lemma: CHG_c_u2_boundary_e1} and $\mathfrak J:=\diag(\ic\tau,1)$, Lemma \ref{lemma: CHG_c_kernel_surjective} gives $\op[\mathfrak R]\Tr^{(2)}u_2=\op[\mathfrak J]\opB \Tr_{\mathbb E}u$.
    By Lemma~\ref{lemma: dynamic_CHG_reduction_c}, we have $\op[\mathfrak T]\opB_{\mathrm{dyn}}\Tr_\mathbb E u=\op[\mathfrak N]\Tr^{(2)}u_2$.
    Together, we obtain $\opB_{\mathrm{dyn}}\Tr_{\mathbb E}u = \op[\mathfrak M_{\mathrm{dyn}}] \opB\Tr_{\mathbb E}u$ with $\mathfrak M_{\mathrm{dyn}} := \mathfrak T^{-1}\mathfrak N\mathfrak R^{-1}\mathfrak J$.
    It remains to show the claimed mapping property of $\mathfrak M_{\mathrm{dyn}}$.

    \medskip

    Since $\det\mathfrak T=\det\mathfrak J=\ic\tau$, $\det\mathfrak R=d_0^+d_1^+$, and $\det\mathfrak N=d$, we have $\det\mathfrak M_{\mathrm{dyn}}=\frac{d}{d_0^+d_1^+}$.
    Hence Corollary~\ref{cor: CHG_c_factor_ellipticity} and the strong $\mathcal N$-ellipticity of $d$ show that $\det\mathfrak M_{\mathrm{dyn}}$ is strongly $\mathcal N$-elliptic with order function $o_{\frac12}+o_0$.
    Direct multiplication gives
    \begin{align*}
        (\mathfrak M_{\mathrm{dyn}})_{11} =\frac{\ic\tau}{d_1^+}, \qquad (\mathfrak M_{\mathrm{dyn}})_{12} =\frac{|\xi'|^2-d_0^+-d_1^+}{d_1^+},
    \end{align*}
    while using also \eqref{eq: dynamic_c_coeff_1}, \eqref{eq: dynamic_c_coeff_2}, and \eqref{eq: dynamic_c_coeff_3} gives
    \begin{align*}
        (\mathfrak M_{\mathrm{dyn}})_{21} &= \frac{q_1-d_0^+(\ic\tau+|\xi'|^2+1)}{d_0^+d_1^+},\\
        (\mathfrak M_{\mathrm{dyn}})_{22} &= \frac{\ic\tau+|\xi'|^2+1}{d_0^+d_1^+} \left( 1-\ic c'\cdot\xi' +c_n\frac{d_0^+-|\xi'|^2}{d_1^+} \right) + (1-c_n)\frac{d_0^--\ic\tau-|\xi'|^2}{d_0^+d_1^+}.
    \end{align*}
    Moreover, $\ell=(\mathfrak M_{\mathrm{dyn}})_{22}/(\mathfrak M_{\mathrm{dyn}})_{12}$.
    By Lemma~\ref{lemma: quadrant_CHG_c}, for $\tau>0$ the symbols $d_0^+$ and $d_1^+$ lie in the open third quadrant and thus $||\xi'|^2-d_0^+-d_1^+|\gtrsim |\xi'^2|+|d_0^+|+|d_1^+|$.
    The same conclusion follows for $\tau<0$ by complex conjugation upon replacing $\xi'$ by $-\xi'$.
    Hence $|\xi'|^2-d_0^+-d_1^+$ is strongly $\mathcal N$-elliptic with order function $o_{\frac12}+o_0$.
    Consequently, $(\mathfrak M_{\mathrm{dyn}})_{12}$ is strongly $\mathcal N$-elliptic with order function $o_0$.
    The displayed formula for $(\mathfrak M_{\mathrm{dyn}})_{22}$ and Corollary~\ref{cor: CHG_c_factor_ellipticity} show that $0$ is a strong upper order function for $(\mathfrak M_{\mathrm{dyn}})_{22}$.
    Thus $-o_0$ is a strong upper order function for $\ell$, and in particular $\op[\ell]\in\call(\mathbb G_{1,\mathrm{dyn}},\mathbb G_{2,\mathrm{dyn}})$.
    For $c=0$, the formula for $(\mathfrak M_{\mathrm{dyn}})_{22}$ reduces to $\frac{1+d_0^-}{d_0^+d_1^+}$, which has upper order function $-o_\frac12$, so that $\ell$ has upper order function $-o_\frac12-o_0$.
    In particular $\op[\ell]\in \call(\mathbb G_{1,\mathrm{dyn}},\mathbb G_H)$ in this case, which shows that $\mathbb G_{\mathrm{dyn}}'$ and $\mathbb G_{1,\mathrm{dyn}}\times \mathbb G_H$ have equivalent norms.

    Define ${\mathfrak M}_{\ell}:=\mathfrak L\mathfrak M_{\mathrm{dyn}}$.
    Then $({\mathfrak M}_{\ell})_{22}=0$ and
    $({\mathfrak M}_{\ell})_{21} =-\det\mathfrak M_{\mathrm{dyn}}/(\mathfrak M_{\mathrm{dyn}})_{12}$, which is strongly $\mathcal N$-elliptic with order function $o_{\frac12}$.
    We choose the column and row orders associated with $\mathbb G$ and $\mathbb G_{1,\mathrm{dyn}}\times\mathbb G_H$ respectively, i.e.,
    \begin{align*}
        t_{1,0} &= 2o_{\frac12}, \qquad & t_{2,0} &= o_{0}, \qquad & s_{1,0} &= 0, \qquad & s_{2,0} &= o_\frac12,\\
        t_{1,1} &= o_{\frac12}, & t_{2,1} &= 0, & s_{1,1} &= -o_0, & s_{2,1} &= 0.
    \end{align*}
    The strong symbol estimates above, Corollary~\ref{cor: CHG_c_factor_ellipticity}, and Lemma~\ref{lemma: dynamic_CHG_reduction_c} give the required upper orders for all entries of ${\mathfrak M}_{\ell}$.
    Moreover, $\det{\mathfrak M}_{\ell}=\det\mathfrak M_{\mathrm{dyn}}$ and $\sum_{j=1}^2t_{j,\mathfrak k}-\sum_{i=1}^2s_{i,\mathfrak k}=o_{\frac12}+o_0$ for $\mathfrak k\in\{0,1\}$.
    Theorem~\ref{thm: main_result} therefore gives $\op[{\mathfrak M}_{\ell}]\in\call_{\mathrm{iso}}(\mathbb G,\mathbb G_{1,\mathrm{dyn}}\times\mathbb G_H)$.
    By the definition of $\mathbb G_{\mathrm{dyn}}'$ it holds $\op[\mathfrak L]\in \call_{\mathrm{iso}}(\mathbb G_{\mathrm{dyn}}',\mathbb G_{1,\mathrm{dyn}}\times\mathbb G_H)$.
    Hence $\op[\mathfrak M_{\mathrm{dyn}}]\in\call_{\mathrm{iso}}(\mathbb G,\mathbb G_{\mathrm{dyn}}')$, which proves the assertion.
\end{proof}
Lemma \ref{lemma: dynamic_CHG_range_c} shows that $\opB_{\mathrm{dyn}}$ fulfills the abstract complementing boundary condition with respect to $\opA_+^+$.
Indeed, if $e\in\cale$ satisfies $(\opB_{\mathrm{dyn}},\opC)e=0$, then $e\in\calk$ since $\opC=q_\calk$, and hence $e=0$ by the injectivity of $\opB_{\mathrm{dyn}}|_{\calk}$ from Lemma \ref{lemma: dynamic_CHG_range_c}.
Thus $(\opB_{\mathrm{dyn}},\opC)$ is injective.

Unlike in the static Cahn--Hilliard--Gurtin problem, however, Lemma \ref{lemma: dynamic_CHG_range_c} does not yet identify the range of $(\opB_{\mathrm{dyn}},\opC)$.
Indeed, the second component $\eta$ of an element $(\opB_{\mathrm{dyn}},\opC)e$ in the range fixes the class of the trace in $\calh=\cale/\calk$ through $\opC e=q_{\calk}e=\eta$, while the possible first components $\opB_{\mathrm{dyn}}e$ are determined by this class up to $\opB_{\mathrm{dyn}}(\calk)$, which is a genuine restriction since $\opB_{\mathrm{dyn}}(\calk)=\mathbb G_{\mathrm{dyn}}'\subsetneq \mathbb G_{\mathrm{dyn}}$.
More precisely, we claim
\begin{align}\label{eq: CHG_c_range_space}
    (\opB_{\mathrm{dyn}},\opC)\mathcal E
    =\{(g,\eta)\in\mathbb G_{\mathrm{dyn}}\times\mathcal H
    \mid g-\opB_{\mathrm{dyn}}e_0\in \mathbb G_{\mathrm{dyn}}'
    \text{ for some }e_0\in q_{\calk}^{-1}(\eta)\}.
\end{align}
The condition is independent of the choice of $e_0$.
Indeed, if $e_0,\widetilde e_0\in q_{\mathcal K}^{-1}(\eta)$, then $e_0-\widetilde e_0\in\mathcal K$, and hence
$\opB_{\mathrm{dyn}}(e_0-\widetilde e_0)\in\mathbb G_{\mathrm{dyn}}'$ by Lemma~\ref{lemma: dynamic_CHG_range_c}.
Thus, the inclusion $\subseteq$ is trivial, since for $(g,\eta)=(\opB_{\mathrm{dyn}}e,q_{\mathcal K}e)$ we may choose $e_0=e$.
Conversely, let $(g,\eta)$ and $e_0\in q_\calk^{-1}(\eta)$ satisfy the condition in \eqref{eq: CHG_c_range_space}.
Then $g-\opB_{\mathrm{dyn}}e_0\in\mathbb G_{\mathrm{dyn}}'$.
By Lemma \ref{lemma: dynamic_CHG_range_c} we may define $k:=(\opB_{\mathrm{dyn}}|_{\mathcal K})^{-1}(g-\opB_{\mathrm{dyn}}e_0)\in\mathcal K$.
For $e:=e_0+k$ we then have $(\opB_{\mathrm{dyn}}e,q_{\mathcal K}e)=(g,\eta)$, which proves the converse inclusion.

\medskip

Consider now $\Gamma\in \call(\mathbb F,\mathbb G_{2,\mathrm{dyn}})$ defined via $\Gamma:=\Gamma_2-\op[\ell]\Gamma_1$ with
\begin{align}\label{eq: CHG_Gamma}
	\Gamma_1 f:=(1-c'\cdot\nabla'-\Delta_{\mathbb R^{n-1}})\Tr_0Q_2f-\Tr_0f_2,
\quad \Gamma_2 f:=\Tr_0f_2+(c'\cdot\nabla'-c_n)\Tr_0Q_2f,
\end{align}
where $Q=(Q_1,Q_2)\in\call(\mathbb F,\mathbb E)$ defined by $Qf:=S_{\mathrm{stat}}^{-1}(f,0)$ is the solution operator of Theorem \ref{thm: CHG_c_stat} for homogeneous boundary data.
We claim that
\begin{align}\label{eq: beta_c_interior_data}
    \opB_{\mathrm{dyn},2}\Tr_{\mathbb E}Q - \op[\ell]\opB_{\mathrm{dyn},1}\Tr_{\mathbb E}Q-\Gamma
  &= -(\Tr_2+\Tr_0) Q_1 -\op[\ell](\Tr_2Q_1+\Delta_{\mathbb R^{n-1}}\Tr_0Q_1)\in \call(\mathbb F,\mathbb G_H).
\end{align}
Let $f\in \mathbb F$ and write $v:=Qf$.
Then $(\partial_t-\Delta)v_1=v_2-c\cdot\nabla v_2-f_2$ and $\Tr_1v_1=\Tr_1v_2=0$, so that
\begin{align*}
    \opB_{\mathrm{dyn},2}\Tr_{\mathbb E}v&=-(\partial_t-\Delta_{\mathbb R^{n-1}}+1)\Tr_0v_1+(1-c_n)\Tr_0v_2\\
    &=-\Tr_2v_1 -\Tr_0v_1 +(c'\cdot\nabla'-c_n)\Tr_0v_2+\Tr_0f_2  =-\Tr_2v_1 -\Tr_0v_1+\Gamma_2 f.
\end{align*}
Similarly,
\begin{align*}
	\opB_{\mathrm{dyn},1}\Tr_{\mathbb E}v&= \partial_t\Tr_0v_1-\Delta_{\mathbb R^{n-1}}\Tr_0v_2
	=\Tr_2v_1+\Delta_{\mathbb R^{n-1}}\Tr_0v_1 + (1-c'\cdot\nabla'-\Delta_{\mathbb R^{n-1}})\Tr_0v_2-\Tr_0f_2\\
	&= \Tr_2v_1+\Delta_{\mathbb R^{n-1}}\Tr_0v_1 + \Gamma_1 f.
\end{align*}
Together, this gives \eqref{eq: beta_c_interior_data}.

\medskip

We now use \eqref{eq: CHG_c_range_space} and \eqref{eq: beta_c_interior_data} to give a simpler description of the resulting compatibility space.
Set
\begin{align}\label{eq: CHG_dyn_data_spc}
\begin{split}
    \mathbb D_{\mathrm{dyn}} &:=
    \set{
        (f,g)\in\mathbb F\times\mathbb G_{\mathrm{dyn}}
        \mid g_2-\op[\ell]g_1-\Gamma f\in \mathbb G'}, \qquad
        \mathbb G':=BH^{(\frac12-\frac1{2p},0),p}_\perp(G^{n-1}),\\
    \|f,g\|_{\mathbb D_{\mathrm{dyn}}} &:=
    \|f\|_{\mathbb F}+\|g\|_{\mathbb G_{\mathrm{dyn}}}
    +\|g_2-\op[\ell]g_1-\Gamma f\|_{\mathbb G'}.
\end{split}
\end{align}
For $c=0$, this space and its norm do not change up to equivalence if we replace $g_2-\op[\ell]g_1-\Gamma f$ by $g_2-\Tr_0 f_2$.
Indeed, in this case $\Gamma_2 f=\Tr_0 f_2$, and $\ell$ has upper order function $-o_\frac12-o_0$, so that $\op[\ell]\in \call(\mathbb G_{1,\mathrm{dyn}},\mathbb G')$ and $\op[\ell]\Gamma_1\in \call(\mathbb F,\mathbb G')$.

We claim that $\mathbb D_{\mathrm{dyn}}=\Comp$ both algebraically and topologically.
By \eqref{eq: CHG_c_range_space} and $\mathsf V(\opA_+^-)^{-1}=\mathsf V\opA_+^+Q =q_\calk\Tr_{\mathbb E}Q$, the condition $(g,\mathsf V(\opA_+^-)^{-1}f)\in(\opB_{\mathrm{dyn}},\opC)\mathcal E$
is equivalent to $g-\opB_{\mathrm{dyn}}\Tr_{\mathbb E}Qf\in\mathbb G_{\mathrm{dyn}}'$, and for $(f,g)\in\Comp$ it holds
\begin{align*}
    b(f,g) =\Tr_{\mathbb E}Qf +(\opB_{\mathrm{dyn}}|_{\calk})^{-1} (g-\opB_{\mathrm{dyn}}\Tr_{\mathbb E}Qf),
\end{align*}
where $b(f,g):=(\opB_{\mathrm{dyn}},\opC)^{-1} (g,\mathsf V(\opA_+^-)^{-1}f)$.
Moreover, the map
\begin{align*}
    R:\Comp\to\mathbb F\times\mathbb G_{\mathrm{dyn}}',
    \qquad
    R(f,g):=(f,g-\opB_{\mathrm{dyn}}\Tr_{\mathbb E}Qf),
\end{align*}
is bijective with inverse $R^{-1}(f,r)=(f,\opB_{\mathrm{dyn}}\Tr_{\mathbb E}Qf+r)$.
Since
\begin{align*}
    \Theta R^{-1}(f,r) = \left(f,\Tr_{\mathbb E}Qf +(\opB_{\mathrm{dyn}}|_{\calk})^{-1}r \right),
\end{align*}
where $\Theta(f,g):=(f,b(f,g))$, it follows $R^{-1}\in \call(\mathbb F\times\mathbb G_{\mathrm{dyn}}',\Comp)$ by the definition of the initial topology on $\Comp$.
Since both $\Comp$ and $\mathbb F\times\mathbb G_{\mathrm{dyn}}'$ are Banach spaces (the former since it is isomorphic to the Banach space $\mathbb E$ by Theorem \ref{thm: abstract_complementing_condition}), the open mapping theorem yields $R\in\call_{\mathrm{iso}}(\Comp,\mathbb F\times\mathbb G_{\mathrm{dyn}}')$.

By \eqref{eq: beta_c_interior_data}, the operator $C:=\opB_{\mathrm{dyn},2}\Tr_{\mathbb E}Q -\op[\ell]\opB_{\mathrm{dyn},1}\Tr_{\mathbb E}Q-\Gamma$ belongs to $\call(\mathbb F,\mathbb G_H)$.
Since $\op[\mathfrak L]\in\call_{\mathrm{iso}}(\mathbb G_{\mathrm{dyn}}',\mathbb G_{1,\mathrm{dyn}}\times\mathbb G_H)$, the map
\begin{align*}
    (f,g)\mapsto (f,g_1-\opB_{\mathrm{dyn},1}\Tr_{\mathbb E}Qf, g_2-\op[\ell]g_1-\Gamma f-Cf)
\end{align*}
is an isomorphism from $\Comp$ onto $\mathbb F\times\mathbb G_{1,\mathrm{dyn}}\times\mathbb G_H$.
Since $\opB_{\mathrm{dyn},1}\Tr_{\mathbb E}Q\in\call(\mathbb F,\mathbb G_{1,\mathrm{dyn}})$ and $C\in\call(\mathbb F,\mathbb G_H)$, a triangular change of variables shows that $(f,g)\mapsto(f,g_1,g_2-\op[\ell]g_1-\Gamma f)$ is an isomorphism from $\Comp$ onto $\mathbb F\times\mathbb G_{1,\mathrm{dyn}}\times\mathbb G_H$.
Since $\mathbb G_H=\mathbb G'\cap\mathbb G_{2,\mathrm{dyn}}$, $\Gamma\in\call(\mathbb F,\mathbb G_{2,\mathrm{dyn}})$, and $\op[\ell]\in\call(\mathbb G_{1,\mathrm{dyn}},\mathbb G_{2,\mathrm{dyn}})$, the norm $\|f\|_{\mathbb F}+\|g_1\|_{\mathbb G_{1,\mathrm{dyn}}} +\|g_2-\op[\ell]g_1-\Gamma f\|_{\mathbb G_H}$ is equivalent to the norm in \eqref{eq: CHG_dyn_data_spc}.
Thus $\mathbb D_{\mathrm{dyn}}=\Comp$ topologically.

\begin{theorem}\label{thm: CHG_c_dyn}
    Assume that $|c|<2$ and let the symbol $\ell$ be given by \eqref{eq: dynamic_CHG_range_c_ell}.
    Define $\mathbb E$, $\mathbb F$ as in \eqref{eq: CHG_spaces}, $\mathbb G_{\mathrm{dyn}}$ as in \eqref{eqn: CHG_dyn_bdry_spc}, and $\mathbb D_{\mathrm{dyn}}$ and $\mathbb G'$ as in \eqref{eq: CHG_dyn_data_spc}.
    Then $(\opL_+,\opB_{\mathrm{dyn}}\Tr_{\mathbb E})\in\call_{\mathrm{iso}}(\mathbb E,\mathbb D_{\mathrm{dyn}})$.
    In particular, for $f\in\mathbb F$ and $g\in\mathbb G_{\mathrm{dyn}}$, the dynamic Cahn--Hilliard--Gurtin system \eqref{eqn: CHG_dynamic} admits a unique solution $u\in\mathbb E$ if and only if $g_2-\op[\ell]g_1-\Gamma f\in\mathbb G'$, and, for $c=0$, if and only if $g_2-\Tr_0f_2\in \mathbb G'$.
    In this case it holds
    \begin{align*}
        \|u\|_{\mathbb E}\simeq\|f,g\|_{\mathbb D_{\mathrm{dyn}}}.
    \end{align*}
\end{theorem}

\begin{proof}
The operator $\opL=\opA^-\opA^+$ is admissible in the sense of Definition \ref{def: admissible}.
By Lemma~\ref{lemma: dynamic_CHG_range_c}, $\opB_{\mathrm{dyn}}$ satisfies the abstract complementing boundary condition with respect to $\opA_+^+$, and the preceding argument identifies the corresponding compatibility space topologically with $\mathbb D_{\mathrm{dyn}}$.
The assertion therefore follows from Theorem~\ref{thm: abstract_complementing_condition}.
\end{proof}
\paragraph{\bf Data availability} No data are available.
\paragraph{\bf Conflict of interest} The authors declare that there is no conflict of interest.


\end{document}